\documentclass[12pt,reqno]{amsart}
\usepackage{a4wide}
\usepackage{enumerate,comment,xparse}

\usepackage{amsmath, amsfonts, amssymb, xcolor, stmaryrd}

\usepackage{hyperref}
\usepackage{url, hypcap}
\hypersetup{colorlinks=true, citecolor=blue, linkcolor=blue}

\usepackage{color,tikz,ifthen}
\usetikzlibrary{calc,through,backgrounds,shapes}
\usetikzlibrary{arrows.meta}

\usepackage{todonotes}

\usepackage{comment}

\usepackage{cleveref}

\crefname{conjecture}{conjecture}{conjectures}
\Crefname{conjecture}{Conjecture}{Conjectures}
\crefname{subsection}{subsection}{subsections}
\Crefname{subsection}{Subsection}{Subsections}

\newcounter{saveeqn}
\newcommand{\alphaeqn}{\setcounter{saveeqn}{\value{equation}}%
\setcounter{equation}{0}%
\global\def\theequation{\mbox{\thesection.\arabic{saveeqn}\alph{equation}}}}
\newcommand{\reseteqn}{\setcounter{equation}{\value{saveeqn}}%
\global\def\theequation{\thesection.\arabic{equation}}}

\usepackage{color}

\newtheorem{theorem}{Theorem}[section]
\newtheorem{proposition}[theorem]{Proposition}
\newtheorem{corollary}[theorem]{Corollary}
\newtheorem{lemma}[theorem]{Lemma}
\newtheorem{conjecture}[theorem]{Conjecture}

\theoremstyle{definition}
\newtheorem{definition}[theorem]{Definition}
\newtheorem{example}[theorem]{Example}
\newtheorem{question}[theorem]{Question}
\newtheorem{openproblem}[theorem]{Open Problem}
\newtheorem{remark}[theorem]{Remark}  
\newtheorem{remarks}[theorem]{Remarks}  

\numberwithin{equation}{section}

\catcode`\@=11
\@namedef{subjclassname@2010}{%
  \textup{2010} Mathematics Subject Classification}
\catcode`\@=12

\def\({\left(}
\def\){\right)}
\def\coef#1{\left\langle#1\right\rangle}

\def\rk{\operatorname{rk}}
\def\width{\operatorname{wd}}

\def\nequiv{\hbox{${}\not\equiv{}$}}
\def\R{\mathbb R}

\def\N{\mathbb N}

\def\om{\omega}
\def\ph{\varphi}
\def\si{\sigma}
\def\ep{\varepsilon}
\def\al{\alpha}
\def\be{\beta}
\def\ga{\gamma}
\def\la{\lambda}

\def\th{\vartheta}
\def\ka{\kappa}
\def\ta{\tau}

\def\rot{R}
\def\rotA{\kappa_{\scriptscriptstyle A}}
\def\rotB{\kappa_{\scriptscriptstyle B}}
\def\rotD{\kappa_{\scriptscriptstyle D}}
\def\Na{\bigtriangledown}
\def\Nam#1#2{\Na_{#1}}
\def\fl#1{\left\lfloor#1\right\rfloor}

\newcommand{\sq}[1]{{\rm #1}} 
\renewcommand{\r}{\sq{r}} 
\newcommand{\s}{\sq{s}} 
\renewcommand{\t}{\sq{t}} 

\def\ifix{i_{\operatorname{fix}}}
\def\rfix{\r_{\operatorname{fix}}}

\def\Cat{\operatorname{Cat}} 
\def\mCat{\operatorname{Cat}^{(m)}} 
\NewDocumentCommand\mCatplus{O{m}}{{\operatorname{Cat}^{(#1)}_+}}

\def\Krewtilde{{\sf K}} 
\def\Krew{\widetilde\Krewtilde} 

\def\Krewplus{\Krew_+} 
\def\Krewplustilde{\Krewtilde_+} 
\def\Krewplusbar{\Krewtilde_\circ} 

\def\Pan{{\sf P}} 
\def\Panplus{{\sf P}_+} 
\def\antichain{\mathcal{A}} 
\def\wo{w_\circ} 
\def\rrr{\rho}
\def\rRr{\nu}
\def\lenS{\ell_\reflS} 
\def\lenR{\ell_\reflR} 
\def\reflR{\mathcal{R}} 
\def\reflS{\mathcal{S}} 
\def\one{{\varepsilon}} 
\def\order{{\operatorname{ord}}} 
\def\leqR{\le_\reflR} 
\def\geqR{\ge_\reflR} 
\def\lec{<_c} 
\def\w{{\rm w}} 
\def\c{{\rm c}} 
\def\wwo{{\rm \wo}} 
\def\NN{N\kern-1ptN(W)} 
\def\NNplus{N\kern-1ptN_+(W)} 
\def\mNN{N\kern-1ptN^{(m)}(W)} 
\def\mNNplus{N\kern-1ptN_+^{(m)}(W)} 

\def\NCsymb{N\kern-1ptC}

\NewDocumentCommand\NC{O{W}O{c}}{\NCsymb(#1,#2)} 
\NewDocumentCommand\NCPlus{O{W}O{c}}{\NCsymb_+(#1,#2)} 

\NewDocumentCommand\mNC{O{W}O{m}O{c}}{\NCsymb^{(#2)}(#1,#3)}
\NewDocumentCommand\mNCPlus{O{W}O{m}O{c}}{\NCsymb^{(#2)}_+(#1,#3)}
\NewDocumentCommand\mNCdelta{O{W}O{m}O{c}}{\NCsymb^{(#2)}_\Delta(#1,#3)}
\NewDocumentCommand\mNCnabla{O{W}O{m}O{c}}{\NCsymb^{(#2)}_\nabla(#1,#3)}

\NewDocumentCommand\NCA{O{N}}{\widetilde{\NCsymb}(#1)}
\NewDocumentCommand\NCAPlus{O{N}}{\widetilde{\NCsymb}_+(#1)}
\NewDocumentCommand\mNCA{O{n-1}O{m}}{\widetilde{\NCsymb}^{(#2)}(A_{#1})}
\NewDocumentCommand\mNCAPlus{O{n-1}O{m}}{\widetilde{\NCsymb}_+^{(#2)}(A_{#1})}

\NewDocumentCommand\NCB{O{N}}{\widetilde{\NCsymb}(\pm #1)}
\NewDocumentCommand\NCBPlus{O{N}}{\widetilde{\NCsymb}(\pm #1)}
\NewDocumentCommand\mNCB{O{n}O{m}}{\widetilde{\NCsymb}^{(#2)}(B_{#1})}
\NewDocumentCommand\mNCBPlus{O{n}O{m}}{\widetilde{\NCsymb}_+^{(#2)}(B_{#1})}

\NewDocumentCommand\mNCD{O{n}O{m}}{\widetilde{\NCsymb}^{(#2)}(D_{#1})}
\NewDocumentCommand\mNCDPlus{O{n}O{m}}{\widetilde{\NCsymb}^{(#2)}_+(D_{#1})}

\def\rootPoset{\Phi^+} 
\NewDocumentCommand\Rword{O{c}}{\mathcal{\reflR}_{#1}} 
\def\invc{\mathbf{inv}_c} 

\def\asc{\operatorname{asc}}
\def\des{\operatorname{des}}

\renewcommand{\th}{\textsuperscript{th}}

\newcommand{\defn}[1]{\emph{\color{red} #1}} 

\newcounter{intege}

\newcommand{\polygon}[5]{ 
  \foreach \t in {1,...,#3} {
    \coordinate (#2\t) at ($#1+(90-\t*360/#3:#4)$);
  }
  \draw[thin,black,fill=white,opacity=0.3,densely dashed] #1 circle (#4);
  \setcounter{intege}{1}
  \pgfmathsetcounter{intege}{1}
  \foreach \object in {#5}{
    \ifthenelse{\not\equal{\object}{}}{
      \filldraw[black] ($#1+($(90-\theintege*360/#3:#4)$)$) circle(2pt);
    }{
    }
    \node[inner sep=0pt] at ($#1+($1.12*(90-\theintege*360/#3:#4)$)$)
         {$\scriptstyle\object$};
    \pgfmathsetcounter{intege}{\theintege+1}
    \setcounter{intege}{\theintege}
  }
}

\newcommand{\polygonlabel}[5]{ 
  \foreach \t in {1,...,#3} {
    \coordinate (#2\t) at ($#1+(90-\t*360/#3:#4)$);
  }
  \draw[thin,black,fill=white,opacity=0.3,densely dashed] #1 circle (#4);
  \setcounter{intege}{1}
  \pgfmathsetcounter{intege}{1}
  \foreach \object in {#5}{
    \ifthenelse{\not\equal{\object}{}}{
      \pgfmathsetmacro\angle{int(90-\theintege*360/#3)}
      \pgfmathsetmacro\opangle{int(270-\theintege*360/#3)}
      \node[circle,draw,inner sep=0pt,fill=black,outer sep=0pt,minimum
size=4pt] at ($#1+($(\angle:#4)$)$) (A) {};
      \node[ellipse,inner sep=2pt,anchor=\opangle,outer sep=2pt] at
(A.\angle) {$\scriptstyle\object$};
    }{
    }
    \pgfmathsetcounter{intege}{\theintege+1}
    \setcounter{intege}{\theintege}
  }
}

\newcommand{\polygonnew}[6]{ 
  \foreach \t in {1,...,#3} {
    \coordinate (#2\t) at ($#1+(90-\t*360/#3:#4)$);
  }
  \draw[thin,black,fill=white,opacity=0.3,densely dashed] #1 circle (#4);
  \setcounter{intege}{1}
  \pgfmathsetcounter{intege}{1}
  \foreach \object in {#5}{
    \ifthenelse{\not\equal{\object}{}}{
      \filldraw[black] ($#1+($(90-\theintege*360/#3:#4)$)$) circle(2pt);
    }{
    }
    \node[inner sep=0pt] at ($#1+($#6*(90-\theintege*360/#3:#4)$)$) {$\scriptstyle\object$};
    \pgfmathsetcounter{intege}{\theintege+1}
    \setcounter{intege}{\theintege}
  }
}

\newcommand{\polygoninner}[5]{ 
  \foreach \t in {1,...,#3} {
    \coordinate (#2\t) at ($#1+(90-\t*360/#3:#4)$);
  }
  \draw[thin,black,fill=white,opacity=0.3,densely dashed] #1 circle
(#4);
  \setcounter{intege}{1}
  \pgfmathsetcounter{intege}{1}
  \foreach \object in {#5}{
    \ifthenelse{\not\equal{\object}{}}{
      \filldraw[black] ($#1+($(90-\theintege*360/#3:#4)$)$)
circle(2pt);
    }{
circle(2pt);
    }
    \node[inner sep=0pt] at ($#1+($0.7*(90-\theintege*360/#3:#4)$)$)
{$\scriptscriptstyle\object$};
    \pgfmathsetcounter{intege}{\theintege+1}
    \setcounter{intege}{\theintege}
  }
}

\begin{document}

\title[Positive $m$-divisible non-crossing partitions]{Positive
  $m$-divisible non-crossing partitions\\ and their Kreweras maps}

\author[C.~Krattenthaler]{Christian Krattenthaler$^*$}
\address[C.~Krattenthaler]{Fakult\"at f\"ur Mathematik, Universit\"at Wien,
Oskar-Morgenstern-\break Platz~1, A-1090 Vienna, Austria}
\thanks{$^*$Research partially supported by the Austrian Science Foundation FWF, grants Z130-N13 and S50-N15,
the latter in the framework of the Special Research Program
``Algorithmic and Enumerative Combinatorics"}

\author[C.~Stump]{Christian Stump$^\dagger$}
\address[C.~Stump]{Fakult\"at f\"ur Mathematik,
Ruhr-Universit\"at Bochum, Germany}
\email{christian.stump@rub.de}
\thanks{$^\dagger$Supported by DFG grants STU 563/2 ``Coxeter-Catalan combinatorics'' and STU 563/\{4--6\}-1 ``Noncrossing phenomena in Algebra and Geometry''.}


\begin{abstract}
  We study positive $m$-divisible non-crossing partitions and their positive Kreweras maps.
  In classical types, we describe their combinatorial realisations as certain non-crossing set partitions.
  We also realise these positive Kreweras maps as pseudo-rotations on a circle, respectively on an annulus.
  We enumerate positive $m$-divisible non-crossing partitions in classical types that are invariant under powers of the
  positive Kreweras maps with respect to several parameters.
  In order to cope with the exceptional types, we develop a different combinatorial model in general type describing positive $m$-divisible non-crossing partitions that are invariant under powers of the positive Kreweras maps.
  We finally show that altogether these results establish several cyclic sieving phenomena.
\end{abstract}

\maketitle

\setcounter{tocdepth}{1}
\tableofcontents

\section{Introduction}
\label{sec:0} 

Non-crossing set partitions have been introduced into combinatorial
mathematics by Kreweras~\cite{KrewAA} in 1972.
Biane~\cite{BianAA} presented an algebraic realisation in terms
of certain factorisations of a long cycle in the symmetric
group. These classical non-crossing partitions were lifted
to all finite Coxeter groups and a given Coxeter element by Bessis~\cite{BesDAA},
and independently by Brady and Watt~\cite{brady2001partial,BRWaAA}.
The earlier combinatorial construction
of non-crossing set partitions for type~$B$ due to Reiner
\cite{ReivAG} turned out to be equivalent to the
type~$B$ situation of the construction of Bessis and of Brady and Watt.
The (algebraic) type~$D$ non-crossing partitions of Bessis and of Brady and Watt
were translated into (combinatorial)
non-crossing set partitions for type~$D$ by Reiner and Athanasiadis in~\cite{AR2004}.

Inspired by a refinement of Kreweras' non-crossing set partitions
due to Edelman\break \cite{EdelAA}, Armstrong~\cite{Arm2006} introduced
a positive integer parameter~$m$ and defined \emph{$m$-divisible
non-crossing partitions} for any finite irreducible Coxeter group~$W$ and any Coxeter element $c \in W$,
together with a natural partial order.
Armstrong showed that, in type~$A_{n-1}$
these objects can be realised as Edelman's refinement, namely 
as non-crossing set partitions of
$\{1,2,\dots,mn\}$ where all block sizes are divisible by~$m$.
In type~$B_n$ they can be realised
as non-crossing set partitions of
$\{1,2,\dots,mn,-1,-2,\dots,-mn\}$ which remain invariant
under the simultaneous replacement $i\mapsto-i$,
and where again all block sizes are divisible by~$m$~\cite[Thm.~4.5.6]{Arm2006}.
The first author and M\"uller~\cite{KrMuAB} (with a missing piece
reported in~\cite{JKimAA})
presented a realisation in type~$D_n$
as non-crossing set partitions of 
$\{1,2,\dots,mn,-1,-2,\dots,-mn\}$ which are arranged on
an annulus subject to similar, but more
technical conditions. In all these realisations as set partitions
(see \Cref{fig:0} for examples),
the above mentioned partial order corresponds to 
the classical refinement order on set partitions.

\begin{figure}
  \centering
  \begin{tikzpicture}[scale=.8]
    \polygon{(0,0)}{obj}{36}{2.5}
      {1,2,3,4,5,6,7,8,9,10,11,12,13,14,15,16,17,18,19,20,21,22,23,24,25,26,27,28,29,30,31,32,33,34,35,36}
     \draw[fill=black,fill opacity=0.1] (obj1) to[bend right=30] (obj5) to[bend right=30] (obj12) to[bend right=30] (obj13) to[bend right=30] (obj14) to[bend left=10] (obj36) to[bend right=30] (obj1);
     \draw[fill=black,fill opacity=0.1] (obj18) to[bend right=30] (obj22) to[bend right=30] (obj29) to[bend right=30] (obj30) to[bend right=30] (obj31) to[bend right=30] (obj35) to[bend left=10] (obj18);
     \draw[fill=black,fill opacity=0.1] (obj2) to[bend right=30] (obj3) to[bend right=30] (obj4) to[bend left=30] (obj2);
     \draw[fill=black,fill opacity=0.1] (obj6) to[bend right=30] (obj7) to[bend right=30] (obj11) to[bend left=30] (obj6);
     \draw[fill=black,fill opacity=0.1] (obj8) to[bend right=30] (obj9) to[bend right=30] (obj10) to[bend left=30] (obj8);
     \draw[fill=black,fill opacity=0.1] (obj15) to[bend right=30] (obj16) to[bend right=30] (obj17) to[bend left=30] (obj15);
     \draw[fill=black,fill opacity=0.1] (obj19) to[bend right=30] (obj20) to[bend right=30] (obj21) to[bend left=30] (obj19);
     \draw[fill=black,fill opacity=0.1] (obj23) to[bend right=30] (obj24) to[bend right=30] (obj28) to[bend left=30] (obj23);
     \draw[fill=black,fill opacity=0.1] (obj25) to[bend right=30] (obj26) to[bend right=30] (obj27) to[bend left=30] (obj25);
     \draw[fill=black,fill opacity=0.1] (obj32) to[bend right=30] (obj33) to[bend right=30] (obj34) to[bend left=30] (obj32);
    \end{tikzpicture}
  \quad
  \begin{tikzpicture}[scale=.8]
    \polygon{(-4,0)}{obj}{24}{2.5}
      {1,2,3,4,5,6,7,8,9,10,11,12,\overline{1},\overline{2},\overline{3},\overline{4},\overline{5},\overline{6},\overline{7},\overline{8},\overline{9},\overline{10},\overline{11},\overline{12}}


     \draw[fill=black,fill opacity=0.1] (obj24) to[bend right=30]
(obj1) to[bend right=30] (obj2) to[bend left=30] (obj24);
     \draw[fill=black,fill opacity=0.1] (obj3) to[bend right=30]
(obj4) to[bend right=30] (obj5) to[bend left=30] (obj3);
     \draw[fill=black,fill opacity=0.1] (obj6) to[bend right=30]
(obj10) to[bend right=30] (obj11) to[bend right=30] (obj18) to[bend
right=30] (obj22) to[bend right=30] (obj23) to[bend right=30] (obj6);
     \draw[fill=black,fill opacity=0.1] (obj7) to[bend right=30]
(obj8) to[bend right=30] (obj9) to[bend left=30] (obj7);
     \draw[fill=black,fill opacity=0.1] (obj12) to[bend right=30]
(obj13) to[bend right=30] (obj14) to[bend left=30] (obj12);
     \draw[fill=black,fill opacity=0.1] (obj15) to[bend right=30]
(obj16) to[bend right=30] (obj17) to[bend left=30] (obj15);
     \draw[fill=black,fill opacity=0.1] (obj19) to[bend right=30]
(obj20) to[bend right=30] (obj21) to[bend left=30] (obj19);
  \end{tikzpicture}
  \quad
  \begin{tikzpicture}[scale=.7]
    \polygonnew{(-4,0)}{objleftout}{30}{3}
      {\overline{15},1,2,3,4,5,6,7,8,9,10,11,12,13,14,15,\overline{1},\overline{2},\overline{3},\overline{4},\overline{5},\overline{6},\overline{7},\overline{8},\overline{9},\overline{10},\overline{11},\overline{12},\overline{13},\overline{14}}{1.1}
    \polygonnew{(-4,0)}{objleftin}{6}{0.75}
      {\overline{17},\overline{16},18,17,16,\overline{18}}{0.65}

     \draw[fill=black,fill opacity=0.1] (objleftin6) to[bend right=5] (objleftout29) to[bend right=30] (objleftout30) to[bend right=5] (objleftin6);
     \draw[fill=black,fill opacity=0.1] (objleftin3) to[bend right=5] (objleftout14) to[bend right=30] (objleftout15) to[bend right=5] (objleftin3);

     \draw[fill=black,fill opacity=0.1] (objleftout2) to[bend right=5] (objleftin1) to[bend left=5] (objleftout1) to[bend right=10] (objleftout2);
     \draw[fill=black,fill opacity=0.1] (objleftout17) to[bend right=5] (objleftin4) to[bend left=5] (objleftout16) to[bend right=10] (objleftout17);

     \draw[fill=black,fill opacity=0.1] (objleftout3) to[bend right=30] (objleftout4) to[bend left=10] (objleftin2) to[bend right=5] (objleftout3);
     \draw[fill=black,fill opacity=0.1] (objleftout18) to[bend right=30] (objleftout19) to[bend left=10] (objleftin5) to[bend right=5] (objleftout18);

     \draw[fill=black,fill opacity=0.1] (objleftout5) to[bend right=30] (objleftout6) to[bend right=30] (objleftout7) to[bend right=30] (objleftout8) to[bend right=30] (objleftout9) to[bend right=30] (objleftout13) to[bend left=5] (objleftout5);
     \draw[fill=black,fill opacity=0.1] (objleftout20) to[bend right=30] (objleftout21) to[bend right=30] (objleftout22) to[bend right=30] (objleftout23) to[bend right=30] (objleftout24) to[bend right=30] (objleftout28) to[bend left=5] (objleftout20);


     \draw[fill=black,fill opacity=0.1] (objleftout10) to[bend right=30] (objleftout11) to[bend right=30] (objleftout12) to[bend left=30] (objleftout10);
     \draw[fill=black,fill opacity=0.1] (objleftout25) to[bend right=30] (objleftout26) to[bend right=30] (objleftout27) to[bend left=30] (objleftout25);
    \end{tikzpicture}
    \caption{Examples of positive $3$-divisible non-crossing set partitions of the respective types $A_{11}$, $B_4$ and $D_6$.}
\label{fig:0}
\end{figure}
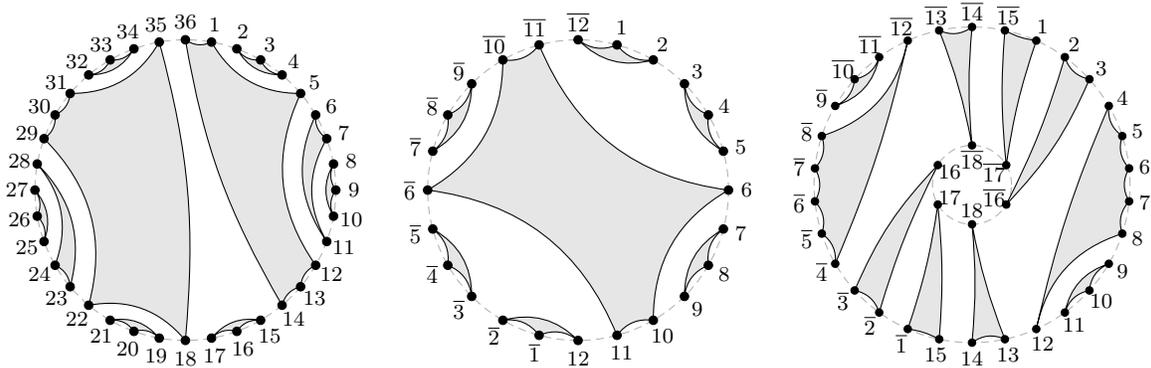

As observed for example in \cite[Prop.~9]{Cha2004},
one of the remarkable facts about these generalised non-crossing
partitions is that their total number is given by the
uniform formula 
\begin{equation}
  \mCat(W) :=
  \prod_{i=1}^n{\frac{mh+d_i}{d_i}}, 
\label{eq:mdivnc}
\end{equation}
where~$n$ denotes the \emph{rank} of~$W$, $d_1 \leq \ldots \leq d_n$ denote the \emph{degrees} of~$W$, and where~$h = d_n$ is the \emph{Coxeter number}; see \Cref{sec:mdivdef} for precise definitions.
This number is called
\defn{Fu\ss--Catalan number of type~$W$}.
These numbers also count various other combinatorial objects associated with~$W$.
Among these are the generalised non-nesting partitions of Athanasiadis~\cite{Ath2004}, the generalised clusters of Fomin and Reading~\cite{FR2005}, and the generalised Coxeter-sortable elements of the second author, Thomas and Williams~\cite{STW2015}.
We refer to the latter article for a detailed discussion of these various
Fu\ss-analogues counted by the Fu\ss--Catalan numbers.

\medskip

In the present article, we study a subset of the $m$-divisible
non-crossing partitions associated with~$W$, namely
\defn{positive $m$-divisible non-crossing partitions of type~$W$}; see \Cref{def:NCpos}.
For $m=1$, these have been introduced
by Reading \cite[Cor.~9.2 plus the following paragraph]{Rea2007}
(implicitly) and by Sim\~oes \cite[Thm.~1.1]{Sim2012} 
(explicitly; only for crystallographic types). 
For generic~$m$, positive
$m$-divisible non-crossing partitions appear in~\cite[Cor.~7.4 plus the following paragraph]{BRT2012}
(again, only for crystallographic types), and in~\cite[Ch.~7]{STW2015}.

The term ``positive" here stems from the observation that the bijection between generalised non-crossing partitions and generalised clusters in~\cite{STW2015} restricts to a bijection between positive non-crossing partitions and positive clusters for the generalised cluster complex of~\cite{FR2005}, where a cluster is {\em positive} if it consists entirely of (coloured) positive roots.

\medskip

The total number of positive $m$-divisible non-crossing partitions is given by
the \defn{positive Fu\ss--Catalan number of type~$W$} (see~\cite[Cor.~7.4]{BRT2012}
for crystallographic types and \cite[Thm.~7.2.1]{STW2015} for all types) defined by
\begin{align}
\mCatplus(W) :=
\prod_{i=1}^n{\frac{mh+d_i-2}{d_i}}.
\label{eq:mdivncplus}
\end{align}
This expression can be found in work by Athanasiadis~\cite{Ath2004}
where he shows that it counts certain bounded regions in the extended
Shi arrangement.
It was then shown in \cite[Cor.~12.2]{FR2005} that positive 
generalised clusters are also counted by this formula, and both 
\cite[Cor.~7.4]{BRT2012} and \cite[Thm.~7.2.1]{STW2015} use the 
earlier mentioned bijective correspondence to derive the same counting 
formula for positive $m$-divisible non-crossing partitions.

\bigskip

Our first focus is on describing positive $m$-divisible non-crossing
partitions for the classical types $A_{n-1}$, $B_n$, and $D_n$ as
subsets of the above mentioned realisations as set
partitions\footnote{\label{foot:NC}As usual, we call a partition~$\pi$ of the (ordered) set
  $\{a_1,a_2,\dots,a_k\}$ \defn{non-crossing}, if the graphical
  representation of~$\pi$, where the elements $a_1,a_2,\dots,a_k$ are
  placed, in this order, around a circle, and successive elements in
  the blocks of~$\pi$ are connected by an arc, can be realised so that
the arcs do not cross each other.}.
We show in \Cref{prop:1} that in type~$A_{n-1}$ being positive means that the block containing~$1$ also contains~$mn$, in \Cref{prop:1B} that in type~$B_n$ being positive means that the block containing~$1$ also contains an element from $\{-1,-2,\dots,-mn\}$, and in \Cref{prop:1D} that in type~$D_n$ being positive means that the block containing~$1$ also contains an element from $\{-1,-2,\dots,-m(n-1)\}$.

\medskip
Our next set of results consists of counting formulae for
types~$A_{n-1}$, $B_n$, and~$D_n$ that refine the total count~\eqref{eq:mdivncplus}. More precisely, we provide formulae for the number of 
(multichains of) positive $m$-divisible non-crossing partitions 
with restrictions on the block structure; see \Cref{sec:enumA,sec:enumB,sec:enumD}.

\medskip
Our final set of results concerns a uniformly defined action on
positive $m$-divisible non-crossing partitions. This action is
given by two variants
of a map introduced by Armstrong~\cite[Def.~3.4.15]{Arm2006}.
Armstrong's map is an extension of a map originally introduced
by Kreweras~\cite[Sec.~3]{KrewAA} commonly known as {\it``Kreweras complement."}
For this reason, we shall call Armstrong's map
({\it generalised\/}) {\it``Kreweras map"}.

Armstrong's map does not
restrict to an action on {\it positive}
$m$-divisible non-crossing partitions.
Inspired by a representation-theoretic functor in the work of
Sim\~oes~\cite{Sim2012} and of Buan, Reiten and Thomas
\cite{BRT2012}, we introduce 
``positive analogues'' of Armstrong's (generalised) Kreweras map on 
positive $m$-divisible non-crossing 
partitions in \Cref{def:positivekreweras}.
One of these \defn{``positive Kreweras maps''} was already described in~\cite[Sec.~7]{AST2010}, where also its order was computed; see \Cref{thm:order}.

Armstrong's original Kreweras map translates into ordinary rotation on the realisations of $m$-divisible non-crossing set partitions in classical types.
It turns out that our variants are slightly more complicated to describe.
They translate into certain ``pseudo-rotations," meaning that they are realised as a ``deformation" of rotation; see \Cref{prop:2,prop:2B,prop:2D}.

In all types, we provide counting formulae for the number of positive $m$-divisible non-crossing partitions that are invariant under a fixed power of either of the two positive Kreweras maps, including refinements; see \Cref{sec:rotenumA,,sec:rotenumB,,sec:rotenumD}.

We finally use the above described results to address cyclic sieving phenomena for the positive $m$-divisible non-crossing partitions.
The phenomenon of cyclic sieving has been introduced by Reiner, Stanton and White in~\cite{RSW2004}.
It refers to a certain enumerative property of
combinatorial objects under the action of a cyclic group.
Armstrong \cite[Conj.~5.4.7]{Arm2006} went
on to conjecture a cyclic sieving phenomenon for his Kreweras map.
A related conjecture on cyclic sieving of $m$-divisible
non-crossing partitions was made by
Bessis and Reiner~\cite[Conj.~6.5]{BR2007}. Both conjectures
were proved by the first author and M\"uller~\cite{KratCG,KrMuAD}.

In this paper,
we prove that the positive $m$-divisible non-crossing partitions together with either of the two variants of the positive Kreweras map
also satisfy a cyclic sieving phenomenon;
see \Cref{thm:CS}. This is in fact a cumulative theorem,
resulting from \Cref{thm:3,thm:3-B,thm:3-D} for the classical types,
and \Cref{thm:cycexc} for the exceptional types.

\medskip
This article is organised as follows.
In \Cref{sec:prel}, we recall the definitions of
$m$-divisible non-crossing partitions associated with a
Coxeter group~$W$, and their Kreweras map(s). 
There, we also define positive $m$-divisible non-crossing
partitions and the positive analogue of the Kreweras map(s).
Moreover, we collect several properties of these objects
(see \Cref{thm:positivedescription} and \Cref{cor:unique}),
and we state one of the main results, namely the cyclic sieving
phenomena in \Cref{thm:CS}.

This is followed by \Cref{sec:typeA,sec:typeB,sec:typeD}, in 
which the positive $m$-divisible non-crossing partitions and
their positive Kreweras maps are discussed in depth for the classical
types. Each of these sections starts with a translation of
the algebraic definition of these objects into the language
of set partitions; see \Cref{prop:1,prop:1B,prop:1D}. 
These translations motivate purely combinatorial definitions of
positive $m$-divisible non-crossing partitions and their positive Kreweras maps
(in terms of a pseudo-rotation) for the classical types;
see \Cref{sec:mapA,sec:mapB,sec:mapD}.
Subsequently, 
a subsection is devoted
to the enumeration of positive $m$-divisible non-crossing
partitions in each type subject to various restrictions on
their block structure; see \Cref{sec:enumA,sec:enumB,sec:enumD}. 
In a further subsection, we characterise, for each type, 
the positive $m$-divisible non-crossing partitions that are invariant
under a given power of either of the two positive Kreweras maps;
see \Cref{sec:rotA,sec:rotB,sec:rotD}. Then, in
\Cref{sec:rotenumA,sec:rotenumB,sec:rotenumD},
we enumerate those that are
invariant. Finally, we use the obtained formulae to deduce
corresponding cyclic sieving phenomena; see
\Cref{thm:3,thm:3-B,thm:3-D}.

In \Cref{sec:typeExc}, we treat the exceptional types
with respect to invariance and cyclic sieving under the action
of the positive Kreweras maps.
We present several problems that were left
open by our analyses in \Cref{sec:open}.

In order to streamline the presentation, we have ``outsourced" most of the (frequently technical) proofs of our results to appendices.
More precisely, our combinatorial proofs of the order
of the (combinatorial version of) the Kreweras maps acting on
positive $m$-divisible non-crossing partitions in the classical
types are contained in \Cref{app:order}.

The various proofs of our enumeration
formulae for positive $m$-divisible non-crossing partitions
of classical type, without
or with invariance under powers of the positive Kreweras maps,
from
\Cref{sec:enumA,,sec:rotenumA,,sec:enumB,,sec:rotenumB,,sec:enumD,,sec:rotenumD}
can be found in \Cref{app:GF}.

This is followed by
\Cref{app:inv} in which we have collected the proofs of our characterisations
of positive $m$-divisible non-crossing partitions that are invariant
under the action of a given power of the positive Kreweras maps,
again in the classical types.

In \Cref{app:unnec} we provide
the proofs for some auxiliary results for our cyclic sieving
results in classical types from \Cref{sec:sievA,,sec:sievB,,sec:sievD}.

Finally, the required computational details
for the verification of our cyclic sieving result for
positive $m$-divisible non-crossing partitions in the exceptional
types in \Cref{thm:cycexc} are carried out in \Cref{app:exc}.

\bigskip\noindent
\textbf{Acknowledgement}
This project was initiated in 2012 when, at the occasion of
the 24th International Conference on ``Formal Power Series and
Algebraic Combinatorics" in Nagoya, Japan, Christian
innocently asked Christian: ``Willst Du ein weiteres
zyklisches Sieben beweisen?"\footnote{``Do you want to prove another
cyclic sieving phenomenon?"} Over time, this turned out to be
slightly more involved than originally anticipated by Christian.
It took multiple meetings at each of the Universit\"at Wien, the
Freie Universit\"at Berlin, the Technische Universit\"at Berlin, and the Ruhr-Universit\"at Bochum, until this article finally materialised in 2025 just before the 37th International Conference on ``Formal Power Series and Algebraic Combinatorics'' in  Sapporo, Japan.

\section{A general setup for positive~\texorpdfstring{$m$}{m}-divisible non-crossing partitions}
\label{sec:prel}

This section provides the setup for $m$-divisible non-crossing partitions in their various incarnations for general finite irreducible Coxeter systems.
In particular, we provide two equivalent descriptions of \emph{positive} $m$-divisible non-crossing partitions.
Most of the background material is standard, though we refer to~\cite{Arm2006,STW2015} for further
detail.
After a brief
reminder of the notions for finite Coxeter systems, we discuss (positive) $m$-divisible non-crossing partitions in \Cref{sec:mdivdef,,sec:alt}, and then introduce our two positive Kreweras maps
in \Cref{sec:poskrewmap}. 
Since they cannot be fully used for $m=1$, we provide an alternative description in
\Cref{sec:poskrewmap2} that makes sense for general $m \geq 1$.
After having established the needed structures, we conclude with
\Cref{sec:sievW}, in which we first recall the definition of the cyclic
sieving phenomenon of Reiner, Stanton and White~\cite{RSW2004}, and
then present two uniformly stated cyclic sieving
phenomena for our positive $m$-divisible non-crossing partitions
in \Cref{thm:CS}.
(More refined cyclic sieving phenomena which are specific
for type~$A_{n-1}$ respectively for type~$B_n$
are contained in \Cref{thm:3,,thm:3-B}.)

\medskip

The following can all be found in standard references such as~\cite{Hum1990}.
Let $(W,\reflS)$ denote a finite irreducible Coxeter system of rank~$n = |\reflS|$ acting on a Euclidean vector space~$V \cong \R^n$ with inner product $\langle \cdot, \cdot \rangle$.
Let
$\Phi\subseteq V$ be a root system for $(W,\reflS)$ with simple roots $\Delta = \{ \alpha_s \mid s \in \reflS \}$, positive roots~$\Phi^+$, and negative roots $\Phi^- = -\Phi^+$.
The elements in $\reflS = \{ r_\alpha \mid \alpha \in \Delta\}$ are called \defn{simple reflections}, and the \defn{reflections} in~$W$ are their conjugates,
\[
  \reflR = \{r_\beta \mid \beta \in \Phi^+ \} = \{ wsw^{-1} \mid w \in W, s \in \reflS\}\,,
\]
where $r_\beta$ denotes the reflection sending~$\beta$ to its negative while fixing pointwise its orthogonal complement~$\beta^\perp = \{v \in V \mid \langle \beta, v \rangle = 0\}$.
The length function~$\lenS : W \rightarrow \N$ sends $w \in W$ to the length of
any reduced $\reflS$-word for~$w$.
There exists a unique element $\wo \in W$ of longest length $N = |\Phi^+| = |\reflR|$.
Since all reduced $\reflS$-words for an element are related by braid relations and, moreover, the two simple reflections on both sides of a braid relation coincide, we may define the \defn{support} of an element $w \in W$ by the set of simple reflections appearing in any of its reduced
$\reflS$-words.
An element $w \in W$ is
\defn{of full support} if its support is all of~$\reflS$, or, equivalently, if it does not live in any proper standard parabolic subgroup.

A \defn{(standard) Coxeter element} $c \in W$ is given by the product of all generators~$\reflS$ in any order, $c = s_1\dots s_n \in W$ for $\reflS = \{s_1,\dots,s_n\}$.
Observe that all standard Coxeter elements are conjugate to each other and that their order is given by the \defn{Coxeter number}~$h$, which moreover
features in the formula $N = nh/2$.
The eigenvalues of a Coxeter element are given by $\{ e^{2\pi i (d_1-1) / h},\ldots,e^{2\pi i (d_n-1) / h}\}$, where $d_1,\ldots,d_n$ are the \defn{degrees} of~$W$.

\begin{example}[\sc Running example in type~$A_2$]
\label{ex:runningA}
  The finite irreducible Coxeter system of type $A_2$ can be represented by the symmetric group acting on three letters.
  In this case, we have
  \[
    \reflS = \big\{(12), (23) \big\} \subset \big\{ (12),(13),(23) \big\} = \reflR \,,
  \]
  where $(ij)$ is the transposition interchanging the letters~$i$ and~$j$.
  The unique longest element $\wo = (12)(23)(12) = (23)(12)(23)$ has length $N = 3$.
  One choice of a Coxeter element is the long cycle $c = (12)(23) = (123)$ and we see that the Coxeter number is~$h = 3$.
  The degrees are moreover $d_1 = 2$ and $d_2 = 3$.
  This yields the (positive)
  Fu\ss--Catalan numbers
  \begin{equation}
  \label{eq:mcatA2}
    \mCat(W) = \tfrac{1}{2}(m+1)(3m+2),\quad \mCatplus(W) = \tfrac{1}{2}m(3m+1)\,.
  \end{equation}
  For $m = 0,1,2$, these evaluate to $(1,5,12)$ and to $(0,2,7)$, respectively.
\end{example}

\subsection{\texorpdfstring{$m$}{m}-divisible non-crossing partitions}
\label{sec:mdivdef}

It is well established that there is also a \emph{dual} viewpoint on Coxeter groups by generating
such a group by all its reflections. Non-crossing partitions play a crucial role in this approach.
The \defn{reflection length} $\lenR : W \rightarrow \N$ is defined analogously as the length with respect to reduced $\reflR$-words.
Analogously to the (left or right) weak order on~$W$, one defines the \defn{absolute order}~$\leqR$ by letting $u\leqR v$ if $\lenR(u)+\lenR(u^{-1}v)=\lenR(v)$.
Thus, $u\leqR v$ if and only if there is a reduced $\reflR$-word for~$v$ which begins with a reduced $\reflR$-word for~$u$.

\medskip

For a Coxeter element~$c$, one defines the set of \defn{non-crossing partitions} $\NC$ as the set of all elements $w\in W$ with $w\leqR c$.
$\NC$ with the partial order $\leqR$ forms a lattice, cf.~\cite{BRWaAB}.
Since $\leqR$ is invariant under conjugation, this lattice is, abstractly, independent of the choice of~$c$.
We refer to~\cite{RRS2017} for an in-depth discussion of this independence of the chosen Coxeter element.

\medskip

For a finite irreducible Coxeter group $W$ of rank~$n$,
Armstrong's \defn{$m$-divisible non-crossing partitions}~\cite{Arm2006} are then defined by
\begin{multline*}
\mNC=\big\{(w_0,w_1,\dots,w_m) \mid w_0w_1\cdots w_m=c\text{ and }\\
\lenR(w_0)+\lenR(w_1)+\dots+\lenR(w_m)=n \big\}.
\end{multline*}
For $m=1$, this definition reduces to the definition of 
non-crossing partitions $\NC$, by identifying
$(w_0,w_1) \in \mNC[W][1][c]$
with the element $w_0 \in \NC$.
The set $\mNC$ becomes a partially ordered set by letting
\[
  (u_0,u_1,\dots,u_m)\leqR
  (v_0,v_1,\dots,v_m)\quad \text{if }
  u_i\geqR v_i,\text{ for }i=1,2,\dots,m.
\]
The elements of $\mNC$
are well-known to be counted by the Fu\ss--Catalan numbers
\[
  |\mNC| = \mCat(W)\,.
\]

\begin{example}
  Following the running example from \Cref{ex:runningA}, we illustrate the twelve $2$-divisible non-crossing partitions $\mNC[A_2][2][(12)(23)]$ in \Cref{fig:ncA3} where we separate the parts by dots.
  \begin{figure}[t]
    \begin{center}
      \resizebox{\textwidth}{!}{
      \begin{tikzpicture}[scale=1.5]
        \tikzstyle{rect}=[rectangle,draw,opacity=.5,fill opacity=1]
        \node[rect, ultra thick, fill=black!10] (01)   at ( 0,0) {$(12)(23)\cdot \one\cdot \one$};
  
        \node[rect, ultra thick, fill=black!10] (11)   at (-5,-2) {$(12)\cdot (23)\cdot \one$};
        \node[rect] (12)   at ( 5,-2) {$(12)\cdot \one\cdot (23)$};
        \node[rect, ultra thick, fill=black!10] (13)   at (-1,-2) {$(13)\cdot (12)\cdot \one$};
        \node[rect, ultra thick, fill=black!10] (14)   at ( 3,-2) {$(13)\cdot \one\cdot (12)$};
        \node[rect, ultra thick, fill=black!10] (15)   at (-3,-2) {$(23)\cdot (13)\cdot \one$};
        \node[rect] (16)   at ( 1,-2) {$(23)\cdot \one\cdot (13)$};
  
        \node[rect, ultra thick, fill=black!10] (21)   at (-4,-4) {$\one \cdot (12)(23) \cdot \one$};
        \node[rect] (22)   at (-2,-4) {$\one \cdot (23)\cdot (13)$};
        \node[rect] (23)   at ( 4,-4) {$\one \cdot \one \cdot (12)(23)$};
        \node[rect] (24)   at ( 2,-4) {$\one \cdot (12)\cdot (23)$};
        \node[rect, ultra thick, fill=black!10] (25)   at ( 0,-4) {$\one \cdot (13)\cdot (12)$};
  
        \draw (01) -- (11);
        \draw (01) -- (12);
        \draw (01) -- (13);
        \draw (01) -- (14);
        \draw (01) -- (15);
        \draw (01) -- (16);
        \draw (11) -- (21);
        \draw (11) -- (22);
        \draw (12) -- (23);
        \draw (12) -- (24);
        \draw (13) -- (21);
        \draw (13) -- (24);
        \draw (14) -- (23);
        \draw (14) -- (25);
        \draw (15) -- (21);
        \draw (15) -- (25);
        \draw (16) -- (22);
        \draw (16) -- (23);
      \end{tikzpicture}}
    \end{center}
    \caption{\label{fig:ncA3}The $2$-divisible non-crossing partitions $\mNC[A_2][2][(12)(23)]$.}
\end{figure}
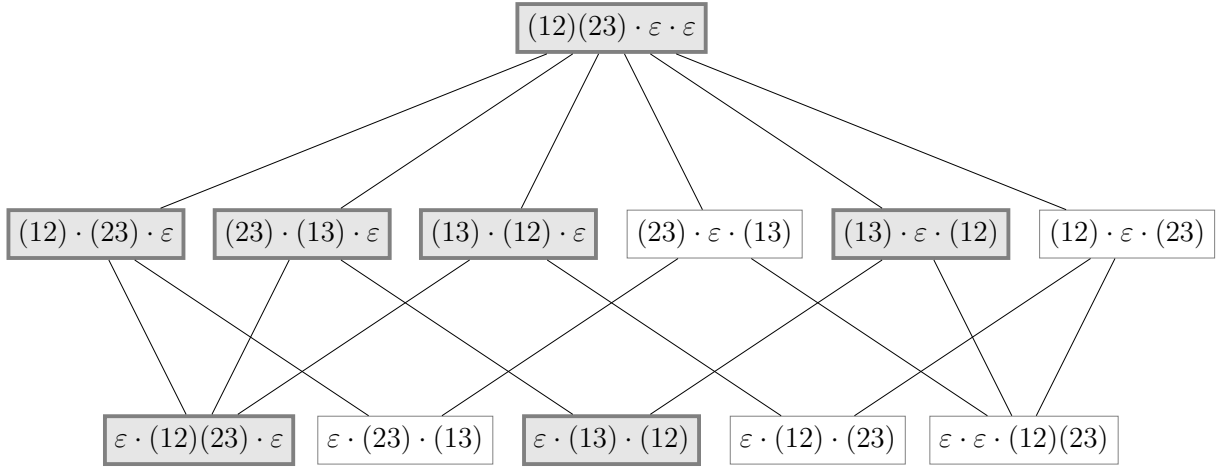
\end{example}

As an abstract partially ordered set, $\mNC$ does again not depend on the
choice of the Coxeter element~$c$. In general, it is not a
lattice but a graded join-semilattice.

\medskip

The main object of study in this paper is the following subset of $m$-divisible non-crossing partitions.

\begin{definition} \label{def:NCpos}
  An $m$-divisible non-crossing partition $(w_0,w_1,\dots,w_m) \in \mNC$ is called \defn{positive} if the product $w_0 \cdots w_{m-1} \in \NC$ has full support.
  The set of \defn{positive $m$-divisible non-crossing partitions} is denoted by $\mNCPlus$.
\end{definition}

The seven positive elements in $\mNC[A_2][2][(12)(23)]$ are highlighted in \Cref{fig:ncA3}.
The above definition is motivated by a bijective correspondence between $m$-divisible non-cros\-sing partitions and the generalised cluster complex defined by Fomin and Reading in~\cite{FR2005}.
In~\cite{BRT2012}, Buan, Reiten, and Thomas provided a
representation-theoretic bijection between facets in the $m$-generalised cluster complex and $m$-divisible non-crossing partitions for crystallographic types.
This bijection was made combinatorial and uniform for all types by the second author, Thomas and Williams in~\cite{STW2015}.
Under this bijection, positive $m$-divisible non-crossing partitions 
correspond exactly to \lq\lq positive clusters\rq\rq\ which are counted by the \defn{positive Fu\ss--Catalan numbers}~\eqref{eq:mdivncplus}; see~\cite{FR2005}.

\subsection{An alternative description of $m$-divisible non-crossing partitions}
\label{sec:alt}

Following~\cite[Ch.~7]{STW2015}, we next provide an alternative description of $m$-divisible non-crossing partitions and their positive analogues.
We first need to recall some background on \emph{Coxeter-sortable elements }in Coxeter groups as introduced by Reading in~\cite{Rea2005} and further studied in various places; see for example~\cite{RS2008,CLS2011, STW2015}.

Fix a reduced word $\c = \s_1 \cdots \s_n$ for the Coxeter element~$c$, and define the \defn{$\c$-sorting word} $\w(\c)$ for $w \in W$ to be the lexicographically first reduced word for~$w$ when written as a subword of $\c^\infty:=c\cdot c\cdot c\cdots$.
(Here and below, we mildly distinguish between the letter $\s$ in the word~$\c$ from the simple reflection $s \in \reflS$, and similarly from a letter $\r$ and the corresponding reflection $r \in W$.)
Observe that all reduced words for~$c$ are obtained from each other by transpositions of consecutive commuting letters, or \defn{up to commutations} for short.
Thus, given the Coxeter element~$c$, the various $\c$-sorting words for $w \in W$ all coincide up to commutations.
Let furthermore $\wwo(\c) = s_{1} \cdots s_{N}$ be the $\c$-sorting word of the longest element $\wo \in W$, starting with an initial copy of $\c = \s_1\cdots \s_n$.
It was shown for example in~\cite[Prop.~7.1]{CLS2011} that, up to commutations, this word ends in a twisted version of the reduced word of the Coxeter element,
namely
\[
  \s_{{N-n+1}} \cdots \s_{N} = \psi(\c) = \psi(\s_1) \cdots \psi(\s_n),
\]
where~$\psi(s) = \wo s \wo$ for $s \in \reflS$.
Depending on the given Coxeter element~$c$, the word $\wwo(\c)$ induces a \defn{reflection ordering}
\begin{equation} \label{eq:refls}
  r_1 \lec r_2 \lec \cdots \lec r_N,
\end{equation}
on $\reflR$
with $r_i = s_{1} \cdots s_{{i-1}} s_{i} s_{{i-1}} \cdots s_{1} \in \reflR$.
For a definition (which is not further used in this paper) and further properties of reflection orderings; see~\cite{Dye1993}.
For later reference, we refer to this reflection ordering by the word~$\invc = \r_1\cdots \r_N$.
We denote the final~$n$ letters of~$\invc$ (i.e., the word $\psi(\c)$) by
$\invc^R$ so that we obtain $\invc = \invc^L \invc^R$.

\begin{example}
  Continuing the above example with $\c = (12)(23)$,
the $\c$-sorting word for the longest element is $\wwo(\c) = (12)(23)(12)$,
while the corresponding reflection ordering is
  \[
    (12) <_c (13) <_c (23)\,.
  \]
  We also see that $\wwo(\c) = (12)(23)(12)$ starts with an initial copy of~$\c = (12)(23)$ and ends with $\psi(\c) = (23)(12)$.
  This gives
  \[
    \invc = (12)(13)(23) = (12) \cdot (13)(23) = \invc^L \cdot \invc^R\,.
  \]
  Note that the middle reflection $(23)$ in $\wwo(\c)$ appears both in the initial copy of~$\c$ and also in~$\psi(\c)$.
  In general, we cannot guarantee that there is a single word equal to $\wwo(\c)$ up to commutations starting with $\c$ and ending with $\psi(\c)$.
  For example, in type $A_3$ with
$\c = (12)(23)(34)$ we have
  \[
    \wwo(\c) = (12)(23)(34)(12)(23)(12)
  \]
  and the third letter $(34)$ appears both in the initial copy of~$\c$ and in 
$\psi(\c)=(34)(23)(12)$.
This has no effect on the
constructions below.
\end{example}

The following property of non-crossing partitions is crucial, we refer to~\cite[Rem.~6.8]{Rea2007} and to \cite[Prop.~4.1.3]{STW2015} for details.

\begin{proposition}
\label{prop:nc_unique_factorization}
  Each $w \in \NC$ has a unique reduced $\reflR$-word $r_1 r_2 \cdots r_p$ with $r_1<_{\c} r_2<_{\c} < \cdots <_{\c} r_p$.
  Moreover, every reduced $\reflR$-word $r_1 r_2 \cdots r_p$ with $r_1<_{\c} r_2<_{\c} < \cdots <_{\c} r_p$ yields an element in~$\NC$.
\end{proposition}

Indeed, this factorisation is the unique factorisation of~$w$ into reflections induced by the EL-labelling of the non-crossing partition lattice in~\cite{ABW2007}.
Together with the definition of $m$-divisible non-crossing partitions, we obtain the following two alternative descriptions.
Define $\mNCdelta$ to be the set of subwords of $\invc^{m+1}$ that
are reduced $\reflR$-words for the Coxeter element~$c$, in which we distinguish the
letters in different copies of~$\invc$,
\begin{multline}
\label{ml:mncdelta}
  \mNCdelta = \{ \w_0\cdot \w_1\cdot\,\dots\,\cdot\w_m \mid \w_i = \t^{(i)}_1 \cdots \t^{(i)}_{k_i} \subset \invc \text{ and }\\ c = t^{(0)}_1 \cdots t^{(0)}_{k_0} \cdots t^{(m)}_1 \cdots t^{(m)}_{k_m} \text{ is reduced} \}\,,
\end{multline}
where we write $\rm u\subset\rm v$ if $\rm u$ is a subword of~$\rm v$.
See \Cref{fig:ncA22} for this set of subwords in the running example of type $A_2$ with $m=2$.
As before, we use $\t_i$ for the letter of~$\invc$ corresponding to the reflection $t_i\in\reflR$. Analogously, the word $\w_i=\t^{(i)}_1 \cdots \t^{(i)}_{k_i}$ is the subword
of~$\invc$ corresponding to the element $w_i=t^{(i)}_1 \cdots t^{(i)}_{k_i}\in\NC$.

We then have the following theorem.

\begin{theorem}[\sc {\cite[Thm.~4.2.5]{STW2015}}]
  Sending a reduced subword in $\mNCdelta$ to the tuple $(w_0,\dots,w_m)\in \mNC$, where $w_i$ is the product of the reflections in the subword inside the $(i+1)$-st copy of~$\invc$, is a bijection.
\end{theorem}

\begin{example}
  The bijection between $\mNC[A_2][2][(12)(23)]$ and $\mNCdelta[A_2][2][(12)(23)]$ is given in \Cref{fig:ncA22}.
  \begin{figure}[t]
    \begin{center}
      \begin{tabular}{ccccc|ccccc}
        \multicolumn{5}{c}{$\mNC[A_2][2][(12)(23)]$} & \multicolumn{5}{c}{$\mNCdelta[A_2][2][(12)(23)]$} \\[3pt]
        \hline
        $(12)(23) $&$\cdot$&$ \one $&$\cdot$&$ \one$ &
        $\textbf{(12)}(13)\textbf{(23)}$&$\cdot$&$(12)(13)(23)$&$\cdot$&$(12)(13)(23)$
        \\[3pt]
        $\one $&$\cdot$&$ \one $&$\cdot$&$ (12)(23)$ &
        $(12)(13)(23)$&$\cdot$&$(12)(13)(23)$&$\cdot$&$\mathbf{(12)}(13)\mathbf{(23)}$
        \\[3pt]
        $\one $&$\cdot$&$ (13) $&$\cdot$&$ (12)$ &
        $(12)(13)(23)$&$\cdot$&$(12)\textbf{(13)}(23)$&$\cdot$&$\textbf{(12)}(13)(23)$
        \\[3pt]
        $(23) $&$\cdot$&$ (13) $&$\cdot$&$ \one$ &
        $(12)(13)\textbf{(23)}$&$\cdot$&$(12)\textbf{(13)}(23)$&$\cdot$&$(12)(13)(23)$
        \\[3pt]
        $(13) $&$\cdot$&$ (12) $&$\cdot$&$ \one$ &
        $(12)\textbf{(13)}(23)$&$\cdot$&$\textbf{(12)}(13)(23)$&$\cdot$&$(12)(13)(23)$
        \\[3pt]
        $(12) $&$\cdot$&$ \one $&$\cdot$&$ (23)$ &
        $\textbf{(12)}(13)(23)$&$\cdot$&$(12)(13)(23)$&$\cdot$&$(12)(13)\textbf{(23)}$
        \\[3pt]
        $\one $&$\cdot$&$ (23) $&$\cdot$&$ (13)$ &
        $(12)(13)(23)$&$\cdot$&$(12)(13)\textbf{(23)}$&$\cdot$&$(12)\textbf{(13)}(23)$
        \\[3pt]
        $\one $&$\cdot$&$ (12)(23) $&$\cdot$&$ \one$ &
        $(12)(13)(23)$&$\cdot$&$\textbf{(12)}(13)\textbf{(23)}$&$\cdot$&$(12)(13)(23)$
        \\[3pt]
        $(12) $&$\cdot$&$ (23) $&$\cdot$&$ \one$ &
        $\textbf{(12)}(13)(23)$&$\cdot$&$(12)(13)\textbf{(23)}$&$\cdot$&$(12)(13)(23)$
        \\[3pt]
        $(23) $&$\cdot$&$ \one $&$\cdot$&$ (13)$ &
        $(12)(13)\textbf{(23)}$&$\cdot$&$(12)(13)(23)$&$\cdot$&$(12)\textbf{(13)}(23)$
        \\[3pt]
        $(13) $&$\cdot$&$ \one $&$\cdot$&$ (12)$ &
        $(12)\textbf{(13)}(23)$&$\cdot$&$(12)(13)(23)$&$\cdot$&$\textbf{(12)}(13)(23)$
        \\[3pt]
        $\one $&$\cdot$&$ (12) $&$\cdot$&$ (23)$ &
        $(12)(13)(23)$&$\cdot$&$\textbf{(12)}(13)(23)$&$\cdot$&$(12)(13)\textbf{(23)}$
      \end{tabular}
    \end{center}
    \caption{The bijection between $\mNC[A_2][2][(12)(23)]$ and $\mNCdelta[A_2][2][(12)(23)]$.}
    \label{fig:ncA22}
  \end{figure}
\end{example}

For an element $(w_0,w_1,\dots,w_m) \in \mNC$ with corresponding $\w_0\cdot\w_1\cdot\,\dots\,\cdot\w_m \in \mNCdelta$, we have $w_0 \in \NC$.
Furthermore, the word $\w_0 \subset \invc$
is redundant since $w_0 = c(w_1\cdots w_m)^{-1}$ with unique reduced increasing word $\w_0$.
Therefore, sending $\w_0\cdot\w_1\cdot\,\dots\,\cdot\w_m$ to $\w_1\cdot\,\dots\,\cdot\w_m$ is a bijection between $\mNCdelta$ and
\begin{multline}
\label{ml:mncnabla}
  \mNCnabla: = \{ \w_1\cdot\,\dots\,\cdot\w_m \mid \w_i = \t^{(i)}_1 \cdots \t^{(i)}_{k_i} \subset \invc \text{ and }\\ t^{(1)}_1 \cdots t^{(1)}_{k_1} \cdots t^{(m)}_1 \cdots t^{(m)}_{k_m} \in \NC \text{ is reduced} \}.
\end{multline}
Moreover, $\mNCdelta[W][m-1] \subset \mNCnabla$ is
a proper subset.

\medskip

These two descriptions can be used to determine positive $m$-divisible non-crossing partitions as follows.
Denote by
\begin{equation} \label{eq:LR} 
  w = w^L\cdot w^R
\end{equation}
the decomposition of~$w \in \NC$ such that $w^L$ and $w^R$ are the parts of the unique subword for~$w$ of~$\invc$ that lie inside $\invc^L$ and inside $\invc^R$, respectively.
The following theorem gives an alternative description of positive $m$-divisible non-crossing partitions.
It was first observed in crystallographic types in~\cite[Thm.~7]{BRT2012}. The general case was established in~\cite[Ch.~7]{STW2015}.

\begin{theorem}
\label{thm:positivedescription}
The tuple $(w_0,w_1,\dots,w_m) \in \mNC$ is positive if and only if $w_m = w_m^L$.
\end{theorem}

We thus call the corresponding elements $\w_0\cdot\,\dots\,\cdot\w_m \in \mNCdelta$ and $\w_1\cdot\,\dots\,\cdot\w_m \in \mNCnabla$ positive if $\w_m \subset \invc^L$.
In \Cref{fig:ncA3}, we highlight the seven positive
elements in $\mNC[A_2][2][(12)(23)]$ and observe in \Cref{fig:ncA22} that these are exactly those where the final component of $\mNCdelta$ is a subword of $\invc^L = (12)$.
We note that this theorem has the following immediate consequence, which we record for later use.

\begin{lemma} \label{cor:unique}
  Let $w \in \NC$, and decompose~$w$ in the form $w=uv$ such that $\lenR(u)+\lenR(v)=\lenR(w)$.
  Suppose $v=r_1\cdots r_k$ for $\r_1\cdots\r_k \subset \invc^R$ and $cu^{-1}$ has full support.
  Then $w^L=u$ and $w^R=v$.
\end{lemma}

\begin{proof}
  Since $cu^{-1}$ has full support, we have $(cu^{-1},u)\in \mNCPlus[W][1]$ by \Cref{def:NCpos}.
  \Cref{thm:positivedescription} then implies that $u^L=u$.
  By definition, $u^L$ is formed by determining the subword of~$\reflR$ that is a reduced $\reflR$-word for $u$, and then disregarding those reflections which lie inside $\invc^R$.
  Since $u^L=u$, nothing had to be disregarded.
  We conclude by observing that $w^L=(uv)^L = u^L = u$ and $w^R = (uv)^R = u^Rv = v$.
\end{proof}

The description of (positive) $m$-divisible non-crossing partitions in terms of $\mNCnabla$ also has the following enumerative implications, namely formulae for the (positive) Fu\ss--Catalan numbers in terms of a sum of binomial coefficients; see \cite[Thm.~4.5.1]{STW2015}.

\begin{corollary}
\label{cor:fusscatcounting}
Let $W$ be a finite irreducible Coxeter group of rank $n$. Then
  \begin{align*}
    \mCat(W) &= \sum \binom{m+1+\asc(\r_1,\ldots,\r_n)}{n}, \\
    \mCatplus(W) &= \sum \binom{m + \des(\r_1,\ldots,\r_n)}{n} = \sum \binom{m + \asc(\r_1,\ldots,\r_n,\rfix)}{n} ,
  \end{align*}
  where all sums range over the factorisations $\r_1\cdots\r_n \subset \invc$ of the Coxeter element~$c$ into reflections $r_i \in \reflR$.
  Here, $\asc(\r_1,\ldots,\r_n)$ and $\des(\r_1,\ldots,\r_n)$ are the number of indices~$i$ such that $r_i <_c r_{i+1}$ and such that $r_i >_c r_{i+1}$, respectively.
  Moreover, $\rfix$ in the last sum is the first letter in the word $\invc^R$.
\end{corollary}

Indeed, the description of the positive Fu\ss--Catalan numbers in terms of descents does not directly follow from the analysis above.
We record it here for its beauty.
On the other hand, the other description using the letter~$\rfix$ will be crucial in the final section where we establish a similar count for positive elements that are invariant under either of the two \emph{positive Kreweras maps}.

\begin{example}
  In our ongoing example, we have $c = (12)(23) = (23)(13) = (13)(12)$ for the ordering $(12) <_c (13) <_c (23)$.
  Moreover, $\rfix = (13)$.
  This yields
  \begin{align*}
    \asc((12),(23)) = 1,&\quad 0 = \asc((13),(12)) = \asc((23),(13)), \\
    \asc((12),(23),(13)) = \asc((13),(12),(13)) = 1,&\quad 0 = \asc((23),(13),(13)),
  \end{align*}
  implying
  \[
    \mCat(W) = 2\binom{m+1}{2} + \binom{m+2}{2},\quad
    \mCatplus(W) = 2\binom{m+1}{2} + \binom{m}{2}\,,
  \]
  in accordance with~\eqref{eq:mcatA2}.
\end{example}

\subsection{The positive Kreweras maps, I}
\label{sec:poskrewmap}

Beside refined counting formulae, we study in this paper a ``positive analogue'' of the \defn{(generalised) Kreweras map} on $m$-divisible non-crossing partitions.
The two variants~$\Krewtilde^{(m)}$ and~$\Krew^{(m)}$ below were introduced by Armstrong in~\cite[Def.~3.4.15]{Arm2006}.
For $m\geq 1$, these are given by
\begin{align*}
   \Krewtilde^{(m)} &:
  \begin{array}{ccc}
     \mNC[W] &\tilde\longrightarrow& \mNC[W] \\
     (w_0,w_1,\ldots,w_m) &\mapsto& (v,cw_mc^{-1},w_1\ldots,w_{m-1}),
  \end{array}\\[15pt]
   \Krew^{(m)} &:
  \begin{array}{ccc}
     \mNC &\tilde\longrightarrow& \mNC \\
     (w_0,w_1,\ldots,w_m) &\mapsto& (cw_mc^{-1},w_0\ldots,w_{m-1}),
  \end{array}
\end{align*}
where $v = c(cw_mc^{-1}w_1 \dots w_{m-1})^{-1} = (cw_mc^{-1})w_0(cw_mc^{-1})^{-1}$.
Both maps are close relatives.
More precisely, the embedding
$$
  \ph :
  \begin{array}{ccc}
    \mNC &\hookrightarrow& \mNC[W][m+1] \\
    (w_0,w_1,\ldots,w_m) &\mapsto& (\one,w_0,w_1,\ldots,w_m)
  \end{array}
$$
has the property that it maps $\Krew^{(m)}$ to $\Krewtilde^{(m+1)}$. In symbols, $\ph \circ \Krew^{(m)} = \Krewtilde^{(m+1)} \circ \ph$.
Moreover, these two variants can easily be interpreted in terms of $\mNCdelta$ and $\mNCnabla$, and the relation comes from the mentioned embedding
of $\mNCdelta$ into $\mNCnabla[W][m+1]$.
We drop the superscripts in $\Krewtilde^{(m)}$ and in $\Krew^{(m)}$ if $m$ is clear from the context and only write $\Krewtilde$ and $\Krew$, respectively.

\medskip

The Kreweras maps do not stabilise the set of positive $m$-divisible non-crossing partitions.
One of the main motivations for this paper is to define and study positive versions of both maps.
Indeed, it turns out that these positive variants can naturally be seen on $\mNCnabla$ and on $\mNCdelta$, and that they are combinatorial descriptions of a representation-theoretic functor on the positive cluster category.

\begin{definition}
\label{def:positivekreweras}
  For $m \geq 2$, the \defn{positive Kreweras map} $\Krewplustilde^{(m)}$ is given by
  \begin{align*}
    \Krewplustilde^{(m)} &:
      \begin{array}{ccc}
        \mNCPlus[W][m] &\tilde\longrightarrow& \mNCPlus[W][m] \\
        (w_0,w_1,\ldots,w_{m}) &\mapsto& (v_+,cw_{m-1}^Rw_{m}c^{-1},w_1,\dots,w_{m-1}^L),
      \end{array} \\
\intertext{where $v_+ = c(cw_{m-1}^Rw_mc^{-1} w_1 \cdots w_{m-1}^L)^{-1} = (cw_{m-1}^Rw_mc^{-1})w_0(cw_{m-1}^Rw_mc^{-1})^{-1}$.
For\break $m \geq 1$, its variant $\Krewplus^{(m)}$ is given by}
    \Krewplus^{(m)} &:
      \begin{array}{ccc}
        \mNCPlus &\tilde\longrightarrow& \mNCPlus \\
        (w_0,w_1,\ldots,w_m) &\mapsto& (cw_{m-1}^Rw_mc^{-1},w_0,w_1,\dots,w_{m-1}^L).
      \end{array}
  \end{align*}
\end{definition}
As for the (original) Kreweras maps, the embedding~$\ph$ of $\mNC$ into $\mNC[W][m+1]$ given above embeds $\mNCPlus$ into $\mNCPlus[W][m+1]$, and we have $\ph \circ \Krewplus^{(m)} = \Krewplustilde^{(m+1)} \circ \ph$.
We again usually drop the superscripts
in $\Krewplustilde^{(m)}$ and $\Krewplus^{(m)}$ and write $\Krewplustilde$ and $\Krewplus$, respectively.
We also emphasise that this version of $\Krewplustilde^{(m)}$ is only defined for $m \geq 2$.
In \Cref{sec:poskrewmap2},
we will present an alternative description that also covers the
case where $m=1$.

\begin{example}
\label{ex:PositiveKrew1}
  We consider $\Krewplustilde^{(2)}$ and both $\Krewplus^{(1)}$ and $\Krewplus^{(2)}$ in our running example.
  We start with $\Krewplustilde^{(2)}$ acting on $\mNCPlus[A_2][2]$:
  \[
    \begin{tikzpicture}
    \node (seq) at (0,0) {
      $(13)\cdot\one\cdot(12) \longmapsto (12)\cdot(23)\cdot\one \longmapsto (23)\cdot(13)\cdot\one \longmapsto (13)\cdot(12)\cdot\one$
    };
    \draw[|->,looseness=3] 
        (seq.mid east) 
        to[out=0,in=0] (seq.south east) 
        -- (seq.south west) 
        to[out=180,in=180] (seq.mid west);

    \node (seq2) at (0,-1.2) {
      $\one\cdot(13)\cdot(12) \longmapsto \one\cdot(12)(23)\cdot\one$
    };
    \draw[|->,looseness=3] 
        (seq2.mid east) 
        to[out=0,in=0] (seq2.south east) 
        -- (seq2.south west) 
        to[out=180,in=180] (seq2.mid west);

    \node (seq2) at (0,-2.4) {
      $(12)(23)\cdot\one\cdot\one$
    };
    \draw[|->,looseness=3] 
        (seq2.mid east) 
        to[out=0,in=0] (seq2.south east) 
        -- (seq2.south west) 
        to[out=180,in=180] (seq2.mid west);
    \end{tikzpicture}
  \]
  Considering only the second orbit and the above inclusion
  $$\varphi : \mNC[A_2][1] \longrightarrow \mNC[A_2][2],
  $$
  we get $\Krewplus^{(1)}$.
  The action of $\Krewplus^{(2)}$ on $\mNCPlus[A_2][2]$ is finally
  \[
    \begin{tikzpicture}
    \node (seq1) at (0,0) {
      $(13)\cdot\one\cdot(12) \mapsto (23)\cdot(13)\cdot\one \mapsto (12)\cdot(23)\cdot\one \mapsto (13)\cdot(12)\cdot\one$
    };
    \node (seq2) at (0,-1) {
      $\one\cdot(12)(23)\cdot\one \longmapsfrom (12)(23)\cdot\one\cdot\one \longmapsfrom \one\cdot(13)\cdot(12)$
    };
    \draw[|->,looseness=3] 
        (seq1.mid east) 
        to[out=0,in=0] (seq2.mid east);
    \draw[|->,looseness=3] 
        (seq2.mid west)
        to[out=180,in=180] (seq1.mid west);
    \end{tikzpicture}
  \]
\end{example}

\subsection{The positive Kreweras maps, II}
\label{sec:poskrewmap2}

The goal of this subsection is to present alternative descriptions of the positive
Kreweras maps $\Krewplustilde^{(m)}$ and~$\Krewplus^{(m)}$ introduced in
\Cref{def:positivekreweras} which enable us to make sense of $\Krewplustilde^{(m)}$
also for $m=1$ (see \Cref{def:K-alt}).
Starting point is the observation that
these positive Kreweras maps can be seen as Coxeter-theoretic versions of the \emph{shift functor} in the \emph{positive generalised cluster category}, compare~\cite{BRT2012,Sim2012}.
Without going into the details of the involved homological algebra, we give here its combinatorial incarnation as given in~\cite{CLS2011,STW2015}.
Up to commutations, it is known that $\c^h = \wwo(\c)\psi(\wwo(\c))$.
This implies the following equality of bi-infinite words, ${}^\infty\c^{\infty} = {}^\infty\big(\wwo(\c)\psi(\wwo(\c))\big)^{\infty}$.
In our ongoing example (with $c=(12)(23)$ and $h=3$), this is
\[
  (12)(23)\cdot(12)(23)\cdot(12)(23) = (12)(23)(12)\cdot(23)(12)(23),
\]
and the bi-infinite word is either
$$\cdots|\ (12)(23)\ |\ (12)(23)\ |\ (12)(23)\ |\cdots$$
or
$$\cdots|\ (12)(23)(12)\ |\ (23)(12)(23)\ |\cdots,$$
respectively.
These bi-infinite words serve as a model for the bounded derived category of finite-dimensional modules over the path algebra of the Dynkin quiver associated with the Coxeter element~$c \in W$.
There are two natural operations on letters in this word:
\begin{itemize}
  \item the \defn{shift functor}~$\Sigma$ sending a letter in a copy of $\wwo(\c)$ (or of $\psi(\wwo(\c))$) to the corresponding letter in the next copy of $\psi(\wwo(\c))$ (or of $\wwo(\c)$, respectively);
  \item the \defn{Auslander--Reiten translate}~$\tau$ sending a letter in a copy of~$\c$ to the corresponding letter in the previous copy of~$\c$.
\end{itemize}
The objects in the positive generalised
cluster category are now orbits in this bi-infinite word under the action of $\Sigma^{m}\circ\tau$.
By construction, a fundamental domain is given by the letters inside the word $\invc^{m-1}\invc^L$.
Denote by~$\Krewplusbar = \Krewplusbar^{(m)}$ this action on letters in $\invc^{m-1}\invc^L$ and extend it letter-wise to subwords.
The following proposition can be found in \cite[Ch.~7]{STW2015}.

\begin{proposition}
\label{prop:PositiveKrewAlt}
For $m \geq 1$, $\Krewplusbar^{(m)}$ acts both on the positive elements in $\mNCnabla$ and on the positive elements in $\mNCdelta[W][m-1]$.
Restricting the above identification between $\mNC$ and $\mNCnabla$ to positive elements, we obtain $\Krewplustilde^{(m)} = \Krewplusbar^{(m)}$.
Restricting further to $\mNCdelta[W][m-1] \subset \mNCnabla$, we obtain also $\Krewplus^{(m-1)} = \Krewplusbar^{(m)}$.
\end{proposition}

\begin{example} \label{ex:AA}
In \Cref{ex:PositiveKrew1},
the letters for $\mNCnabla[A_2][2]$ are
  \[
    \invc\invc^L = (12)^{(1)}(13)^{(1)}(23)^{(1)}\cdot(12)^{(2)}\color{lightgray}{(13)(23)},
  \]
where the final two letters $\color{lightgray}{(13)(23)}$ are identified with the initial two letters $(12)(13)$ in this order, and where we indicated
by superscripts
for each reflection in which copy of~$\invc$ it occurs.
  The action of~$\Krewplusbar$ on these letters is
  \[
    \begin{tikzpicture}
      \node (seq) at (0,0) {
        $(12)^{(1)} \longmapsto (12)^{(2)} \longmapsto (23)^{(1)} \longmapsto (13)^{(1)}$
      };
      \draw[|->,looseness=3] 
          (seq.mid east) 
          to[out=0,in=0] (seq.south east) 
          -- (seq.south west) 
          to[out=180,in=180] (seq.mid west);
    \end{tikzpicture}
  \]
  This yields that the order of~$\Krewplusbar$ is four in this case, and extending this action to subwords yields the above orbit structures.
\end{example}

Following \Cref{prop:PositiveKrewAlt}, we may consider both versions of the positive Kreweras maps $\Krewplustilde^{(m)}$ and $\Krewplus^{(m)}$ as the action of~$\Krewplusbar$ on the positive elements in $\mNCnabla$,
respectively in $\mNCdelta$.
This has the additional advantage that this definition also allows to consider the action of~$\Krewplustilde^{(1)}$ on positive elements in $\mNCnabla[W][1]$.
This situation was not covered using the above version from \Cref{def:positivekreweras}.
In the above example, this case corresponds to the single letter word $\invc^L = (12)\color{lightgray}{(13)(23)}$ with $\color{lightgray}{(13)}$ identified with $(12)$ and $\color{lightgray}{(23)}$ identified with $\color{lightgray}{(13)}$.
In particular, the order of~$\Krewplustilde$ in this case is one and we have the two single element orbits $\one$ and $(12)$.
This yields the following definition for $m=1$.

\begin{definition} \label{def:K-alt}
  The \defn{positive Kreweras map} $\Krewplustilde^{(1)} : \mNCPlus[W][1] \tilde\longrightarrow \mNCPlus[W][1]$ is given by sending $(cw^{-1},w)$ with $w = w^L = t_1 \cdots t_k$ to the element obtained by sending every reflection to its first conjugate inside
$\invc^L$ --- that is, $t_i$ is sent to $c^jt_ic^{-j} \in \invc^L$ with
$j \in \{1,2,\dots\}$ minimal with this property --- and then rearranging the reflections according to the reflection ordering.
\end{definition}

In our example, we see that there is only one reflection in $\invc^L = (12)$.
Conjugation by the Coxeter
element yields $(12) \mapsto (23) \mapsto (13) \mapsto (12) \in \invc^L$.
In this example, $j = 3$ in the definition, and
the positive Kreweras map is given
by extending this operation to words.

\begin{theorem}[\sc {\cite[Thm.~7.3.4]{STW2015}}]
\label{thm:order}
  Let $m\ge 1$.
  The order of\/ $\Krewplusbar^{(m)}$ on positive elements in $\mNCnabla$ and in $\mNCdelta[W][m-1]$ are both given by
  \[
    \order(\Krewplusbar^{(m)}) =
    \begin{cases}
      mh/2 - 1, & \text{if } \psi(s) = s \text{ for all } s \in \reflS, \\
      mh - 2, & \text{otherwise,}
    \end{cases}
  \]
  with the only exception that the order in type $I_2(h)$ for $\Krewplusbar^{(2)}$ on positive elements in\break $\mNCdelta[W][1]$ equals $mh-2$ even for odd~$h$.
  In particular, in all cases we have $\left(\Krewplusbar^{(m)}\right)^{mh-2}=\mathbf{1}$.
\end{theorem}

Indeed, the case $m=1$ for the action of~$\Krewplusbar$ on positive elements in $\mNCnabla[W][1]$ was not treated in \cite[Thm.~7.3.4]{STW2015} but the given proof covers that case as well.
We also make this order concrete in the classical types.

\begin{corollary} \label{cor:order}
  The order of $\Krewplusbar^{(m)}$ on positive elements in $\mNCnabla$ is given by
  \[
    \order(\Krewplusbar^{(m)}) =
    \begin{cases}
      mn - 2, & \text{in type } A_{n-1}, \\
      mn - 1, & \text{in type } B_{n}, \\
      m(n-1) - 1, & \text{in type } D_{n} \text{ with } n \text{ even}, \\
      2m(n-1) - 2, & \text{in type } D_{n} \text{ with } n \text{ odd}.
    \end{cases}
  \]
\end{corollary}

\subsection{Cyclic sieving}
\label{sec:sievW}

We conclude this section with stating the cyclic sieving phenomena
concerning the positive Kreweras maps.
The notion of cyclic sieving was introduced by Reiner, Stanton and
White~\cite{RSW2004}. Given a finite set $S$,
an action on $S$ of a cyclic group~$C$, and a polynomial $P(q)$ in $q$
with non-negative integer coefficients, we say
that the triple $\big(S,P(q),C\big)$ \defn{exhibits the cyclic sieving
  phenomenon}, if for all~$\rrr$ dividing~$\#C$ and
all~$g\in C$ of order~$\rrr$, we have\footnote{This is not the
original definition \cite[Def.-Prop.(i)]{RSW2004}, but is easily seen
to be equivalent to it.}
\begin{equation} \label{eq:siev} 
\#\{s\in S\mid g(s)=s\}=
P(e^{2\pi i/\rrr}).
\end{equation}

Recall that the usual $q$-extension of the positive Fu\ss--Catalan numbers is given by
\[
  \mCatplus(W;q) = \prod_{i=1}^n{\frac{[mh+d_i-2]_q}{[d_i]_q}},
\]
where we use the standard $q$-notation $[\alpha]_q:=(1-q^\alpha)/(1-q) = 1 + q + \dots + q^{\alpha-1}$.

\begin{theorem}
\label{thm:CS}
For a given positive integer $m$, let $W$ be a finite irreducible Coxeter group~$W$, let~$C$ be the cyclic group generated by the Kreweras map~$\Krewplustilde^{(m+1)}$, viewed as a group of order~$(m+1)h-2$, and let~$\widetilde C$ be the cyclic group generated by the Kreweras map~$\Krewplus^{(m)}$, viewed again as a group of order~$(m+1)h-2$.
Then the triples
  \[
\Big(\mNCPlus[W][m+1],\ \mCatplus[m+1](W;q),\ C\Big) \text{ and } 
    \Big(\mNCPlus,\ \mCatplus(W;q),\ \widetilde C\Big)
  \]
  both exhibit the cyclic sieving phenomenon.
\end{theorem}

We will prove this theorem for the classical types by a case-by-case analysis in  \Cref{sec:typeA,sec:typeB,sec:typeD}. Furthermore, in \Cref{sec:typeExc} we develop a tool to count (positive) $m$-divisible non-crossing partitions that are invariant under powers of the positive Kreweras maps for a given type in a finite computation.
We then apply this approach to each exceptional type individually to derive the proposed cyclic sieving phenomenon; see \Cref{thm:cycexc} and \Cref{app:exc}.

\section{Positive non-crossing set partitions in 
type~\texorpdfstring{$A$}{A}}
\label{sec:typeA}

In this section, we study positive $m$-divisible non-crossing partitions and the positive Kreweras maps in the situation of the symmetric group $A_{n-1} = \mathfrak{S}_n$.
This situation corresponds to the Coxeter system
\begin{equation} \label{eq:WSA} 
(W,\reflS) = \big( A_{n-1},\{ s_i = (i,i+1) \mid 1 \leq i < n\} \big)
\end{equation}
together with the Coxeter element $c = s_1 \cdots s_{n-1} \in A_{n-1}$ being the
long cycle $(1,2,\ldots,n) \in \mathfrak{S}_n$. 

\medskip

We start with \Cref{sec:realA},
in which we first make the positivity condition for
$m$-divisible non-crossing partitions under the choice~\eqref{eq:WSA} 
of Coxeter system explicit; see \Cref{prop:posA}. We then use
this to describe \emph{positive} $m$-divisible non-crossing partitions
within the combinatorial
model of $m$-divisible non-crossing partitions due to Edelman
\cite{EdelAA}, via the translation due to Armstrong
\cite[Thm.~4.3.8]{Arm2006}; see \Cref{prop:1}.

In \Cref{sec:mapA}, we show that, under Armstrong's translation,
for $m\ge2$ the positive Kreweras map~$\Krewplustilde^{(m)}$
acts as a pseudo-rotation; see \Cref{prop:2}.
This pseudo-rotation is denoted by~$\rotA$.

In \Cref{sec:combA}, we introduce a natural extension of~$\rotA$
which also makes
sense for\break $m=1$; see \Cref{def:phiallg}. We argue in \Cref{prop:2-m=1}
that, under Armstrong's translation, for $m=1$
this extension corresponds exactly to the map~$\Krewplustilde^{(1)}$
from \Cref{def:K-alt}.
The order of the (extended) map~$\rotA$ is of course 
given by \Cref{thm:order} for type~$A$; see \Cref{lem:N-2}.
For the sake of completeness,
we also provide an independent, {\it combinatorial}, proof for that
order; see \Cref{app:order-A}.

\Cref{sec:enumA} is devoted to the enumeration of positive
$m$-divisible non-crossing partitions. The most refined and most general
result is \Cref{thm:countingA} which provides a
formula for the number of multichains in the poset of these
non-crossing partitions in which the ranks of the elements in the
multichain as well as the block structure of the bottom element are prescribed. All other results in that subsection,
\Cref{cor:countingA1,cor:countingA2,cor:2,cor:3}, 
are obtained as consequences.

\Cref{sec:rotA} presents the characterisation of
type~$A$ positive $m$-divisible non-crossing partitions that are
invariant under powers of the pseudo-rotation~$\rotA$; see \Cref{lem:allA}.

In \Cref{sec:rotenumA}, we then exploit this characterisation to provide a formula for the number of 
type~$A$ positive $m$-divisible non-crossing partitions that are
invariant under powers of the pseudo-rotation~$\rotA$ and
have a given block structure; see \Cref{thm:2}. Once again,
several less refined enumeration results may be obtained by summation;
see \Cref{cor:4,cor:5}.

In the final subsection, \Cref{sec:sievA}, we state, and prove, several
cyclic sieving phenomena refining \Cref{thm:CS} in type~$A$; see~\Cref{thm:3}.

\subsection{Combinatorial realisation of the positive non-crossing partitions}
\label{sec:realA}
With the choice \eqref{eq:WSA} of Coxeter system, 
we have the following simple combinatorial description of 
positive $m$-divisible non-crossing partitions.

\begin{proposition}
\label{prop:posA}
  Let $m$ and $n$ be positive integers with  $m\ge2$.
  The tuple
  \[
    (w_0,w_1,\dots,w_m) \in \mNC[A_{n-1}]
  \]
  is positive if and only if $w_m(n)=n$, or, equivalently, if $(w_0w_1\cdots
  w_{m-1})(n)=1$. 
\end{proposition}

\begin{proof}
According to \Cref{def:NCpos}, the tuple $(w_0,w_1,\dots,w_m)$ is
positive if and only if $cw_m^{-1}=w_0w_1\cdots w_{m-1}$ has full support in our generators $s_1,\dots,s_{n-1}$ given in~\eqref{eq:WSA}. The
cycles in the disjoint cycle decomposition of~$cw_m^{-1}$ define
a non-crossing partition. Thus, $cw_m^{-1}$ will not have full support if and only if 
$(cw_m^{-1})(n)\ne1$. By acting on the left on both sides of
this relation by the inverse of the
Coxeter element $c=(1,2,\dots,n)$, we see that this is equivalent
to $w_m^{-1}(n)\ne n$. We are interested in the contrapositive:
$cw_m^{-1}$ has full support if and only if $w_m(n)=n$.
\end{proof}

Next we recall Armstrong's bijection between $\mNC[A_{n-1}]$ and 
Edelman's $m$-divisible non-crossing partitions of
$\{1,2,\dots,mn\}$ (with the classical definition of ``non-crossing" as
recalled in \Cref{foot:NC}), which we denote here by
$\mNCA$.
We write $\NCA$ for the set of \emph{all} non-crossing partitions of $\{1,2,\dots,N\}$ and observe that
\[
  \mNCA = \{ \pi \in \NCA[mn] \mid \text{all block sizes of $\pi$ are divisible by~$m$}\}\,.
\]
Given an element $(w_0,w_1,\dots,w_m)\in \mNC[A_{n-1}]$,
the bijection, $\Nam{A_{n-1}}m$ say, 
from\break \cite[Thm.~4.3.8]{Arm2006} 
works as follows: for $i=1,2,\dots,m$, let $\ta_{m,i}$ be
the transformation which maps a permutation $w\in A_{n-1}$ to a
permutation $\ta_{m,i}(w)\in \mathfrak{S}_{mn}$ by letting
$$(\ta_{m,i}(w))(mk+i-m)=mw(k)+i-m,\quad k=1,2,\dots,n,$$ 
and 
$(\ta_{m,i}(w))(l)=l$ for all $l\nequiv i$~(mod~$m$).
With the choice of Coxeter element $c=(1,2,\dots,n)$, 
the announced bijection maps 
$(w_0,w_1,\dots,w_m)\in \mNC[A_{n-1}]$ to
\begin{equation} \label{eq:Na}
\Nam{A_{n-1}}m(w_0,w_1,\dots,w_m)=
(1,2,\dots,mn)\,(\ta_{m,1}(w_1))^{-1}\,(\ta_{m,2}(w_2))^{-1}\,\cdots\,
(\ta_{m,m}(w_m))^{-1},
\end{equation}
where the cycles in the disjoint cycle decomposition correspond to
the blocks in the non-crossing partition in
$\mNCA$.
We refer the reader to \cite[Sec.~4.3.2]{Arm2006} for the details.
For example, let $n=7$, $m=3$, $w_0=(4,5,6)$,
$w_1=(3,6)$,
$w_2=(1,7)$, and
$w_3=(1,2,6)$. Then $(w_0,w_1,w_2,w_3)$ is mapped to
\begin{align*}
\Nam{A_{6}}3(w_0,w_1,w_2,w_3)&=(1,2,\dots,21)\,(7,16)\,
(2,20)\,(18,6,3)\\
&=(1,2,21)\,(3,19,20)\,(4,5,6)\,(7,17,18)\,(8,9,\dots,16).
\end{align*}

Armstrong \cite[Thm.~4.3.13]{Arm2006} has shown that the cycle
structure of the first
component of an $m$-divisible non-crossing partition
determines the block structure of its image under his bijection.

\begin{proposition} \label{prop:block}
Let $(w_0,w_1,\dots,w_m)\in\mNC[A_{n-1}]$. The non-crossing partition
$\Nam{A_{n-1}}m(w_0,w_1,\dots,w_m)\in\mNCA$ has $b_i$ blocks of
size~$mi$ if and only if $w_0$ has $b_i$ cycles of length~$i$
in its disjoint cycle decomposition. 
\end{proposition}

In the earlier example, we have $w_0=(4,5,6)=(1)(2)(3)(7)(4,5,6)$ which has one cycle
of length~3 and four cycles of length~1. Indeed, the image of
$(w_0,w_1,w_2,w_3)$ in the example has one block of size $3\cdot 3=9$
and four blocks of size $3\cdot 1=3$.

\begin{remark} \label{rem:id}
A simple consequence of \Cref{prop:block} is that the image of
the set of all non-crossing partitions in $\mNC[A_{n-1}]$ of the form
$(\ep,w_1,\dots,w_m)$ under the map $\Nam{A_{n-1}}m$ is
the set of all non-crossing partitions in $\mNCA$ in which {\em all}
blocks have size~$m$.
\end{remark}

\begin{theorem} \label{prop:1}
Let $m$ and $n$ be positive integers with $m\ge2$.
The image under $\Nam{A_{n-1}}m$ of the positive $m$-divisible
non-crossing partitions in $\mNC[A_{n-1}]$ are those $m$-divisible
non-crossing partitions in $\mNCA$ where $1$ and
$mn$ are in the same block.
\end{theorem}

\begin{proof}
First, let $(w_0,w_1,\dots,w_m)\in \mNC[A_{n-1}]$.
By definition, we know that\break $w_m(n)=n$, and hence
$\ta_{m,m}(w_m))^{-1}(mn)=mn$. Consequently,
\begin{align*}
\Big(\Nam{A_{n-1}}m(&w_0,w_1,\dots,w_m)\Big)(mn)\\
&=
\big((1,2,\dots,mn)\,(\ta_{m,1}(w_1))^{-1}\,(\ta_{m,2}(w_2))^{-1}\,\cdots\,
(\ta_{m,m}(w_m))^{-1}\big)(mn)\\
&=
\big((1,2,\dots,mn)\big)(mn)=1.
\end{align*}
Translated to (combinatorial) non-crossing partitions, 
this means that $mn$ and $1$
belong to the same block, which is what we claimed.

Conversely, let $\pi\in \mNCA$ be an
$m$-divisible non-crossing partition in which $1$ and $mn$ are in the
same block. In other words, if $\pi$
is interpreted as a permutation, $\pi(mn)=1$. Let
$(w_0,w_1,\dots,w_m)$ be the element of $\mNC[A_{n-1}]$ such that
$$
\Nam{A_{n-1}}m(w_0,w_1,\dots,w_m)=\pi.
$$
We must have
$$
\Big(\Nam{A_{n-1}}m(w_0,w_1,\dots,w_m)\Big)(mn)=1.
$$
The definition \eqref{eq:Na} of $\Nam{A_{n-1}}m$
and the fact that $(\ta_{m,i}(w_i))^{-1}$ leaves multiples of~$m$
fixed as long as $1\le i\le m-1$ together imply that 
$(\ta_{m,m}(w_m))^{-1}(mn)=mn$, or, equivalently, that
$w_m(n)=n$. By \Cref{prop:posA}, this means that
$(w_0,w_1,\dots,w_m)$ is positive.
\end{proof} 

We let $\mNCAPlus$ denote the ``positive" elements of $\mNCA$, that is, those for which~$1$ and~$mn$ are contained in the same block.
In accordance with our previous conventions, we write $\NCAPlus$ for the set of \emph{all} non-crossing partitions of $\{1,2,\dots,N\}$ for which~$1$ and~$N$ are in the same block.
We observe that
\[
  \mNCAPlus = \{ \pi \in \NCAPlus[mn] \mid \text{all block sizes of $\pi$ are divisible by~$m$}\}\,.
\]
We furthermore call the block containing both $1$ and $mn$ the \defn{special block} of the partition under consideration.

\subsection{Combinatorial realisation of the positive Kreweras maps}
\label{sec:mapA}

Next we translate the positive Kreweras map $\Krewplustilde$
into a ``rotation action" on the positive elements of
$\mNCA$. In order to do so, we need to explicitly describe 
the decomposition~\eqref{eq:LR} in type~$A$.

\begin{lemma} \label{lem:LRA}
Let~$w$ be an element of $\NC[A_{n-1}]$, and let
$w^{-1}(n)=a$. Then
$$w=w^Lw^R,\quad \quad \text {where $w^R=(a,n)$}.$$
In particular, we have $w^L(n)=n$.
\end{lemma}

\begin{proof}
The last $n-1$ reflections in $\reflR$ given in \eqref{eq:refls}
are $(i,n)$, $i=1,2,\dots,n-1$, in this order. Thus, the factorisation
$w=w'\circ(a,n)$ for $a=w^{-1}(n)$ is reduced. This implies that
$w^L=w'$ and $w^R=(a,n)$, as desired.

The final claim of the lemma is obvious.
\end{proof}

%
%

Next, we translate the positive Kreweras map $\Krewplustilde$ from
\Cref{def:positivekreweras} for type $A_{n-1}$ into combinatorial
language. 

\begin{theorem} \label{prop:2}
Let $m$ and $n$ be positive integers with $m\ge2$.
Under the bijection $\Nam{A_{n-1}}m$, the map $\Krewplustilde^{(m)}$ translates into
the following map~$\rotA$ on $\mNCAPlus$:
if $mn-1$ is in the same block as $1$ and $mn$ of an element~$\pi$
of~$\mNCAPlus$, then $\rotA$ 
rotates all blocks of~$\pi$
by one unit in clockwise direction. On the other hand, if $mn-1$ is
in a different block, say $\{b_1,b_2,\dots,mn-1\}$ and
$\{1,\dots,a,mn\}$ are two different blocks, where
$$1<\dots<a<b_1<b_2\dots<mn-1<mn,$$
then the image of~$\pi$ under~$\rotA$ is the partition which
contains the blocks $\{1,b_2+1,b_3+,\dots,\break mn\}$ and $\{2,\dots,a+1,b_1+1\}$
and the blocks which arise from the remaining blocks of~$\pi$ by
rotating them by one unit in clockwise direction.
\end{theorem}

\begin{figure}
  \centering
  \begin{tikzpicture}[scale=1]
    \polygon{(-3.5,0)}{obj}{45}{2.5}
      {1,,,,,,\hspace{5pt},,,,,,,,,\hspace{5pt},,,,a,,,,,b_1,,,,b_2,,,,,,,\hspace{15pt},,,,,,mn-1\hspace{4pt},,mn\hspace{2pt},}

    \draw[line width=2.5pt,black] (obj1) to[bend left=50] (obj44);

    \draw[dash pattern=on 2pt off 2pt on 2pt off 2pt on 2pt off 2pt on 2pt off 2pt on 2pt off 2pt on 2pt off 46pt on 2pt off 2pt on 2pt off 2pt on 2pt off 2pt on 2pt off 2pt] (obj7) to[bend right=50] (obj16);
    
    \draw (obj16) to[bend right=50]
          (obj20) to[bend left=10]
          (obj44) to[bend right=50]
          (obj1)  to[bend right=50]
          (obj7);

    \draw[dash pattern=on 2pt off 2pt on 2pt off 2pt on 2pt off 2pt on 2pt off 2pt on 2pt off 2pt on 2pt off 38pt on 2pt off 2pt on 2pt off 2pt on 2pt off 2pt on 2pt off 2pt] (obj29) to[bend right=50] (obj36);
    
    \draw (obj36) to[bend right=50]
          (obj42) to[bend left=10]
          (obj25) to[bend right=50] (obj29);

    \draw[dotted] (obj2)  to[bend right=50] (obj6);
    \draw[dotted] (obj37) to[bend right=50] (obj41);
    \node at ($($0.92*($(obj4)+(3.5,0)$)$)-(3.5,0)$) {\tiny$X$};
    \node at ($($0.92*($(obj39)+(3.5,0)$)$)-(3.5,0)$) {\tiny$Y$};

    \node[inner sep=0pt] at (0,0) {$\mapsto$};

    \polygon{(3.5,0)}{obj}{45}{2.5}
      {1,,2,,,,,,\hspace{5pt},,,,,,,,,\hspace{5pt},,,,a+1,,,,,b_1+1\hspace{8pt},,,,b_2+1\hspace{10pt},,,,,,,\hspace{15pt},,,,,,mn\hspace{2pt},}

    \draw[line width=2.5pt,black] (obj1) to[bend left=50] (obj44);

    \draw[dash pattern=on 2pt off 2pt on 2pt off 2pt on 2pt off 2pt on 2pt off 2pt on 2pt off 2pt on 2pt off 46pt on 2pt off 2pt on 2pt off 2pt on 2pt off 2pt on 2pt off 2pt] (obj9) to[bend right=50] (obj18);
    
    \draw (obj18) to[bend right=50]
          (obj22) to[bend right=50]
          (obj27) to[bend left=10]
          (obj3)  to[bend right=50]
          (obj9);

    \draw[dash pattern=on 2pt off 2pt on 2pt off 2pt on 2pt off 2pt on 2pt off 2pt on 2pt off 2pt on 2pt off 28pt on 2pt off 2pt on 2pt off 2pt on 2pt off 2pt on 2pt off 2pt] (obj31) to[bend right=20] (obj38);
    
    \draw (obj38) to[bend right=50]
          (obj44) to[bend right=50]
          (obj1)  to[bend left=10]
          (obj31);

    \draw[dotted] (obj4)  to[bend right=50] (obj8);
    \draw[dotted] (obj39) to[bend right=50] (obj43);
    \node at ($($0.92*($(obj6)-(3.5,0)$)$)+(3.5,0)$) {\tiny$X'$};
    \node at ($($0.92*($(obj41)-(3.5,0)$)$)+(3.5,0)$) {\tiny$Y'$};

    \end{tikzpicture}
  \caption{The action of the pseudo-rotation $\rotA$}
\label{fig:7}
\end{figure}

\begin{figure}
  \centering
  \begin{tikzpicture}[scale=1]
    \polygon{(-3.5,0)}{objleft}{24}{2.5}
      {1,2,3,4,5,6,7,8,9,10,11,12,13,14,15,16,17,18,19,20,21,22,23,24}
     \draw[line width=2.5pt,black] (objleft1) to[bend left] (objleft24);
     \draw[fill=black,fill opacity=0.1] (objleft1) to[bend right=30] (objleft5) to[bend right=30] (objleft6) to[bend right=30] (objleft13) to[bend right=30] (objleft14) to[bend right=30] (objleft24) to[bend right=30] (objleft1);
     \draw[fill=black,fill opacity=0.1] (objleft2) to[bend right=30] (objleft3) to[bend right=30] (objleft4) to[bend left=30] (objleft2);
     \draw[fill=black,fill opacity=0.1] (objleft7) to[bend right=30] (objleft8) to[bend right=30] (objleft12) to[bend left=30] (objleft7);
     \draw[fill=black,fill opacity=0.1] (objleft9) to[bend right=30] (objleft10) to[bend right=30] (objleft11) to[bend left=30] (objleft9);
     \draw[fill=black,fill opacity=0.1] (objleft15) to[bend right=30] (objleft19) to[bend right=30] (objleft23) to[bend left=30] (objleft15);
     \draw[fill=black,fill opacity=0.1] (objleft16) to[bend right=30] (objleft17) to[bend right=30] (objleft18) to[bend left=30] (objleft16);
     \draw[fill=black,fill opacity=0.1] (objleft20) to[bend right=30] (objleft21) to[bend right=30] (objleft22) to[bend left=30] (objleft20);

     \node[inner sep=0pt] at (0,0) {$\mapsto$};

     \polygon{(3.5,0)}{objright}{24}{2.5}
      {1,2,3,4,5,6,7,8,9,10,11,12,13,14,15,16,17,18,19,20,21,22,23,24}
     \draw[line width=2.5pt,black] (objright1) to[bend left] (objright24);
     \draw[fill=black,fill opacity=0.1] (objright2) to[bend right=30] (objright6) to[bend right=30] (objright7) to[bend right=30] (objright14) to[bend right=30] (objright15) to[bend right=30] (objright16) to[bend right=30] (objright2);
     \draw[fill=black,fill opacity=0.1] (objright3) to[bend right=30] (objright4) to[bend right=30] (objright5) to[bend left=30] (objright3);
     \draw[fill=black,fill opacity=0.1] (objright8) to[bend right=30] (objright9) to[bend right=30] (objright13) to[bend left=30] (objright8);
     \draw[fill=black,fill opacity=0.1] (objright10) to[bend right=30] (objright11) to[bend right=30] (objright12) to[bend left=30] (objright10);
     \draw[fill=black,fill opacity=0.1] (objright17) to[bend right=30] (objright18) to[bend right=30] (objright19) to[bend left=30] (objright17);
     \draw[fill=black,fill opacity=0.1] (objright20) to[bend right=30] (objright24) to[bend right=30] (objright1) to[bend left=30] (objright20);
     \draw[fill=black,fill opacity=0.1] (objright21) to[bend right=30] (objright22) to[bend right=30] (objright23) to[bend left=30] (objright21);
    \end{tikzpicture}
  \caption{The action of the pseudo-rotation $\rotA$ in a concrete example}
\label{fig:1}
\end{figure}
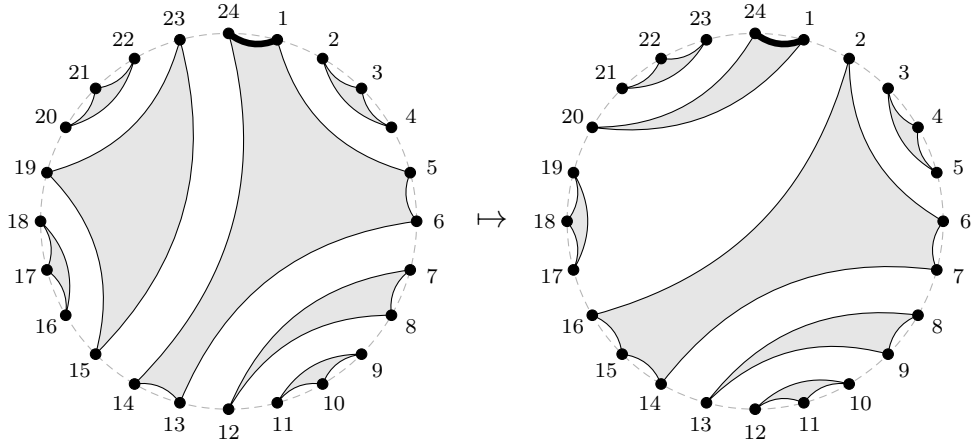

\begin{remark}
\Cref{fig:7} provides a schematic illustration of 
the construction in the statement of 
the above theorem.
\Cref{fig:1} then provides a concrete example for the parameters $m=3$, $n=8$, $a=14$,
$b_1=15$, and $b_2=19$.
Since this map~$\rotA$ essentially acts as rotation, possibly except for the ``neighbourhood'' of~$1$, we call it frequently a \defn{pseudo-rotation}.
\end{remark}

\begin{proof}[Proof of \Cref{prop:2}]
The only claim which needs proof is the one on the image of the
blocks $\{b_1,b_2,\dots, mn-1\}$ and $\{1,\dots,a,mn\}$ in the case where
$mn-1$ is not in the same block as $1$ and $mn$. By the structure of
$m$-divisible non-crossing partitions, we have 
$$
b_1\equiv b_2-1\equiv a-(m-1)\equiv 0\pmod m.
$$
Let us write 
$b_1=mB_1$,
$b_2=mB_2+1$, and
$a=mA-1$, for suitable integers 
$B_1,B_2,A$. If $(w_0,w_1,\dots,w_m)$ denotes the
element of $\mNCPlus[A_{n-1}]$ which corresponds to~$\pi$ under the
bijection $\Nam{A_{n-1}}m$, then we have
\begin{align*}
w_{m-1}(n)&=A,\\
w_{m-1}(B_1)&=n,\\
w_m(B_2)&=B_1,\\
w_m(n)&=n.
\end{align*}
Consequently, we have $w_{m-1}^R=(B_1,n)$ and
\begin{align*}
cw_{m-1}^Rw_mc^{-1}(1)&=B_1+1,\\
cw_{m-1}^Rw_mc^{-1}(B_2+1)&=1,\\
w_{m-1}^L(B_1)&=A,\\
w_{m-1}^L(n)&=n.
\end{align*}
In view of the definition, it follows that
$\Nam{A_{n-1}}m\big(\Krewplustilde(w_0,w_1,\dots,w_m)\big)$ maps 
$a+1=mA$ to $mB_1+1=b_1+1$,
$b_1+1=mB_1+1$ to $m\cdot 1-(m-1)+1=2$,
$1=m\cdot 1-(m-1)$ to $m(B_2+1)-(m-1)+1=b_2+1$,
as we claimed.
\end{proof}

\subsection{The combinatorial positive Kreweras maps}
\label{sec:combA}

The definition of $\rotA$ from \Cref{prop:2} to
make sense on an \emph{arbitrary} non-crossing
partition needed that blocks should be of size at least~$2$.
It is a simple matter to extend this so that it
makes sense for non-crossing partitions without any restrictions.
See \Cref{fig:8} for an illustration of the following definition.

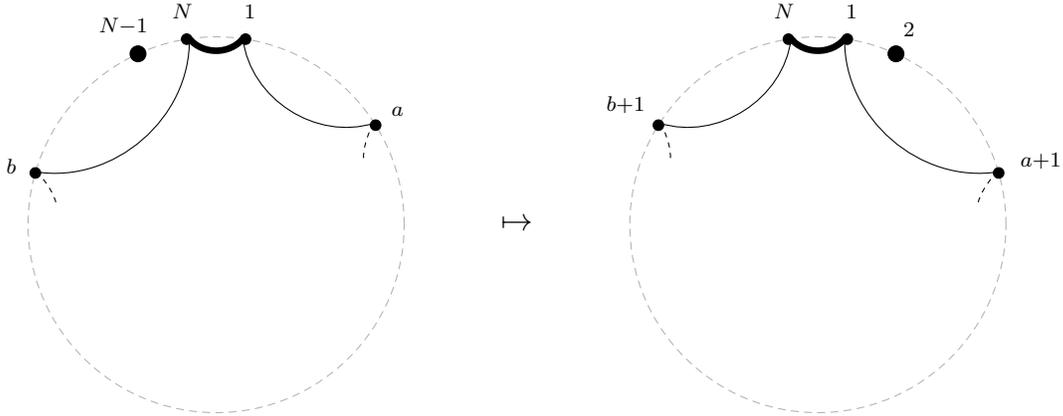
\begin{figure}
\begin{center}
    \begin{tikzpicture}[scale=1]
    \polygonlabel{(-4,0)}{obj}{44}{2.5}
      {1,,,,,,a,,,,,,,,,,,,,,,,,,,,,,,,,,,,b,,,,,,N-1,,N,}

    \draw[line width=2.5pt,black] (obj1) to[bend left=50] (obj43);
    \draw (obj35) to[bend right=50]
          (obj43) to[bend right=50]
          (obj1) to[bend right=50]
          (obj7);

    \draw[black, dash pattern=on 2pt off 2pt on 2pt off 2pt on 2pt off
2pt on 2pt off 35pt] (obj7) to[bend right=50] (obj11);
    \draw[black, dash pattern=on 2pt off 2pt on 2pt off 2pt on 2pt off
2pt on 2pt off 35pt] (obj35) to[bend left=50] (obj31);

    \filldraw[black] (obj41) circle(3pt);

    \node[inner sep=0pt] at (0,0) {$\mapsto$};

    \polygonlabel{( 4,0)}{obj}{44}{2.5}
      {1,,2,,,,,,a+1,,,,,,,,,,,,,,,,,,,,,,,,,,,,b+1,,,,,,N,}

    \draw[line width=2.5pt,black] (obj1) to[bend left=50] (obj43);
    \draw (obj37) to[bend right=50]
          (obj43) to[bend right=50]
          (obj1) to[bend right=50]
          (obj9);

    \draw[black, dash pattern=on 2pt off 2pt on 2pt off 2pt on 2pt off
2pt on 2pt off 35pt] (obj9) to[bend right=50] (obj13);
    \draw[black, dash pattern=on 2pt off 2pt on 2pt off 2pt on 2pt off
2pt on 2pt off 35pt] (obj37) to[bend left=50] (obj33);

    \filldraw[black] (obj3) circle(3pt);

\end{tikzpicture}
\end{center}
\caption{Definition of the pseudo-rotation under presence of 
singleton blocks}
\label{fig:8}
\end{figure}

\begin{definition} \label{def:phiallg}
We define the action of $\rotA$ on a positive non-crossing partition~$\pi$
of $\{1,2,\dots,N\}$ to be given by \Cref{prop:2} with $mn$ replaced by
$N$ (cf.\ \Cref{fig:7,fig:1}), except when the block containing $N-1$ is a
singleton block (in which case \Cref{prop:2} would not make sense).
If $N-1$ forms its own singleton block, then the image of~$\rotA$
is the non-crossing partition which arises from~$\pi$ by replacing the
singleton block $\{N-1\}$ by the singleton block $\{2\}$, and by adding $1$ to
each element of the remaining blocks of~$\pi$, except to $N$ and 
$1$ which remain untouched; see \Cref{fig:8} for an illustration.
\end{definition}

As it turns out, this definition is also in agreement with the
image under $\Nam{A_{n-1}}1$ of the map $\Krewplustilde^{(1)}$ in
\Cref{def:K-alt} for the case where $m=1$.

\begin{theorem} \label{prop:2-m=1}
Let $n$ be a positive integer.
Under the bijection $\Nam{A_{n-1}}1$, the map $\Krewplustilde^{(1)}$
from \Cref{def:K-alt} translates to the combinatorial map~$\rotA$ described in
\Cref{def:phiallg}  with $N=n$.
\end{theorem}

\begin{figure}
  \centering
  \begin{tikzpicture}[scale=1]
    \polygon{(-3.5,0)}{obj}{45}{2.5}
      {1,,,,,,\hspace{5pt},,,,,,,,,\hspace{5pt},,,,a,,,,,b_1,,,,b_2,,,,,,,\hspace{15pt},,,,,,n-1\hspace{10pt},,\hspace{2pt}n,}
    \polygon{(-3.5,0)}{objXXX}{45}{2.5}
      {,,,,,,,,,,,,,,,,,,,,\hspace{1pt},,,\hspace{12pt}\color{red}b_1-1,,\hspace{1pt},,\hspace{1pt},,,,,,,,,,,,,,,\color{red}n-1}

    \draw[line width=2.5pt,black] (obj1) to[bend left=50] (obj44);
    \draw[line width=2.5pt,red] (objXXX21) to[bend right=30] (objXXX24) to[bend right=10] (objXXX43);
    \draw[line width=2.5pt,red] (objXXX26) to[bend right=30] (objXXX28);

    \draw[dash pattern=on 2pt off 2pt on 2pt off 2pt on 2pt off 2pt on 2pt off 2pt on 2pt off 2pt on 2pt off 46pt on 2pt off 2pt on 2pt off 2pt on 2pt off 2pt on 2pt off 2pt] (obj7) to[bend right=50] (obj16);

    \draw (obj16) to[bend right=50]
          (obj20) to[bend left=10]
          (obj44) to[bend right=50]
          (obj1)  to[bend right=50]
          (obj7);

    \draw[dash pattern=on 2pt off 2pt on 2pt off 2pt on 2pt off 2pt on 2pt off 2pt on 2pt off 2pt on 2pt off 38pt on 2pt off 2pt on 2pt off 2pt on 2pt off 2pt on 2pt off 2pt] (obj29) to[bend right=50] (obj36);
    
    \draw (obj36) to[bend right=50]
          (obj42) to[bend left=10]
          (obj25) to[bend right=50] (obj29);

    \draw[dotted] (obj2)  to[bend right=50] (obj6);
    \draw[dotted] (obj37) to[bend right=50] (obj41);
    \node at ($($0.92*($(obj4)+(3.5,0)$)$)-(3.5,0)$) {\tiny$X$};
    \node at ($($0.92*($(obj39)+(3.5,0)$)$)-(3.5,0)$) {\tiny$Y$};

    \node[inner sep=0pt] at (0,0) {$\mapsto$};

    \polygon{(3.5,0)}{obj}{45}{2.5}
      {1,,2,,,,,,\hspace{5pt},,,,,,,,,\hspace{5pt},,,,a+1,\hspace{1pt},,,\hspace{3pt}\color{red}b_1,b_1+1\hspace{5pt},,,,b_2+1\hspace{10pt},,,,,,,\hspace{15pt},,,,,,n\hspace{2pt},}

    \polygon{(3.5,0)}{objXXX}{45}{2.5}
      {,\color{red}1,,,,,,,,,,,,,,,,,,,,,,,,,,\color{red}b_1+1\hspace{20pt},,\hspace{1pt}}

    \draw[line width=2.5pt,black] (obj1) to[bend left=50] (obj44);
    \draw[line width=2.5pt,red] (objXXX30) to[bend left=30] (objXXX28) to[bend left=10] (objXXX2);
    \draw[line width=2.5pt,red] (objXXX23) to[bend right=30] (objXXX26);

    \draw[dash pattern=on 2pt off 2pt on 2pt off 2pt on 2pt off 2pt on 2pt off 2pt on 2pt off 2pt on 2pt off 46pt on 2pt off 2pt on 2pt off 2pt on 2pt off 2pt on 2pt off 2pt] (obj9) to[bend right=50] (obj18);
    
    \draw (obj18) to[bend right=50]
          (obj22) to[bend right=50]
          (obj27) to[bend left=10]
          (obj3)  to[bend right=50]
          (obj9);

    \draw[dash pattern=on 2pt off 2pt on 2pt off 2pt on 2pt off 2pt on 2pt off 2pt on 2pt off 2pt on 2pt off 28pt on 2pt off 2pt on 2pt off 2pt on 2pt off 2pt on 2pt off 2pt] (obj31) to[bend right=20] (obj38);
    
    \draw (obj38) to[bend right=50]
          (obj44) to[bend right=50]
          (obj1)  to[bend left=10]
          (obj31);

    \draw[dotted] (obj4)  to[bend right=50] (obj8);
    \draw[dotted] (obj39) to[bend right=50] (obj43);
    \node at ($($0.92*($(obj6)-(3.5,0)$)$)+(3.5,0)$) {\tiny$X'$};
    \node at ($($0.92*($(obj41)-(3.5,0)$)$)+(3.5,0)$) {\tiny$Y'$};
    \end{tikzpicture}
  \caption{The action of the pseudo-rotation $\rotA$ for $m=1$}
\label{fig:1m=1}
\end{figure}
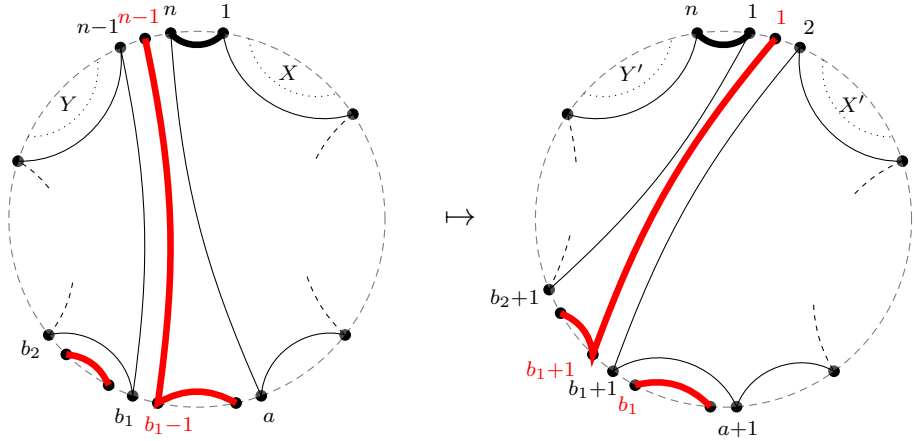

\begin{proof}[Sketch of proof]
Recall that the bijection $\Nam{A_{n-1}}1$ in the case $m=1$ reduces
to the simple map $w\mapsto cw^{-1}$, where due to our
choice~\eqref{eq:WSA} of Coxeter system, $c=(1,2,\dots,n)$.
This map is usually called the
``Kreweras complement". It has an elegant combinatorial realisation
(cf.\ \cite[Fig.~4.3 in Sec.~4.2]{Arm2006}) on which our arguments
will rely.

Let $w=w^L=t_1\cdots t_k$ be an element of $\mNCPlus[A_{n-1}][1]$,
and let $\pi$ be the non-crossing partition in $\NCAPlus[n]$
corresponding to the Kreweras complement $(1,2,\dots,n)w^{-1}$.
According to \Cref{def:K-alt}, the Kreweras
map~$\Krewplustilde^{(1)}$ acts by sending the transposition~$t_i$ to
$c^jt_ic^{-j}$  
with $j$ minimal such that $c^jt_ic^{-j}\in \invc^L$.
Now, since $c=(1,2,\dots,n)$,
this action will always be $(a,b)\mapsto(a+1,b+1)$, except if $b=n-1$.
In the latter case, the action will be $(a,n-1)\mapsto(1,a+2)$
(with $j=2$) if $a<n-2$, and it will be $(n-2,n-1)\mapsto(1,2)$
(with $j=3$). Hence, given the combinatorial realisation of the
Kreweras complement, the
map~$\Krewplustilde^{(1)}$ becomes rotation by one unit in
clockwise direction under the bijection $\Nam{A_{n-1}}1$ as long as
no $t_i$ is of the form $(a,n-1)$. 
In particular, a singleton block $\{k\}$ in~$\pi$ --- which under the Kreweras
complement corresponds to an adjacent transposition $(k-1,k)$ among
the~$t_i$'s --- is mapped to $\{k+1\}$ if $k\le n-1$.

On the other hand, let us assume that there is a $t_i$ with
$t_i=(b_1-1,n-1)$ for some $b_1<n$.
This situation is schematically illustrated in \Cref{fig:1m=1}.
The red arcs there correspond to transpositions~$t_i$.
In particular, among the $t_i$'s there may occur $(a_0,b_1-1)$, for
some $a_0$ with $a_0\le b_1-1$, and $(b_1,b_2-1)$, for some $b_2$ with
$b_1\le b_2-1$ (in both cases, equality corresponds to the
degenerate case where the ``doubled" number does not appear in the $t_i$'s). 
If $b_1<n-1$, then, as discussed just above,
$\Krewplustilde^{(1)}$ maps $(b_1-1,n-1)$ to $(1,b_1+1)$; see
\Cref{fig:1m=1}. In the Kreweras complement~$\pi$, this means that,
before the action of~$\Krewplustilde^{(1)}$, there are two
blocks $\{b_1,b_2,\dots,n-1\}$ and $\{1,\dots,a,n\}$ with
$$1<\dots<a<b_1<b_2\dots<n-1<n;$$
after the action of~$\Krewplustilde^{(1)}$, there are the two
blocks $(1,b_2+1,\dots,n)$ and $(2,\dots,a+1,b_1+1)$.
It should be noted that this transformation of the blocks is
caused by the fact that the overlapping (red) arcs $(a_0,b_1-1)$ and
$(b_1-1,n-1)$ become the non-overlapping arcs $(a_0+1,b_1)$ and
$(b_1+1,1)$ in the image after application of~$\Krewplustilde^{(1)}$,
while the non-overlapping arcs $(b_1-1,n-1)$ and 
$(b_1,b_2-1)$ become the overlapping arcs $(b_1+1,1)$ and
$(b_1+1,b_2)$ in the image.
All other blocks in the combinatorial realisation~$\pi$ of $w=w^L=t_1\cdots
t_k$ are rotated by one unit in clockwise direction.

Finally, if one of the $t_i$'s should equal $(n-2,n-1)$, then this
corresponds to the singleton block $\{n-1\}$ in~$\pi$.
As discussed above, the action
of $\Krewplustilde^{(1)}$ maps $t_i=(n-2,n-1)$ to $(1,2)$,
corresponding to the singleton block $\{2\}$ in the Kreweras complement.

All this is exactly in agreement with \Cref{def:phiallg}.
\end{proof}

We conclude with the following theorem describing the order of the combinatorial\break map~$\rotA$. In view of our previous discussion, it follows from
\Cref{cor:order} for the type~$A_{N-1}$ with $m=1$.
However, since we believe it is interesting for its own sake, we
provide a direct, combinatorial, proof in \Cref{app:order-A}.

\begin{theorem} \label{lem:N-2}
Let $N$ be a positive integer with $N\ge2$.
The order of the map~$\rotA$ acting on the set of positive non-crossing 
partitions~$\pi$ of\/ $\{1,2,\dots,N\}$ is $N-2$.
\end{theorem}

\subsection{Enumeration of positive non-crossing partitions}
\label{sec:enumA}

Here we present our enumeration results for positive non-crossing
partitions in $\mNCAPlus$. Our proofs are based on a generating
function approach and are largely deferred to \Cref{app:GF}.

The most refined and most general
result is \Cref{thm:countingA} which provides a
formula for the number of multichains in the poset of these
non-crossing partitions in which the ranks of the elements of the
multichain as well as the block structure of the bottom element are prescribed. 
By specialisation and summation, we obtain, among others, results on the number of
such partitions with prescribed block structure and with prescribed
rank, and the total number of multichains of given length in
$\mNCAPlus$.

\begin{theorem} \label{thm:countingA}
Let $m,n,l$ be positive integers,
and let $s_1,s_2,\dots,s_l,b_1,b_2,\dots,b_n$ be
non-negative integers with $s_1+s_2+\dots+s_l=n-1$. 
The number of multichains 
$\pi_1\le \pi_2\le \dots\le \pi_{l-1}$ in the poset 
of positive $m$-divisible non-crossing partitions of\/ $\{1,2,\dots,mn\}$
with the property that $\rk(\pi_i)=s_1+s_2+\cdots +s_i$, 
$i=1,2,\dots,l-1$, and that the
number of blocks of size~$mi$ of~$\pi_1$ is~$b_i$, $i=1,2,\dots,n$,
is given by
\begin{equation} \label{eq:multichains-bi-si} 
\frac {mn-b_1-b_2-\dots-b_n} {(mn-1)(b_1+b_2+\dots+b_n)}
\binom {b_1+b_2+\dots+b_n}{b_1,b_2,\dots,b_n}
\binom {mn-1} {s_2}\cdots
\binom {mn-1} {s_l}
\end{equation}
if $b_1+2b_2+\dots+nb_n=n$ and $s_1+b_1+b_2+\dots+b_n=n$, and\/ $0$ otherwise.
\end{theorem}

For the proof of the theorem, see the end of \Cref{app:GF-A}.

The $l=2$ special case of the previous theorem yields an
enumeration formula for positive $m$-divisible non-crossing partitions
with a given block structure.

\begin{corollary} \label{cor:countingA1}
Let $m$ and $n$ be positive integers,
For non-negative integers $b_1,b_2,\dots,b_n$,
the number 
of positive $m$-divisible non-crossing partitions of\/ $\{1,2,\dots,mn\}$
which have exactly $b_i$
blocks of size~$mi$, $i=1,2,\dots,n$,
is given by
\begin{equation} \label{eq:multichains-bi} 
\frac {1} {mn-1}
\binom {b_1+b_2+\dots+b_n}{b_1,b_2,\dots,b_n}
\binom {mn-1} {b_1+b_2+\dots+b_n}
\end{equation}
if $b_1+2b_2+\dots+nb_n=n$, and\/ $0$ otherwise.
\end{corollary}

On the other hand, if we sum Expression~\eqref{eq:multichains-bi-si} over all
$s_2,s_3,\dots,s_l$ with $s_2+s_3+\dots+s_l=n-s_1-1$
using iterated Chu--Vandermonde summation
(see e.g.\ \cite[Sec.~5.1, Eq.~(5.27)]{GrKPAA}), then we obtain
the following generalisation.

\begin{corollary} \label{cor:countingA2}
Let $m,n,l$ be positive integers,
For non-negative integers $b_1,b_2,\dots,b_n$,
the number of multichains 
$\pi_1\le \pi_2\le \dots\le \pi_{l-1}$ in the poset  
of positive $m$-divisible non-crossing partitions of\/ $\{1,2,\dots,mn\}$
for which the number of blocks of size~$mi$ of~$\pi_1$ is~$b_i$,
$i=1,2,\dots,n$,
is given by
\begin{equation} \label{eq:multichains-bia} 
\frac {mn-b_1-b_2-\dots-b_n} {(mn-1)(b_1+b_2+\dots+b_n)}
\binom {b_1+b_2+\dots+b_n}{b_1,b_2,\dots,b_n}
\binom {(l-1)(mn-1)} {b_1+b_2+\dots+b_n-1}
\end{equation}
if $b_1+2b_2+\dots+nb_n=n$, and\/ $0$ otherwise.
\end{corollary}

Summing the result in \Cref{thm:countingA} over all possible block
structures, we arrive at the following result.

\begin{corollary} \label{cor:2}
Let $m,n,l$ be positive integers,
and let $s_1,s_2,\dots,s_l$ be
non-negative integers with $s_1+s_2+\dots+s_l=n-1$. 
The number of multi-chains 
$\pi_1\le \pi_2\le \dots\le \pi_{l-1}$ in the poset 
of positive $m$-divisible non-crossing partitions of\/ $\{1,2,\dots,mn\}$
with the property that $\rk(\pi_i)=s_1+s_2+\cdots +s_i$, 
$i=1,2,\dots,l-1$, 
is given by
\begin{equation} \label{eq:multichains-si} 
\frac {mn-s_2-s_3-\dots-s_l-1} {(mn-1)n}
\binom {n}{s_1}
\binom {mn-1} {s_2}\cdots
\binom {mn-1} {s_l}.
\end{equation}
\end{corollary}

\begin{proof}
As indicated above the statement of the corollary, the expression~\eqref{eq:multichains-si} can be obtained from~\eqref{eq:multichains-bi-si} by summing over all possible~$b_i$'s,
$i=1,2,\dots,n$, with $b_1+2b_2+\dots+nb_n=n$ and 
$s_1+b_1+b_2+\dots+b_n=n$. A simpler way to arrive at the same result
is to go to~\eqref{eq:x^b}, set all~$x_i$'s equal to~$1$, to
see that the number which we want to compute is given by
\begin{multline*}
\frac {mn-s_2-s_3-\dots-s_l-1} {(mn-1)(s_2+s_3+\dots+s_l+1)}
\binom {mn-1}{s_2}\binom {mn-1}{s_3}\cdots\binom {mn-1}{s_{l}}\\
\times
\coef{z^{mn}}
\(\frac {z^m} {1-z^{mi}}\)^{s_2+s_3+\dots+s_l+1}\\
=\frac {mn-s_2-s_3-\dots-s_l-1} {(mn-1)(s_2+s_3+\dots+s_l+1)}
\binom {mn-1}{s_2}\binom {mn-1}{s_3}\cdots\binom {mn-1}{s_{l}}\\
\times
\binom {n-1}{n-s_2-s_3-\dots-s_l-1},
\end{multline*}
which, because of $s_1+s_2+\dots+s_l=n-1$, is easily seen to be
equivalent to~\eqref{eq:multichains-si}.
\end{proof}

Summing the expression in \eqref{eq:multichains-si} over all
possible~$s_i$'s with $s_1+s_2+\dots+s_l=n-1$, we obtain the number
of all multichains of a given length, and, thus, the \emph{zeta
  polynomial} of the poset of positive $m$-divisible non-crossing
partitions of $\{1,2,\dots,mn\}$. As a special case, the \emph{total
  number} of positive $m$-divisible non-crossing
partitions of $\{1,2,\dots,mn\}$ falls out.

\begin{corollary} \label{cor:3}
Let $m,n,l$ be positive integers,
The number of multichains 
$\pi_1\le \pi_2\le \dots\le \pi_{l-1}$ in the poset 
of positive $m$-divisible non-crossing partitions of\/ $\{1,2,\dots,mn\}$
is given by
\begin{equation} \label{eq:multichains-l} 
\frac {1+(l-1)(m-1)} {n-1}
\binom {n-1+(l-1)(mn-1)} {n-2}.
\end{equation}
In particular, the total number of elements in this poset equals
\begin{equation} \label{eq:total} 
\frac {1} {n}
\binom {(m+1)n-2} {n-1}.
\end{equation}
\end{corollary}

\begin{proof}
We compute 
\begin{align*}
\sum_{s_1+\dots+s_l=n-1}&
\frac {mn-s_2-s_3-\dots-s_l-1} {(mn-1)n}
\binom {n}{s_1}
\binom {mn-1} {s_2}\cdots
\binom {mn-1} {s_l}\\
&=
\sum_{s_1=0}^{n-1}
\frac {mn-(n-s_1-1)-1} {(mn-1)n}
\binom {n}{s_1}
\binom {(l-1)(mn-1)} {n-s_1-1}\\
&=
\sum_{s_1=0}^{n-1}
\Bigg(\frac {1} {n}
\binom {n}{s_1}
\binom {(l-1)(mn-1)} {n-s_1-1}\\
&\kern2cm
-
\frac {(l-1)(mn-1)} {(mn-1)n}
\binom {n}{s_1}
\binom {(l-1)(mn-1)-1} {n-s_1-2}\Bigg)\\
&=
\frac {1} {n}
\binom {n+(l-1)(mn-1)}{n-1}
-
\frac {l-1} {n}
\binom {(n-1)+(l-1)(mn-1)}{n-2},
\end{align*}
where, to obtain the second line and the last line, 
we used again the Chu--Vandermonde summation formula.
Upon little manipulation, the last expression can be converted\break into~\eqref{eq:multichains-l}. Formula~\eqref{eq:total} results from~\eqref{eq:multichains-l} by setting $l=2$.
\end{proof}

The formula in \Cref{cor:3} counts multichains of~$l-1$ elements.
For $m=1$, this formula thus equals the zeta polynomial of the poset of positive non-crossing partitions of $\{1,2,\dots,n\}$
According to~\cite[Lem.~2.2(iii)]{MR0595933}, evaluating the zeta polynomial at $l = -1$ yields the M\"obius function.
We record this in the following corollary.

\begin{corollary}
\label{cor:moebiusA}
  The M\"obius function between bottom and top element of the poset of positive non-crossing partitions of $\{1,2,\dots,n\}$ is given by
  \begin{equation}
    \frac {(-1)^{n-2}} {n-1}
    \binom {2n-4} {n-2}.
  \end{equation}
\end{corollary}

\subsection{Characterisation of pseudo-rotationally invariant elements}
\label{sec:rotA}

In this subsection, we describe the elements of
$\NCAPlus$ that are invariant under a
power of~$\rotA$. To be specific,
we call such an element
\defn{$\rrr$-pseudo-rotationally invariant} if it is invariant
under~$\rotA^{(N-2)/\rrr}$ for some positive integer~$\rrr$ that divides~$N-2$.
We achieve the description
by providing two constructions of such elements
together with \Cref{lem:allA} which then tells us how to obtain all
such invariant elements from these two constructions. 

In order to be able to describe the two constructions, we need to
introduce a few notions. Given a sequence $a_1,a_2,\dots,a_k$, we
shall frequently consider the cyclic rotation of these elements and
invariably denote it by~$\rot$, so that $\rot(a_i)=a_{i+1}$ and the
indices are taken modulo~$k$. (Although $\rot$ depends on the
sequence $a_1,a_2,\dots,a_k$, we do not indicate this in the notation
since the sequence will always be clear from the context.)
We call a block~$B$ of a non-crossing partition of
$\{a_1,a_2,\dots,a_k\}$ \defn{central\/} if $B=\rot^l(B)$ for
some~$l$ with $1\le l<k$. See \Cref{fig:central} for an example
with $k=20$ and $l=5$.

\begin{figure}
  \centering
 \begin{tikzpicture}[scale=1]
    \polygonnew{(0,0)}{obj}{20}{3}
      {a_1,a_2,a_3,a_4,a_5,a_6,a_7,a_8,a_9,a_{10},a_{11},a_{12},a_{13},a_{14
     },a_{15},a_{16},a_{17},a_{18},a_{19},a_{20}}{1.1}

     \draw[fill=black,fill opacity=0.1] (obj2) to[bend right=15]
(obj4) to[bend right=15] (obj7) to[bend right=15] (obj9) to[bend
right=15] (obj12) to[bend right=15] (obj14) to[bend right=15] (obj17)
to[bend right=15] (obj19) to[bend right=15] (obj2);
    \end{tikzpicture}
  \caption{A central block}
\label{fig:central}
\end{figure}
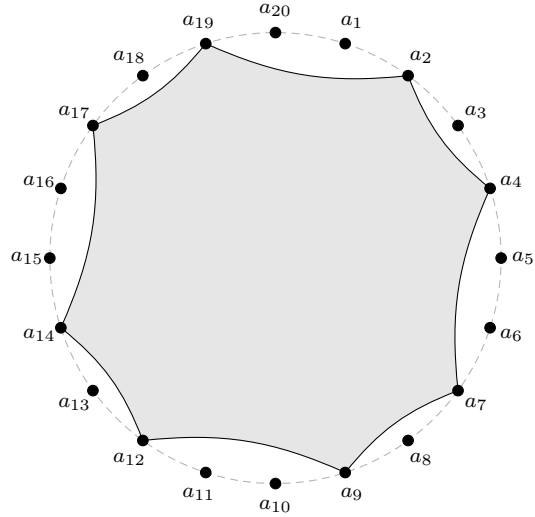

\medskip
We now present two constructions of (pseudo-rotationally invariant,
as it will turn out in \Cref{lem:allA}) positive non-crossing partitions.
In the sequel, $\rrr $ will always be a positive integer dividing $N-2$
and different from $1$.

\begin{figure}
  \centering
  \begin{tikzpicture}[scale=1]
    \polygon{(0,0)}{obj}{30}{2.5}
      {1,2,3,4,5,6,7,8,9,10,11,12,13,14,15,16,17,18,19,20,21,22,23,24,25,26,27,28,29,30}
     \draw[line width=2.5pt,black] (obj1) to[bend left] (obj30);
     \draw[fill=black,fill opacity=0.1] (obj1) to[bend right=30] (obj2) to[bend right=30] (obj9) to[bend right=30] (obj16) to[bend right=30] (obj23) to[bend right=30] (obj30) to[bend right=30] (obj1);
     \draw[fill=black,fill opacity=0.1] (obj3) to[bend right=30] (obj7) to[bend right=30] (obj8) to[bend left=30] (obj3);
     \draw[fill=black,fill opacity=0.1] (obj4) to[bend right=30] (obj5) to[bend right=30] (obj6) to[bend left=30] (obj4);
     \draw[fill=black,fill opacity=0.1] (obj10) to[bend right=30] (obj14) to[bend right=30] (obj15) to[bend left=30] (obj10);
     \draw[fill=black,fill opacity=0.1] (obj11) to[bend right=30] (obj12) to[bend right=30] (obj13) to[bend left=30] (obj11);
     \draw[fill=black,fill opacity=0.1] (obj17) to[bend right=30] (obj21) to[bend right=30] (obj22) to[bend left=30] (obj17);
     \draw[fill=black,fill opacity=0.1] (obj18) to[bend right=30] (obj19) to[bend right=30] (obj20) to[bend left=30] (obj18);
     \draw[fill=black,fill opacity=0.1] (obj24) to[bend right=30] (obj28) to[bend right=30] (obj29) to[bend left=30] (obj24);
     \draw[fill=black,fill opacity=0.1] (obj25) to[bend right=30] (obj26) to[bend right=30] (obj27) to[bend left=30] (obj25);
    \end{tikzpicture}
  \caption{A $4$-pseudo-rotationally invariant non-crossing partition in
$\NCAPlus[30]$ (and also in $\mNCAPlus[9][3]$)}
\label{fig:3}
\end{figure}
\begin{figure}
  \centering
  \begin{tikzpicture}[scale=1]
    \polygon{(0,0)}{obj}{30}{2.5}
      {1,2,3,4,5,6,7,8,9,10,11,12,13,14,15,16,17,18,19,20,21,22,23,24,25,26,27,28,29,30}
     \draw[line width=2.5pt,black] (obj1) to[bend left] (obj30);
     \draw[fill=black,fill opacity=0.1] (obj3) to[bend right=30] (obj4) to[bend right=30] (obj11) to[bend right=30] (obj18) to[bend right=30] (obj25) to[bend right=30] (obj26) to[bend right=30] (obj3);
     \draw[fill=black,fill opacity=0.1] (obj5) to[bend right=30] (obj9) to[bend right=30] (obj10) to[bend left=30] (obj5);
     \draw[fill=black,fill opacity=0.1] (obj6) to[bend right=30] (obj7) to[bend right=30] (obj8) to[bend left=30] (obj6);
     \draw[fill=black,fill opacity=0.1] (obj12) to[bend right=30] (obj16) to[bend right=30] (obj17) to[bend left=30] (obj12);
     \draw[fill=black,fill opacity=0.1] (obj13) to[bend right=30] (obj14) to[bend right=30] (obj15) to[bend left=30] (obj13);
     \draw[fill=black,fill opacity=0.1] (obj19) to[bend right=30] (obj23) to[bend right=30] (obj24) to[bend left=30] (obj19);
     \draw[fill=black,fill opacity=0.1] (obj20) to[bend right=30] (obj21) to[bend right=30] (obj22) to[bend left=30] (obj20);
     \draw[fill=black,fill opacity=0.1] (obj27) to[bend right=30] (obj28) to[bend right=30] (obj29) to[bend left=30] (obj27);
     \draw[fill=black,fill opacity=0.1] (obj30) to[bend right=30] (obj1) to[bend right=30] (obj2) to[bend left=30] (obj30);
    \end{tikzpicture}
  \caption{Another $4$-pseudo-rotationally invariant non-crossing partition in
$\NCAPlus[30]$ (and also in $\mNCAPlus[9][3]$)}
\label{fig:4}
\end{figure}

\medskip
\noindent

{\sc Construction 1.}
We start with a
non-crossing partition of $\{2,3,\dots,N-1\}$ that is invariant under
$\rot^{(N-2)/\rrr}$ (with $\rot$ acting cyclically on $2<3<\dots<N-1$) and that
contains $2$ in the central block. We also allow the
degenerate case that the central block is empty.
Then we add $1$ and $N$ to the central block, thereby creating a
positive non-crossing partition. If we are in
the degenerate case where we started with an empty central block, the
construction has to be understood in the sense that a block $\{1,N\}$
is added to the $\rrr $-pseudo-rotationally invariant non-crossing partition of
$\{2,3,\dots,N-1\}$ in which $\{2,3,\dots,\frac {N-2} {\rrr }+1\}$ is
a union of blocks.
\Cref{fig:3,fig:4,fig:22} illustrate this construction.
In the first two figures, we chose $\rrr=4$ and $N=30$.
In order to obtain the non-crossing partition in \Cref{fig:3},
one starts with the non-crossing partition
\begin{multline*}
\{\{2,9,16,23\},\{3,7,8\},\{4,5,6\},
\{10,14,15\},\{11,12,13\},\\
\{17,21,22\},\{18,19,20\},
\{24,28,29\},\{25,26,27\}\}
\end{multline*}
of $\{2,3,\dots,29\}$, which is indeed invariant under~$\rot^{(30-2)/4}=\rot^7$.
Then one adds $1$ and $30$ to the
central block $\{2,9,16,23\}$ to obtain the partition shown in
\Cref{fig:3}. If one applies $\rotA$ twice to this
partition, then the result is the partition in \Cref{fig:4}.

\begin{figure}
\begin{center}
 \begin{tikzpicture}[scale=1]
    \polygon{(-3.5,0)}{objleft}{32}{2.5}
      {1,2,3,4,5,6,7,8,9,10,11,12,13,14,15,16,17,18,19,20,21,22,23,24,25,26,
     27,28,29,30,31,32}
     \draw[line width=2.5pt,black] (objleft1) to[bend left] (objleft32);
     \draw (objleft3) to[bend right] (objleft4);
     \draw (objleft6) to[bend right] (objleft7);
     \draw (objleft13) to[bend right] (objleft14);
     \draw (objleft16) to[bend right] (objleft17);
     \draw (objleft23) to[bend right] (objleft24);
     \draw (objleft26) to[bend right] (objleft27);
     \draw (objleft32) to[bend right] (objleft1);

     \draw[fill=black,fill opacity=0.1] (objleft2) to[bend right=30]
(objleft5) to[bend right=30] (objleft8) to[bend left=30] (objleft2);
     \draw[fill=black,fill opacity=0.1] (objleft9) to[bend right=30]
(objleft10) to[bend right=30] (objleft11) to[bend left=30] (objleft9);
     \draw[fill=black,fill opacity=0.1] (objleft12) to[bend right=30]
(objleft15) to[bend right=30] (objleft18) to[bend left=30]
     (objleft12);
      \draw[fill=black,fill opacity=0.1] (objleft19) to[bend right=30]
(objleft20) to[bend right=30] (objleft21) to[bend left=30]
(objleft19);
     \draw[fill=black,fill opacity=0.1] (objleft22) to[bend right=30]
(objleft25) to[bend right=30] (objleft28) to[bend left=30]
(objleft22);
     \draw[fill=black,fill opacity=0.1] (objleft29) to[bend right=30]
(objleft30) to[bend right=30] (objleft31) to[bend left=30]
(objleft29);

    \polygon{(3.5,0)}{objright}{32}{2.5}
      {1,2,3,4,5,6,7,8,9,10,11,12,13,14,15,16,17,18,19,20,21,22,23,24,25,26,
     27,28,29,30,31,32}
     \draw[line width=2.5pt,black] (objright1) to[bend left] (objright32);
     \draw (objright4) to[bend right] (objright5);
     \draw (objright7) to[bend right] (objright8);
     \draw (objright14) to[bend right] (objright15);
     \draw (objright17) to[bend right] (objright18);
     \draw (objright24) to[bend right] (objright25);
     \draw (objright27) to[bend right] (objright28);
 
     \draw (objright30) to[bend right] (objright2);

     \draw[fill=black,fill opacity=0.1] (objright3) to[bend right=30]
(objright6) to[bend right=30] (objright9) to[bend left=30]
(objright3);
     \draw[fill=black,fill opacity=0.1] (objright10) to[bend right=30]
(objright11) to[bend right=30] (objright12) to[bend left=30]
(objright10);
     \draw[fill=black,fill opacity=0.1] (objright13) to[bend right=30]
(objright16) to[bend right=30] (objright19) to[bend left=30]
(objright13);
     \draw[fill=black,fill opacity=0.1] (objright20) to[bend right=30]
(objright21) to[bend right=30] (objright22) to[bend left=30]
(objright20);
     \draw[fill=black,fill opacity=0.1] (objright23) to[bend right=30]
(objright26) to[bend right=30] (objright29) to[bend left=30]
     (objright23);
      \draw[fill=black,fill opacity=0.1] (objright31) to[bend right=30]
(objright32) to[bend right=30] (objright1) to[bend left=30]
(objright31);
 \end{tikzpicture}
\end{center}
\caption{Two $3$-pseudo-rotationally invariant non-crossing partitions in
$\NCAPlus[32]$ (and also in $\mNCAPlus[31][1]$)}
\label{fig:22}
\end{figure}

An example of the degenerate case where one starts with a rotationally
symmetric non-crossing partition of $\{2,3,\dots,N-1\}$ with empty
central block is given in \Cref{fig:22}. The parameter values in
the figure are $N=32$ and $\rho = 3$.
In this example we start with the $3$-rotationally invariant
non-crossing partition 
\begin{multline*}
\{\{2,5,8\},\{3,4\},\{6,7\},\{9,10,11\},
\{12,15,18\},\{13,14\},\{16,17\},\{19,20,21\},\\\{22,25,28\},
\{23,24\},\{26,27\},\{29,30,31\}\}
\end{multline*}
of $\{2,3,\dots,31\}$
with empty central block. Adding $\{32,1\}$ as central block,
we obtain the $3$-pseudo-rotationally invariant non-crossing
partition on the left of \Cref{fig:22}. The right part of the figure
shows the result of one application of the pseudo-rotation~$\rotA$.

\begin{figure}
  \centering
  \begin{tikzpicture}[scale=1]
    \polygon{(0,0)}{obj}{36}{2.5}
      {1,2,3,4,5,6,7,8,9,10,11,12,13,14,15,16,17,18,19,20,21,22,23,24,25,26,27,28,29,30,31,32,33,34,35,36}
     \draw[line width=2.5pt,black] (obj1) to[bend left] (obj36);
     \draw[fill=black,fill opacity=0.1] (obj1) to[bend right=30] (obj5) to[bend right=30] (obj12) to[bend right=30] (obj13) to[bend right=30] (obj14) to[bend left=10] (obj36) to[bend right=30] (obj1);
     \draw[fill=black,fill opacity=0.1] (obj18) to[bend right=30] (obj22) to[bend right=30] (obj29) to[bend right=30] (obj30) to[bend right=30] (obj31) to[bend right=30] (obj35) to[bend left=10] (obj18);
     \draw[fill=black,fill opacity=0.1] (obj2) to[bend right=30] (obj3) to[bend right=30] (obj4) to[bend left=30] (obj2);
     \draw[fill=black,fill opacity=0.1] (obj6) to[bend right=30] (obj7) to[bend right=30] (obj11) to[bend left=30] (obj6);
     \draw[fill=black,fill opacity=0.1] (obj8) to[bend right=30] (obj9) to[bend right=30] (obj10) to[bend left=30] (obj8);
     \draw[fill=black,fill opacity=0.1] (obj15) to[bend right=30] (obj16) to[bend right=30] (obj17) to[bend left=30] (obj15);
     \draw[fill=black,fill opacity=0.1] (obj19) to[bend right=30] (obj20) to[bend right=30] (obj21) to[bend left=30] (obj19);
     \draw[fill=black,fill opacity=0.1] (obj23) to[bend right=30] (obj24) to[bend right=30] (obj28) to[bend left=30] (obj23);
     \draw[fill=black,fill opacity=0.1] (obj25) to[bend right=30] (obj26) to[bend right=30] (obj27) to[bend left=30] (obj25);
     \draw[fill=black,fill opacity=0.1] (obj32) to[bend right=30] (obj33) to[bend right=30] (obj34) to[bend left=30] (obj32);
    \end{tikzpicture}
  \caption{A $2$-pseudo-rotationally invariant non-crossing partition in
$\NCAPlus[36]$ (and also in $\mNCAPlus[11][3]$)}
\label{fig:5}
\end{figure}
\begin{figure}
  \centering
  \begin{tikzpicture}[scale=1]
    \polygon{(0,0)}{obj}{36}{2.5}
      {1,2,3,4,5,6,7,8,9,10,11,12,13,14,15,16,17,18,19,20,21,22,23,24,25,26,27,28,29,30,31,32,33,34,35,36}
     \draw[line width=2.5pt,black] (obj1) to[bend left] (obj36);
     \draw[fill=black,fill opacity=0.1] (obj3) to[bend right=30] (obj7) to[bend right=30] (obj14) to[bend right=30] (obj15) to[bend right=30] (obj16) to[bend right=30] (obj20) to[bend right=10] (obj3);
     \draw[fill=black,fill opacity=0.1] (obj24) to[bend right=30] (obj31) to[bend right=30] (obj32) to[bend right=30] (obj33) to[bend right=30] (obj34) to[bend right=30] (obj2) to[bend left=10] (obj24);
     \draw[fill=black,fill opacity=0.1] (obj4) to[bend right=30] (obj5) to[bend right=30] (obj6) to[bend left=30] (obj4);
     \draw[fill=black,fill opacity=0.1] (obj8) to[bend right=30] (obj9) to[bend right=30] (obj13) to[bend left=30] (obj8);
     \draw[fill=black,fill opacity=0.1] (obj10) to[bend right=30] (obj11) to[bend right=30] (obj12) to[bend left=30] (obj10);
     \draw[fill=black,fill opacity=0.1] (obj17) to[bend right=30] (obj18) to[bend right=30] (obj19) to[bend left=30] (obj17);
     \draw[fill=black,fill opacity=0.1] (obj21) to[bend right=30] (obj22) to[bend right=30] (obj23) to[bend left=30] (obj21);
     \draw[fill=black,fill opacity=0.1] (obj25) to[bend right=30] (obj26) to[bend right=30] (obj30) to[bend left=30] (obj25);
     \draw[fill=black,fill opacity=0.1] (obj27) to[bend right=30] (obj28) to[bend right=30] (obj29) to[bend left=30] (obj27);
     \draw[fill=black,fill opacity=0.1] (obj35) to[bend right=30] (obj36) to[bend right=30] (obj1) to[bend left=30] (obj35);
    \end{tikzpicture}
  \caption{Another $2$-pseudo-rotationally invariant non-crossing partition in
$\NCAPlus[36]$ (and also in $\mNCAPlus[11][3]$)}
\label{fig:6}
\end{figure}

\medskip
\noindent
{\sc Construction 2.}
This construction requires $\rrr =2$.
We start with a non-crossing partition~$\bar\pi$ of
$\{N,1,2,\dots,(N-2)/2\}$ in which $N$ and $1$ are in the
same block. Out of these data, we form a positive non-crossing
partition~$\pi$ of $\{1,2,\dots,N\}$ in the following way:
each block of~$\bar\pi$ is also a block of~$\pi$. Furthermore,
for each
block $B$ of~$\bar\pi$ not containing $N$ and $1$, its ``twin block''
$\rot^{(N-2)/2}(B)$ is also a block of~$\pi$. Finally,
let $B_0$ denote
the block containing~$N$ and~$1$ of~$\bar\pi$, say
$B_0=\{N,1,v_1,v_2,\dots,v_{a-2}\}$. Then
$$\rot^{(N-2)/2}\big(\{v_1,v_2,\dots,v_{a-2}\}\big)\cup
\{N/2,N-1\}$$
is also a block of~$\pi$.
\Cref{fig:5,fig:6} illustrate this construction.
In these figures, we chose $N=36$.
In order to obtain the non-crossing partition in \Cref{fig:5},
one starts with the non-crossing partition
$$
\{\{1,5,12,13,14,36\},\{2,3,4\},\{6,7,11\},
\{8,9,10\},\{15,16,17\}\}
$$
of $\{36,1,2,\dots,17\}$. Then the blocks not containing $1$ and
$36$ are rotated, that is,\break $\rot^{(36-2)/2}=
\rot^{17}$ is applied to them, yielding the blocks
$$
\{19,20,21\},\ \{23,24,28\},\
\{25,26,27\},\ \{32,33,34\}.
$$
All of these become blocks of the partition to be constructed;
see \Cref{fig:5}. Finally, we take the block $B_0$ containing
$1$ and $36$, $B_0=\{1,5,12,13,14,36\}$, remove $1$ and $36$,
apply $\rot^{17}$ to the remaining set,
$$                                                                             \rot^{17}(B_0\setminus\{1,36\})=\{22,29,30,31\},
$$
and finally add $36/2=18$ and $36-1=35$ to this
to obtain the last block of the partition to be constructed;
see \Cref{fig:5}.
If one applies $\rotA$ twice to this
partition, then the result is the partition in \Cref{fig:6}.

\begin{theorem} \label{lem:allA}
Let $N$ and $\rrr$ be positive integers with $N,\rrr\ge2$ and $\rrr\mid(N-2)$.

\smallskip
{\em(1)} If $\rrr \ge3$,
all $\rrr $-pseudo-rotationally invariant elements of
$\NCAPlus$ 
are obtained by starting with a non-crossing partition
from Construction~{\em1}
and applying the pseudo-rotation~$\rotA$ repeatedly to it.

\smallskip
{\em(2)}
All $2$-pseudo-rotationally invariant elements of $\NCAPlus$ 
are obtained by starting with a non-crossing partition
from Construction~{\em1} or Construction~{\em2}
and applying the pseudo-rotation~$\rotA$ repeatedly to it.
\end{theorem}

The proof of this theorem is given in \Cref{app:inv-A}.

\subsection{Enumeration of pseudo-rotationally invariant elements}
\label{sec:rotenumA}

With the characterisation of pseudo-rotationally invariant elements
of $\NCAPlus$ at hand, we can now embark on
the enumeration of such elements. Since the map~$\rotA$ does in general not
preserve the partial order on non-crossing partitions, and the
repeated application of~$\rotA$ is essential in the aforementioned
characterisation (see \Cref{lem:allA}), we are not able to provide
results on chain enumeration (which however still may exist).
Rather, the most refined and most general
result in this subsection is \Cref{thm:2} which provides a
formula for the number of these non-crossing partitions in which
the block structure is prescribed. In the statement of the theorem,
and also later,
we use an extended version of the earlier introduced notion of
``central block":
for a $\rrr $-pseudo-rotationally invariant
element~$\pi$ of $\NCAPlus$
(that is, $\pi$ is invariant with respect to~$\rotA^{(N-2)/\rrr }$), 
a block $B$ of~$\pi$ is called
\defn{$\rotA$-central block\/} (or simply \defn{central block} if $\rotA$
is clear from the context)
if $B=\rotA^{(N-2)/\rrr }(B)$.
(How $\rotA$ acts on a block~$B$ should be intuitively clear from
the definition of~$\rotA$, cf.\ \Cref{fig:7,fig:1,fig:8}.)
For example, both non-crossing partitions in \Cref{fig:3,fig:4} 
have central blocks: in the partition of \Cref{fig:3}
the central block is $\{1,2,9,16,23,30\}$, while in 
the partition of \Cref{fig:4}
the central block is $\{3,4,11,18,25,26\}$. On the other hand,
neither of the non-crossing partitions in \Cref{fig:5,fig:6} 
has a central block.
While a singleton block can never be a central block, it may
happen that --- slightly counter-intuitively --- $\{N,1\}$ is a central block. 
An example with $N=32$ and $\rrr=3$ is shown in \Cref{fig:22},
where indeed
$\rotA^{(32-2)/3}(\(1,32\))=\rotA^{10}(\(1,32\))=\(1,32\)$.

\begin{theorem} \label{thm:2}
Let $m,n,\rrr $ be positive integers with $\rrr \ge2$ and $\rrr \mid (mn-2)$.
Furthermore, let $b_1,b_2,\dots,b_n$ be non-negative integers.
The number of positive $m$-divisible non-crossing partitions of 
$\{1,2,\dots,mn\}$ which are invariant under
the $\rrr $-pseudo-rotation~$\rotA^{(mn-2)/\rrr }$, the
number of non-central 
blocks of size~$mi$ being $\rrr b_i$, $i=1,2,\dots,n$,
the central block having size $ma=mn-m\rrr \sum_{j=1}^njb_j$,
is given by
\begin{equation} \label{eq:multichains-bi-si-pos} 
\binom {b_1+b_2+\dots+b_n}{b_1,b_2,\dots,b_n}
\binom {(mn-2)/\rrr } {b_1+b_2+\dots+b_n}
\end{equation}
if $b_1+2b_2+\dots+nb_n<{n/\rrr }$, 
or if $\rrr =2$ and $b_1+2b_2+\dots+nb_n=n/2$, 
and\/ $0$ otherwise.
\end{theorem}

The proof of this theorem is found in \Cref{app:GF-A-inv}.

\begin{corollary} \label{cor:4}
Let $m,n,\rrr $ be positive integers with $\rrr \ge2$ and $\rrr \mid (mn-2)$.
Furthermore, let $b$ be a non-negative integer.
The number
of positive $m$-divisible non-crossing partitions of\/ $\{1,2,\dots,mn\}$
which are invariant under the $\rrr $-pseudo-rotation~$\rotA^{(mn-2)/\rrr }$,
have $\rrr b$ non-central blocks and a central block of size~$ma$,
is given by
\begin{equation} \label{eq:multichains-si-pos-a} 
\binom {{(n-a-\rrr )/\rrr }}{b-1}
\binom {(mn-2)/\rrr }{b}
\end{equation}
if $a\equiv n$~{\em(mod~$\rrr $)} and $0$ otherwise,
while the number of these partitions without restricting the size 
{\em(}or existence{\em)} of the central block is
\begin{equation} \label{eq:multichains-si-pos} 
\binom {\fl{n/\rrr }}{b}
\binom {(mn-2)/\rrr }{b}.
\end{equation}
\end{corollary}

The proof of this corollary is also found in \Cref{app:GF-A-inv}.

\begin{corollary} \label{cor:5}
Let $m,n,a,\rrr $ be positive integers with $\rrr \ge2$ and $\rrr \mid (mn-2)$.
The number
of positive $m$-divisible non-crossing partitions of\/ $\{1,2,\dots,mn\}$
which are invariant under the $\rrr $-pseudo-rotation~$\rotA^{(mn-2)/\rrr }$ and have a central block of size~$ma$
is given by
\begin{equation} \label{eq:multichains-pos-a} 
\binom {{((m+1)n-a-\rrr -2)/\rrr }}{{(n-a)/\rrr }}
\end{equation}
if $a\equiv n$~{\em(mod~$\rrr $)} and $0$ otherwise,
while the number of these partitions without restricting the size
{\em(}or existence{\em)} of
the central block is
\begin{equation} \label{eq:multichains-pos} 
\binom {\fl{((m+1)n-2)/\rrr }}{\fl{n/\rrr }}.
\end{equation}
\end{corollary}

\begin{proof}
One sums Expressions~\eqref{eq:multichains-si-pos-a} and~\eqref{eq:multichains-si-pos} over all~$b$.
The resulting sums can be evaluated by means of the Chu--Vandermonde
summation formula.
\end{proof}

\subsection{Cyclic sieving}
\label{sec:sievA}

With the enumeration results in \Cref{sec:enumA,sec:rotenumA}, we are
now ready to derive our cyclic sieving results in type~$A$; see
\Cref{thm:3} below.
The definition of cyclic sieving (given in~\eqref{eq:siev})
requires certain properties of the ``cyclic sieving polynomial"
$P(q)$. Namely, one has to show that
it is indeed a polynomial and has non-negative
coefficients, and one has to compute the evaluation of this
polynomial at certain roots of unity.
We verify these properties separately
in \Cref{lem:1,lem:2}.

\medskip
We use the standard $q$-notations $[\alpha]_q:=(1-q^\alpha)/(1-q)$,
$[n]_q!:=[n]_q\, [n-1]_q\cdots[1]_q$, where $[0]_q!:=1$,
the $q$-binomial coefficient
$$
\begin{bmatrix} M\\m\end{bmatrix}_q:=\begin{cases}
\frac {[M]_q\, [M-1]_q\cdots [M-m+1]_q} {[m]_q!},&\text{if $k\ge1$},\\
1,&\text{if $k=0$,}
\end{cases}
$$
and the $q$-multinomial coefficient
$$
\begin{bmatrix} {M}\\{m_1,m_2,\dots,m_k}\end{bmatrix}_q
:=\begin{cases}
\frac {[M]_q\,[M-1]_q\cdots [M-m_1-m_2-\dots-m_k+1]_q} {[m_1]_q!\,[m_2]_q!\cdots
[m_k]_q!},&\text {if $m_1+\dots+m_k\ge1$,}\\
1,&\text {if $m_1+\dots+m_k=0$.}
\end{cases}
$$

\begin{lemma} \label{lem:1}
Let $m,n$ be positive integers and $s$ and  $b_1,b_2,\dots,b_n$ 
non-neg\-ative integers such that $0\le s\le n-1$ and
\begin{equation}
\label{eq:SB}
b_1+2b_2+\dots+nb_n=n.
\end{equation}
Then the expressions
\begin{gather}
\frac {1} {[n]_q}\begin{bmatrix} mn-2\\n-1\end{bmatrix}_q,
\label{eq:SD}
\\
\frac {1} {[n]_q}
\begin{bmatrix} {n}\\{s}\end{bmatrix}_q
\begin{bmatrix} {mn-2}\\ {n-s-1}\end{bmatrix}_q,
\label{eq:SE}
\\
\frac {1} {[b_1+b_2+\dots+b_n]_q}
\begin{bmatrix} {b_1+b_2+\dots+b_n}\\{b_1,b_2,\dots,b_n}\end{bmatrix}_q
\begin{bmatrix} {mn-2}\\ {b_1+b_2+\dots+b_n-1}\end{bmatrix}_q
\label{eq:SF}
\end{gather}
are polynomials in $q$ with non-negative integer coefficients.
\end{lemma}

The proof of this lemma is found in \Cref{app:unnec-pol}.

\begin{lemma} \label{lem:2}
Let $m,n,s,b_1,b_2,\dots,b_n$ be integers satisfying
the hypotheses in \Cref{lem:1}. 
Furthermore, let $\si$ be a positive integer such that $\si\mid (mn-2)$ and 
$\si<mn-2$. Writing $\rrr =(mn-2)/\si$ and 
$\om_\rrr =e^{2\pi i\si/(mn-2)}=e^{2\pi i/\rrr }$, we have
\begin{align}
\frac {1} {[n]_q}\bmatrix (m+1)n-2\\n-1\endbmatrix_q
\Bigg\vert_{q=\om_\rrr }
&=\binom {\fl{((m+1)n-2)/\rrr }} {\fl{n/\rrr }},
\label{eq:SI}
\\
\frac {1} {[n]_q}\bmatrix mn-2\\n-1\endbmatrix_q
\Bigg\vert_{q=\om_\rrr }
&=\begin{cases}
\binom {(mn-2)/2} {{n/2}},&\text{if\/ $\rrr =2$ and $n\equiv 0$ \rm (mod $2$)},\\
\binom {(mn-2)/\rrr } {\fl{n/\rrr }},&\text{if\/ $n\equiv 1$ \rm (mod $\rrr $)},\\
0,&\text{otherwise,}
\end{cases}
\label{eq:SJ}
\\
\notag
\frac {1} {[n]_q}
&\bmatrix {n}\\{s}\endbmatrix_q
\bmatrix {mn-2}\\ {n-s-1}\endbmatrix_q
\Bigg\vert_{q=\om_\rrr }
&\\
&\hskip-1cm
=\begin{cases}
\binom {\fl{(n-1)/\rrr }} {\fl{s/\rrr }}
\binom {(mn-2)/\rrr } {{(n-s-1)/\rrr }},
&\text{if\/ $n-s-1\equiv 0$ \rm (mod $\rrr $)},\\
\binom {(n-2)/2} {{s/2}}
\binom {(mn-2)/2} {{(n-s)/2}},
&\text{if\/ $\rrr =2$ and $n\equiv s\equiv0$ \rm (mod $2$),}\\
0,&\text{otherwise,}
\end{cases}
\notag
\\
\label{eq:SK}
\\
\notag
&\hskip-4.5cm
\frac {1} {[b_1+b_2+\dots+b_n]_q}
\bmatrix {b_1+b_2+\dots+b_n}\\{b_1,b_2,\dots,b_n}\endbmatrix_q
\bmatrix {mn-2}\\ {b_1+b_2+\dots+b_n-1}\endbmatrix_q\Bigg\vert_{q=\om_\rrr }\\
&\hskip-4cm=\begin{cases}
\binom {{(b_1+b_2+\dots+b_n-1)/\rrr }}
{\fl{b_1/\rrr },\fl{b_2/\rrr },\dots,\fl{b_n/\rrr }}
\binom {(mn-2)/\rrr } {{(b_1+b_2+\dots+b_n-1)/\rrr }},
&\text{if\/ $b_k\equiv0$ \rm (mod $\rrr $) for $k\ne j$}\\
&\text{and $b_j\equiv1$ \rm (mod $\rrr $), for some $j$},\\
\binom {{(b_1+b_2+\dots+b_n)/2}}
{{b_1/2},{b_2/2},\dots,{b_n/2}}
\binom {(mn-2)/2} {\fl{(b_1+b_2+\dots+b_n)/2}},\hskip-4cm\\
&\hskip-1cm
\text{if\/ $\rrr =2$ and $n\equiv b_1\equiv\dots\equiv b_n\equiv0$ \rm (mod $2$),}\\
0,&\text{otherwise.}
\end{cases}
\notag
\\
\label{eq:SL}
\end{align}
\end{lemma}

The proof of this lemma is found in \Cref{app:unnec-eval}.

Now we are ready to state, and prove, our cyclic sieving results
in type~$A$.

\begin{theorem}
\label{thm:3}
Let $m,n,s,b_1,b_2,\dots,b_n$ be integers satisfying
the hypotheses in\break \Cref{lem:1}. 
Furthermore, 
let $C$ be the cyclic group of pseudo-rotations of an $mn$-gon
generated by~$\rotA$. Then
the triple $(M,P,C)$ exhibits the cyclic sieving phenomenon
for the following choices of sets $M$ and polynomials $P$:

\begin{enumerate}
\item[\em(1)]
$M=\mNCAPlus$, and
$$P(q)=\frac {1} {[n]_q}\left[\begin{matrix}
(m+1)n-2\\n-1\end{matrix}\right]_q;$$
\item[\em(2)]
$M$ consists of the elements of
$\mNCAPlus$ all of whose blocks have size~$m$, and
$$P(q)=\frac {1} {[n]_q}\left[\begin{matrix}
mn-2\\n-1\end{matrix}\right]_q;$$
\item[\em(3)]
$M$ consists of all elements of
$\mNCAPlus$ which have rank $s$
{\em(}or, equivalently, their number of blocks is $n-s${\em)},
and 
$$P(q)=\frac {1} {[n]_q}
\bmatrix {n}\\{s}\endbmatrix_q
\bmatrix {mn-2}\\ {n-s-1}\endbmatrix_q;
$$
\item[\em(4)]
$M$ consists of all elements of
$\mNCAPlus$ whose
number of blocks of size~$mi$ is~$b_i$, $i=1,2,\dots,n$,
and 
$$P(q)=\frac {1} {[b_1+b_2+\dots+b_n]_q}
\bmatrix {b_1+b_2+\dots+b_n}\\{b_1,b_2,\dots,b_n}\endbmatrix_q
\bmatrix {mn-2}\\ {b_1+b_2+\dots+b_n-1}\endbmatrix_q.
$$
\end{enumerate}
\end{theorem}

\begin{proof}
The polynomials $P(q)$ in the assertion of the theorem are indeed
polynomials with non-negative coefficients due to \Cref{lem:1}.

Now, given a positive integer $\rrr$ with $\rrr\mid (mn-2)$, we must show  
$$
P(q)\big\vert_{q=\om_\rrr }=
(\text{number of elements of $\mNCAPlus$ invariant under
  $\rotA^{(mn-2)/\rrr}$}),
$$
where $\om_\rrr=e^{2\pi i/\rrr}$.

For Item~(1), this follows by combining \eqref{eq:multichains-pos} and~\eqref{eq:SI}.
Item~(2) is the special case of Item~(4) where $b_1=n$
(and, hence, $b_2=b_3=\dots=b_n=0$), and therefore does not need
to be treated separately.

\medskip
We now turn to Item~(3). 
We distinguish two cases, depending on $\rrr $. 

\medskip
{\sc Case 1: $\rrr \ge3$.} Then, according to \Cref{lem:10}, any 
positive $m$-divisible non-crossing partition which is
$\rrr $-pseudo-rotationally invariant must have a central block. 
Moreover, \Cref{lem:allA}(1) says that all of these
partitions arise by
Construction~1 and subsequent applications of the
pseudo-rotation~$\rotA$. As a consequence,
the number of non-central blocks must be divisible by~$\rrr $.
Consequently, for any such non-crossing partition~$\pi$, we have
$$
\#(\text{blocks of $\pi$})=1+\rrr B=n-s,
$$
where $B$ is some non-negative integer. In particular, we must
have $n-s-1\equiv 0$~(mod~$\rrr $).
The number of these non-crossing partitions is given by
Formula~\eqref{eq:multichains-si-pos} with $b$ replaced by
$B=(n-s-1)/\rrr $. It is not difficult to see that this agrees with
the expression in the first alternative on the right-hand side of~\eqref{eq:SK}.

\medskip
{\sc Case 2: $\rrr =2$.} In that case, \Cref{lem:allA}(2) says that we
obtain all relevant positive $m$-divisible non-crossing partitions
by either Construction~1 or~2 and subsequent applications of~$\rotA$. 
In the first case, there is a central block and
the previous arguments apply again. In the second case, there is
no central block. Moreover, if $b_i$ denotes the number of blocks
of size~$mi$, then all~$b_i$'s are even. In particular,
the number of blocks, $n-s$, must be even.
Since $n=b_1+2b_2+\dots+nb_n$, also $n$, and thus also
$s$, must be even. The number of relevant non-crossing partitions
is then given by~\eqref{eq:multichains-si-pos-a} with $a=0$ and
$b$ replaced by $(n-s)/2$. The resulting expression is identical
with the expression in the second alternative on the right-hand side of~\eqref{eq:SK}.

In all other cases, there are no positive $m$-divisible non-crossing
partitions which are $\rrr $-pseudo-rotationally invariant.

\medskip
Finally, we address Item~(4). 
Again, we distinguish two cases depending on $\rrr $.

\medskip
{\sc Case 1: $\rrr \ge3$.} Let again $b_i$ denote the number of blocks
of size~$mi$.
Repeating the arguments for Item~(3),
we see that there is one $b_j$ (namely the one for which $mj$ is
the size of the central block) which is congruent to $1$ modulo~$\rrr $,
while all other~$b_k$'s, $k\ne j$, are divisible by~$\rrr $.
The number of relevant non-crossing partitions is then given
by~\eqref{eq:multichains-bi-si-pos} with $b_j$ replaced by
$(b_j-1)/\rrr $ and $b_k$ replaced by $b_k/\rrr $, $k\ne j$. 
This agrees with the expression in the first alternative
on the right-hand side of~\eqref{eq:SL}.

\medskip
{\sc Case 2: $\rrr =2$.} In that case, \Cref{lem:allA}(2) says that we
obtain all relevant positive $m$-divisible non-crossing partitions
by either Construction~1 or~2 and subsequent application
of~$\rotA$. In the first case, there is a central block and
the previous arguments apply again. In the second case, by
repeating the arguments for Item~(3), we see that
all~$b_i$'s, and thus also $n$, must be even. 
The number of relevant non-crossing partitions
is then given by~\eqref{eq:multichains-bi-si-pos} with 
$b_k$ replaced by $b_k/2$, $k=1,2,\dots,n$. 
The resulting expression is identical
with the expression in second alternative on the right-hand side of~\eqref{eq:SL}.

In all other cases, there are no positive $m$-divisible non-crossing
partitions which are $\rrr $-pseudo-rotationally invariant.

\medskip
This completes the proof of the theorem.
\end{proof}

\begin{corollary} \label{cor:CS-A}
  The conclusion of \Cref{thm:CS} holds for type $A_{n-1}$.
\end{corollary}

\begin{proof}
The cyclic sieving phenomenon for
  $\Big(\mNCPlus[A_{n-1}][m+1],\ \mCatplus[m+1](A_{n-1};q),\ C\Big)$
follows directly from \Cref{thm:3}(1) with~$m$ replaced by~$m+1$.

The image of the
embedding of $\mNCPlus[A_{n-1}][m]$
in $\mNCPlus[A_{n-1}][m+1]$ as described in
\Cref{prop:PositiveKrewAlt} is given by all tuples
$(w_0,w_1,\dots,w_{m+1})$ with $w_0=\ep$. By
\Cref{rem:id}, these tuples correspond to non-crossing set partitions
in $\mNCA[n-1][m+1]$ all of whose blocks have size~$m+1$. 
The cyclic sieving phenomenon for
$\Big(\mNCPlus[A_{n-1}][m],\ \mCatplus(A_{n-1};q),\ \widetilde C\Big)$
thus follows from \Cref{thm:3}(2) with $m$ replaced by~$m+1$.
\end{proof}

\section{Positive non-crossing set partitions in
type~\texorpdfstring{$B$}{B}}
\label{sec:typeB}

In this section, we study positive $m$-divisible non-crossing partitions and the positive Kreweras maps in the situation of the hyperoctahedral group $B_n$.
Recall that $B_n$ can be combinatorially realised as the group of permutations~$\pi$ of $\{1,2,\dots,n,-1,-2,\dots,-n\}$ satisfying $\pi(-i)=-\pi(i)$ for $i=1,2,\dots,n$.
In the sequel, we shall most often write $\overline i$ instead of~$-i$.
We adopt the type~$B$ cycle notation from~\cite{BRWaAA}
consisting of the conventions
\begin{align*}
((i_1,i_2,\dots,i_k))&:=(i_1,i_2,\dots,i_k)
(\overline{i_1},\overline{i_2},\dots,\overline{i_k}),\\
[i_1,i_2,\dots,i_k]&:=(i_1,i_2,\dots,i_k,
\overline{i_1},\overline{i_2},\dots,\overline{i_k}).
\end{align*}

The Coxeter system corresponding to the hyperoctahedral group~$B_n$ is
\begin{equation} \label{eq:WSB} 
(W,\reflS) = \big( B_{n},\{ s_i = ((i,i+1)) \mid 1 \leq i < n\}\cup
\{s_n=(n,\overline{n})=[n]\} \big)
\end{equation}
together with the Coxeter element $c = s_1 \cdots
s_{n}=[1,2,\ldots,n]$.

\medskip

The structure of this section is the same as the one of \Cref{sec:typeA}.
Namely, in \Cref{sec:realB}
we first make the positivity condition for type~$B$
$m$-divisible non-crossing partitions under the choice~\eqref{eq:WSB} 
of Coxeter system explicit; see \Cref{lem:wmn>0}. We then use
this to describe \emph{positive} $m$-divisible non-crossing partitions
within Armstrong's combinatorial
model of type~$B$ $m$-divisible non-crossing partitions
\cite[Def.~4.5.5 and Thm.~4.5.6]{Arm2006}; see \Cref{prop:1B}.

In \Cref{sec:mapB}, we show that also here, under Armstrong's
translation,
for $m\ge2$ the positive Kreweras map acts as a pseudo-rotation; see
\Cref{prop:2B}.
This pseudo-rotation is denoted by~$\rotB$.

In \Cref{sec:combB}, we introduce a natural extension of~$\rotB$
which also makes
sense for $m=1$; see \Cref{def:phiallgB}.
Also here it turns out
that, under Armstrong's translation, for $m=1$
this extension corresponds exactly to the map~$\Krewplustilde^{(1)}$
from \Cref{def:K-alt}; see \Cref{prop:2-m=1B}.
The order of the (extended) map~$\rotB$ is of course 
given by \Cref{thm:order} for type~$B$; see \Cref{lem:N-2B}.
For the sake of completeness,
we also provide an independent, {\it combinatorial}, proof for that
order; see \Cref{app:order-B}.

\Cref{sec:enumB} is devoted to the enumeration of type~$B$ positive
$m$-divisible non-crossing partitions. Here, we avoid attempting to
find results on chain enumeration since this seems to require more
refined techniques as the ones we have developed here. Thus, our
most refined and most general
result is \Cref{thm:1-B} which provides a
formula for the number of these
non-crossing partitions in which the block structure is prescribed.
All other results in that subsection,
\Cref{cor:2-B,cor:3-B}, 
are obtained as consequences.

\Cref{sec:rotB} is then devoted to the characterisation of
type~$B$ positive $m$-divisible non-crossing partitions that are
invariant under powers of the pseudo-rotation~$\rotB$; see \Cref{lem:allB}.

In \Cref{sec:rotenumB}, we then exploit this characterisation to provide a formula for the number of 
type~$B$ positive $m$-divisible non-crossing partitions that are
invariant under powers of the pseudo-rotation~$\rotB$ and
have a given block structure; see \Cref{thm:2-B}. Once again,
several less refined enumeration results may be obtained by summation;
see \Cref{cor:4-B,cor:5-B}.

In the final subsection, \Cref{sec:sievB}, we establish several
cyclic sieving phenomena refining \Cref{thm:CS} in type~$B$;
see \Cref{thm:3-B}.

\subsection{Combinatorial realisation of the positive non-crossing partitions}
\label{sec:realB}

With the choice~\eqref{eq:WSB} of Coxeter system, 
we have the following simple combinatorial description of 
positive $m$-divisible non-crossing partitions.

\begin{proposition} \label{lem:wmn>0}
Let $m$ and $n$ be positive integers with  $m\ge2$.
  $(w_0,w_1,\dots,w_m) \in \mNC[B_{n}]$ is positive if and only if
  $w_m(n)>0$.
\end{proposition}

\begin{proof}
According to \Cref{def:NCpos}, the tuple $(w_0,w_1,\dots,w_m)$ is
positive if and only if $cw_m^{-1}=w_0w_1\cdots w_{m-1}$ has full support in our generators $\{s_1,\dots,s_n\}$ given in~\eqref{eq:WSB}. The
cycles in the disjoint cycle decomposition of~$cw_m^{-1}$ define
a non-crossing partition. Thus, $cw_m^{-1}$ will not have full support if and only if the preimage of~$1$ under
$cw_m^{-1}$ is a positive number. In other words, there is an $l>0$
such that $(cw_m^{-1})(l)=1$.
By acting on the left on both sides
of this relation by the inverse of the
Coxeter element $c=[1,2,\dots,n]$, we see that this is equivalent
to $w_m^{-1}(l)=\overline n$, or, again equivalently, to $w_m(n)=\overline l$,
for some $l>0$. We are interested in the contrapositive:
$cw_m^{-1}$ has full support if and only if $w_m(n)>0$.
\end{proof}

Next we recall Armstrong's bijection between $\mNC[B_{n}]$ and 
$m$-divisible non-crossing partitions of
$\{1,2,\dots,mn,\overline1,\overline2,\dots,\overline{mn}\}$ that are 
invariant under substitution of~$i$ by~$-i$, for all~$i$, which we denote here by
$\mNCB$.
We write $\NCB$ for the set of \emph{all}
non-crossing partitions of
$\{1,2,\dots,N,\overline1,\overline2,\dots,\overline N\}$ that are
invariant under substitution of~$i$ by~$-i$, for all~$i$, and observe that
\[
  \mNCB = \{ \pi \in
  \NCB[mn] \mid \text{all block sizes of $\pi$ are divisible by~$m$}\}\,. 
  \]
Following~\cite{ReivAG}, we agree here to call a block
that is itself invariant under substitution of~$i$ by~$-i$, for
all~$i$, a \defn{zero block}, a notion that will be used frequently in
the sequel. It should be noted that a non-crossing partition can have
at most one zero block and that, for parity reasons, the size of
a zero block is always divisible by~$2m$.  

\medskip
Given an element $(w_0,w_1,\dots,w_m)\in \mNC[B_{n}]$,
the bijection, $\Nam{B_n}m$ say, 
from \cite[Thm.~4.5.6]{Arm2006} 
works essentially in the same way as in type~$A$: 
namely
$(w_0,w_1,\dots,w_m)\in \mNC[B_{n}]$ is mapped to
\begin{equation} \label{eq:Nb}
\Nam{B_n}m(w_0,w_1,\dots,w_m)=
[1,2,\dots,mn]\,(\bar\ta_{m,1}(w_1))^{-1}\,(\bar\ta_{m,2}(w_2))^{-1}\,\cdots\,
(\bar\ta_{m,m}(w_m))^{-1},
\end{equation}
where $\bar\ta_{m,i}$ is the obvious extension of the 
transformation $\ta_{m,i}$ from \Cref{sec:typeA}.
More precisely, we let
$$(\bar\ta_{m,i}(w))(mk+i-m)=mw(k)+i-m,\quad 
k=1,2,\dots,n,\overline1,\overline2,\dots,\overline n,$$ 
and 
$(\bar\ta_{m,i}(w))(l)=l$ and
$(\bar\ta_{m,i}(w))(\overline l)=\overline l$ for all $l\not\equiv i$~(mod~$m$),
where $m\overline k+i-m$ is identified with $\overline{mk+i-m}$ for all~$k$
and~$i$. 
Again, the cycles in the disjoint cycle decomposition correspond to
the blocks in the non-crossing partition in
$\mNCB$.
We refer the reader to \cite[Sec.~4.5]{Arm2006} for the details.
For example, let $n=5$, $m=3$, $w_0=((2,4))$,
$w_1=[1]=(1,\overline 1)$,
$w_2=((1,4))$, and
$w_3=((2,3))\,((4,5))$. Then $(w_0;w_1,w_2,w_3)$ is mapped to
\begin{align} \notag
\Nam{B_{5}}3(w_0;w_1,w_2,w_3)&=[1,2,\dots,15]\,[1]\,((2,11))\,
((6,9))\,((12,15))\\
&=((1,\overline2,\overline{12}))\,((3,4,5,6,10,11))\,((7,8,9))\,((13,14,15)).
\label{eq:NCB}
\end{align} 

Similarly as in type $A_{n-1}$, the cycle
structure of the first
component of an $m$-divisible non-crossing partition
determines the block structure of its image under Armstrong's
bijection~$\Nam{B_n}m$; see \cite[Proof of Thm.~4.3.13]{Arm2006}.

\begin{proposition} \label{prop:block-B}
Let $(w_0,w_1,\dots,w_m)\in\mNC[B_n]$. The non-crossing partition
$\pi=\Nam{B_n}m(w_0,w_1,\dots,w_m)\in\mNCB$ has a zero block consisting
of $2km$~elements if and only if $w_0$ contains a cycle of the form
$[i_1,i_2,\dots,i_k]$ in its disjoint cycle decomposition.
Furthermore, $\pi$ has $2b_k$ non-zero blocks of
size~$mk$ if and only if $w_0$ has $b_k$ cycles of the form
$((i_1,i_2,\dots,i_k))$
in its disjoint cycle decomposition.\footnote{The reader must recall
that a ``cycle'' of the form $((i_1,i_2,\dots,i_k))$ consists actually
of {\em two} disjoint cycles of length~$k$, while
a cycle of the form $[i_1,i_2,\dots,i_k]$ is indeed a single cycle,
of length~$2k$.}
\end{proposition}

In the earlier example, we have
$w_0=((2,4))=((1))((3))((4))((5))((2,4))$. Indeed, the image of
$(w_0,w_1,w_2,w_3)$ in the example has no zero block,
two non-zero blocks of size $3\cdot 2=6$
and six non-zero blocks of size $3\cdot 1=3$.

\begin{remark} \label{rem:id-B}
A simple consequence of \Cref{prop:block-B} is that the image of
the set of all non-crossing partitions in $\mNC[B_n]$ of the form
$(\ep,w_1,\dots,w_m)$ under the map $\Nam{B_n}m$ is
the set of all non-crossing partitions in $\mNCB$ 
in which {\em all} blocks have size~$m$.
\end{remark}

\begin{theorem} \label{prop:1B}
Let $m$ and $n$ be positive integers with  $m\ge2$.
The image under $\Nam{B_n}m$ of the positive $m$-divisible
non-crossing partitions in $\mNC[B_{n}]$ are those $m$-divisible
non-crossing partitions in $\mNCB$ where 
the block containing $1$ also contains a negative number.
\end{theorem}

\begin{proof}
First, let $(w_0,w_1,\dots,w_m)\in \mNC[B_{n}]$.
By definition, we know that $w_m(n)>0$. Let us write $l$
for $w_m(n)$. We then have
\begin{align*}
\Big(\Nam{B_n}m(&w_0,w_1,\dots,w_m)\Big)(\overline{ml})\\
&=
\big([1,2,\dots,mn]\,(\ta_{m,1}(w_1))^{-1}\,(\ta_{m,2}(w_2))^{-1}\,\cdots\,
(\ta_{m,m}(w_m))^{-1}\big)(\overline{ml})\\
&=
\big([1,2,\dots,mn]\big)(\overline{mn})=1.
\end{align*}
Translated to non-crossing partitions, this means that $\overline{ml}$ and $1$
belong to the same block, which implies our claim.

Conversely, let $\pi\in \mNCB$ be an
$m$-divisible non-crossing partition in which the block containing
$1$ also contains a negative number. Let $\overline L$ be the negative number
that is connected with~$1$ in the geometric representation of~$\pi$.
Since in between $\overline L$ and $1$ there ``sit" blocks all of whose sizes
are divisible by~$m$, we must necessarily have $\overline L=\overline{ml}$,
for some $l$ between~$1$ and~$n$.
In other words, if $\pi$
is interpreted as a permutation, $\pi(\overline{ml})=1$. Let
$(w_0,w_1,\dots,w_m)$ be the element of $\mNC[B_{n}]$ such that
$$
\Nam{B_n}m(w_0,w_1,\dots,w_m)=\pi.
$$
We must have
$$
\Big(\Nam{B_n}m(w_0,w_1,\dots,w_m)\Big)(\overline{ml})=1.
$$
The definition \eqref{eq:Nb} of $\Nam{B_n}m$
and the fact that $(\ta_{m,i}(w_i))^{-1}$ leaves multiples of~$m$
fixed as long as $1\le i\le m-1$ together imply that 
$(\ta_{m,m}(w_m))^{-1}(\overline{ml})=\overline{mn}$.
Phrased differently, we have
$(\ta_{m,m}(w_m))^{-1}(ml)=mn$, or, equivalently, 
$w_m(n)=l>0$. By definition, this means that
$(w_0,w_1,\dots,w_m)$ is positive.
\end{proof} 

\subsection{Combinatorial realisation of the positive Kreweras maps}
\label{sec:mapB}

The next step consists in translating the positive Kreweras map $\Krewplustilde$
into a ``rotation action" on the positive elements of
$\mNC[B_{n}]$. In order to do so, we need to explicitly describe 
the decomposition~\eqref{eq:LR} in type~$B$.

\begin{lemma} \label{lem:LRB}
Let~$w$ be an element of $\NC[B_n]$, and consider the cycle $z$
in the (type~$B$) disjoint cycle decomposition of~$w$ that contains~$n$. 
\begin{enumerate} 
\item[\em(1)] If $z=((i_1,i_2,\dots,i_k,n))$ with $i_1>0$
{\em(}and consequently $0<i_1<i_2<\dots<i_k<n${\em)}, we put
$a=n$.
\item[\em(2)] If $z=((i_1,\dots,i_s,i_{s+1},\dots,i_k,n))$ 
with $i_s<0$ and $i_{s+1}>0$
{\em(}and consequently $i_1,i_2,\dots,i_s<0$,
$i_{s+1},\dots,i_{k-1},i_k>0$,
and $0<\vert i_1\vert<\vert i_2\vert<\dots<\vert i_k\vert<n$;
the sequence $i_{s+1},\dots,i_k$ has to be interpreted as
the empty sequence if $s=k${\em)}, 
we put $a=i_s$.
\item[\em(3)] If $z=[i_1,i_2,\dots,i_k,n]$
{\em(}and consequently $0<i_1<i_2<\dots<i_k<n${\em)}, we put
$a=\overline n$.
\end{enumerate}
Then we have
$$w=w^Lw^R,\quad \quad \text {where $w^R=((a,n))$}.$$
{\em(}By slight abuse of notation, here we identify $((\overline n,n))$
with $[n]$.{\em)}
This covers all possible cases.
\end{lemma}

\begin{proof}
The last $n$ reflections in $\reflR$ given in \eqref{eq:refls}
are 
\begin{equation} \label{eq:reflsB} 
((i,n)),\ i=\overline1,\overline2,\dots,\overline {n-1},\text{ and }[n]=(\overline n,n),
\end{equation}
in this order. 
Thus, by definition, we have $w^L=w\circ (w^R)^{-1}$, where the factor
$w^R$ consists of a product of these~$n$ reflections (we shall see,
however, that in our case it will be always at most one reflection), 
and $\lenR(w^L)+\lenR(w^R)=\lenR(w)$.
Furthermore, using \Cref{lem:wmn>0} with
$m=1$, we see that $w^L$ must have the property that $w^L(n)>0.$
The arguments below are heavily based on the uniqueness of the
decomposition $w=w^l\circ w^R$ in the sense of 
\Cref{cor:unique}.

In Case~(1), we have $w(n)=i_1>0$. Hence, we may choose
$w^L=w$ and $w^R=\one=((n,n))$, and uniqueness of decomposition
guarantees that this is the correct choice.

In Case~(2), the choice $w^R=((i_s,n))$ implies that
$w^L$ would contain the cycles\break
$((i_1,\dots,i_s))((i_{s+1},\dots,i_k,n))$. 
Hence, we would
have $w^L(n)=i_{s+1}>0$ if~$s<k$, and $w^L(n)=n$ if $s=k$. 
By uniqueness of decomposition, 
this must be the correct choice.


Finally we address Case~(3). If we choose $w^R=[n]$,
then $w^L$ would contain the cycle
$((i_1,i_2,\dots,i_k,n))$. 
Hence, we would
have $w^L(n)=i_1>0$. Again,
by uniqueness of decomposition, 
this must be the correct choice.
\end{proof}

Next, we translate the positive Kreweras map $\Krewplustilde$ from
\Cref{def:positivekreweras} for type $B_{n}$ into combinatorial
language. 

\begin{theorem} \label{prop:2B}
Let $m$ and $n$ be positive integers with $m\ge2$.
Under the bijection $\Nam{B_n}m$, the map $\Krewplustilde^{(m)}$ translates to
the following map~$\rotB$ on $\mNCBPlus$:
if the block of~$\overline{mn}$ contains another negative element,
then $\rotB$ rotates all blocks of~$\pi$ by one unit in clockwise direction.

On the other hand, 
if the block of~$\overline{mn}$ of some element~$\pi\in\mNCBPlus$ contains no other negative 
element, then, first of all, $1$ is also contained in this block.
In other words, $\overline {mn}$ and $1$ are successive elements
in a block {\em(}and consequently also $mn$ and $\overline 1${\em)}.
For the description of the operation, we distinguish 
between two cases:

\begin{enumerate} 
\item[\em(1)] If there is a block not containing $1,\overline1,mn,\overline{mn}$ in
  which there are positive {\em and} negative elements, then
let $b$ be minimal such that $\overline d$ and $b$ are successive
elements in a block, with $\overline d$ negative and $b$
positive\footnote{\label{foot:d=b}We have $b\ne d$. For, if $b=d$,
then $\{b,\overline d\}=\{b,\overline b\}$ is a block by itself.
It is a zero block of size~2. Thus, necessarily $m=2$.
However, for parity reasons, such non-crossing partitions of
type~$B$ do not exist.}.
Furthermore, let $b$ be followed by~$e$ in this block, and let
$a$ and $\overline{mn}$ be successive elements
in a block. {\em(}The element~$a$ must be positive by the assumption
defining the subcase in which we are, while $e$ can be positive
or negative. See the left half of \Cref{fig:10} for an
illustration of the various definitions.{\em)} 
Then the image of~$\pi$ under~$\rotB$ is the partition 
in which $\overline{d+1}$, $1$, and $e+1$ are three 
successive elements in a
block, and $a+1$, $b+1$, and $2$ are successive elements in 
another block. All other succession relations in~$\pi$
are rotated by one unit in clockwise direction.
See \Cref{fig:10} for a schematic illustration of this operation.
\item[\em(2)] If there is {\em no} block other than the ones containing
$1,\overline1,mn,\overline{mn}$ with
positive {\em and} negative elements, then let $a$,
$\overline{mn}$, and~$1$ be successive elements in a block.
{\em(}Again, the element~$a$ must be positive.{\em)}
The image of~$\pi$ under~$\rotB$ is the partition 
in which $\overline{a+1}$, $1$, and $\overline2$ are successive
elements in a block, and all other succession relations in~$\pi$
are rotated by one unit in clockwise direction. 
See \Cref{fig:11} for a schematic illustration of this operation.
\end{enumerate}
\end{theorem}

\begin{figure}
\begin{center}
  \begin{tikzpicture}[scale=1]
    \polygon{(-4,0)}{obj}{44}{2.5}
       {1,,,,a,,,,,,b,,,,e,,,d,,,,mn,\overline{1},,,,\overline{a},,,,,,\overline{b},,,,\overline{e},,,\overline{d},,,,\overline{mn}}

    \draw[line width=2.5pt,black] (obj1) to[bend left=50] (obj44);
    \draw[line width=2.5pt,black] (obj23) to[bend left=50] (obj22);

    \draw[dash pattern=on 2pt off 2pt on 2pt off 2pt on 2pt off 2pt on
2pt off 16pt on 2pt off 2pt on 2pt off 2pt] (obj1) to[bend right=50]
(obj5);
    \draw[dash pattern=on 2pt off 2pt on 2pt off 2pt on 2pt off 2pt on
2pt off 16pt on 2pt off 2pt on 2pt off 2pt] (obj23) to[bend right=50]
(obj27);

    \draw (obj44) to[bend right=50] (obj5);
    \draw (obj22) to[bend right=50] (obj27);

    \draw[dotted] (obj6) to[bend right=50] (obj10);
    \node at ($($0.92*($(obj8)+(4.0,0)$)$)-(4.0,0)$) {\tiny$X$};

    \draw[dotted] (obj28) to[bend right=50] (obj32);
    \node at ($($0.92*($(obj30)+(4.0,0)$)$)-(4.0,0)$) {\tiny$\overline{X}$};

    \draw (obj40) to[bend right=20] (obj11) to[bend right=50] (obj15);
    \draw[dash pattern=on 2pt off 2pt on 2pt off 2pt on 2pt off 2pt on
2pt off 2pt on 2pt off 2pt on 2pt off 95pt on 2pt off 2pt on 2pt off
2pt on 2pt] (obj15) to[bend left=10] (obj40);
    \draw (obj18) to[bend right=20] (obj33) to[bend right=50] (obj37);
    \draw[dash pattern=on 2pt off 2pt on 2pt off 2pt on 2pt off 2pt on
2pt off 2pt on 2pt off 2pt on 2pt off 95pt on 2pt off 2pt on 2pt off
2pt on 2pt] (obj37) to[bend left=10] (obj18);

    \node[inner sep=0pt] at (0,0) {$\mapsto$};

    \polygon{(4,0)}{obj}{44}{2.5}
       {1,2,,,,a+1,,,,,,b+1,,,,e+1,,,d+1,,,mn,\overline{1},\overline{2},,,,\overline{a+1},,,,,,\overline{b+1},,,,\overline{e+1},,,\overline{d+1},,,\overline{mn}}

    \draw[line width=2.5pt,black] (obj1) to[bend left=50] (obj41);
    \draw[line width=2.5pt,black] (obj23) to[bend left=50] (obj19);

    \draw[dash pattern=on 2pt off 2pt on 2pt off 2pt on 2pt off 2pt on
2pt off 16pt on 2pt off 2pt on 2pt off 2pt] (obj2) to[bend right=50]
(obj6);
    \draw[dash pattern=on 2pt off 2pt on 2pt off 2pt on 2pt off 2pt on
2pt off 16pt on 2pt off 2pt on 2pt off 2pt] (obj24) to[bend right=50]
(obj28);

    \draw (obj6) to[bend right=50] (obj12);
    \draw (obj28) to[bend right=50] (obj34);

    \draw (obj2) to[bend right=50] (obj12);
    \draw (obj24) to[bend right=50] (obj34);

    \draw[dotted] (obj7) to[bend right=50] (obj11);
    \node at ($($0.92*($(obj9)-(4.0,0)$)$)+(4.0,0)$) {\tiny$X$};
    \draw[dotted] (obj29) to[bend right=50] (obj33);
    \node at ($($0.92*($(obj31)-(4.0,0)$)$)+(4.0,0)$) {\tiny$X$};

    \draw (obj1) to[bend right=20] (obj16);
    \draw[dash pattern=on 2pt off 2pt on 2pt off 2pt on 2pt off 2pt on
2pt off 2pt on 2pt off 2pt on 2pt off 95pt on 2pt off 2pt on 2pt off
2pt on 2pt] (obj16) to[bend left=10] (obj41);
    \draw (obj23) to[bend right=20] (obj38);
    \draw[dash pattern=on 2pt off 2pt on 2pt off 2pt on 2pt off 2pt on
2pt off 2pt on 2pt off 2pt on 2pt off 95pt on 2pt off 2pt on 2pt off
2pt on 2pt] (obj38) to[bend left=10] (obj19);

    \end{tikzpicture}
\end{center}
\caption{The action of the pseudo-rotation $\rotB$,
Case (1)}
\label{fig:10}
\end{figure}

\begin{figure}
  \begin{tikzpicture}[scale=1]
    \polygon{(-4,0)}{obj}{24}{2.5}
      {1,2,3,4,5,6,7,8,9,10,11,12,\overline{1},\overline{2},\overline{3},\overline{4},\overline{5},\overline{6},\overline{7},\overline{8},\overline{9},\overline{10},\overline{11},\overline{12}}

      \draw[line width=2.5pt,black] (obj1) to[bend left=50] (obj24);
      \draw[line width=2.5pt,black] (obj13) to[bend left=50] (obj12);

     \draw[fill=black,fill opacity=0.1] (obj24) to[bend right=30]
(obj1) to[bend right=30] (obj2) to[bend left=30] (obj24);
     \draw[fill=black,fill opacity=0.1] (obj3) to[bend right=30]
(obj4) to[bend right=30] (obj5) to[bend left=30] (obj3);
     \draw[fill=black,fill opacity=0.1] (obj6) to[bend right=30]
(obj10) to[bend right=30] (obj11) to[bend right=30] (obj18) to[bend
right=30] (obj22) to[bend right=30] (obj23) to[bend right=30] (obj6);
     \draw[fill=black,fill opacity=0.1] (obj7) to[bend right=30]
(obj8) to[bend right=30] (obj9) to[bend left=30] (obj7);
     \draw[fill=black,fill opacity=0.1] (obj12) to[bend right=30]
(obj13) to[bend right=30] (obj14) to[bend left=30] (obj12);
     \draw[fill=black,fill opacity=0.1] (obj15) to[bend right=30]
(obj16) to[bend right=30] (obj17) to[bend left=30] (obj15);
     \draw[fill=black,fill opacity=0.1] (obj19) to[bend right=30]
(obj20) to[bend right=30] (obj21) to[bend left=30] (obj19);

    \node[inner sep=0pt] at (0,0) {$\mapsto$};

    \polygon{(4,0)}{obj}{24}{2.5}
      {1,2,3,4,5,6,7,8,9,10,11,12,\overline{1},\overline{2},\overline{3},\overline{4},\overline{5},\overline{6},\overline{7},\overline{8},\overline{9},\overline{10},\overline{11},\overline{12}}

      \draw[line width=2.5pt,black] (obj1) to[bend left=50] (obj24);
      \draw[line width=2.5pt,black] (obj13) to[bend left=50] (obj12);

     \draw[fill=black,fill opacity=0.1] (obj2) to[bend right=30]
(obj3) to[bend right=30] (obj7) to[bend left=30] (obj2);
     \draw[fill=black,fill opacity=0.1] (obj4) to[bend right=30]
(obj5) to[bend right=30] (obj6) to[bend left=30] (obj4);
     \draw[fill=black,fill opacity=0.1] (obj1) to[bend right=10]
(obj11) to[bend right=30] (obj12) to[bend right=30] (obj13) to[bend
right=10] (obj23) to[bend right=30] (obj24) to[bend right=30] (obj1);
     \draw[fill=black,fill opacity=0.1] (obj8) to[bend right=30]
(obj9) to[bend right=30] (obj10) to[bend left=30] (obj8);
     \draw[fill=black,fill opacity=0.1] (obj14) to[bend right=30]
(obj15) to[bend right=30] (obj19) to[bend left=30] (obj14);
     \draw[fill=black,fill opacity=0.1] (obj16) to[bend right=30]
(obj17) to[bend right=30] (obj18) to[bend left=30] (obj16);
     \draw[fill=black,fill opacity=0.1] (obj20) to[bend right=30]
(obj21) to[bend right=30] (obj22) to[bend left=30] (obj20);
    \end{tikzpicture}
  \caption{The action of the pseudo-rotation $\rotB$, Case (1)}
\label{fig:10b}
\end{figure}

\begin{figure}
\begin{center}
  \begin{tikzpicture}[scale=1]
    \polygon{(-4,0)}{obj}{44}{2.5}
       {1,,,,a,,,,,,,,,,,,,,,,\hspace{20pt}mn-1,mn,\overline{1},,,,\overline{a},,,,,,,,,,,,,,,,\overline{mn-1}\hspace{20pt},\overline{mn}}

    \draw[line width=2.5pt,black] (obj1) to[bend left=50] (obj44);
    \draw[line width=2.5pt,black] (obj23) to[bend left=50] (obj22);

    \draw[dash pattern=on 2pt off 2pt on 2pt off 2pt on 2pt off 2pt on
2pt off 16pt on 2pt off 2pt on 2pt off 2pt] (obj1) to[bend right=50]
(obj5);
    \draw[dash pattern=on 2pt off 2pt on 2pt off 2pt on 2pt off 2pt on
2pt off 16pt on 2pt off 2pt on 2pt off 2pt] (obj23) to[bend right=50]
(obj27);

    \draw (obj44) to[bend right=50] (obj5);
    \draw (obj22) to[bend right=50] (obj27);

    \draw[dotted] (obj6) to[bend right=20] (obj21);
    \node at ($($0.62*($(obj13)+(4.0,0)$)$)-(4.0,0)$) {\tiny$X$};
    \draw[dotted] (obj28) to[bend right=20] (obj43);
    \node at ($($0.62*($(obj35)+(4.0,0)$)$)-(4.0,0)$) {\tiny$\overline{X}$};

    \node[inner sep=0pt] at (0,0) {$\mapsto$};

    \polygon{(4,0)}{obj}{44}{2.5}
       {1,2,,,,a+1,,,,,,,,,,,,,,,,mn,\overline{1},\overline{2},,,,\overline{a+1},,,,,,,,,,,,,,,,\overline{mn}}

    \draw[line width=2.5pt,black] (obj1) to[bend left=20] (obj28);
    \draw[line width=2.5pt,black] (obj23) to[bend left=20] (obj6);

    \draw (obj1) to[bend right=0] (obj24);
    \draw (obj23) to[bend right=0] (obj2);

    \draw[dash pattern=on 2pt off 2pt on 2pt off 2pt on 2pt off 2pt on
2pt off 16pt on 2pt off 2pt on 2pt off 2pt] (obj2) to[bend right=50]
(obj6);
    \draw[dash pattern=on 2pt off 2pt on 2pt off 2pt on 2pt off 2pt on
2pt off 16pt on 2pt off 2pt on 2pt off 2pt] (obj24) to[bend right=50]
(obj28);

    \draw[dotted] (obj7) to[bend right=20] (obj22);
    \node at ($($0.62*($(obj14)-(4.0,0)$)$)+(4.0,0)$) {\tiny$X$};
    \draw[dotted] (obj29) to[bend right=20] (obj44);
    \node at ($($0.62*($(obj36)-(4.0,0)$)$)+(4.0,0)$) {\tiny$\overline{X}$};

    \end{tikzpicture}
\end{center}
\caption{The action of the pseudo-rotation $\rotB$,
Case (2)}
\label{fig:11}
\end{figure}

\begin{figure}
  \begin{tikzpicture}[scale=1]
    \polygon{(-4,0)}{obj}{24}{2.5}
      {1,2,3,4,5,6,7,8,9,10,11,12,\overline{1},\overline{2},\overline{3},\overline{4},\overline{5},\overline{6},\overline{7},\overline{8},\overline{9},\overline{10},\overline{11},\overline{12}}

      \draw[line width=2.5pt,black] (obj1) to[bend left=50] (obj24);
      \draw[line width=2.5pt,black] (obj13) to[bend left=50] (obj12);

     \draw[fill=black,fill opacity=0.1] (obj24) to[bend right=30]
(obj1) to[bend right=30] (obj8) to[bend left=30] (obj24);
     \draw[fill=black,fill opacity=0.1] (obj2) to[bend right=30]
(obj6) to[bend right=30] (obj7) to[bend left=30] (obj2);
     \draw[fill=black,fill opacity=0.1] (obj3) to[bend right=30]
(obj4) to[bend right=30] (obj5) to[bend left=30] (obj3);
     \draw[fill=black,fill opacity=0.1] (obj9) to[bend right=30]
(obj10) to[bend right=30] (obj11) to[bend left=30] (obj9);
     \draw[fill=black,fill opacity=0.1] (obj12) to[bend right=30]
(obj13) to[bend right=30] (obj20) to[bend left=30] (obj12);
     \draw[fill=black,fill opacity=0.1] (obj14) to[bend right=30]
(obj18) to[bend right=30] (obj19) to[bend left=30] (obj14);
     \draw[fill=black,fill opacity=0.1] (obj15) to[bend right=30]
(obj16) to[bend right=30] (obj17) to[bend left=30] (obj15);
     \draw[fill=black,fill opacity=0.1] (obj21) to[bend right=30]
(obj22) to[bend right=30] (obj23) to[bend left=30] (obj21);

    \node[inner sep=0pt] at (0,0) {$\mapsto$};

    \polygon{(4,0)}{obj}{24}{2.5}
      {1,2,3,4,5,6,7,8,9,10,11,12,\overline{1},\overline{2},\overline{3},\overline{4},\overline{5},\overline{6},\overline{7},\overline{8},\overline{9},\overline{10},\overline{11},\overline{12}}

      \draw[line width=2.5pt,black] (obj1) to[bend left=30] (obj21);
      \draw[line width=2.5pt,black] (obj13) to[bend left=30] (obj9);

     \draw[fill=black,fill opacity=0.1] (obj2) to[bend right=30]
(obj9) to[bend right=30] (obj13) to[bend left=10] (obj2);
     \draw[fill=black,fill opacity=0.1] (obj3) to[bend right=30]
(obj7) to[bend right=30] (obj8) to[bend left=30] (obj3);
     \draw[fill=black,fill opacity=0.1] (obj4) to[bend right=30]
(obj5) to[bend right=30] (obj6) to[bend left=30] (obj4);
     \draw[fill=black,fill opacity=0.1] (obj10) to[bend right=30]
(obj11) to[bend right=30] (obj12) to[bend left=30] (obj10);

     \draw[fill=black,fill opacity=0.1] (obj14) to[bend right=30]
(obj21) to[bend right=30] (obj1) to[bend left=10] (obj14);
     \draw[fill=black,fill opacity=0.1] (obj15) to[bend right=30]
(obj19) to[bend right=30] (obj20) to[bend left=30] (obj15);
     \draw[fill=black,fill opacity=0.1] (obj16) to[bend right=30]
(obj17) to[bend right=30] (obj18) to[bend left=30] (obj16);
     \draw[fill=black,fill opacity=0.1] (obj22) to[bend right=30]
(obj23) to[bend right=30] (obj24) to[bend left=30] (obj22);

    \end{tikzpicture}
\caption{The action of the pseudo-rotation $\rotB$,
Case (2)}
\label{fig:11b}
\end{figure}

\begin{remarks}
(i) \Cref{fig:10} provides a schematic illustration of 
the construction in Case~(1) of the statement of 
the above theorem, while \Cref{fig:10b} provides a concrete
example. In this example, $m=3$, $n=4$, $a=2$, $b=6$, $d=11$, and $e=10$.
Case~(2) of the statement of \Cref{prop:2B} is illustrated in
\Cref{fig:11}, while \Cref{fig:11b} provides a concrete example, in
which $m=3$, $n=4$, and $a=8$.

\medskip
(ii) Case~(2) can be seen as a degenerate special case of Case~(1).
To see this, one chooses $b=mn$ and $d=a$ in Case~(1). We shall
occasionally refer to this point of view in order to simplify
arguments in proofs.
\end{remarks}

\begin{proof}[Proof of \Cref{prop:2B}]
Let us first assume that, 
in the positive $m$-divisible type~$B$ non-crossing partition~$\pi$,
the block of~$\overline{mn}$ contains another negative element.
Among those, let $\overline{mi_1-1}$ be the one which is the
predecessor of~$\overline{mn}$ in the block.
Obviously, we have $i_1>0$.
\Cref{fig:16} provides sketches of such situations. 
Let $(w_0,w_1,\dots,w_m)$ denote the
element of $\mNCPlus[B_n]$ which corresponds to~$\pi$ under the
bijection $\Nam{B_n}m$. By the definition of $\Nam{B_n}m$, we
see that $w_{m-1}(n)=i_1$. From
\Cref{lem:LRB}(1), we infer $w_{m-1}^L=w_{m-1}$ and
$w_{m-1}^R=\one$. This means that the pseudo-rotation~$\rotB$
reduces to ordinary rotation in this special case, as we claimed. 

\begin{figure}
\begin{center}
\begin{tikzpicture}
    \polygonlabel{(-4,-21)}{obj}{44}{2.5}
      {1,,,,,,,,,,,,,,,,,,,,,,,,,,,,\hspace{1pt},,,,,,,,\overline{mi_1-1},,,,,,\overline{mn},}

    \draw[line width=2.5pt,black] (obj1) to[bend left=30] (obj29);
    \draw (obj1)  to[bend left=30]
          (obj29);
    \draw (obj43) to[bend left=50]
          (obj37);

    \draw[dash pattern=on 22pt off 100pt] (obj43) to[bend left=20] (obj29);
    \draw[dash pattern=on 22pt off 100pt] (obj37) to[bend left=20] (obj29);
    \draw[dash pattern=on 22pt off 140pt] (obj1)  to[bend right=20] (obj8);

    \node[inner sep=0pt] at (0,-21) {};

    \polygonlabel{( 4,-21)}{obj}{44}{2.5}
      {1,,,,,,,,,,,,,,,,,,,,,,,,,,,,,,,,,,,,\overline{mi_1-1},,,,,,\overline{mn},}

    \draw[line width=2.5pt,black] (obj1) to[bend left=50] (obj43);
    \draw (obj43) to[bend left=50]
          (obj37);

    \draw[dash pattern=on 22pt off 100pt] (obj37) to[bend left=20] (obj29);
    \draw[dash pattern=on 22pt off 140pt] (obj1)  to[bend right=20] (obj8);

\end{tikzpicture}
\end{center}
  \caption{The block of $\overline{mn}$ contains another negative element}
\label{fig:16}
\end{figure}
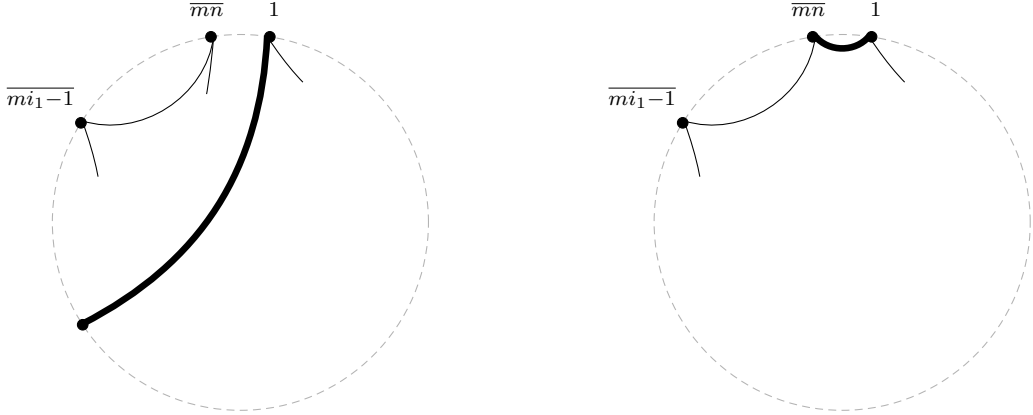

\medskip
For the remainder of the proof, we assume that the block
of~$\overline{mn}$ contains no other negative element. In particular,
it must contain a positive number (since $m\ge2$).
By \Cref{prop:1B}, we know that the block of~$1$ contains a negative
number. Hence, if the partition wants to be non-crossing, necessarily
$1$ must be in the same block as $\overline{mn}$.

\medskip
Let again $(w_0,w_1,\dots,w_m)$ denote the
element of $\mNCPlus[B_n]$ which corresponds to~$\pi$ under the
bijection $\Nam{B_n}m$. 

\begin{figure}
\begin{center}
\begin{tikzpicture}
    \polygonlabel{(-4,0)}{obj}{44}{2.5}
      {1,,,a,a+1,,,a_2,a_2+1,,,a_3,,,a_{s-1}+1,,,b-1,b,,,,,,,,,,,,,,,,,,,\overline{d},,,,,,\overline{N}}

    \draw[line width=2.5pt,black] (obj1) to[bend left=50] (obj44);
    \draw (obj44) to[bend right=50]
          (obj4);

    \draw (obj5) to[bend right=90, looseness=2]
          (obj8);

    \draw (obj9) to[bend right=90, looseness=2]
          (obj12);

    \draw (obj15) to[bend right=90, looseness=2]
          (obj18);

    \draw (obj19) to [bend left=10] (obj38);

    \draw[dash pattern=on 2pt off 10pt on 2pt off 2pt on 2pt off 2pt]
(obj12) to[bend right=50] (obj15);

\end{tikzpicture}
\end{center}
  \caption{The block structure in Case (1)}
\label{fig:16b}
\end{figure}
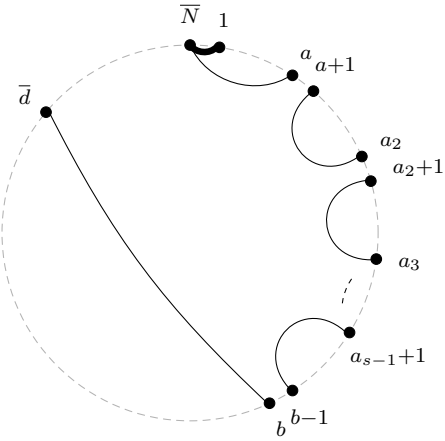

\smallskip
{\sc Case (1).} 
Let $a_2,a_3,\dots,a_{s-1}$ be positive integers with
$a<a_2<a_3<\dots<a_{s-1}$ such that 
$a_2$ and $a+1$ (sic!) are successive elements in a block of~$\pi$,
$a_3$ and $a_2+1$ are successive elements in a block,
\dots, and
$b-1$ and $a_{s-1}+1$ are successive elements in a block.
\Cref{fig:16b} provides a sketch of this situation.

The congruence classes of $a,b,d,e,a_2,\dots,a_{s-1}$ 
modulo~$m$ are forced: indeed, 
we have $a\equiv-1$~(mod~$m$) since $a$ is the predecessor
of~$\overline{mn}$ in their block,
we have $a_2\equiv-1$~(mod~$m$) since $a_2$ is the predecessor of
$a+1$ in their block,
we have $a_3\equiv-1$~(mod~$m$) since $a_3$ is the predecessor of
$a_2+1$ in their block, \dots,
we have $b\equiv0$~(mod~$m$) since $b-1$ is the predecessor of
$a_{s-1}+1$ in their block, and
we have $d\equiv-1$~(mod~$m$) and
$e\equiv1$~(mod~$m$) since $\overline d$, $b$, $e$ are successive
elements in their block. We write
$a=mA-1$, 
$b=mB$,
$d=mD-1$, 
$e=mE+1$, and $a_i=mA_i-1$, $i=2,3,\dots,s-1$, 
for suitable integers 
$A,B,D,E,A_2,A_3,\dots,A_{s-1}$. 
By the definition of $\Nam{B_n}m$,
we have
\begin{align}
\notag
w_{m-1}(\overline n)&=A,\\
\notag
w_{m-1}(A)&=A_2,\\
\notag
w_{m-1}(A_2)&=A_3,\\
\notag
w_{m-1}(A_{s-1})&=B,\\
\notag
w_{m-1}(B)&=\overline D,\\
\label{eq:wmex3}
w_m(E)&=B,\\
\label{eq:wmex4}
w_m(n)&=n.
\end{align}


\noindent
Consequently, the cycle of $w_{m-1}$ containing $n$ is of the
form 
$$((\overline A,\overline A_2,\dots,\overline A_{s-1},\overline B,D,\dots,n)).$$
By \Cref{lem:LRB}(2), we have 
$$w^R_{m-1}=((\overline B,n)),\quad \text{and }
w^L_{m-1}\text{ contains the cycles }((A,A_2,\dots,A_{s-1},B))
((D,\dots,n)).$$
Together with \eqref{eq:wmex3} and \eqref{eq:wmex4}, 
this implies the relations
\begin{align*}
w^L_{m-1}(n)&=D,\\
w^L_{m-1}(A)&=A_2,\\
w^L_{m-1}(A_2)&=A_3,\\
w^L_{m-1}(A_{s-1})&=B,\\
w^L_{m-1}(B)&=A,\\
(cw^R_{m-1}w_mc^{-1})(1)&=B+1,\\
(cw^R_{m-1}w_mc^{-1})(E+1)&=1.
\end{align*}
By the definition of $\Krewplustilde^{(m)}$ and of $\Nam{B_n}m$,
this means that
$a+1$, $b+1$, and~$2$ are successive elements in their block
in the image of~$\pi$ under~$\Krewplustilde^{(m)}$ (conjugated by~$\Nam{B_n}m$),
that $\overline{d+1}$, $1$, and $e+1$ are successive elements
in a block,
while all other ``block connections" are rotated by one unit
in clockwise direction.
This is in accordance with the corresponding assertion in the
proposition.

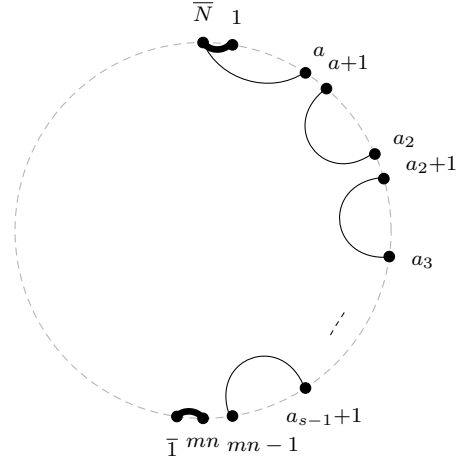
\begin{figure}
\begin{center}
\begin{tikzpicture}
    \polygonlabel{( 4,0)}{obj}{44}{2.5}
      {1,,,a,a+1,,,a_2,a_2+1,,,a_3,,,,,,a_{s-1}+1,,,\hbox{\kern20pt\tiny$mn-1$},mn,\overline{1},,,,,,,,,,,,,,,,,,,,,\overline{N}}

    \draw[line width=2.5pt,black] (obj1) to[bend left=50] (obj44);
    \draw[line width=2.5pt,black] (obj22) to[bend right=50] (obj23);
    \draw (obj44) to[bend right=50]
          (obj4);

    \draw (obj5) to[bend right=90, looseness=2]
          (obj8);

    \draw (obj9) to[bend right=90, looseness=2]
          (obj12);

    \draw (obj18) to[bend right=90, looseness=2]
          (obj21);

    \draw[dash pattern=on 2pt off 25pt on 2pt off 2pt on 2pt off 2pt]
(obj12) to[bend right=10] (obj18);

\end{tikzpicture}
\end{center}
  \caption{The block structure in Case (2)}
\label{fig:16c}
\end{figure}

\smallskip
{\sc Case (2).} 
Let $a_2,a_3,\dots,a_k$ be positive integers with
$a<a_2<a_3<\dots<a_k$ such that 
$a_2$ and $a+1$ are successive elements in a block,
$a_{i+1}$ and $a_i+1$ are successive elements in a block,
$i=2,3,\dots,s-2$, and
$mn-1$ and $a_{s-1}+1$ are successive elements in a block.
\Cref{fig:16c} provides a sketch of this situation.

As in Case~(1), we must have $a\equiv-1$~(mod~$m$),
and $a_i\equiv-1$~(mod~$m$), for $i=2,3,\dots,s-1$.
We write again $a=mA-1$ and $a_i=mA_i-1$ for suitable 
integers $A$ and $A_i$, for $i=2,3,\dots,s-1$.
By the definition of  $\Nam{B_n}m$,
we have
\begin{align}
\notag
w_{m-1}(\overline n)&=A,\\
\notag
w_{m-1}(A)&=A_2,\\
\notag
w_{m-1}(A_2)&=A_3,\\
\notag
w_{m-1}(A_{s-2})&=A_{s-1},\\
\notag
w_{m-1}(A_{s-1})&=n,\\
\label{eq:wmex4b}
w_m(n)&=n.
\end{align}
Consequently, the cycle of $w_{m-1}$ containing $n$ is of the
form 
$$[A,A_2,\dots,\dots,A_{s-1},n].$$
By \Cref{lem:LRB}(3), we have 
$$w^R_{m-1}=[n],\quad \text{and }
w^L_{m-1}\text{ contains the cycle }((A,A_2,\dots,A_{s-1},n)).$$
Together with \eqref{eq:wmex4b}, 
this implies the relations
\begin{align*}
w^L_{m-1}(n)&=A,\\
w^L_{m-1}(A)&=A_2,\\
w^L_{m-1}(A_2)&=A_3,\\
w^L_{m-1}(A_{s-2})&=A_{s-1},\\
w^L_{m-1}(A_{s-1})&=n,\\
(cw^R_{m-1}w_mc^{-1})(1)&=\overline 1.
\end{align*}
By the definition of  $\Nam{B_n}m$, this means that
$\overline{a+1}$, $1$, and~$\overline 2$ are successive elements in 
their block in the image of~$\pi$ under~$\Krewplustilde^{(m)}$ (conjugated by~$\Nam{B_n}m$),
while all other ``block connections" are rotated by one unit in clockwise direction.
This is in accordance with the corresponding assertion in the
proposition.
Before finishing the argument, we would like to offer, to the first
three readers who made it thus far, a one-dollar bill hand-signed by Christian.
The bill is claimed by sending us a {\it non-trivial\/}
positive non-crossing partition
in $\mNCBPlus[4][2]$ and its image under the positive Kreweras map~$\rotB$.
This completes the proof of the proposition.
\end{proof}

\subsection{The combinatorial positive Kreweras maps}
\label{sec:combB}

The definition of $\rotB$ from \Cref{prop:2B} to
make sense on an \emph{arbitrary} non-crossing
partition needed that blocks should be of size at least~$2$.
It is a simple matter to extend this so that it
makes sense for non-crossing partitions without any restrictions.

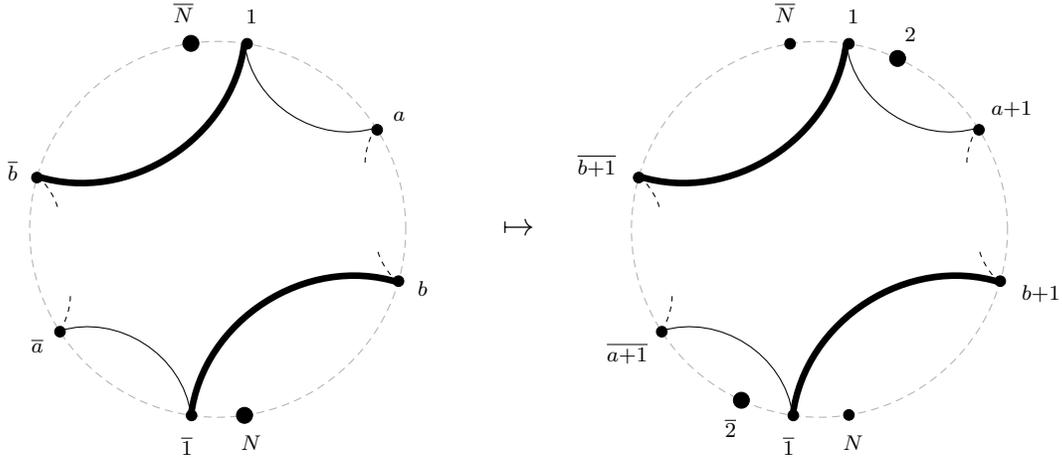
\begin{figure}
\begin{center}
    \begin{tikzpicture}[scale=1]
    \polygonlabel{(-4,0)}{obj}{44}{2.5}
      {1,,,,,,a,,,,,,b,,,,,,,,N,,\overline{1},,,,,,\overline{a},,,,,,\overline{b},,,,,,,,\overline{N},}

    \draw[line width=2.5pt,black] (obj1) to[bend left=50] (obj35);
    \draw (obj35) to[bend right=50]
          (obj1) to[bend right=50]
          (obj7);

    \draw[black, dash pattern=on 2pt off 2pt on 2pt off 2pt on 2pt off
2pt on 2pt off 35pt] (obj7) to[bend right=50] (obj11);
    \draw[black, dash pattern=on 2pt off 2pt on 2pt off 2pt on 2pt off
2pt on 2pt off 35pt] (obj35) to[bend left=50] (obj31);

    \filldraw[black] (obj43) circle(3pt);

    \draw[line width=2.5pt,black] (obj23) to[bend left=50] (obj13);
    \draw (obj13) to[bend right=50]
          (obj23) to[bend right=50]
          (obj29);

    \draw[black, dash pattern=on 2pt off 2pt on 2pt off 2pt on 2pt off
2pt on 2pt off 35pt] (obj29) to[bend right=50] (obj33);
    \draw[black, dash pattern=on 2pt off 2pt on 2pt off 2pt on 2pt off
2pt on 2pt off 35pt] (obj13) to[bend left=50] (obj9);

    \filldraw[black] (obj21) circle(3pt);

    \node[inner sep=0pt] at (0,0) {$\mapsto$};

    \polygonlabel{( 4,0)}{obj}{44}{2.5}
      {1,,2,,,,a+1,,,,,,b+1,,,,,,,,N,,\overline{1},,\overline{2},,,,\overline{a+1},,,,,,\overline{b+1},,,,,,,,\overline{N},}

    \draw[line width=2.5pt,black] (obj1) to[bend left=50] (obj35);
    \draw (obj35) to[bend right=50]
          (obj1) to[bend right=50]
          (obj7);

    \draw[black, dash pattern=on 2pt off 2pt on 2pt off 2pt on 2pt off
2pt on 2pt off 35pt] (obj7) to[bend right=50] (obj11);
    \draw[black, dash pattern=on 2pt off 2pt on 2pt off 2pt on 2pt off
2pt on 2pt off 35pt] (obj35) to[bend left=50] (obj31);

    \filldraw[black] (obj3) circle(3pt);

    \draw[line width=2.5pt,black] (obj23) to[bend left=50] (obj13);
    \draw (obj13) to[bend right=50]
          (obj23) to[bend right=50]
          (obj29);

    \draw[black, dash pattern=on 2pt off 2pt on 2pt off 2pt on 2pt off
2pt on 2pt off 35pt] (obj29) to[bend right=50] (obj33);
    \draw[black, dash pattern=on 2pt off 2pt on 2pt off 2pt on 2pt off
2pt on 2pt off 35pt] (obj13) to[bend left=50] (obj9);

    \filldraw[black] (obj25) circle(3pt);

\end{tikzpicture}
\end{center}
 \caption{Definition of the pseudo-rotation under presence of
singleton blocks}
\label{fig:8B}
\end{figure}

\begin{definition} \label{def:phiallgB}
We define the action of $\rotB$ on a positive type~$B$ 
non-crossing partition~$\pi$
of $\{1,2,\dots, N,\overline1,\overline2,\dots,\overline N\}$ 
to be given by \Cref{prop:2B} with $mn$ replaced by
$N$ (cf.\ \Cref{fig:10,fig:11}), except when the block containing $N$ is a
singleton block (in which case also $\{\overline{N}\}$ is a singleton 
block; here, \Cref{prop:2B} would not make sense).
If $\{N\}$ is a singleton block, then the image of~$\rotB$
is the non-crossing partition which arises from~$\pi$ 
by replacing the singleton block $\{N\}$ 
by the singleton block $\{\overline 2\}$, by replacing the
singleton block $\{\overline N\}$ by the singleton block $\{2\}$, 
and $i$ by~$i+1$ and $\overline i$ by~$\overline{i+1}$ for
$i=2,3,\dots,N-1$
in each of the remaining blocks of~$\pi$, leaving~$1$ and~$\overline1$
untouched; see \Cref{fig:8B}
for an illustration.
In the degenerate case of Case~(1) in \Cref{prop:2B} where $b=d$,
the partition $\pi$ contains $\{b,\overline b\}=\{d,\overline d\}$ as zero
block\footnote{See \Cref{sec:realB} for the
definition of zero block.} (see the left half of
\Cref{fig:d=b} for a schematic illustration of that situation,
which can only occur for $m=1$; cf.\ \Cref{foot:d=b}).
The image of~$\pi$ under~$\rotB$ is then defined as the partition 
in which $1$ and $\overline1$ form a block,
and $a+1$, $b+1$, and $2$ are successive elements in 
another block. All other succession relations in~$\pi$
are rotated by one unit in clockwise direction.
See \Cref{fig:d=b} for a schematic illustration of this operation
and \Cref{fig:d=b2} for a concrete example in which $N=7$, $a=2$,
and $b=d=3$.
\end{definition}

\begin{figure}
\begin{center}
    \begin{tikzpicture}[scale=1]
    \polygonlabel{(-4,0)}{obj}{44}{2.5}
      {1,,,,,,a,,,,,,b,,,,,,,,N,,\overline{1},,,,,,\overline{a},,,,,,\overline{b},,,,,,,,\overline{N},}

    \draw[line width=2.5pt,black] (obj1) to[bend left=50] (obj43);
    \draw (obj43) to[bend right=50]
          (obj7);

    \draw[black, dash pattern=on 2pt off 2pt on 2pt off 2pt on 2pt off
2pt on 2pt off 35pt] (obj7) to[bend left=20] (obj1);

    \draw[line width=2.5pt,black] (obj23) to[bend left=50] (obj21);
    \draw (obj21) to[bend right=50]
          (obj29);

    \draw[black, dash pattern=on 2pt off 2pt on 2pt off 2pt on 2pt off
2pt on 2pt off 35pt] (obj29) to[bend left=20] (obj23);

    \draw[black,fill=black,fill opacity=0.1] (obj13) to[bend left=20] (obj35)
                         to[bend left=20] (obj13);

    \node[inner sep=0pt] at (0,0) {$\mapsto$};

    \polygonlabel{( 4,0)}{obj}{44}{2.5}
      {1,,2,,,,a+1,,,,,,b+1,,,,,,,,N,,\overline{1},,\overline{2},,,,\overline{a+1},,,,,,\overline{b+1},,,,,,,,\overline{N},}

    \draw[black] (obj3) to[bend right=50] (obj13) to[bend left=30] (obj7);

    \draw[black, dash pattern=on 2pt off 2pt on 2pt off 2pt on 2pt off
20pt] (obj7) to[bend left=20] (obj3);

    \draw[black] (obj25) to[bend right=50] (obj35) to[bend left=30] (obj29);

    \draw[black, dash pattern=on 2pt off 2pt on 2pt off 2pt on 2pt off
20pt] (obj29) to[bend left=20] (obj25);

    \draw[line width=2.5pt,black,fill=black,fill opacity=0.1] (obj1)
to[bend left=20] (obj23)
                         to[bend left=20] (obj1);
\end{tikzpicture}
\end{center}
\caption{The action of the pseudo-rotation $\rotB$,
Case (1) with $b=d$}
\label{fig:d=b}
\end{figure}

\begin{figure}
\begin{center}
  \begin{tikzpicture}[scale=1]
    \polygon{(-4,0)}{obj}{14}{2.5}
      {1,2,3,4,5,6,7,\overline{1},\overline{2},\overline{3},\overline{4},\overline{5},\overline{6},\overline{7}}

      \draw[line width=2.5pt,black] (obj1) to[bend left=30] (obj14);
      \draw[line width=2.5pt,black] (obj8) to[bend left=30] (obj7);

     \draw[fill=black,fill opacity=0.1] (obj14) to[bend right=30]
(obj1) to[bend right=30] (obj2) to[bend left=30] (obj14);
     \draw[fill=black,fill opacity=0.1] (obj7) to[bend right=30]
(obj8) to[bend right=30] (obj9) to[bend left=30] (obj7);

     \draw[fill=black,fill opacity=0.1] (obj4) to[bend right=50]
(obj6) to[bend left=20] (obj4);
     \draw[fill=black,fill opacity=0.1] (obj11) to[bend right=50]
(obj13) to[bend left=20] (obj11);

     \draw[fill=black,fill opacity=0.1] (obj3) to[bend right=20]
(obj10) to[bend right=20] (obj3);

    \node[inner sep=0pt] at (0,0) {$\mapsto$};

    \polygon{(4,0)}{obj}{14}{2.5}
      {1,2,3,4,5,6,7,\overline{1},\overline{2},\overline{3},\overline{4},\overline{5},\overline{6},\overline{7}}

     \draw[fill=black,fill opacity=0.1] (obj2) to[bend right=30]
(obj3) to[bend right=30] (obj4) to[bend left=30] (obj2);
     \draw[fill=black,fill opacity=0.1] (obj9) to[bend right=30]
(obj10) to[bend right=30] (obj11) to[bend left=30] (obj9);

     \draw[fill=black,fill opacity=0.1] (obj5) to[bend right=50]
(obj7) to[bend left=20] (obj5);
     \draw[fill=black,fill opacity=0.1] (obj12) to[bend right=50]
(obj14) to[bend left=20] (obj12);

     \draw[line width=2.5pt,fill=black,fill opacity=0.1] (obj1)
to[bend right=20]
(obj8) to[bend right=20] (obj1);
    \end{tikzpicture}
\end{center}
\caption{The action of the pseudo-rotation $\rotB$,
Case (1) with $b=d$}
\label{fig:d=b2}
\end{figure}

As it turns out, this definition is also in agreement with the
image under $\Nam{B_{n}}1$ of the map $\Krewplustilde^{(1)}$ in
\Cref{def:K-alt} for the case where $m=1$.

\begin{theorem} \label{prop:2-m=1B}
Let $n$ be a positive integer.
Under the bijection $\Nam{B_{n}}1$, the map $\Krewplustilde^{(1)}$
from \Cref{def:K-alt} translates to the combinatorial map~$\rotB$ described in
\Cref{def:phiallgB}  with $N=n$.
\end{theorem}

Since the proof of the theorem is completely analogous to the proof
of \Cref{prop:2-m=1}, we omit it here.

\medskip
We conclude with the following theorem describing the order of the combinatorial\break map~$\rotB$. In view of our previous discussion, it follows from
\Cref{cor:order} for the type~$B_N$ with $m=1$.
However, since we believe it is interesting for its own sake, we
provide a direct, combinatorial, proof in \Cref{app:order-B}.

\begin{theorem} \label{lem:N-2B}
Let $N$ be a positive integer.
The order of the map~$\rotB$ acting on the set of positive type~$B$ 
non-crossing partitions~$\pi$
of\/ $\{1,2,\dots, N,\overline1,\overline2,\dots,\overline N\}$ 
is $N-1$.
\end{theorem}

\subsection{Enumeration of positive non-crossing partitions}
\label{sec:enumB}

Here we present our enumeration results for positive non-crossing
partitions in $\mNCBPlus$.
The most refined and most general
result is \Cref{thm:1-B} which provides a
formula for the number of these
non-crossing partitions in which the block structure is prescribed. 
By specialisation and summation, we obtain results on the number of
such partitions with prescribed rank or size of zero block, including
the total number of elements in $\mNCBPlus$.

\begin{theorem} \label{thm:1-B}
Let $m$ and $n$ be positive integers, and let $a$ and
$b_1,b_2,\dots,b_n$ be non-negative integers.
The number of positive $m$-divisible non-crossing partitions of 
$\{1,2,\dots,\break mn,\overline1,\overline2,\dots,\overline{mn}\}$ of type~$B$, where
the number of non-zero blocks of size~$mi$ is~$2b_i$, $i=1,2,\dots,n$,
the zero block having size $2ma=2mn-2m \sum_{j=1}^njb_j$
{\em(}if $a=0$ this is interpreted as the absence of a zero block{\em)},
is given by
\begin{equation} \label{eq:block-B} 
\binom {b_1+b_2+\dots+b_n}{b_1,b_2,\dots,b_n}
\binom {mn-1} {b_1+b_2+\dots+b_n}.
\end{equation}
\end{theorem}

The proof of this theorem is found in \Cref{app:GF-B-inv}.

\begin{corollary} \label{cor:2-B}
Let $m$ and $n$ be positive integers, and let
$a$ and $b$ be non-negative integers.
The number
of positive $m$-divisible non-crossing partitions of 
$\{1,2,\dots,mn,\overline1,\overline2,\dots,\overline{mn}\}$ of type~$B$
which have $2b$ non-zero blocks and a zero block of size~$2ma$
{\em(}if $a=0$ this is interpreted as the absence of a zero block{\em)},
is given by
\begin{equation} \label{eq:multichains-si-pos-a-B-1} 
\binom {{n-a-1}}{b-1}
\binom {mn-1}{b},
\end{equation}
while the number of these partitions without restricting the size 
{\em(}or existence{\em)} of the zero block is
\begin{equation} \label{eq:multichains-si-pos-B-1} 
\binom {{n}}{b}
\binom {mn-1}{b}.
\end{equation}
Via the relation $s=n-b$, the last expression also equals the number
of elements of $\mNCBPlus$ of fixed rank~$s$.
\end{corollary}

The proof of this corollary is also found in \Cref{app:GF-B-inv}.

\begin{corollary} \label{cor:3-B}
Let $m$ and $n$ be positive integers, and let $a$ be a non-negative integer.
The number
of positive $m$-divisible non-crossing partitions of 
$\{1,2,\dots,mn,\overline1,\overline2,\dots,\overline{mn}\}$ of type~$B$
which have a zero block of size~$2ma$
{\em(}if $a=0$ this is interpreted as the absence of a zero block{\em)}
is given by
\begin{equation} \label{eq:multichains-pos-a-B-1} 
\binom {{(m+1)n-a-2 }}{{n-a }},
\end{equation}
while the number of these partitions without restricting the size
{\em(}or existence{\em)} of the zero block is
\begin{equation} \label{eq:multichains-pos-B-1} 
\binom {{(m+1)n-1}}{{n}}.
\end{equation}
\end{corollary}

\begin{proof}
One sums Expressions~\eqref{eq:multichains-si-pos-a-B-1} 
and~\eqref{eq:multichains-si-pos-B-1} 
over all~$b$.
The resulting sums can be evaluated by means of the Chu--Vandermonde
summation formula.
\end{proof}

\subsection{Characterisation of pseudo-rotationally invariant elements}
\label{sec:rotB}

In this subsection, we describe the elements of
$\NCBPlus$ that are invariant under a
power of~$\rotB$. 
We call an element of $\NCBPlus$
\defn{$2\rRr $-pseudo-rotationally 
invariant\/} if it is invariant under the action of
$\rotB^{2(N-1)/(2\rRr) }=\rotB^{(N-1)/\rRr }$.
This --- perhaps seemingly unintuitive --- definition has its background 
in the cyclic sieving result(s) that we aim to eventually state and prove
in \Cref{sec:sievB}. While it is shown in \Cref{lem:N-2B} 
that the order of~$\rotB$ is~$N-1$, in the statement of the cyclic sieving
phenomena in \Cref{thm:3-B} we need to consider the cyclic group generated 
by~$\rotB$ as a group of order $2(N-1)$ in order to fit with the uniform 
statement of cyclic sieving in \Cref{thm:CS}.

We achieve the description of pseudo-rotationally invariant elements
in $\NCBPlus$
by providing a construction of such elements (see below)
together with \Cref{lem:allB} which then tells us how to obtain all
such invariant elements from the construction. 
To this end, we 
adopt the notions of rotation~$\rot$ of a sequence and of a central
block from \Cref{sec:rotA}.

\begin{figure}
\begin{center}
  \begin{tikzpicture}[scale=1]
    \polygon{(0,0)}{obj}{30}{2.5}
      {1,2,3,4,5,6,7,8,9,10,11,12,13,14,15,\overline{1},\overline{2},\overline{3},\overline{4},\overline{5},\overline{6},\overline{7},\overline{8},\overline{9},\overline{10},\overline{11},\overline{12},\overline{13},\overline{14},\overline{15}}
     \draw[fill=black,fill opacity=0.1] (obj1) to[bend right=30]
(obj2) to[bend right=30] (obj9) to[bend right=30] (obj16) to[bend
right=30] (obj17) to[bend right=30] (obj24) to[bend right=30] (obj1);
     \draw[fill=black,fill opacity=0.1] (obj3) to[bend right=30]
(obj7) to[bend right=30] (obj8) to[bend left=30] (obj3);
     \draw[fill=black,fill opacity=0.1] (obj4) to[bend right=30]
(obj5) to[bend right=30] (obj6) to[bend left=30] (obj4);
     \draw[fill=black,fill opacity=0.1] (obj10) to[bend right=30]
(obj14) to[bend right=30] (obj15) to[bend left=30] (obj10);
     \draw[fill=black,fill opacity=0.1] (obj11) to[bend right=30]
(obj12) to[bend right=30] (obj13) to[bend left=30] (obj11);
     \draw[fill=black,fill opacity=0.1] (obj18) to[bend right=30]
(obj22) to[bend right=30] (obj23) to[bend left=30] (obj18);
     \draw[fill=black,fill opacity=0.1] (obj19) to[bend right=30]
(obj20) to[bend right=30] (obj21) to[bend left=30] (obj19);
     \draw[fill=black,fill opacity=0.1] (obj25) to[bend right=30]
(obj29) to[bend right=30] (obj30) to[bend left=30] (obj25);
     \draw[fill=black,fill opacity=0.1] (obj26) to[bend right=30]
(obj27) to[bend right=30] (obj28) to[bend left=30] (obj26);
    \end{tikzpicture}
\end{center}
  \caption{A $4$-pseudo-rotationally invariant non-crossing partition in
$\NCBPlus[15]$ (and also in $\mNCBPlus[5][3]$)}
\label{fig:34}
\end{figure}

\begin{figure}
\begin{center}
\begin{tikzpicture}[scale=1]
    \polygon{(0,0)}{obj}{30}{2.5}
      {1,2,3,4,5,6,7,8,9,10,11,12,13,14,15,\overline{1},\overline{2},\overline{3},\overline{4},\overline{5},\overline{6},\overline{7},\overline{8},\overline{9},\overline{10},\overline{11},\overline{12},\overline{13},\overline{14},\overline{15}}
     \draw[fill=black,fill opacity=0.1] (obj6) to[bend right=30]
(obj7) to[bend right=30] (obj14) to[bend right=30] (obj21) to[bend
right=30] (obj22) to[bend right=30] (obj29) to[bend right=30] (obj6);
     \draw[fill=black,fill opacity=0.1] (obj2) to[bend right=30]
(obj3) to[bend right=30] (obj4) to[bend left=30] (obj2);
     \draw[fill=black,fill opacity=0.1] (obj30) to[bend right=30]
(obj1) to[bend right=30] (obj5) to[bend left=30] (obj30);
     \draw[fill=black,fill opacity=0.1] (obj9) to[bend right=30]
(obj10) to[bend right=30] (obj11) to[bend left=30] (obj9);
     \draw[fill=black,fill opacity=0.1] (obj8) to[bend right=30]
(obj12) to[bend right=30] (obj13) to[bend left=30] (obj8);
     \draw[fill=black,fill opacity=0.1] (obj17) to[bend right=30]
(obj18) to[bend right=30] (obj19) to[bend left=30] (obj17);
     \draw[fill=black,fill opacity=0.1] (obj15) to[bend right=30]
(obj16) to[bend right=30] (obj20) to[bend left=30] (obj15);
     \draw[fill=black,fill opacity=0.1] (obj24) to[bend right=30]
(obj25) to[bend right=30] (obj26) to[bend left=30] (obj24);
     \draw[fill=black,fill opacity=0.1] (obj23) to[bend right=30]
(obj27) to[bend right=30] (obj28) to[bend left=30] (obj23);
    \end{tikzpicture}
\end{center}
  \caption{Another $4$-pseudo-rotationally invariant non-crossing partition in
$\NCBPlus[15]$ (and also in $\mNCBPlus[5][3]$)}
\label{fig:35}
\end{figure}

\medskip
\noindent
{\sc Construction.} 
We start with an (ordinary)
$2\rRr $-rotationally invariant 
non-crossing partition of $\{2,3,\dots,N,\overline
2,\overline3,\dots,\overline{N}\}$, $2$ being
contained in the zero block. We also allow the
degenerate case that the zero block is empty. 
Then we add $1$ and $\overline1$ to the zero block, thereby creating a
positive non-crossing partition.
We illustrate this construction in the case where $\rRr=2$ and $N=15$.
We start with the non-crossing partition 
\begin{multline*}
\{\{2,9,\overline2,\overline9\},\{3,7,8\},\{4,5,6\},
\{10,14,15\},\{11,12,13\},\\
\{\overline3,\overline7,\overline8\},\{\overline4,\overline5,\overline6\},
\{\overline{10},\overline{14},\overline{15}\},
\{\overline{11},\overline{12},\overline{13}\}\}
\end{multline*}
of $\{2,3,\dots,\overline{15}\}$, 
which is indeed $4$-rotationally invariant
in the above sense. Then one adds $1$ and $\overline1$ to the
zero block $\{2,9,\overline2,\overline9\}$ to obtain the partition shown in
\Cref{fig:34}. If one applies $\rotB$ five times to this
partition, then the result is the partition in \Cref{fig:35}.

\begin{theorem} \label{lem:allB}
Let $N$ and $\rRr$ be positive integers with $\rRr \ge2$. Then
all $2\rRr $-pseudo-rotationally invariant positive non-crossing
  partitions of\/ $\{1,2,\dots,N,\overline1,\overline2,\dots,\overline{N}\}$
of type~$B$ 
are obtained by starting with a non-crossing partition
obtained by the above Construction
and applying the pseudo-rotation~$\rotB$ repeatedly to it.
\end{theorem}

The proof of this theorem is given in \Cref{app:inv-B}.

%

\subsection{Enumeration of pseudo-rotationally invariant elements}
\label{sec:rotenumB}

With the characterisation of pseudo-rotationally invariant elements
of $\NCBPlus$ at hand, we can now embark on
the enumeration of such elements. 
The most refined and most general
result in this subsection is \Cref{thm:2-B} which provides a
formula for the number of these non-crossing partitions in which
the block structure is prescribed.

\begin{theorem} \label{thm:2-B}
Let $m,n,a,\rRr $ be positive integers with $\rRr \mid (mn-1)$.
Furthermore, let $b_1,b_2,\dots,b_n$ be non-negative integers.
The number of positive $m$-divisible non-crossing partitions of 
$\{1,2,\dots,mn,\overline1,\overline2,\dots,\overline{mn}\}$ of type~$B$
which are invariant under the
$2\rRr $-pseudo-rotation~$\rotB^{(mn-1)/\rRr }$, where the
number of non-zero 
blocks of size~$mi$ is~$2\rRr b_i$, $i=1,2,\dots,n$,
the zero block having size $2ma=2mn-2m\rRr \sum_{j=1}^njb_j>0$,
is given by
\begin{equation} \label{eq:multichains-bi-si-pos-B} 
\binom {b_1+b_2+\dots+b_n}{b_1,b_2,\dots,b_n}
\binom {(mn-1)/\rRr } {b_1+b_2+\dots+b_n}.
\end{equation}
\end{theorem}

\begin{corollary} \label{cor:4-B}
Let $m,n,a,\rRr $ be positive integers with $\rRr \ge2$ and $\rRr \mid (mn-1)$.
Furthermore, let $b$ be a non-negative integer.
The number
of positive $m$-divisible non-crossing partitions of 
$\{1,2,\dots,mn,\overline1,\overline2,\dots,\overline{mn}\}$ of type~$B$
which are invariant under the $2\rRr $-pseudo-rotation~$\rotB^{(mn-1)/\rRr }$,
have $2\rRr b$ non-zero blocks and a zero block of size~$2ma$,
is given by
\begin{equation} \label{eq:multichains-si-pos-a-B} 
\binom {{(n-a-\rRr )/\rRr }}{b-1}
\binom {(mn-1)/\rRr }{b}
\end{equation}
if $a\equiv n$~{\em(mod~$\rRr $)} and $0$ otherwise,
while the number of these partitions without restricting the size 
of the zero block is
\begin{equation} \label{eq:multichains-si-pos-B} 
\binom {\fl{n/\rRr }}{b}
\binom {(mn-1)/\rRr }{b}.
\end{equation}
\end{corollary}

\begin{corollary} \label{cor:5-B}
Let $m,n,a,\rRr $ be positive integers with $\rRr \ge2$ and $\rRr \mid (mn-1)$.
The number
of positive $m$-divisible non-crossing partitions of 
$\{1,2,\dots,mn,\overline1,\overline2,\dots,\overline{mn}\}$ of type~$B$
which are invariant
under~$\rotB^{(mn-1)/\rRr }$ and have a zero block of size~$2ma$
is given by
\begin{equation} \label{eq:multichains-pos-a-B} 
\binom {{((m+1)n-a-\rRr -1)/\rRr }}{{(n-a)/\rRr }}
\end{equation}
if $a\equiv n$~{\em(mod~$\rRr $)} and $0$ otherwise,
while the number of these partitions without restricting the size
of the zero block is
\begin{equation} \label{eq:multichains-pos-B} 
\binom {\fl{((m+1)n-1)/\rRr }}{\fl{n/\rRr }}.
\end{equation}
\end{corollary}

\begin{proof}
One sums Expressions~\eqref{eq:multichains-si-pos-a-B} 
and~\eqref{eq:multichains-si-pos-B} 
over all~$b$.
The resulting sums can be evaluated by means of the Chu--Vandermonde
summation formula.
\end{proof}

\subsection{Cyclic sieving}
\label{sec:sievB}

With the enumeration results in \Cref{sec:enumB,sec:rotenumB}, we are
now ready to derive our cyclic sieving results in type~$B$; see
\Cref{thm:3-B} below.
\Cref{lem:2-B} prepares for the proof of this theorem.

\begin{lemma} \label{lem:2-B}
Let $m,n,s,b_1,b_2,\dots,b_n$ be integers 
such that $0\le s\le n$ and
\begin{equation}
\label{eq:SB-B}
b_1+2b_2+\dots+nb_n\le n.
\end{equation}
Furthermore, let $\si$ be a positive integer such that $\si\mid (2mn-2)$.
Writing $\rrr =2(mn-1)/\si$ and 
$\om_\rrr =e^{2\pi i\si/(2mn-2)}=e^{2\pi i/\rrr }$, we have
\begin{align}
\bmatrix (m+1)n-1\\n\endbmatrix_{q^2}
\Bigg\vert_{q=\om_\rrr }
&=\begin{cases} 
\binom {\fl{2((m+1)n-1)/\rrr }} {\fl{2n/\rrr }},&\text{if $\rrr$ is
  even},\\
\binom {\fl{((m+1)n-1)/\rrr }} {\fl{n/\rrr }},&\text{if $\rrr$ is
  odd},
\end{cases}
\label{eq:SI-B}
\\
\bmatrix mn-1\\n\endbmatrix_{q^2}
\Bigg\vert_{q=\om_\rrr }
&=\begin{cases}
\binom {mn-1} {{n}},&\text{if\/ $\rrr =1$ or $\rrr =2$},\\
0,&\text{otherwise,}
\end{cases}
\label{eq:SJ-B}
\\
\notag
\bmatrix {n}\\{s}\endbmatrix_{q^2}
\bmatrix {mn-1}\\ {n-s}\endbmatrix_{q^2}
\Bigg\vert_{q=\om_\rrr }
&=\begin{cases}
\binom {\fl{2n/\rrr }} {\fl{2s/\rrr }}
\binom {2(mn-1)/\rrr } {{2(n-s)/\rrr }},
&\text{if\/ $\rrr$ is even and $n\equiv s$ \rm (mod $\rrr /2$)},\\
\binom {\fl{n/\rrr }} {\fl{s/\rrr }}
\binom {(mn-1)/\rrr } {{(n-s)/\rrr }},
&\text{if\/ $\rrr$ is odd and $n\equiv s$ \rm (mod $\rrr $)},\\
0,&\text{otherwise,}
\end{cases}
\notag
\\
\label{eq:SK-B}
\\
\notag
&\hskip-4.5cm
\bmatrix {b_1+b_2+\dots+b_n}\\{b_1,b_2,\dots,b_n}\endbmatrix_{q^2}
\bmatrix {mn-1}\\ {b_1+b_2+\dots+b_n}\endbmatrix_{q^2}\Bigg\vert_{q=\om_\rrr }\\
&\hskip-4cm=\begin{cases}
\binom {{2(b_1+b_2+\dots+b_n)/\rrr }}
{\fl{2b_1/\rrr },\fl{2b_2/\rrr },\dots,\fl{2b_n/\rrr }}
\binom {2(mn-1)/\rrr } {{2(b_1+b_2+\dots+b_n)/\rrr }},
&\kern-1.6pt
\text{if\/ $\rrr$ is even and $b_k\equiv0$ \rm (mod $\rrr/2 $) for all $k$,}\\
\binom {{(b_1+b_2+\dots+b_n)/\rrr }}
{\fl{b_1/\rrr },\fl{b_2/\rrr },\dots,\fl{b_n/\rrr }}
\binom {(mn-1)/\rrr } {{(b_1+b_2+\dots+b_n)/\rrr }},
&\kern-1.6pt
\text{if\/ $\rrr$ is odd and $b_k\equiv0$ \rm (mod $\rrr $) for all $k$,}\\
0,&\kern-1.6pt\text{otherwise.}
\end{cases}
\notag
\\
\label{eq:SL-B}
\end{align}
\end{lemma}

The proof of this lemma is found in \Cref{app:unnec-eval}.

Now we are ready to state, and prove, our cyclic sieving results
in type~$B$.

\begin{theorem} \label{thm:3-B}
Let $m,n,s,b_1,b_2,\dots,b_n$ be integers satisfying
the hypotheses in\break \Cref{lem:2-B}. 
Furthermore, 
let $C$ be the cyclic group of pseudo-rotations of the $2mn$-gon
consisting of the elements
$\{1,2,\dots,mn,\overline1,\overline2,\dots,\overline{mn}\}$
generated by~$\rotB$, viewed as a group of order~$2mn-2$. Then
the triple $(M,P,C)$ exhibits the cyclic sieving phenomenon
for the following choices of sets $M$ and polynomials $P$:

\begin{enumerate}
\item[\em(1)]
$M=\mNCBPlus$, and
$$P(q)=\left[\begin{matrix}
(m+1)n-1\\n\end{matrix}\right]_{q^2};$$
\item[\em(2)]
$M$ consists of the elements of
$\mNCBPlus$ all of whose blocks have size~$m$, and
$$P(q)=\left[\begin{matrix}
mn-1\\n\end{matrix}\right]_{q^2};$$
\item[\em(3)]
$M$ consists of all elements of
$\mNCBPlus$ which have rank $s$
{\em(}or, equivalently, their number of non-zero blocks is $2(n-s)${\em)},
and 
$$P(q)=
\bmatrix {n}\\{s}\endbmatrix_{q^2}
\bmatrix {mn-1}\\ {n-s}\endbmatrix_{q^2};
$$
\item[\em(4)]
$M$ consists of all elements of
$\mNCBPlus$ whose
number of non-zero blocks of size~$mi$ is~$2b_i$, $i=1,2,\dots,n$,
and 
$$P(q)=
\bmatrix {b_1+b_2+\dots+b_n}\\{b_1,b_2,\dots,b_n}\endbmatrix_{q^2}
\bmatrix {mn-1}\\ {b_1+b_2+\dots+b_n}\endbmatrix_{q^2}.
$$
\end{enumerate}
\end{theorem}

\begin{proof}
The polynomials $P(q)$ in the assertion of the theorem are indeed
polynomials with non-negative coefficients since $q$-binomial
and $q$-multinomial coefficients have this property.

Now, given a positive integer $\rrr$ with $\rrr\mid (2mn-2)$, we must show  
$$
P(q)\big\vert_{q=\om_\rrr }=
\text{number of elements of $\mNCBPlus$ which are invariant under
  $\rotB^{(2mn-2)/\rrr}$},
$$
where $\om_\rrr=e^{2\pi i/\rrr}$.

\medskip
We begin with Item~(1).
If $\rrr$ is even, this follows by combining~\eqref{eq:multichains-pos-B}
with $\rRr=\rrr/2$
and the first alternative on the right-hand side of~\eqref{eq:SI-B}.
If $\rrr$ is odd, then we claim that an element of $\mNCBPlus$
that is invariant under the action of
$\rotB^{(2mn-2)/\rrr}$ is automatically also
invariant under $\rotB^{(2mn-2)/2\rrr}=\rotB^{(mn-1)/\rrr}$. Indeed,
let $\la$ be the multiplicative inverse of~$2$ modulo~$\rrr$. Then
there exists a positive integer $\ka$ such that $2\la=1+\rrr\ka$.
Thus, we have
\begin{equation} \label{eq:rotB/2} 
\la\cdot\frac {2mn-2} {\rrr}
=\frac {2\la} {\rrr}(mn-1)
=\frac {1} {\rrr}(mn-1)+\ka(mn-1).
\end{equation}
According to \Cref{lem:N-2B}. the order of~$\rotB$ is~$mn-1$, proving
our claim. In this case, the assertion~(1) follows 
by combining~\eqref{eq:multichains-pos-B} with $\rRr=\rrr$
and the second alternative on the right-hand side of~\eqref{eq:SI-B}.

\medskip
Item~(2) is the special case of Item~(4) where $b_1=n$
(and, hence, $b_2=b_3=\dots=b_n=0$), and therefore does not need
to be treated separately.

\medskip
We now turn to Item~(3). 
Again, let first $\rrr$ be even.
If $\rrr=2$, then the assertion follows immediately from
\Cref{lem:N-2B} and~\eqref{eq:multichains-si-pos-B-1}.
If $\rrr>2$, then,
according to \Cref{lem:allB}, any
positive $m$-divisible non-crossing partition which is
$\rrr $-pseudo-rotationally invariant arises from the
Construction in \Cref{sec:rotB} and subsequent applications of the
pseudo-rotation~$\rotB$. As a consequence,
the number of non-zero blocks must be divisible by~$\rrr $.
Hence, for any such non-crossing partition~$\pi$, we have
$$
\#(\text{non-zero blocks of $\pi$})=\rrr B=2(n-s),
$$
where $B$ is some non-negative integer. In particular, we must
have $n\equiv s$~(mod~$\rrr /2$).
The number of these non-crossing partitions is given by
Formula~\eqref{eq:multichains-si-pos-B} with $\rRr=\rrr/2$ 
and $b$ replaced by $B=2(n-s)/\rrr $.
It is not difficult to see that this agrees with
the expression in the first alternative on the right-hand side of~\eqref{eq:SK-B}.

On the other hand, if $\rrr$ is odd, then we reuse the earlier argument
that any element of $\mNCBPlus$
that is invariant under the action of~$\rotB^{(2mn-2)/\rrr}$ is also
invariant under $\rotB^{(2mn-2)/2\rrr}=\rotB^{(mn-1)/\rrr}$. 
Hence, \Cref{lem:N-2B}, Equation~\eqref{eq:multichains-si-pos-B-1},
\Cref{lem:allB}, and
Equation~\eqref{eq:multichains-si-pos-B} with $\rRr=\rrr$ and $b$ replaced by
$(n-s)/\rrr $ together confirm that the $\rrr$-pseudo-rotationally
invariant elements of $\mNCBPlus$ that are relevant for the
current case are given by the second alternative on the right-hand
side of~\eqref{eq:SK-B}.

In all other cases, there are no positive $m$-divisible non-crossing
partitions which are $\rrr $-pseudo-rotationally invariant.

\medskip
Finally, we address Item~(4). 
Again, we distinguish two cases depending on whether $\rrr $ is even
or odd.

Let first $\rrr$ be even. 
Repeating the arguments from Item~(3), 
we see that all numbers~$b_k$ are divisible by~$\rrr/2$.
If $\rrr=2$,
the number of relevant non-crossing partitions is then given
by \Cref{thm:1-B}. If $\rrr>2$, the number of relevant non-crossing
partitions is given by \Cref{thm:2-B} with the appropriate replacements.
The resulting expression agrees with the expression in the first
alternative on the right-hand side of~\eqref{eq:SL-B}.

If $\rrr$ is odd, then, by repeating the earlier argument again
that any element of $\mNCBPlus$
that is invariant under the action of
$\rotB^{(2mn-2)/\rrr}$ is also
invariant under $\rotB^{(2mn-2)/2\rrr}=\rotB^{(mn-1)/\rrr}$, we infer
that all numbers~$b_k$ are divisible by~$\rrr$. 
Also here, the number of relevant non-crossing
partitions is given by \Cref{thm:2-B} with the appropriate replacements.
The resulting expression agrees with the expression in the second
alternative on the right-hand side of~\eqref{eq:SL-B}.

In all other cases, there are no positive $m$-divisible non-crossing
partitions which are $\rrr $-pseudo-rotationally invariant.

\medskip
This completes the proof of the theorem.
\end{proof}

\begin{corollary} \label{cor:CS-B}
  The conclusion of \Cref{thm:CS} holds for type $B_n$.
\end{corollary}

\begin{proof}
The cyclic sieving phenomenon for
  $\Big(\mNCPlus[B_n][m+1],\ \mCatplus[m+1](B_n;q),\ C\Big)$
follows directly from \Cref{thm:3-B}(1) with~$m$ replaced by~$m+1$.

The image of the
embedding of $\mNCPlus[B_n][m]$
in $\mNCPlus[B_n][m+1]$ as described in
\Cref{prop:PositiveKrewAlt} is given by all tuples
$(w_0,w_1,\dots,w_{m+1})$ with $w_0=\ep$. By
\Cref{rem:id-B}, these tuples correspond to non-crossing set partitions
in $\mNCB[n][m+1]$ all of whose non-zero blocks have size~$m+1$. 
The cyclic sieving phenomenon for
$\Big(\mNCPlus[B_n][m],\ \mCatplus(B_n;q),\ \widetilde C\Big)$
thus follows from \Cref{thm:3-B}(2) with $m$ replaced by~$m+1$.
\end{proof}

\section{Positive non-crossing set partitions in 
type~\texorpdfstring{$D$}{D}}
\label{sec:typeD}

In this section, we study positive $m$-divisible non-crossing
partitions and the positive Kreweras maps in the situation of the
even hyperoctahedral group $D_n$. Recall that $D_n$ can be combinatorially
realised as the group of permutations~$\pi$ of
$\{1,2,\dots,n,-1,-2,\dots,-n\}$ satisfying
$\pi(-i)=-\pi(i)$ for $i=1,2,\dots,n$ which maps an even number
of the numbers $\{1,2,\dots,n\}$ to a negative number.
Again, we shall most often write $\overline i$ instead of~$-i$.
Moreover, the type~$B$ cycle notation explained
in \Cref{sec:typeB} is also used here.

The Coxeter system corresponding to the group $D_n$ is
\begin{equation} \label{eq:WSD}
(W,\reflS) = \big( D_{n},\{ s_i = ((i,i+1)) \mid 1 \leq i \le n-1\}\cup
\{s_{n}=((n-1,\overline n))\} \big)
\end{equation}
together with the Coxeter element $c = s_1 \cdots s_{n}=[1,2,\ldots,n-1][n]$.

\medskip

The structure of this section is again the same as the one of \Cref{sec:typeA}.
Namely, in \Cref{sec:realD}
we first make the positivity condition for type~$D$
$m$-divisible non-crossing partitions under the choice~\eqref{eq:WSD} 
of Coxeter system explicit; see \Cref{lem:wmn-1>0}. We then use
this to describe \emph{positive} $m$-divisible non-crossing partitions
within the combinatorial
model of type~$D$ $m$-divisible non-crossing partitions on an annulus
from~\cite{KrMuAB} (with the correction reported in~\cite{JKimAA}); see \Cref{prop:1D}.

In \Cref{sec:mapD}, we show that, under the translation
described in \cite[Sec.~7]{KrMuAB}, for $m\ge2$ 
the positive Kreweras map acts as a pseudo-rotation on an annulus; see
\Cref{prop:2D}.
This pseudo-rotation is denoted by~$\rotD$.

In \Cref{sec:combD}, we then show how this realisation naturally gives rise to an extension which also makes
sense for $m=1$; see \Cref{def:phiallgD}.
Again it turns out
that, under the translation from~\cite{KrMuAB}, for $m=1$
this extension corresponds exactly to the map~$\Krewplustilde^{(1)}$
from \Cref{def:K-alt}; see \Cref{prop:2-m=1D}.
The order of the (extended) map~$\rotD$ is of course 
given by \Cref{thm:order} for type~$D$.
However, for enumerative purposes,
we need more refined order properties of~$\rotD$ that also take into
account the combinatorial structure of the type~$D$ positive $m$-divisible
non-crossing partitions.
These are the subject of \Cref{lem:N-2D}.
They require {\it combinatorial\/} proofs independent of
\Cref{thm:order}. These proofs are found in \Cref{app:order-D}.

\Cref{sec:enumD} is devoted to the enumeration of type~$D$ positive
$m$-divisible non-crossing partitions. For the sake of simplicity,
we content ourselves with the enumeration of all such non-crossing
partitions, and those all of whose block sizes are equal to~$m$;
see \Cref{cor:2-D}.

\Cref{sec:rotD} is then devoted to the characterisation of
type~$D$ positive $m$-divisible non-crossing partitions that are
invariant under powers of the pseudo-rotation~$\rotD$; see \Cref{lem:allD}.
Here, the invariant non-crossing partitions arise from three different
constructions. 

In \Cref{sec:rotenumD}, we then use this characterisation and, for
each of the aforementioned three constructions, 
provide a formula for the number of corresponding 
type~$D$ positive $m$-divisible non-crossing partitions that are
invariant under powers of the pseudo-rotation~$\rotD$;
see \Cref{thm:enumD-1,thm:enumD-2,thm:enumD-3}.

In the final subsection, \Cref{sec:sievD}, we establish two
cyclic sieving phenomena which prove \Cref{thm:CS} in type~$D$;
see \Cref{thm:3-D}.

\subsection{Combinatorial realisation of the positive non-crossing partitions}
\label{sec:realD}

With the choice \eqref{eq:WSD} of Coxeter system, 
we have the following simple combinatorial description of 
positive $m$-divisible non-crossing partitions.

\begin{proposition} \label{lem:wmn-1>0}
Let $m$ and $n$ be positive integers with  $m\ge2$. The tuple
  $(w_0,w_1,\dots,\break w_m) \in \mNC[D_{n}]$ is positive if and only if
  $0<w_m(n-1)\le n-1$.
\end{proposition}

\begin{proof}
According to \Cref{def:NCpos}, the tuple $(w_0,w_1,\dots,w_m)$ is
positive if and only if $cw_m^{-1}=w_0w_1\cdots w_{m-1}$ has full support in our generators $\{s_1,\dots,s_n\}$ given in~\eqref{eq:WSD}. The
cycles in the disjoint cycle decomposition of~$cw_m^{-1}$ define
a non-crossing partition. Thus, $cw_m^{-1}$ will not have full support if and only if the preimage of~$1$ under
$cw_m^{-1}$ is a positive number or $\overline n$.
In other words, there is either an $l>0$
such that $(cw_m^{-1})(l)=1$, or $(cw_m^{-1})(\overline n)=1$.
By acting on the left on both sides
of this relation by the inverse of the
Coxeter element $c=[1,2,\dots,n-1][n]$, we see that this is equivalent
to either $w_m^{-1}(l)=\overline {n-1}$, or $w_m^{-1}(\overline n)=\overline {n-1}$.
Again an equivalent statement is that we have either $w_m(n-1)=\overline 
l$, for some $l>0$,
or $w_m(n-1)=n$. We are interested in the contrapositive:
$cw_m^{-1}$ has full support if and only if $0<w_m(n-1)\le n-1$.
\end{proof}

Next we recall the bijection from \cite[Sec.~7]{KrMuAB} between $\mNC[D_{n}]$ and
$m$-divisible non-crossing partitions on an annulus, with
the numbers
$\{1,2,\dots,mn-m,\overline1,\overline2,\dots,\break\overline{mn-m}\}$
placed clockwise
around the outer circle, the numbers
$\{mn-m+1,\break mn-m+2,\dots,mn,
\overline{mn-m+1},\overline{mn-m+2},\dots,\overline{mn}\}$
placed {\it counter-clockwise} around the inner circle,
that are invariant under substitution of~$i$ by~$-i$ for all~$i$, 
and which satisfy the following additional
restrictions:

\begin{enumerate} 
\item [(D1)]
Successive elements in a block must be
successive numbers when taken modulo~$m$. 
\item [(D2)]
A zero block (see right below for the definition)
must contain \emph{all\/} elements of the inner
circle and at least two elements of the outer circle.
\item [(D3)]
In the case that there is no block that contains elements
from the outer \emph{and\/} the inner circle, the blocks
on the inner circle are completely determined by the blocks
on the outer circle in the following way: consider a block, $B$ say,
on the outer circle that is ``visible'' from the inner circle.
(Here, ``visible'' means that it is possible to travel from
a point of the inner circle to an element of that block
without crossing any block connection.) Let $b$ be the ``right-most''
element in $B$, which, by definition, is determined by considering
the elements of $B\cup(-B)$ in clockwise order and, among these,
choosing the element of~$B$ that is followed by an element
of~$-B$. There is a unique $a\equiv b$~(mod~$m$) between
$mn-m+1$ and $mn$. Then one block on the inner circle
consists of the $m$ consecutive elements in counter-clockwise
order on the inner circle ending in $a$, and the other
block consists of the negatives of these elements.
\end{enumerate}

(Restriction (D3) was missing
in~\cite{KrMuAB}, but is reported in \cite[Sec.~7]{JKimAA}.)
We denote these latter $m$-divisible non-crossing partitions by
$\mNCD$.

Again, following~\cite{ReivAG}, we agree here to call a block
that is itself invariant under substitution of~$i$ by~$-i$, for
all~$i$, a \defn{zero block}.
Furthermore, a block is called
\defn{bridging} if it contains elements of the outer \emph{and\/} the inner circle.

\medskip
Given an element $(w_0,w_1,\dots,w_m)\in \mNC[D_{n}]$,
the bijection, $\Nam{D_n}m$ say,
from \cite[Sec.~7]{KrMuAB}
works essentially in the same way as in type~$B$:
namely
$(w_0,w_1,\dots,w_m)\in \mNC[D_{n}]$ is mapped to
\begin{multline} \label{eq:Nd}
\Nam{D_n}m(w_0,w_1,\dots,w_m)=
[1,2,\dots,m(n-1)]\,[mn-m+1,\dots,mn-1,mn]\\
\circ(\bar\ta_{m,1}(w_1))^{-1}\,(\bar\ta_{m,2}(w_2))^{-1}\,\cdots\,
(\bar\ta_{m,m}(w_m))^{-1},
\end{multline}
where $\bar\ta_{m,i}$ is defined as in \Cref{sec:typeB}
for type~$B$.
Again, the cycles in the disjoint cycle decomposition correspond to
the blocks in the non-crossing partition in
$\mNCD$.
We refer the reader to \cite[Sec.~7]{KrMuAB} for the details.
For example, let $n=6$, $m=3$, $w_0=((2,\overline 4))$,
$w_1=((2,\overline 6))\,((4,5))$,
$w_2=((1,\overline 5))\,((2,3))$, and
$w_3=((3,6))$. Then $(w_0;w_1,w_2,w_3)$ is mapped to
\begin{align} \notag
\Nam{D_{6}}3(w_0;&w_1,w_2,w_3)\\
\notag
&=[1,2,\dots,15]\,[16,17,18]
((4,\overline{16}))\,((10,13))\,((2,\overline{14}))\,
((5,8))\,((9,18))\\
&=((1,2,\overline{15}))\,((3,4,\overline{17},\overline{18},\overline{10},
\overline{14}))\,((5,9,\overline{16}))\,((6,7,8))\,((11,12,13)).
\label{eq:NCD}
\end{align}

Similarly as in types $A_{n-1}$ and $B_n$, the cycle
structure of the first
component of an $m$-divisible non-crossing partition
determines the block structure of its image under the above
bijection~$\Nam{D_n}m$.

\begin{proposition} \label{prop:block-D}
Let $(w_0,w_1,\dots,w_m)\in\mNC[D_n]$. The non-crossing partition
$\pi=\Nam{D_n}m(w_0,w_1,\dots,w_m)\in\mNCD$ has a zero block consisting
of $2km+2m$ elements if and only if $w_0$ contains 
$[i_1,i_2,\dots,i_k][n]$ for some $i_1,i_2,\dots,i_k$
in its disjoint cycle decomposition.
Furthermore, $\pi$ has $2b_k$ non-zero blocks of
size~$mk$ if and only if $w_0$ has $b_k$ cycles of the form
$((i_1,i_2,\dots,i_k))$
in its disjoint cycle decomposition.\footnote{The reader must recall
that a ``cycle'' of the form $((i_1,i_2,\dots,i_k))$ consists actually
of {\em two} disjoint cycles of length~$k$, while
a cycle of the form $[i_1,i_2,\dots,i_k]$ is indeed a single cycle,
of length~$2k$.}
\end{proposition}

In the earlier example, we have $w_0=((2,\overline4))=((1))((3))((5))((6))((2,\overline4))$. Indeed, the image of
$(w_0,w_1,w_2,w_3)$ in the example has no zero block,
two non-zero blocks of size $3\cdot 2=6$
and eight non-zero blocks of size $3\cdot 1=3$.

\begin{remark} \label{rem:id-D}
A simple consequence of \Cref{prop:block-D} is that the image of
the set of all non-crossing partitions in $\mNC[D_n]$ of the form
$(\ep,w_1,\dots,w_m)$ under the map $\Nam{D_n}m$ is
the set of all non-crossing partitions in $\mNCD$ 
in which {\em all} blocks have size~$m$.
\end{remark}

\begin{theorem} \label{prop:1D}
Let $m$ and $n$ be positive integers with $m\ge2$.
The image under $\Nam{D_n}m$ of the positive $m$-divisible
non-crossing partitions in $\mNC[D_{n}]$ are those $m$-divisible
non-crossing partitions in $\mNCD$ where
the predecessor of\/ $1$ in its block is a negative number
on the outer circle of the annulus.
\end{theorem}

\begin{proof}
First, let $(w_0,w_1,\dots,w_m)\in \mNCPlus[D_{n}]$.
By definition, we know that\break $0<w_m(n-1)\le n-1$. Let us write $l$
for $w_m(n-1)$. We then have
\begin{align*}
\Big(\Nam{D_n}m(&w_0,w_1,\dots,w_m)\Big)(\overline{ml})\\
&=
\big([1,2,\dots,mn-m][mn-m+1,\dots,mn]\\
&\kern2cm
\,(\ta_{m,1}(w_1))^{-1}\,(\ta_{m,2}(w_2))^{-1}\,\cdots\,
(\ta_{m,m}(w_m))^{-1}\big)(\overline{ml})\\
&=
\big([1,2,\dots,mn-m][mn-m+1,\dots,mn]\big)(\overline{m(n-1)})=1.
\end{align*}
Translated to non-crossing partitions, this means that $\overline{ml}$ and $1$
belong to the same block, which implies our claim.

Conversely, let $\pi\in \mNCD$ be an
$m$-divisible non-crossing partition in which
$1$ is preceded in its block by a negative number,
$\overline L$ say, on the outer circle.
Since in between $\overline L$ and $1$ there ``sit'' blocks all of whose
sizes
are divisible by~$m$, we must necessarily have $\overline L=\overline{ml}$,
for some $l$ between~$1$ and~$n-1$.
In other words, if $\pi$
is interpreted as a permutation, $\pi(\overline{ml})=1$. Let
$(w_0,w_1,\dots,w_m)$ be the element of $\mNC[D_{n}]$ such that
$$
\Nam{D_n}m(w_0,w_1,\dots,w_m)=\pi.
$$
We must have
$$
\Big(\Nam{D_n}m(w_0,w_1,\dots,w_m)\Big)(\overline{ml})=1.
$$
The definition \eqref{eq:Nd} of $\Nam{D_n}m$
and the fact that $(\ta_{m,i}(w_i))^{-1}$ leaves multiples of~$m$
fixed as long as $1\le i\le m-1$ together imply that
$(\ta_{m,m}(w_m))^{-1}(\overline{ml})=\overline{m(n-1)}$.
Phrased differently, we have
$(\ta_{m,m}(w_m))^{-1}(ml)=m(n-1)$, or, equivalently,
$w_m(n-1)=l>0$. By definition, this means that
$(w_0,w_1,\dots,w_m)$ is positive.
\end{proof}

\subsection{Combinatorial realisation of the positive Kreweras maps}
\label{sec:mapD}

The next step consists in translating the positive Kreweras map $\Krewplustilde$
into a ``rotation action'' on the elements of
$\mNCPlus[D_{n}]$. In order to do so, we need to explicitly describe
the decomposition~\eqref{eq:LR} in type~$D$.

\begin{lemma} \label{lem:LRD}
Let~$w$ be an element of $\NC[D_n]$, and consider the cycle $z$
in the {\em(}type~$D${\em)} disjoint cycle decomposition of~$w$ that contains~$n-1$.
\begin{enumerate}
\item[\em(1)] If $z=((i_1,i_2,\dots,i_k,n-1))$ with
$0<i_1<i_2<\dots<i_k<n-1$, then
$w^R=\one$.
\item[\em(2)] If $z=((i_1,\dots,i_s,i_{s+1},\dots,i_k,n-1))$ 
with $i_1,i_2,\dots,i_s<0$,
$i_{s+1},\dots,i_{k-1},i_k>0$,
and $\vert i_1\vert<\vert i_2\vert<\dots<\vert i_k\vert<n-1$
{\em(}as before,
the sequence $i_{s+1},\dots,i_k$ has to be interpreted as
the empty sequence if $s=k${\em)},
then $w^R=((i_s,n-1))$.
\item[\em(3)] If $z=((i_1,\dots,i_s,n-1,j_1,\dots,j_t,\pm n))$
with $0<i_1<i_2<\dots<i_s<n-1$ and
$0>j_1>j_2>\dots>j_t>-i_1$,
then $w^R=((n-1,\pm n))$.
\item[\em(4)] If $z=[i_1,i_2,\dots,i_k,n-1][n]$
{\em(}and consequently $0<i_1<i_2<\dots<i_k<n-1${\em)}, then
$w^R=((n-1,n))((n-1,\overline n))=[n-1][n]$.
\end{enumerate}
Here, in Item~{\em(3)}, $\pm n$ means either $n$ or $\overline n$,
and the two alternatives must be read consistently.
This covers all possible cases.
\end{lemma}

\begin{proof}
The last $n$ reflections in $\reflR$ given in \eqref{eq:refls}
are
\begin{equation} \label{eq:reflsD}
((i,n-1)),\ i=\overline1,\overline2,\dots,\overline {n-2},\text{ and }
((n-1,\overline n)),\ ((n-1,n)),
\end{equation}
in this order.
Thus, by definition, we have $w^L=w\circ (w^R)^{-1}$, where the factor
$w^R$ consists of a product of these~$n$ reflections,
and $\lenR(w^L)+\lenR(w^R)=\lenR(w)$.
Furthermore, using \Cref{lem:wmn-1>0} with
$m=1$, we see that $w^L$ must have the property that
$0< w^L(n-1)\le n-1.$
The arguments below are heavily based on the uniqueness of the
decomposition $w=w^l\circ w^R$ in the sense of
\Cref{cor:unique}.

In Case~(1), we have $w(n-1)=i_1>0$. Hence, we may choose
$w^L=w$ and $w^R=\one$, and uniqueness of decomposition
guarantees that this is the correct choice.

In Case~(2), the choice $w^R=((i_s,n-1))$ implies that
$w^L$ would contain the cycles
$((i_1,\dots,i_s))((i_{s+1},\dots,i_k,n-1))$.
Hence, we would
have $w^L(n-1)=i_{s+1}>0$ if~$s<k$, and $w^L(n-1)=n-1$ if $s=k$.
By uniqueness of decomposition,
this must be the correct choice.

Case~(3) is similar. The choice $w^R=((n-1,\pm n))$
implies that
$w^L$ would contain the cycles
$((i_1,i_2,\dots,i_s,n-1))((j_1,j_2,\dots,j_t,\pm n))$.
Hence, we would
have $w^L(n-1)=i_1>0$, which,
by uniqueness of decomposition,
must be the correct choice.

Finally we address Case~(4). If we choose $w^R=[n-1][n]$,
then $w^L$ would contain the cycle
$((i_1,i_2,\dots,i_k,n-1))$.
Hence, we would
have $w^L(n-1)=i_1>0$. Again,
by uniqueness of decomposition,
this must be the correct choice.
\end{proof}

Next, we translate the positive Kreweras map $\Krewplustilde$ from
\Cref{def:positivekreweras} for type $D_{n}$ into combinatorial
language. 

\begin{theorem} \label{prop:2D}
Let $m$ and $n$ be positive integers with $m\ge2$.
Under the bijection $\Nam{D_n}m$, the map $\Krewplustilde^{(m)}$ translates to
the following map~$\rotD$ on $\mNCDPlus$:
if the block of $\overline{m(n-1)}$ contains another negative element
which is different from $\overline{mn-1}$,
then $\rotD$ rotates all blocks of~$\pi$ by one unit
{\em(}in clockwise direction on the outer circle of the annulus,
and in counter-clockwise direction on the inner circle{\em)}.

On the other hand,
if the block of $\overline{m(n-1)}$ of some element~$\pi\in\mNCDPlus$
contains no other negative
element which is different from $\overline{mn-1}$,
then, first of all, $1$ is also contained in this block.
In other words, $\overline{m(n-1)}$ and\/ $1$ are successive elements
in a block {\em(}and consequently also $m(n-1)$ and\/ $\overline 1${\em)}.
For the description of the operation, we distinguish
between three cases:

\begin{enumerate}
\item[\em(1)] If there is a block not containing
$1,\overline1,m(n-1),\overline{m(n-1)}$ in
which there are positive {\em and} negative elements
from the outer circle, then
let $b$ be minimal such that $\overline d$ and $b$ are successive
elements in such a block,
with $\overline d$ negative and $b$ positive.
Furthermore, let $b$ be followed by~$e$ in this block, and let
$a$ and $\overline{m(n-1)}$ be successive elements
in a block. {\em(}The element~$a$ must be positive and on the
outer circle by the assumption
defining the subcase in which we are\footnote{Indeed, it is that block
not containing $1,\overline1,m(n-1),\overline{m(n-1)}$ in which there are
positive {\it and\/} negative elements on the outer circle which
constitutes a ``barrier" making it possible for~$a$ to lie on the
inner circle.}, while $e$ can be positive
or negative, and on the outer or inner circle. See
the left half of \Cref{fig:12} for an
illustration of the various definitions.{\em)}
Then the image of~$\pi$ under~$\rotD$ is the partition
in which $\overline{d+1}$, $1$, and $e+1$ are three
successive elements in a
block, and $a+1$, $b+1$, and $2$ are successive elements in
another block. All other succession relations in~$\pi$
are rotated by one unit in clockwise direction on the outer circle
and in counter-clockwise direction on the inner circle.
See \Cref{fig:12} for a schematic illustration of this operation.
\item[\em(2)]If there is no block not containing
$1,\overline1,m(n-1),\overline{m(n-1)}$ in
which there are positive {\em and} negative elements
from the outer circle, and
if the predecessor of~$\mp mn$, denoted by~$\overline d$
(with $d$ positive),
is on the outer circle, then let
$\overline b$ be the successor of~$\mp mn$ and~$a$ be the predecessor of
$\overline{m(n-1)}$.
{\em(}The element~$a$ must be positive and on the
outer circle or equal to $\mp(mn-1)$ by the assumption
defining the subcase in which we are. Furthermore,
the element $\overline b$ is either a negative element on the 
outer circle or equal to $\pm(mn-m+1)$.  
See the left half of \Cref{fig:13} for an
illustration of the described situation.{\em)}
Then the image of~$\pi$ under~$\rotD$ is the partition
in which $\overline{d+1}$, $1$, and $\overline{b+1}$ are three
successive elements in a
block, and $a+1$, $\pm(mn-m+1)$, and $2$ are successive elements in
another block. All other succession relations in~$\pi$
are rotated by one unit in clockwise direction on the outer circle
and in counter-clockwise direction on the inner circle.
See \Cref{fig:13} for a schematic illustration of this operation.
\item[\em(3)] If there is no block containing elements from the outer
\emph{and} the inner circle, then let $a$,
$\overline{m(n-1)}$, and~$1$ be successive elements in a block.
{\em(}Again, the element~$a$ must be positive.{\em)}
By Condition~{\em(D3)} in the definition of $\mNCD$, the blocks on the inner circle
are $\{\overline{mn},mn-m+1,\dots,mn-1\}$ and its negative.
The image of~$\pi$ under~$\rotD$ is the partition
in which $\overline{a+1}$, $1$, and $\overline2$ are successive
elements, other succession relations on the outer circle 
are rotated by one unit in clockwise direction,
while on the inner circle the blocks in the image partition are
$\{mn-m+2,\dots,mn-1,mn,\overline{mn-m+1}\}$
and its negative. {\em(}These are the original blocks rotated by
{\em two} units in counter-clockwise direction{\em)}
See \Cref{fig:14} for a schematic illustration of this
operation.
\end{enumerate}
\end{theorem}

\begin{figure}
\begin{center}
  \begin{tikzpicture}[scale=1]

    \polygon{(-4,0)}{obj}{44}{2.5}
       {1,,,,a,,,,,,b,,,,e,,,d,,,,\hspace*{25pt}m(n-1),\overline{1},,,,\overline{a},,,,,,\overline{b},,,,\overline{e},,,\overline{d},,,,\overline{m(n-1)}\hspace*{25pt}}

    \polygoninner{(-4,0)}{objleftin}{8}{0.6}
      {,,,,,,,}

    \draw[line width=2.5pt,black] (obj1) to[bend left=50] (obj44);
    \draw[line width=2.5pt,black] (obj23) to[bend left=50] (obj22);

    \draw[dash pattern=on 2pt off 2pt on 2pt off 2pt on 2pt off 2pt on
2pt off 16pt on 2pt off 2pt on 2pt off 2pt] (obj1) to[bend right=50]
(obj5);
    \draw[dash pattern=on 2pt off 2pt on 2pt off 2pt on 2pt off 2pt on
2pt off 16pt on 2pt off 2pt on 2pt off 2pt] (obj23) to[bend right=50]
(obj27);

    \draw (obj44) to[bend right=50] (obj5);
    \draw (obj22) to[bend right=50] (obj27);

    \draw[dotted] (obj6) to[bend right=50] (obj10);
    \node at ($($0.92*($(obj8)+(4.0,0)$)$)-(4.0,0)$) {\tiny$X$};

    \draw[dotted] (obj28) to[bend right=50] (obj32);
    \node at ($($0.92*($(obj30)+(4.0,0)$)$)-(4.0,0)$) {\tiny$\overline{X}$};

    \draw (obj40) to[bend right=20] (obj11) to[bend right=50] (obj15);
    \draw[dash pattern=on 2pt off 2pt on 2pt off 2pt on 2pt off 2pt on
2pt off 2pt on 2pt off 2pt on 2pt off 95pt on 2pt off 2pt on 2pt off
2pt on 2pt] (obj15) to[bend left=10] (obj40);
    \draw (obj18) to[bend right=20] (obj33) to[bend right=50] (obj37);
    \draw[dash pattern=on 2pt off 2pt on 2pt off 2pt on 2pt off 2pt on
2pt off 2pt on 2pt off 2pt on 2pt off 95pt on 2pt off 2pt on 2pt off
2pt on 2pt] (obj37) to[bend left=10] (obj18);

    \node[inner sep=0pt] at (0,0) {$\mapsto$};

    \polygon{(4,0)}{obj}{44}{2.5}
       {1,2,,,,a+1,,,,,,\,b+1,,,,e+1,,,d+1,,,\hspace*{20pt}m(n-1),\overline{1},\overline{2},,,,\overline{a+1},,,,,,\overline{b+1}\,,,,,\overline{e+1},,,\overline{d+1},,,\overline{m(n-1)}\hspace*{20pt}}

    \polygoninner{(4,0)}{objleftin}{8}{0.6}
      {,,,,,,,}

    \draw[line width=2.5pt,black] (obj1) to[bend left=50] (obj41);
    \draw[line width=2.5pt,black] (obj23) to[bend left=50] (obj19);

    \draw[dash pattern=on 2pt off 2pt on 2pt off 2pt on 2pt off 2pt on
2pt off 16pt on 2pt off 2pt on 2pt off 2pt] (obj2) to[bend right=50]
(obj6);
    \draw[dash pattern=on 2pt off 2pt on 2pt off 2pt on 2pt off 2pt on
2pt off 16pt on 2pt off 2pt on 2pt off 2pt] (obj24) to[bend right=50]
(obj28);

    \draw (obj6) to[bend right=50] (obj12);
    \draw (obj28) to[bend right=50] (obj34);

    \draw (obj2) to[bend right=50] (obj12);
    \draw (obj24) to[bend right=50] (obj34);

    \draw[dotted] (obj7) to[bend right=50] (obj11);
    \node at ($($0.92*($(obj9)-(4.0,0)$)$)+(4.0,0)$) {\tiny$X$};
    \draw[dotted] (obj29) to[bend right=50] (obj33);
    \node at ($($0.92*($(obj31)-(4.0,0)$)$)+(4.0,0)$) {\tiny$X$};

    \draw (obj1) to[bend right=20] (obj16);
    \draw[dash pattern=on 2pt off 2pt on 2pt off 2pt on 2pt off 2pt on
2pt off 2pt on 2pt off 2pt on 2pt off 95pt on 2pt off 2pt on 2pt off
2pt on 2pt] (obj16) to[bend left=10] (obj41);
    \draw (obj23) to[bend right=20] (obj38);
    \draw[dash pattern=on 2pt off 2pt on 2pt off 2pt on 2pt off 2pt on
2pt off 2pt on 2pt off 2pt on 2pt off 95pt on 2pt off 2pt on 2pt off
2pt on 2pt] (obj38) to[bend left=10] (obj19);

    \end{tikzpicture}
\end{center}
\caption{The action of the pseudo-rotation $\rotD$,
Case (1)}
\label{fig:12}
\end{figure}

\begin{figure}
  \centering
  \begin{tikzpicture}[scale=1]
    \polygon{(-4,0)}{obj}{24}{2.5}
      {1,2,3,4,5,6,7,8,9,10,11,12,\overline{1},\overline{2},\overline{3},\overline{4},\overline{5},\overline{6},\overline{7},\overline{8},\overline{9},\overline{10},\overline{11},\overline{12}}

    \polygoninner{(-4,0)}{objleftin}{6}{0.6}
      {\overline{14},\overline{13},15,14,13,\overline{15}}

      \draw[line width=2.5pt,black] (obj1) to[bend left=50] (obj24);
      \draw[line width=2.5pt,black] (obj13) to[bend left=50] (obj12);

     \draw[fill=black,fill opacity=0.1] (obj24) to[bend right=30]
(obj1) to[bend right=30] (obj2) to[bend left=30] (obj24);
     \draw[fill=black,fill opacity=0.1] (obj3) to[bend right=30]
(obj4) to[bend right=30] (obj5) to[bend left=30] (obj3);
     \draw[fill=black,fill opacity=0.1] (obj6) to[bend right=30]
(obj10) to[bend right=30] (obj11) to[bend right=30] (obj18) to[bend
right=30] (obj22) to[bend right=30] (obj23) to[bend right=30] (obj6);
     \draw[fill=black,fill opacity=0.1] (obj7) to[bend right=30]
(obj8) to[bend right=30] (obj9) to[bend left=30] (obj7);
     \draw[fill=black,fill opacity=0.1] (obj12) to[bend right=30]
(obj13) to[bend right=30] (obj14) to[bend left=30] (obj12);
     \draw[fill=black,fill opacity=0.1] (obj15) to[bend right=30]
(obj16) to[bend right=30] (obj17) to[bend left=30] (obj15);
     \draw[fill=black,fill opacity=0.1] (obj19) to[bend right=30]
(obj20) to[bend right=30] (obj21) to[bend left=30] (obj19);

    \polygon{(4,0)}{obj}{24}{2.5}
      {1,2,3,4,5,6,7,8,9,10,11,12,\overline{1},\overline{2},\overline{3},\overline{4},\overline{5},\overline{6},\overline{7},\overline{8},\overline{9},\overline{10},\overline{11},\overline{12}}

    \node[inner sep=0pt] at (0,0) {$\mapsto$};

    \polygoninner{(4,0)}{objleftin}{6}{0.6}
      {\overline{14},\overline{13},15,14,13,\overline{15}}

      \draw[line width=2.5pt,black] (obj1) to[bend left=50] (obj24);
      \draw[line width=2.5pt,black] (obj13) to[bend left=50] (obj12);

     \draw[fill=black,fill opacity=0.1] (obj2) to[bend right=30]
(obj3) to[bend right=30] (obj7) to[bend left=30] (obj2);
     \draw[fill=black,fill opacity=0.1] (obj4) to[bend right=30]
(obj5) to[bend right=30] (obj6) to[bend left=30] (obj4);
     \draw[fill=black,fill opacity=0.1] (obj1) to[bend left=10]
(obj11) to[bend right=30] (obj12) to[bend right=30] (obj13) to[bend
left=10] (obj23) to[bend right=30] (obj24) to[bend right=30] (obj1);
     \draw[fill=black,fill opacity=0.1] (obj8) to[bend right=30]
(obj9) to[bend right=30] (obj10) to[bend left=30] (obj8);
     \draw[fill=black,fill opacity=0.1] (obj14) to[bend right=30]
(obj15) to[bend right=30] (obj19) to[bend left=30] (obj14);
     \draw[fill=black,fill opacity=0.1] (obj16) to[bend right=30]
(obj17) to[bend right=30] (obj18) to[bend left=30] (obj16);
     \draw[fill=black,fill opacity=0.1] (obj20) to[bend right=30]
(obj21) to[bend right=30] (obj22) to[bend left=30] (obj20);
    \end{tikzpicture}
  \caption{The action of the pseudo-rotation $\rotD$, Case (1)}
\label{fig:23}
\end{figure}

\begin{figure}
  \centering
\begin{tikzpicture}[scale=1]
    \polygon{(-4,0)}{obj}{24}{2.5}
      {1,2,3,4,5,6,7,8,9,10,11,12,\overline{1},\overline{2},\overline{3},\overline{4},\overline{5},\overline{6},\overline{7},\overline{8},\overline{9},\overline{10},\overline{11},\overline{12}}

    \polygoninner{(-4,0)}{objin}{12}{0.6}
      {\overline{14},,\overline{13},,15,,14,,13,,\overline{15},}

      \draw[line width=2.5pt,black] (obj1) to[bend left=50] (obj24);
      \draw[line width=2.5pt,black] (obj13) to[bend left=50] (obj12);

     \draw[fill=black,fill opacity=0.1] (obj24) to[bend right=30]
(obj1) to[bend right=30] (obj2) to[bend left=30] (obj24);
     \draw[fill=black,fill opacity=0.1] (obj3) to[bend right=30]
(obj4) to[bend right=30] (obj5) to[bend left=30] (obj3);
     \draw[fill=black,fill opacity=0.1] (obj6) to[bend right=30]
(obj10) to[bend right=25] (obj23) to[bend left=10] (obj6);
     \draw[fill=black,fill opacity=0.1] (obj11) to[bend left=10] (obj18) to[bend
right=30] (obj22) to[bend right=25] (obj11);
     \draw[fill=black,fill opacity=0.1] (obj7) to[bend right=30]
(obj8) to[bend right=30] (obj9) to[bend left=30] (obj7);
     \draw[fill=black,fill opacity=0.1] (obj12) to[bend right=30]
(obj13) to[bend right=30] (obj14) to[bend left=30] (obj12);
     \draw[fill=black,fill opacity=0.1] (obj15) to[bend right=30]
(obj16) to[bend right=30] (obj17) to[bend left=30] (obj15);
     \draw[fill=black,fill opacity=0.1] (obj19) to[bend right=30]
(obj20) to[bend right=30] (obj21) to[bend left=30] (obj19);

     \draw[fill=black,fill opacity=0.1] (objin9) to[bend left=40]
(objin11) to[bend left=40] (objin1) to[bend right=100,looseness=2] (objin9);

     \draw[fill=black,fill opacity=0.1] (objin3) to[bend left=40]
(objin5) to[bend left=40] (objin7) to[bend right=100,looseness=2] (objin3);

    \polygon{(4,0)}{obj}{24}{2.5}
      {1,2,3,4,5,6,7,8,9,10,11,12,\overline{1},\overline{2},\overline{3},\overline{4},\overline{5},\overline{6},\overline{7},\overline{8},\overline{9},\overline{10},\overline{11},\overline{12}}

    \polygoninner{(4,0)}{objin}{12}{0.6}
      {\overline{15},,\overline{14},,\overline{13},,15,,14,,13,}

      \draw[line width=2.5pt,black] (obj1) to[bend left=50] (obj24);
      \draw[line width=2.5pt,black] (obj13) to[bend left=50] (obj12);

     \draw[fill=black,fill opacity=0.1] (obj2) to[bend right=30]
(obj3) to[bend right=30] (obj7) to[bend left=30] (obj2);
     \draw[fill=black,fill opacity=0.1] (obj4) to[bend right=30]
(obj5) to[bend right=30] (obj6) to[bend left=30] (obj4);
     \draw[fill=black,fill opacity=0.1] (obj1) to[bend left=20]
(obj11) to[bend right=20] (obj24) to[bend right=30] (obj1);
     \draw[fill=black,fill opacity=0.1] (obj23) to[bend right=20]
(obj12) to[bend right=30] (obj13) to[bend left=20] (obj23);
     \draw[fill=black,fill opacity=0.1] (obj8) to[bend right=30]
(obj9) to[bend right=30] (obj10) to[bend left=30] (obj8);
     \draw[fill=black,fill opacity=0.1] (obj14) to[bend right=30]
(obj15) to[bend right=30] (obj19) to[bend left=30] (obj14);
     \draw[fill=black,fill opacity=0.1] (obj16) to[bend right=30]
(obj17) to[bend right=30] (obj18) to[bend left=30] (obj16);
     \draw[fill=black,fill opacity=0.1] (obj20) to[bend right=30]
(obj21) to[bend right=30] (obj22) to[bend left=30] (obj20);
     \draw[fill=black,fill opacity=0.1] (objin9) to[bend left=40]
(objin11) to[bend left=40] (objin1) to[bend right=100,looseness=2] (objin9);

     \draw[fill=black,fill opacity=0.1] (objin3) to[bend left=40]
(objin5) to[bend left=40] (objin7) to[bend right=100,looseness=2] (objin3);

    \end{tikzpicture}
  \caption{The action of the pseudo-rotation $\rotD$, Case (1)}
\label{fig:23a}
\end{figure}

\begin{figure}
\begin{center}
  \begin{tikzpicture}[scale=1]

    \polygon{(-4,0)}{obj}{44}{3}
       {1,,,,a,,,,,,b,,,,d,,,,,,,\hspace*{25pt}m(n-1),\overline{1},,,,\overline{a},,,,,,\overline{b},,,,\overline{d},,,,,,,\overline{m(n-1)}\hspace*{25pt}}

    \polygoninner{(-4,0)}{objin}{16}{1.2}
      {\pm mn\hspace*{5pt},,,,,,,,\hspace*{5pt}\mp mn,,,,,,}

    \draw[line width=2.5pt,black] (obj1) to[bend left=50] (obj44);
    \draw[line width=2.5pt,black] (obj23) to[bend left=50] (obj22);

    \draw[dash pattern=on 2pt off 2pt on 2pt off 2pt on 2pt off 2pt on
2pt off 16pt on 2pt off 2pt on 2pt off 2pt] (obj1) to[bend right=20]
(obj5);
    \draw[dash pattern=on 2pt off 2pt on 2pt off 2pt on 2pt off 2pt on
2pt off 16pt on 2pt off 2pt on 2pt off 2pt] (obj23) to[bend right=20]
(obj27);

    \draw (obj44) to[bend right=50] (obj5);
    \draw (obj22) to[bend right=50] (obj27);

    \draw[dotted] (obj16) to[bend right=20] (obj21);
    \node at ($($0.92*($(obj19)+(4.0,0)$)$)-(4.0,0)$) {\tiny$X$};
    \draw[dotted] (obj38) to[bend right=20] (obj43);
    \node at ($($0.92*($(obj41)+(4.0,0)$)$)-(4.0,0)$) {\tiny$\overline{X}$};

    \draw (obj11) to[bend right=30] (objin1) to[bend left=20] (obj15);
    \draw (obj33) to[bend right=30] (objin9) to[bend left=20] (obj37);

    \draw[dash pattern=on 2pt off 2pt on 2pt off 2pt on 2pt off 2pt on
2pt off 25pt on 2pt off 2pt on 2pt off 2pt] (obj11) to[bend right=20]
(obj15);
    \draw[dash pattern=on 2pt off 2pt on 2pt off 2pt on 2pt off 2pt on
2pt off 25pt on 2pt off 2pt on 2pt off 2pt] (obj33) to[bend right=20]
(obj37);

    \node[inner sep=0pt] at (0,0) {$\mapsto$};

    \polygon{(4,0)}{obj}{44}{3}
       {1,2,,,,a+1,,,,,,b+1,,,,d+1,,,,,,\hspace*{25pt}m(n-1),\overline{1},\overline{2},,,,\overline{a+1},,,,,,\overline{b+1},,,,\overline{d+1},,,,,,\overline{m(n-1)}\hspace*{25pt}}

    \polygoninner{(4,0)}{objin}{16}{1.2}
      {,,\pm (m(n-1)+1)\hspace*{35pt},\mp
mn\hspace*{5pt},,,,,,,\hspace*{35pt}\mp (m(n-1)+1),\hspace*{5pt}\pm
mn,,,,}


    \draw[dash pattern=on 2pt off 2pt on 2pt off 2pt on 2pt off 2pt on
2pt off 16pt on 2pt off 2pt on 2pt off 2pt] (obj2) to[bend right=50]
(obj6);
    \draw[dash pattern=on 2pt off 2pt on 2pt off 2pt on 2pt off 2pt on
2pt off 16pt on 2pt off 2pt on 2pt off 2pt] (obj24) to[bend right=50]
(obj28);

    \draw (obj2) to[bend right=30] (objin3) to[bend left=20] (obj6);
    \draw (obj24) to[bend right=30] (objin11) to[bend left=20] (obj28);

    \draw[dotted] (obj17) to[bend right=20] (obj22);
    \node at ($($0.92*($(obj20)-(4.0,0)$)$)+(4.0,0)$) {\tiny$X$};
    \draw[dotted] (obj39) to[bend right=20] (obj44);
    \node at ($($0.92*($(obj42)-(4.0,0)$)$)+(4.0,0)$) {\tiny$\overline{X}$};

    \draw (obj12) to[bend right=20] (obj23);
    \draw (obj34) to[bend right=20] (obj1);
    \draw[line width=2.5pt,black] (obj23) to[bend left=20] (obj16);
    \draw[line width=2.5pt,black] (obj1) to[bend left=20] (obj38);

    \draw[dash pattern=on 2pt off 2pt on 2pt off 2pt on 2pt off 2pt on
2pt off 25pt on 2pt off 2pt on 2pt off 2pt] (obj12) to[bend right=20]
(obj16);
    \draw[dash pattern=on 2pt off 2pt on 2pt off 2pt on 2pt off 2pt on
2pt off 25pt on 2pt off 2pt on 2pt off 2pt] (obj34) to[bend right=20]
(obj38);

    \end{tikzpicture}
\end{center}
\caption{The action of the pseudo-rotation $\rotD$, Case (2)}
\label{fig:13}
\end{figure}

\begin{figure}
  \centering
  \begin{tikzpicture}[scale=1]
    \polygonnew{(-4,0)}{objleftout}{40}{3}
      {1,2,3,4,5,6,7,8,9,10,11,12,13,14,15,16,17,18,19,20,\overline{1},\overline{2},\overline{3},\overline{4},\overline{5},\overline{6},\overline{7},\overline{8},\overline{9},\overline{10},\overline{11},\overline{12},\overline{13},\overline{14},\overline{15},\overline{16},\overline{17},\overline{18},\overline{19},\overline{20}}{1.1}
    \polygonnew{(-4,0)}{objleftin}{8}{0.75}
      {\overline{22},\overline{21},24,23,22,21,\overline{24},\overline{23}}{0.65}

\draw[fill=black,fill opacity=0.1] (objleftout40) to[bend right=30] (objleftout1) to[bend right=5] (objleftin1) to[bend right=50] (objleftin8) to[bend right=5] (objleftout40);

     \draw[fill=black,fill opacity=0.1] (objleftout2) to[bend right=30] (objleftout3) to[bend left=5] (objleftout16) to[bend right=10] (objleftin2) to[bend right=10] (objleftout2);

     \draw[fill=black,fill opacity=0.1] (objleftout4) to[bend right=30] (objleftout9) to[bend right=10] (objleftout10) to[bend right=30] (objleftout11) to[bend left=30] (objleftout4);

     \draw[fill=black,fill opacity=0.1] (objleftout5) to[bend right=30] (objleftout6) to[bend right=30] (objleftout7) to[bend right=30] (objleftout8) to[bend left=30] (objleftout5);

     \draw[fill=black,fill opacity=0.1] (objleftout12) to[bend right=30] (objleftout13) to[bend right=30] (objleftout14) to[bend right=30] (objleftout15) to[bend left=30] (objleftout12);

     \draw[fill=black,fill opacity=0.1] (objleftout17) to[bend right=30] (objleftout18) to[bend right=30] (objleftout19) to[bend right=5] (objleftin3) to[bend right=5] (objleftout17);

     \draw[fill=black,fill opacity=0.1] (objleftout20) to[bend right=30] (objleftout21) to[bend right=5] (objleftin5) to[bend right=50] (objleftin4) to[bend right=5] (objleftout20);

     \draw[fill=black,fill opacity=0.1] (objleftout22) to[bend right=30] (objleftout23) to[bend left=5] (objleftout36) to[bend right=10] (objleftin6) to[bend right=10] (objleftout22);

     \draw[fill=black,fill opacity=0.1] (objleftout24) to[bend right=30] (objleftout29) to[bend right=10] (objleftout30) to[bend right=30] (objleftout31) to[bend left=30] (objleftout24);

     \draw[fill=black,fill opacity=0.1] (objleftout25) to[bend right=30] (objleftout26) to[bend right=30] (objleftout27) to[bend right=30] (objleftout28) to[bend left=30] (objleftout25);

     \draw[fill=black,fill opacity=0.1] (objleftout32) to[bend right=30] (objleftout33) to[bend right=30] (objleftout34) to[bend right=30] (objleftout35) to[bend left=30] (objleftout32);

     \draw[fill=black,fill opacity=0.1] (objleftout37) to[bend right=30] (objleftout38) to[bend right=30] (objleftout39) to[bend right=5] (objleftin7) to[bend right=5] (objleftout37);

     \node[inner sep=0pt] at (0,0) {$\mapsto$};

    \polygonnew{(4,0)}{objrightout}{40}{3}
      {2,3,4,5,6,7,8,9,10,11,12,13,14,15,16,17,18,19,20,\overline{1},\overline{2},\overline{3},\overline{4},\overline{5},\overline{6},\overline{7},\overline{8},\overline{9},\overline{10},\overline{11},\overline{12},\overline{13},\overline{14},\overline{15},\overline{16},\overline{17},\overline{18},\overline{19},\overline{20},1}{1.1}
    \polygonnew{(4,0)}{objrightin}{8}{0.75}
      {\overline{23},\overline{22},\overline{21},24,23,22,21,\overline{24}}{0.65}

     \draw[fill=black,fill opacity=0.1] (objrightin7) to[bend right=10] (objrightout1) to[bend right=5] (objrightin1) to[bend right=50] (objrightin8) to[bend right=50] (objrightin7);

     \draw[fill=black,fill opacity=0.1] (objrightout2) to[bend right=30] (objrightout3) to[bend left=5] (objrightout16) to[bend right=10] (objrightin2) to[bend right=10] (objrightout2);

     \draw[fill=black,fill opacity=0.1] (objrightout4) to[bend right=30] (objrightout9) to[bend right=10] (objrightout10) to[bend right=30] (objrightout11) to[bend left=30] (objrightout4);

     \draw[fill=black,fill opacity=0.1] (objrightout5) to[bend right=30] (objrightout6) to[bend right=30] (objrightout7) to[bend right=30] (objrightout8) to[bend left=30] (objrightout5);

     \draw[fill=black,fill opacity=0.1] (objrightout12) to[bend right=30] (objrightout13) to[bend right=30] (objrightout14) to[bend right=30] (objrightout15) to[bend left=30] (objrightout12);

     \draw[fill=black,fill opacity=0.1] (objrightout17) to[bend right=30] (objrightout18) to[bend right=30] (objrightout19) to[bend right=30] (objrightout20) to[bend left=30] (objrightout17);

     \draw[fill=black,fill opacity=0.1] (objrightin3) to[bend right=10] (objrightout21) to[bend right=5] (objrightin5) to[bend right=50] (objrightin4) to[bend right=50] (objrightin3);

     \draw[fill=black,fill opacity=0.1] (objrightout22) to[bend right=30] (objrightout23) to[bend left=5] (objrightout36) to[bend right=10] (objrightin6) to[bend right=10] (objrightout22);

     \draw[fill=black,fill opacity=0.1] (objrightout24) to[bend right=30] (objrightout29) to[bend right=10] (objrightout30) to[bend right=30] (objrightout31) to[bend left=30] (objrightout24);

     \draw[fill=black,fill opacity=0.1] (objrightout25) to[bend right=30] (objrightout26) to[bend right=30] (objrightout27) to[bend right=30] (objrightout28) to[bend left=30] (objrightout25);

     \draw[fill=black,fill opacity=0.1] (objrightout32) to[bend right=30] (objrightout33) to[bend right=30] (objrightout34) to[bend right=30] (objrightout35) to[bend left=30] (objrightout32);

     \draw[fill=black,fill opacity=0.1] (objrightout37) to[bend right=30] (objrightout38) to[bend right=30] (objrightout39) to[bend right=30] (objrightout40) to[bend left=30] (objrightout37);
    \end{tikzpicture}
  \caption{The action of the pseudo-rotation $\rotD$, Case (2)}
\label{fig:19}
\end{figure}
\begin{figure}
  \centering
  \begin{tikzpicture}[scale=1]
    \polygonnew{(-4,0)}{objleftout}{30}{3}
      {\overline{15},1,2,3,4,5,6,7,8,9,10,11,12,13,14,15,\overline{1},\overline{2},\overline{3},\overline{4},\overline{5},\overline{6},\overline{7},\overline{8},\overline{9},\overline{10},\overline{11},\overline{12},\overline{13},\overline{14}}{1.1}
    \polygonnew{(-4,0)}{objleftin}{6}{0.75}
      {\overline{17},\overline{16},18,17,16,\overline{18}}{0.65}

     \draw[fill=black,fill opacity=0.1] (objleftin6) to[bend right=5] (objleftout29) to[bend right=30] (objleftout30) to[bend right=5] (objleftin6);
     \draw[fill=black,fill opacity=0.1] (objleftin3) to[bend right=5] (objleftout14) to[bend right=30] (objleftout15) to[bend right=5] (objleftin3);

     \draw[fill=black,fill opacity=0.1] (objleftout2) to[bend right=5] (objleftin1) to[bend left=5] (objleftout1) to[bend right=10] (objleftout2);
     \draw[fill=black,fill opacity=0.1] (objleftout17) to[bend right=5] (objleftin4) to[bend left=5] (objleftout16) to[bend right=10] (objleftout17);

     \draw[fill=black,fill opacity=0.1] (objleftout3) to[bend right=30] (objleftout4) to[bend left=10] (objleftin2) to[bend right=5] (objleftout3);
     \draw[fill=black,fill opacity=0.1] (objleftout18) to[bend right=30] (objleftout19) to[bend left=10] (objleftin5) to[bend right=5] (objleftout18);

     \draw[fill=black,fill opacity=0.1] (objleftout5) to[bend right=30] (objleftout9) to[bend right=30] (objleftout13) to[bend left=5] (objleftout5);
     \draw[fill=black,fill opacity=0.1] (objleftout20) to[bend right=30] (objleftout24) to[bend right=30] (objleftout28) to[bend left=5] (objleftout20);

     \draw[fill=black,fill opacity=0.1] (objleftout6) to[bend right=30] (objleftout7) to[bend right=30] (objleftout8) to[bend left=30] (objleftout6);
     \draw[fill=black,fill opacity=0.1] (objleftout21) to[bend right=30] (objleftout22) to[bend right=30] (objleftout23) to[bend left=30] (objleftout21);

     \draw[fill=black,fill opacity=0.1] (objleftout10) to[bend right=30] (objleftout11) to[bend right=30] (objleftout12) to[bend left=30] (objleftout10);
     \draw[fill=black,fill opacity=0.1] (objleftout25) to[bend right=30] (objleftout26) to[bend right=30] (objleftout27) to[bend left=30] (objleftout25);

     \node[inner sep=0pt] at (0,0) {$\mapsto$};

    \polygonnew{(4,0)}{objrightout}{30}{3}
      {1,2,3,4,5,6,7,8,9,10,11,12,13,14,15,\overline{1},\overline{2},\overline{3},\overline{4},\overline{5},\overline{6},\overline{7},\overline{8},\overline{9},\overline{10},\overline{11},\overline{12},\overline{13},\overline{14},\overline{15}}{1.1}
    \polygonnew{(4,0)}{objrightin}{6}{0.75}
      {\overline{18},\overline{17},\overline{16},18,17,16}{0.65}

     \draw[fill=black,fill opacity=0.1] (objrightout1) to[bend left=30] (objrightout29) to[bend right=30] (objrightout30) to[bend right=30] (objrightout1);
     \draw[fill=black,fill opacity=0.1] (objrightout16) to[bend left=30] (objrightout14) to[bend right=30] (objrightout15) to[bend right=30] (objrightout16);

     \draw[fill=black,fill opacity=0.1] (objrightout2) to[bend right=5] (objrightin1) to[bend right=50] (objrightin6) to[bend right=10] (objrightout2);
     \draw[fill=black,fill opacity=0.1] (objrightout17) to[bend right=5] (objrightin4) to[bend right=50] (objrightin3) to[bend right=10] (objrightout17);

     \draw[fill=black,fill opacity=0.1] (objrightout3) to[bend right=30] (objrightout4) to[bend left=10] (objrightin2) to[bend right=5] (objrightout3);
     \draw[fill=black,fill opacity=0.1] (objrightout18) to[bend right=30] (objrightout19) to[bend left=10] (objrightin5) to[bend right=5] (objrightout18);

     \draw[fill=black,fill opacity=0.1] (objrightout5) to[bend right=30] (objrightout9) to[bend right=30] (objrightout13) to[bend left=5] (objrightout5);
     \draw[fill=black,fill opacity=0.1] (objrightout20) to[bend right=30] (objrightout24) to[bend right=30] (objrightout28) to[bend left=5] (objrightout20);

     \draw[fill=black,fill opacity=0.1] (objrightout6) to[bend right=30] (objrightout7) to[bend right=30] (objrightout8) to[bend left=30] (objrightout6);
     \draw[fill=black,fill opacity=0.1] (objrightout21) to[bend right=30] (objrightout22) to[bend right=30] (objrightout23) to[bend left=30] (objrightout21);

     \draw[fill=black,fill opacity=0.1] (objrightout10) to[bend right=30] (objrightout11) to[bend right=30] (objrightout12) to[bend left=30] (objrightout10);
     \draw[fill=black,fill opacity=0.1] (objrightout25) to[bend right=30] (objrightout26) to[bend right=30] (objrightout27) to[bend left=30] (objrightout25);
    \end{tikzpicture}
  \caption{The action of the pseudo-rotation $\rotD$, Case (2)}
\label{fig:20}
\end{figure}

\begin{figure}
\begin{tikzpicture}[scale=1]

\polygonnew{(-4,0)}{objleftout}{40}{3}
  {2,3,4,5,6,7,8,9,10,11,12,13,14,15,16,17,18,19,20,%
   \overline{1},\overline{2},\overline{3},\overline{4},\overline{5},
   \overline{6},\overline{7},\overline{8},\overline{9},\overline{10},
   \overline{11},\overline{12},\overline{13},\overline{14},\overline{15},
   \overline{16},\overline{17},\overline{18},\overline{19},\overline{20},1}{1.1}

\polygonnew{(-4,0)}{objleftin}{8}{0.75}
  {22,21,\overline{24},\overline{23},\overline{22},\overline{21},24,23}{0.65}

\draw[fill=black,fill opacity=0.1]
      (objleftout40) to[bend right=30] (objleftout1)
                     to[bend right= 5] (objleftin8)
                     to[bend right= 5] (objleftout39)
                     to[bend right=30] (objleftout40);

\draw[fill=black,fill opacity=0.1]
      (objleftout2)  to[bend right=30] (objleftout3)
                     to[bend left = 5] (objleftout16)
                     to[bend right=20] (objleftin1)
                     to[bend right=10] (objleftout2);

\draw[fill=black,fill opacity=0.1]
      (objleftout4)  to[bend right=30] (objleftout9)
                     to[bend right=10] (objleftout10)
                     to[bend right=30] (objleftout11)
                     to[bend left =30] (objleftout4);

\draw[fill=black,fill opacity=0.1]
      (objleftout5)  to[bend right=30] (objleftout6)
                     to[bend right=30] (objleftout7)
                     to[bend right=30] (objleftout8)
                     to[bend left =30] (objleftout5);

\draw[fill=black,fill opacity=0.1]
      (objleftout12) to[bend right=30] (objleftout13)
                     to[bend right=30] (objleftout14)
                     to[bend right=30] (objleftout15)
                     to[bend left =30] (objleftout12);

\draw[fill=black,fill opacity=0.1]
      (objleftout17) to[bend right=30] (objleftout18)
                     to[bend left =10] (objleftin3)
                     to[bend right=50] (objleftin2)
                     to[bend left =10] (objleftout17);

\draw[fill=black,fill opacity=0.1]
      (objleftout19) to[bend right=30] (objleftout20)
                     to[bend right=30] (objleftout21)
                     to[bend right= 5] (objleftin4)
                     to[bend right= 5] (objleftout19);

\draw[fill=black,fill opacity=0.1]
      (objleftout22) to[bend right=30] (objleftout23)
                     to[bend left = 5] (objleftout36)
                     to[bend right=20] (objleftin5)
                     to[bend right=10] (objleftout22);

\draw[fill=black,fill opacity=0.1]
      (objleftout24) to[bend right=30] (objleftout29)
                     to[bend right=10] (objleftout30)
                     to[bend right=30] (objleftout31)
                     to[bend left =30] (objleftout24);

\draw[fill=black,fill opacity=0.1]
      (objleftout25) to[bend right=30] (objleftout26)
                     to[bend right=30] (objleftout27)
                     to[bend right=30] (objleftout28)
                     to[bend left =30] (objleftout25);

\draw[fill=black,fill opacity=0.1]
      (objleftout32) to[bend right=30] (objleftout33)
                     to[bend right=30] (objleftout34)
                     to[bend right=30] (objleftout35)
                     to[bend left =30] (objleftout32);
\draw[fill=black,fill opacity=0.1]
      (objleftout37) to[bend right=30] (objleftout38)
                     to[bend left =10] (objleftin7)
                     to[bend right=50] (objleftin6)
                     to[bend left =10] (objleftout37);

\node[inner sep=0pt] at (0,0) {$\mapsto$};

\polygonnew{(4,0)}{objrightout}{40}{3}
  {2,3,4,5,6,7,8,9,10,11,12,13,14,15,16,17,18,19,20,%
   \overline{1},\overline{2},\overline{3},\overline{4},\overline{5},
   \overline{6},\overline{7},\overline{8},\overline{9},\overline{10},
   \overline{11},\overline{12},\overline{13},\overline{14},\overline{15},
   \overline{16},\overline{17},\overline{18},\overline{19},\overline{20},1}{1.1}

\polygonnew{(4,0)}{objrightin}{8}{0.75}
  {\overline{21},24,23,22,21,\overline{24},\overline{23},\overline{22}}{0.65}

\draw[fill=black,fill opacity=0.1]
      (objrightout1) to[bend right=30] (objrightout2)
                     to[bend left =20] (objrightin2)
                     to[bend right=50] (objrightin1)
                     to[bend right=20] (objrightout1);

\draw[fill=black,fill opacity=0.1]
      (objrightout3) to[bend right=30] (objrightout4)
                     to[bend left = 5] (objrightout17)
                     to[bend right=10] (objrightin3)
                     to[bend right=20] (objrightout3);

\draw[fill=black,fill opacity=0.1]
      (objrightout5) to[bend right=30] (objrightout10)
                     to[bend right=10] (objrightout11)
                     to[bend right=30] (objrightout12)
                     to[bend left =30] (objrightout5);

\draw[fill=black,fill opacity=0.1]
      (objrightout6) to[bend right=30] (objrightout7)
                     to[bend right=30] (objrightout8)
                     to[bend right=30] (objrightout9)
                     to[bend left =30] (objrightout6);

\draw[fill=black,fill opacity=0.1]
      (objrightout13) to[bend right=30] (objrightout14)
                      to[bend right=30] (objrightout15)
                      to[bend right=30] (objrightout16)
                      to[bend left =30] (objrightout13);

\draw[fill=black,fill opacity=0.1]
      (objrightout18) to[bend right=30] (objrightout19)
                      to[bend right =30] (objrightout20)
                      to[bend right = 5] (objrightin4)
                      to[bend right = 5] (objrightout18);

\draw[fill=black,fill opacity=0.1]
      (objrightin5)  to[bend right=10] (objrightout21)
                     to[bend right =20] (objrightout22)
                     to[bend left  =20] (objrightin6)
                     to[bend right =50] (objrightin5);

\draw[fill=black,fill opacity=0.1]
      (objrightout23) to[bend right=30] (objrightout24)
                      to[bend left = 5] (objrightout37)
                      to[bend right=10] (objrightin7)
                      to[bend right=20] (objrightout23);

\draw[fill=black,fill opacity=0.1]
      (objrightout25) to[bend right=30] (objrightout30)
                      to[bend right=10] (objrightout31)
                      to[bend right=30] (objrightout32)
                      to[bend left =30] (objrightout25);

\draw[fill=black,fill opacity=0.1]
      (objrightout26) to[bend right=30] (objrightout27)
                      to[bend right=30] (objrightout28)
                      to[bend right=30] (objrightout29)
                      to[bend left =30] (objrightout26);

\draw[fill=black,fill opacity=0.1]
      (objrightout33) to[bend right=30] (objrightout34)
                      to[bend right=30] (objrightout35)
                      to[bend right=30] (objrightout36)
                      to[bend left =30] (objrightout33);

\draw[fill=black,fill opacity=0.1]
      (objrightout38) to[bend right=30] (objrightout39)
                      to[bend right=30] (objrightout40)
                      to[bend right= 5] (objrightin8)
                      to[bend right= 5] (objrightout38);
\end{tikzpicture}
\caption{The action of the pseudo-rotation $\rotD$, Case (2)}
\label{fig:20B}
\end{figure}

\begin{figure}
\begin{center}
  \begin{tikzpicture}[scale=1]

    \polygon{(-4,0)}{obj}{44}{3}
       {1,,,,a,,,,,,,,,,,,,,,,,\hspace*{25pt}m(n-1),\overline{1},,,,\overline{a},,,,,,,,,,,,,,,,,\overline{m(n-1)}\hspace*{25pt}}

    \polygoninner{(-4,0)}{objin}{16}{1.2}
      {,,\overline{m(n-1)+1}\hspace*{25pt},mn,mn-1,,,,,,\hspace*{25pt}m(n-1)+1,\overline{mn},\overline{mn-1},,,}

    \draw[line width=2.5pt,black] (obj1) to[bend left=50] (obj44);
    \draw[line width=2.5pt,black] (obj23) to[bend left=50] (obj22);

    \draw[dash pattern=on 2pt off 2pt on 2pt off 2pt on 2pt off 2pt on
2pt off 16pt on 2pt off 2pt on 2pt off 2pt] (obj1) to[bend right=20]
(obj5);
    \draw[dash pattern=on 2pt off 2pt on 2pt off 2pt on 2pt off 2pt on
2pt off 16pt on 2pt off 2pt on 2pt off 2pt] (obj23) to[bend right=20]
(obj27);

    \draw (obj44) to[bend right=50] (obj5);
    \draw (obj22) to[bend right=50] (obj27);

     \draw[dashed] (objin5) to[bend left=40] (objin6) to[bend left=40]
(objin7) to[bend left=40] (objin8) to[bend left=40] (objin9) to[bend
left=40] (objin10) to[bend left=40] (objin11);
     \draw (objin11) to[bend left=40] (objin12) to[bend
right=100,looseness=2] (objin5);
     \draw[dashed] (objin13) to[bend left=40] (objin14) to[bend
left=40] (objin15) to[bend left=40] (objin16) to[bend left=40]
(objin1) to[bend left=40] (objin2) to[bend left=40] (objin3);
     \draw (objin3) to[bend left=40] (objin4) to[bend
right=100,looseness=2] (objin13);

    \node[inner sep=0pt] at (0,0) {$\mapsto$};

    \polygon{(4,0)}{obj}{44}{3}
       {1,2,,,,a+1,,,,,,,,,,,,,,,,\hspace*{25pt}m(n-1),\overline{1},\overline{2},,,,\overline{a+1},,,,,,,,,,,,,,,,\overline{m(n-1)}\hspace*{25pt}}

    \polygoninner{(4,0)}{objin}{16}{1.2}
      {,\overline{m(n-1)+2}\hspace*{25pt},\overline{m(n-1)+1}\hspace*{25pt},mn,,,,,,\hspace*{25pt}m(n-1)+2,\hspace*{25pt}m(n-1)+1,\overline{mn},,,,}

    \draw[line width=2.5pt,black] (obj1) to[bend right=40] (obj28);
    \draw[line width=2.5pt,black] (obj23) to[bend right=40] (obj6);

    \draw[dash pattern=on 2pt off 2pt on 2pt off 2pt on 2pt off 2pt on
2pt off 16pt on 2pt off 2pt on 2pt off 2pt] (obj2) to[bend right=10]
(obj6);
    \draw[dash pattern=on 2pt off 2pt on 2pt off 2pt on 2pt off 2pt on
2pt off 16pt on 2pt off 2pt on 2pt off 2pt] (obj24) to[bend right=10]
(obj28);

    \draw (obj2) to[bend left=60] (obj23);
    \draw (obj24) to[bend left=60] (obj1);

     \draw[dashed] (objin4) to[bend left=40] (objin5) to[bend left=40]
(objin6) to[bend left=40] (objin7) to[bend left=40] (objin8) to[bend
left=40] (objin9) to[bend left=40] (objin10);
     \draw (objin4) to[bend right=40] (objin3) to[bend
left=100,looseness=2] (objin10);
     \draw[dashed] (objin12) to[bend left=40] (objin13) to[bend
left=40] (objin14) to[bend left=40] (objin15) to[bend left=40]
(objin16) to[bend left=40] (objin1) to[bend left=40] (objin2);
     \draw (objin12) to[bend right=40] (objin11) to[bend
left=100,looseness=2] (objin2);

    \end{tikzpicture}
\end{center}
\caption{The action of the pseudo-rotation $\rotD$, Case (3)}
\label{fig:14}
\end{figure}

\begin{figure}
  \centering
  \begin{tikzpicture}[scale=1.35]
    \polygon{(-2.9,0)}{obj}{26}{2.5}
      {1,2,3,4,5,6,7,8,9,10,11,,12,\overline{1},\overline{2},\overline{3},\overline{4},\overline{5},\overline{6},\overline{7},\overline{8},\overline{9},\overline{10},\overline{11},,\overline{12}}

    \polygoninner{(-2.9,0)}{objin}{12}{0.5}
      {15,,14,,13,,\overline{15},,\overline{14},,\overline{13}}

      \draw[line width=2.5pt,black] (obj1) to[bend left=50] (obj26);
      \draw[line width=2.5pt,black] (obj14) to[bend left=50] (obj13);

     \draw[fill=black,fill opacity=0.1] (obj26) to[bend right=30]
(obj1) to[bend right=30] (obj8) to[bend left=30] (obj26);
     \draw[fill=black,fill opacity=0.1] (obj2) to[bend right=30]
(obj6) to[bend right=30] (obj7) to[bend left=30] (obj2);
     \draw[fill=black,fill opacity=0.1] (obj3) to[bend right=30]
(obj4) to[bend right=30] (obj5) to[bend left=30] (obj3);
     \draw[fill=black,fill opacity=0.1] (obj9) to[bend right=30]
(obj10) to[bend right=30] (obj11) to[bend left=30] (obj9);
     \draw[fill=black,fill opacity=0.1] (obj13) to[bend right=30]
(obj14) to[bend right=30] (obj21) to[bend left=30] (obj13);
     \draw[fill=black,fill opacity=0.1] (obj15) to[bend right=30]
(obj19) to[bend right=30] (obj20) to[bend left=30] (obj15);
     \draw[fill=black,fill opacity=0.1] (obj16) to[bend right=30]
(obj17) to[bend right=30] (obj18) to[bend left=30] (obj16);
     \draw[fill=black,fill opacity=0.1] (obj22) to[bend right=30]
(obj23) to[bend right=30] (obj24) to[bend left=30] (obj22);

     \draw[fill=black,fill opacity=0.1] (objin9) to[bend left=40]
(objin11) to[bend left=40] (objin1) to[bend right=100,looseness=2] (objin9);

     \draw[fill=black,fill opacity=0.1] (objin3) to[bend left=40]
(objin5) to[bend left=40] (objin7) to[bend right=100,looseness=2] (objin3);

    \polygon{(2.9,0)}{obj}{26}{2.5}
      {1,,2,3,4,5,6,7,8,9,10,11,12,\overline{1},,\overline{2},\overline{3},\overline{4},\overline{5},\overline{6},\overline{7},\overline{8},\overline{9},\overline{10},\overline{11},\overline{12}}

    \node[inner sep=0pt] at (0,0) {$\mapsto$};

    \polygoninner{(2.9,0)}{objin}{12}{0.5}
      {15,,14,,13,,\overline{15},,\overline{14},,\overline{13}}

      \draw[line width=2.5pt,black] (obj1) to[bend left=30] (obj23);
      \draw[line width=2.5pt,black] (obj14) to[bend left=30] (obj10);

     \draw[fill=black,fill opacity=0.1] (obj3) to[bend right=30]
(obj10) to[bend right=30] (obj14) to (obj3);
     \draw[fill=black,fill opacity=0.1] (obj4) to[bend right=30]
(obj8) to[bend right=30] (obj9) to[bend left=30] (obj4);
     \draw[fill=black,fill opacity=0.1] (obj5) to[bend right=30]
(obj6) to[bend right=30] (obj7) to[bend left=30] (obj5);
     \draw[fill=black,fill opacity=0.1] (obj11) to[bend right=30]
(obj12) to[bend right=30] (obj13) to[bend left=30] (obj11);

     \draw[fill=black,fill opacity=0.1] (obj16) to[bend right=30]
(obj23) to[bend right=30] (obj1) to (obj16);
     \draw[fill=black,fill opacity=0.1] (obj17) to[bend right=30]
(obj21) to[bend right=30] (obj22) to[bend left=30] (obj17);
     \draw[fill=black,fill opacity=0.1] (obj18) to[bend right=30]
(obj19) to[bend right=30] (obj20) to[bend left=30] (obj18);
     \draw[fill=black,fill opacity=0.1] (obj24) to[bend right=30]
(obj25) to[bend right=30] (obj26) to[bend left=30] (obj24);

     \draw[fill=black,fill opacity=0.1] (objin11) to[bend left=40]
(objin1) to[bend left=40] (objin3) to[bend right=100,looseness=2] (objin11);

     \draw[fill=black,fill opacity=0.1] (objin5) to[bend left=40]
(objin7) to[bend left=40] (objin9) to[bend right=100,looseness=2] (objin5);
    \end{tikzpicture}
  \caption{The action of the pseudo-rotation $\rotD$, Case (3)}
\label{fig:24}
\end{figure}

\begin{remarks}
(i) \Cref{fig:12} provides a schematic illustration of 
the construction in Case~(1) of the statement of 
the above theorem, while \Cref{fig:23,fig:23a} provide two concrete
examples. In both examples, $m=3$, $n=5$, $a=2$, $b=6$, $d=11$, and
$e=10$.
However, they differ in the existence of a zero block: in \Cref{fig:23}
there is a zero block, while in \Cref{fig:23a} there is none.
Case~(2) of the statement of \Cref{prop:2D} is illustrated in
\Cref{fig:13}, while \Cref{fig:19,fig:20,fig:20B}
provide three concrete examples.
In these examples,
$m=4$, $n=6$, $a=\overline{23}$, $b=17$, and $d=19$, respectively
$m=3$, $n=6$, $a=\overline{17}$, $b=13$, and $d=14$, respectively
$m=3$, $n=6$, $a=23$, $b=21$, and $d=19$.
Finally, \Cref{fig:14} provides a schematic illustration of 
the construction in Case~(3) of the statement of 
the above theorem, while \Cref{fig:24} contains a concrete example,
in which $m=3$, $n=5$, and $a=8$.

\medskip
(ii) Case~(2) can be seen as a degenerate special case of Case~(1).
To see this, one chooses $b=\mp mn$ in Case~(1). 
\end{remarks}

\begin{proof}[Proof of \Cref{prop:2D}]
Let us first assume that, 
in the positive $m$-divisible type~$D$ non-crossing partition~$\pi$,
the block of $\overline{m(n-1)}$ contains another negative element
which is different from~$\overline{mn-1}$.
Among those, let $\overline{mi_1-1}$ be the one which is the
predecessor of $\overline{m(n-1)}$ in the block.
Obviously, we have $i_1>0$.
\Cref{fig:16D} provides sketches of such situations. 
Let $(w_0,w_1,\dots,w_m)$ denote the
element of $\mNCPlus[D_n]$ which corresponds to~$\pi$ under the
bijection $\Nam{D_n}m$. By the definition of $\Nam{D_n}m$, we
see that $w_{m-1}(n-1)=i_1$. From
\Cref{lem:LRD}(1), we infer $w_{m-1}^L=w_{m-1}$ and
$w_{m-1}^R=\one$. This means that the pseudo-rotation~$\rotD$
reduces to ordinary rotation in this special case, as we claimed. 

\begin{figure}
\begin{center}
\begin{tikzpicture}
    \polygonlabel{(-4,-21)}{obj}{44}{2.5}
      {1,,,,,,,,,,,,,,,,,,,,,,,,,,,,\hspace{1pt},,,,,,,,\overline{mi_1-1},,,,,,\overline{m(n-1)},}

    \draw[line width=2.5pt,black] (obj1) to[bend left=30] (obj29);
    \draw (obj1)  to[bend left=30]
          (obj29);
    \draw (obj43) to[bend left=50]
          (obj37);

    \draw[dash pattern=on 22pt off 100pt] (obj43) to[bend left=20] (obj29);
    \draw[dash pattern=on 22pt off 100pt] (obj37) to[bend left=20] (obj29);
    \draw[dash pattern=on 22pt off 140pt] (obj1)  to[bend right=20] (obj8);

    \node[inner sep=0pt] at (0,-21) {};

    \polygonlabel{( 4,-21)}{obj}{44}{2.5}
      {1,,,,,,,,,,,,,,,,,,,,,,,,,,,,,,,,,,,,\overline{mi_1-1},,,,,,\overline{m(n-1)},}

    \draw[line width=2.5pt,black] (obj1) to[bend left=50] (obj43);
    \draw (obj43) to[bend left=50]
          (obj37);

    \draw[dash pattern=on 22pt off 100pt] (obj37) to[bend left=20] (obj29);
    \draw[dash pattern=on 22pt off 140pt] (obj1)  to[bend right=20] (obj8);

\end{tikzpicture}
\end{center}
  \caption{The block of $\overline{m(n-1)}$ contains another negative element}
\label{fig:16D}
\end{figure}

\medskip
For the remainder of the proof, we assume that the block of
$\overline{m(n-1)}$ contains no other negative element which
is different from~$\overline{mn-1}$. By arguments that are completely
analogous to the corresponding ones in the proof of
\Cref{prop:2B}, one shows that~$1$ is an element 
of that block.

\medskip
Let again $(w_0,w_1,\dots,w_m)$ denote the
element of $\mNCPlus[D_n]$ which corresponds to~$\pi$ under the
bijection $\Nam{D_n}m$. 

\smallskip
{\sc Case (1).} 
This is proved in the same way as Case~(1) in the proof of
\Cref{prop:2B}.

\smallskip
{\sc Case (2).} This case is also similar to Case~(1) in the proof
of \Cref{prop:2B}. We are at times sketchy here,
leaving some details to the reader.

By the assumptions of this case, and by what we already
showed, we know that $a$, $\overline{m(n-1)}$, and~$1$
are successive elements in a block, and that
$\overline d$, $\mp{mn}$ and $\overline b$ are successive elements
in a block.

We write $a=mA-1$, $b=mB+1$, and $d=mD-1$,
for suitable integers~$A$, $B$, and $D$. Here, $D$ is positive.
The integer $A$ can be between $1$ and $n-1$ or equal to $\mp n$,
and $B$ can be between $A$ and $D-1$ or equal to $\mp(n-1)$.
In view of the above observations,
by the definition of $\Nam{D_n}m$, we have
\begin{align}
\notag
w_{m-1}(\overline{n-1})&=A,\\
\notag
w_{m-1}(\pm n)&=D,\\
\label{eq:wmex4c}
w_{m}(n-1)&=n-1,\\
\label{eq:wmex4d}
w_m(B)&=\pm n.
\end{align}
Consequently, the cycle of $w_{m-1}$ containing $n-1$ is of the
form 
$$((D,\dots,n-1,\overline A,\dots,\pm n)).$$
(At this point, one has to argue that $n-1$ and $\pm n$ are indeed
in the same cycle of $w_{m-1}$.)
By \Cref{lem:LRD}(3), we have 
$$w^R_{m-1}=((n-1,\pm n)),\quad \text{and}\quad 
w^L_{m-1}\text{ contains the cycles }((D,\dots,n-1))((\overline A,\dots,\pm n)).$$
Together with \eqref{eq:wmex4c} and \eqref{eq:wmex4d}, 
this implies the relations
\begin{align*}
w^L_{m-1}(n-1)&=D,\\
w^L_{m-1}(\pm n)&=\overline A,\\
(cw^R_{m-1}w_mc^{-1})(1)&=\pm n,\\
(cw^R_{m-1}w_mc^{-1})(B+1)&=\overline 1.
\end{align*}
By the definition of $\Nam{D_n}m$, this means that
$\overline{d+1}$, $1$, and~$\overline{b+1}$ are successive elements in 
their block in the image of~$\pi$ under~$\Krewplustilde^{(m)}$
(conjugated by~$\Nam{D_n}m$), that
$a+1$, $\pm(mn-m+1)$, and~$2$ are successive elements in 
their block, 
while all other ``block connections" are rotated by one unit
in clockwise direction on the outer circle and by one unit in
counter-clockwise direction on the inner circle.
This is in accordance with the corresponding assertion in the
proposition.

\smallskip
{\sc Case (3).} 
We are again somewhat sketchy here.
We write again $a=mA-1$ for a suitable positive integer~$A$.
By the definition of $\Nam{D_n}m$,
we have
\begin{align}
\notag
w_{m-1}(n-1)&=\overline A,\\
\notag
w_{m-1}(n)&=\overline n,\\
\label{eq:wmex4f}
w_m(n-1)&=n-1,\\
\label{eq:wmex4g}
w_m(n)&=n.
\end{align}
Consequently, the cycle of $w_{m-1}$ containing $n-1$ is of the
form 
$$[A,\dots,\dots,n-1][n].$$
By \Cref{lem:LRD}(4), we have 
$$w^R_{m-1}=[n-1][n],\quad \text{and }\quad 
w^L_{m-1}\text{ contains the cycles }((A,\dots,n-1))((n)).$$
Together with \eqref{eq:wmex4f} and \eqref{eq:wmex4g}, 
this implies the relations
\begin{align*}
w^L_{m-1}(n-1)&=A,\\
w^L_{m-1}(n)&=n,\\
(cw^R_{m-1}w_mc^{-1})(1)&=\overline 1,\\
(cw^R_{m-1}w_mc^{-1})(n)&=\overline n.
\end{align*}
By the definition of $\Nam{D_n}m$, this means that
$\overline{a+1}$, $1$, and~$\overline 2$ are successive elements in 
their block in the image of~$\pi$ under~$\Krewplustilde^{(m)}$
(conjugated by~$\Nam{D_n}m$), that $mn$, $\overline{mn-m+1}$, and $mn-m+2$
are successive elements in their block,
while the other ``block connections" are rotated by one unit
in clockwise direction on the outer circle and by {\it two} units in
counter-clockwise direction on the inner circle.
This is in accordance with the corresponding assertion in the
proposition.

This completes the proof of the proposition.
\end{proof}

\subsection{The combinatorial positive Kreweras maps}
\label{sec:combD}

The definition of $\rotD$ from \Cref{prop:2D} to
make sense on an \emph{arbitrary} non-crossing
partition needed that blocks should be of size at least~$2$.
It is a simple matter to extend this so that it
makes sense for non-crossing partitions without any restrictions.

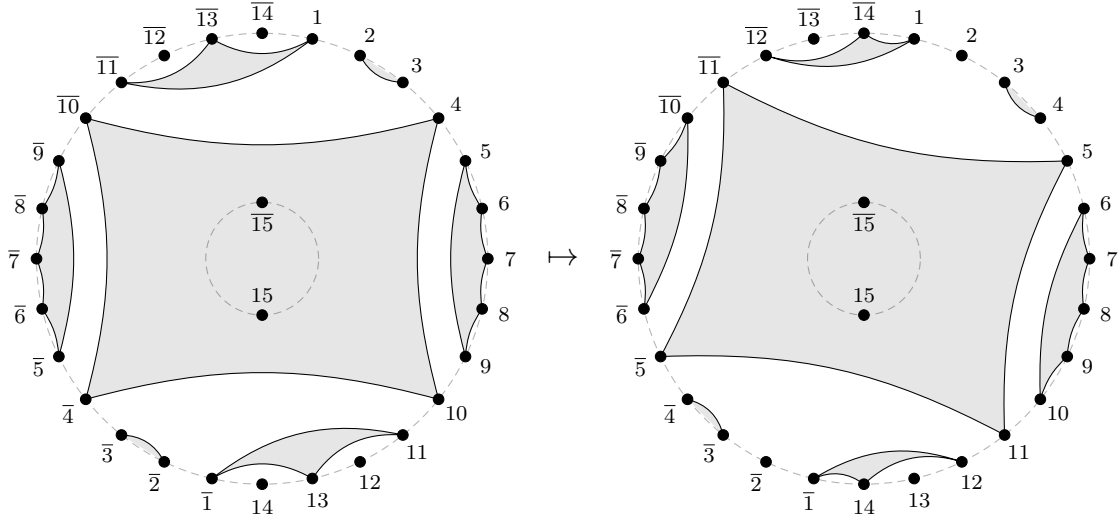
\begin{figure}
\begin{center}
  \begin{tikzpicture}[scale=1]
    \polygonnew{(-4,0)}{obj}{28}{3}
      {1,2,3,4,5,6,7,8,9,10,11,12,13,14,\overline{1},\overline{2},\overline{3},\overline{4},\overline{5},\overline{6},\overline{7},\overline{8},\overline{9},\overline{10},\overline{11},\overline{12},\overline{13},\overline{14}}{1.1}
    \polygonnew{(-4,0)}{objin}{2}{0.75}
      {15,\overline{15}}{0.65}

     \draw[fill=black,fill opacity=0.1] (obj18) to[bend right=15]
(obj24) to[bend right=15] (obj4) to[bend right=15] (obj10) to[bend
right=15] (obj18);
     \draw[fill=black,fill opacity=0.1] (obj19) to[bend right=15]
(obj20) to[bend right=15] (obj21) to[bend right=15] (obj22) to[bend
right=15] (obj23) to[bend left=15] (obj19);
     \draw[fill=black,fill opacity=0.1] (obj5) to[bend right=15]
(obj6) to[bend right=15] (obj7) to[bend right=15] (obj8) to[bend
right=15] (obj9) to[bend left=15] (obj5);
     \draw[fill=black,fill opacity=0.1] (obj1) to[bend left=25]
(obj25) to[bend right=30] (obj27) to[bend right=30] (obj1);
     \draw[fill=black,fill opacity=0.1] (obj15) to[bend left=25]
(obj11) to[bend right=30] (obj13) to[bend right=30] (obj15);
     \draw[fill=black,fill opacity=0.1] (obj2) to[bend right=30] (obj3);
     \draw[fill=black,fill opacity=0.1] (obj16) to[bend right=30] (obj17);

    \node[inner sep=0pt] at (0,0) {$\mapsto$};

    \polygonnew{(4,0)}{obj}{28}{3}
      {1,2,3,4,5,6,7,8,9,10,11,12,13,14,\overline{1},\overline{2},\overline{3},\overline{4},\overline{5},\overline{6},\overline{7},\overline{8},\overline{9},\overline{10},\overline{11},\overline{12},\overline{13},\overline{14}}{1.1}
    \polygonnew{(4,0)}{objin}{2}{0.75}
      {15,\overline{15}}{0.65}

     \draw[fill=black,fill opacity=0.1] (obj19) to[bend right=15]
(obj25) to[bend right=15] (obj5) to[bend right=15] (obj11) to[bend
right=15] (obj19);
     \draw[fill=black,fill opacity=0.1] (obj20) to[bend right=15]
(obj21) to[bend right=15] (obj22) to[bend right=15] (obj23) to[bend
right=15] (obj24) to[bend left=15] (obj20);
     \draw[fill=black,fill opacity=0.1] (obj6) to[bend right=15]
(obj7) to[bend right=15] (obj8) to[bend right=15] (obj9) to[bend
right=15] (obj10) to[bend left=15] (obj6);
     \draw[fill=black,fill opacity=0.1] (obj1) to[bend left=25]
(obj26) to[bend right=30] (obj28) to[bend right=30] (obj1);
     \draw[fill=black,fill opacity=0.1] (obj15) to[bend left=25]
(obj12) to[bend right=30] (obj14) to[bend right=30] (obj15);
     \draw[fill=black,fill opacity=0.1] (obj3) to[bend right=30] (obj4);
     \draw[fill=black,fill opacity=0.1] (obj17) to[bend right=30] (obj18);
    \end{tikzpicture}
\end{center}
\caption{Example of the action of the pseudo-rotation $\rotD$ under
presence of singleton blocks}
\label{fig:B8}
\end{figure}

\begin{definition} \label{def:phiallgD}
We define the action of $\rotD$ on a positive type~$D$ 
non-crossing partition~$\pi$
of $\{1,2,\dots, n,\overline1,\overline2,\dots,\overline n\}$ 
as in \Cref{prop:2D} with $m=1$
(cf.\ \Cref{fig:12,fig:13,fig:14}), 
except when the block containing $n-1$ is a
singleton block (in which case also $\{\overline{n-1}\}$ is a singleton 
block; here, \Cref{prop:2D} would not make sense).
If $\{n-1\}$ is a singleton block, then the image of~$\rotD$
is the non-crossing partition which arises from~$\pi$ 
by replacing the singleton block $\{n-1\}$ 
by the singleton block $\{\overline 2\}$, by replacing the
singleton block $\{\overline {n-1}\}$ by the singleton block $\{2\}$, 
and $i$ by~$i+1$ and $\overline i$ by~$\overline{i+1}$ for
$i=2,3,\dots,n-2$
in each of the remaining blocks of~$\pi$, leaving~$1$, $\overline1$, $n$,
and~$\overline n$ untouched; see \Cref{fig:B8} for an example
with $n=15$.
\end{definition}

As it turns out, this definition is also in agreement with the
image under $\Nam{D_{n}}1$ of the map $\Krewplustilde^{(1)}$ in
\Cref{def:K-alt} for the case where $m=1$.

\begin{theorem} \label{prop:2-m=1D}
Let $n$ be a positive integer.
Under the bijection $\Nam{D_{n}}1$, the map $\Krewplustilde^{(1)}$
from \Cref{def:K-alt} translates to the combinatorial map~$\rotD$ described in
\Cref{def:phiallgD}.
\end{theorem}

Since the proof of the theorem is completely analogous to the proof
of \Cref{prop:2-m=1}, we omit it here.

\medskip
We conclude with the following theorem describing the order of the
combinatorial\break map~$\rotD$. In view of our previous discussion, the
first part of the theorem follows from
\Cref{cor:order} for the type~$D_N$ with $m=1$.
We provide a direct, combinatorial, proof 
which also allows us to show the more refined assertions in the
second part of the theorem. See \Cref{app:order-D}.

\begin{theorem} \label{lem:N-2D}
Let $m$ and $n$ be positive integers.
The order of the map~$\rotD$ acting on $\mNCDPlus$
is $m(n-1)-1$ if $n$ is even, and it is equal
to $2m(n-1)-2$ if $n$ is odd. More precisely, given an element~$\pi$ of
$\mNCDPlus$ without bridging block, we have $\rotD^{m(n-1)-1}(\pi)=\pi$.
On the other hand, if $\pi\in \mNCDPlus$ contains a bridging block, 
then $\rotD^{m(n-1)-1}(\pi)=\pi$ if $n$ is even, while only
$\rotD^{2m(n-1)-2}(\pi)=\pi$ if $n$ is odd.
\end{theorem}

\subsection{Enumeration of positive non-crossing partitions}
\label{sec:enumD}

Here we present our enumeration results for positive non-crossing
partitions in $\mNCDPlus$.
In the theorem below we give the total number of these
non-crossing partitions, and as well the number of those of them
in which all block sizes are equal to~$m$.

\begin{theorem} \label{cor:2-D}
Let $m$ and $n$ be positive integers.
The number
of positive $m$-divisible non-crossing partitions of 
$\{1,2,\dots,mn,\overline1,\overline2,\dots,\overline{mn}\}$ of type~$D$
equals
\begin{equation} \label{eq:multichains-si-pos-a-D-1} 
\frac {2m(n-1)+n-2} {n}
\binom {(m+1)(n-1)-1} {n-1},
\end{equation}
while the number of these partitions of which all blocks have size~$m$ equals
\begin{equation} \label{eq:multichains-si-pos-D-1} 
\frac {2m(n-1)-n} {n}
\binom {m(n-1)-1} {n-1}.
\end{equation}
\end{theorem}

\begin{proof}
Elements of $\mNCDPlus$ may or may not have a bridging block.
The case of those without bridging block may be reduced 
to a type~$B$ situation. For the other case, we refer
to the special case of an enumeration result from \Cref{sec:rotenumD}.

To be precise,
as we argue in the proof of \Cref{lem:N-2D} given in \Cref{app:order-D},
partitions in 
$\mNCDPlus$ without a bridging block are in bijection with
partitions in $\mNCBPlus[n-1]$. By~\eqref{eq:multichains-pos-B-1} with $n$
replaced by~$n-1$, the number of these partitions is
$$
\binom {(m+1)(n-1)-1}{n-1}.
$$
On the other hand, a partition in $\mNCDPlus$ with a bridging
block must arise from Construction~3 in \Cref{sec:rotD} below with $\rRr=1$.
By \Cref{thm:enumD-2} below with $\rRr=1$, the number of these
partitions is
$$
2\binom {(m+1)(n-1)-1}{n}.
$$
It is easy to verify that the sum of these two numbers equals~\eqref{eq:multichains-si-pos-a-D-1}.

Similarly, by \eqref{eq:multichains-si-pos-B-1} with $b=n$ and
subsequent replacement of~$n$ by~$n-1$, the number of partitions in 
$\mNCDPlus$ without a bridging block, all of whose blocks have
size~$m$, is given by
$$
\binom {m(n-1)-1}{n-1}.
$$
Moreover, by \Cref{thm:enumD-4} below with $\rRr=1$, the number of 
partitions in $\mNCDPlus$ that contain a bridging block and 
all of whose blocks have size~$m$, is given by
$$
2\binom {m(n-1)-1}{n}.
$$
It is easy to verify that the sum of these two numbers equals~\eqref{eq:multichains-si-pos-D-1}.
\end{proof}

\subsection{Characterisation of pseudo-rotationally invariant elements}
\label{sec:rotD}

In this subsection, we describe the elements of
$\mNCDPlus$ that are invariant under a
power of~$\rotD$. 
We call an element of $\mNCDPlus$
\defn{$\rrr $-pseudo-rotationally 
  invariant\/} if it is invariant under the action
of~$\rotD^{2(m(n-1)-1)/\rrr }$.

We achieve the description of pseudo-rotationally invariant elements
in $\mNCDPlus$
by providing three constructions of such elements (see below)
together with \Cref{lem:allD} which then tells us how to obtain all
such invariant elements from the constructions. 

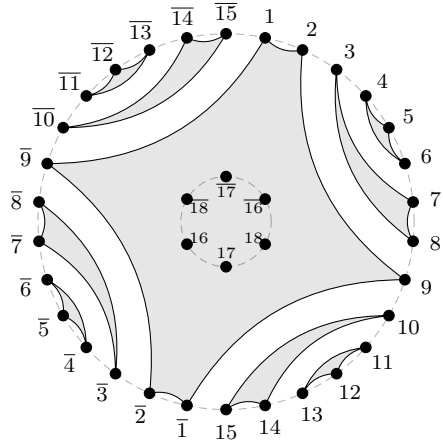
\begin{figure}
  \centering
  \begin{tikzpicture}[scale=1]
    \polygon{(0,0)}{obj}{30}{2.5}
      {1,2,3,4,5,6,7,8,9,10,11,12,13,14,15,\overline{1},\overline{2},\overline{3},\overline{4},\overline{5},\overline{6},\overline{7},\overline{8},\overline{9},\overline{10},\overline{11},\overline{12},\overline{13},\overline{14},\overline{15}}
    \polygoninner{(0,0)}{objin}{12}{0.6}
            {,\overline{16},,18,,17,,16,,\overline{18},,\overline{17},}

     \draw[fill=black,fill opacity=0.1] (obj1) to[bend right=30]
(obj2) to[bend right=30] (obj9) to[bend right=30] (obj16) to[bend
right=30] (obj17) to[bend right=30] (obj24) to[bend right=30] (obj1);
     \draw[fill=black,fill opacity=0.1] (obj3) to[bend right=30]
(obj7) to[bend right=30] (obj8) to[bend left=30] (obj3);
     \draw[fill=black,fill opacity=0.1] (obj4) to[bend right=30]
(obj5) to[bend right=30] (obj6) to[bend left=30] (obj4);
     \draw[fill=black,fill opacity=0.1] (obj10) to[bend right=30]
(obj14) to[bend right=30] (obj15) to[bend left=30] (obj10);
     \draw[fill=black,fill opacity=0.1] (obj11) to[bend right=30]
(obj12) to[bend right=30] (obj13) to[bend left=30] (obj11);
     \draw[fill=black,fill opacity=0.1] (obj18) to[bend right=30]
(obj22) to[bend right=30] (obj23) to[bend left=30] (obj18);
     \draw[fill=black,fill opacity=0.1] (obj19) to[bend right=30]
(obj20) to[bend right=30] (obj21) to[bend left=30] (obj19);
     \draw[fill=black,fill opacity=0.1] (obj25) to[bend right=30]
(obj29) to[bend right=30] (obj30) to[bend left=30] (obj25);
     \draw[fill=black,fill opacity=0.1] (obj26) to[bend right=30]
(obj27) to[bend right=30] (obj28) to[bend left=30] (obj26);
    \end{tikzpicture}
  \caption{A $4$-pseudo-rotationally invariant non-crossing
    partition
  in $\mNCDPlus[6][3]$}
\label{fig:25}
\end{figure}

\medskip
\noindent
{\sc Construction 1.} 
We start with a
$2\rRr $-pseudo-rotationally invariant positive
non-crossing partition of type~$B$ of $\{1,2,\dots,m(n-1),\overline
1,\overline2,\dots,\overline{m(n-1)}\}$,
all of whose blocks have size divisible by~$m$.
By \Cref{lem:allB}, this partition contains a zero block.
We add the elements 
$\{m(n-1)+1,\dots,mn,\overline{m(n-1)+1},\dots,\overline{mn}\}$
of the inner circle to this zero block, thereby creating a
$2\rRr $-pseudo-rotationally invariant 
positive $m$-divisible non-crossing partition of type~$D$.
We illustrate this construction in the case where $\rRr=2$, $m=3$, and $n=6$.
We start with the non-crossing partition 
\begin{multline*}
\{\{1,2,9,\overline1,\overline2,\overline9\},\{3,7,8\},\{4,5,6\},
\{10,14,15\},\{11,12,13\},\\
\{\overline3,\overline7,\overline8\},\{\overline4,\overline5,\overline6\},
\{\overline{10},\overline{14},\overline{15}\},
\{\overline{11},\overline{12},\overline{13}\}\}
\end{multline*}
of $\{1,2,\dots,\overline{15}\}$, 
which is indeed $4$-pseudo-rotationally invariant
in the above sense (see \Cref{fig:34}). Then one adds 
the elements $\{16,17,18,\overline{16},\overline{17},\overline{18}\}$
of the inner circle to the zero block; see \Cref{fig:25}.

\begin{figure}
  \centering
  \begin{tikzpicture}[scale=1]
    \polygon{(-4,0)}{obj}{24}{2.5}
      {1,2,3,4,5,6,7,8,9,10,11,12,\overline{1},\overline{2},\overline{3},\overline{4},\overline{5},\overline{6},\overline{7},\overline{8},\overline{9},\overline{10},\overline{11},\overline{12}}

    \polygoninner{(-4,0)}{objin}{12}{0.6}
      {,,\overline{14},,\overline{13},,15,,14,,13,,\overline{15}}

      \draw[line width=2.5pt,black] (obj1) to[bend left=50] (obj24);
      \draw[line width=2.5pt,black] (obj13) to[bend left=50] (obj12);

     \draw[fill=black,fill opacity=0.1] (obj3) to[bend right=30]
(obj4) to[bend right=30] (obj5) to[bend left=30] (obj3);
     \draw[fill=black,fill opacity=0.1] (obj2) to[bend right=30]
(obj6) to[bend right=30] (obj7) to[bend left=30] (obj2);
     \draw[fill=black,fill opacity=0.1] (obj1) to[bend right=10]
 (obj8) to[bend left=10] (obj24) to[bend right=30] (obj1);
     \draw[fill=black,fill opacity=0.1] (obj9) to[bend right=30]
(obj10) to[bend right=30] (obj11) to[bend left=30] (obj9);

     \draw[fill=black,fill opacity=0.1] (obj15) to[bend right=30]
(obj16) to[bend right=30] (obj17) to[bend left=30] (obj15);
     \draw[fill=black,fill opacity=0.1] (obj14) to[bend right=30]
(obj18) to[bend right=30] (obj19) to[bend left=30] (obj14);
     \draw[fill=black,fill opacity=0.1] (obj13) to[bend right=10]
 (obj20) to[bend left=10] (obj12) to[bend right=30] (obj13);
     \draw[fill=black,fill opacity=0.1] (obj21) to[bend right=30]
(obj22) to[bend right=30] (obj23) to[bend left=30] (obj21);

     \draw[fill=black,fill opacity=0.1] (objin9) to[bend left=40]
(objin12) to[bend left=40] (objin1) to[bend right=100,looseness=2]
(objin9);

     \draw[fill=black,fill opacity=0.1] (objin3) to[bend left=40]
(objin5) to[bend left=40] (objin7) to[bend right=100,looseness=2]
(objin3);

    \polygon{(4,0)}{obj}{24}{2.5}
      {1,2,3,4,5,6,7,8,9,10,11,12,\overline{1},\overline{2},\overline{3},\overline{4},\overline{5},\overline{6},\overline{7},\overline{8},\overline{9},\overline{10},\overline{11},\overline{12}}

    \polygoninner{(4,0)}{objin}{12}{0.6}
      {,,\overline{14},,\overline{13},,15,,14,,13,,\overline{15}}

      \draw[line width=2.5pt,black] (obj1) to[bend left=20] (obj21);
      \draw[line width=2.5pt,black] (obj13) to[bend left=20] (obj9);

     \draw[fill=black,fill opacity=0.1] (obj4) to[bend right=30]
(obj5) to[bend right=30] (obj6) to[bend left=30] (obj4);
     \draw[fill=black,fill opacity=0.1] (obj3) to[bend right=30]
(obj7) to[bend right=30] (obj8) to[bend left=30] (obj3);
     \draw[fill=black,fill opacity=0.1] (obj2) to[bend right=10]
(obj9) to[bend right=20] (obj13) to[bend right=20] (obj2);
     \draw[fill=black,fill opacity=0.1] (obj10) to[bend right=30]
(obj11) to[bend right=30] (obj12) to[bend left=30] (obj10);

     \draw[fill=black,fill opacity=0.1] (obj16) to[bend right=30]
(obj17) to[bend right=30] (obj18) to[bend left=30] (obj16);
     \draw[fill=black,fill opacity=0.1] (obj15) to[bend right=30]
(obj19) to[bend right=30] (obj20) to[bend left=30] (obj15);
     \draw[fill=black,fill opacity=0.1] (obj14) to[bend right=10]
(obj21) to[bend right=20] (obj1) to[bend right=20] (obj14);
     \draw[fill=black,fill opacity=0.1] (obj22) to[bend right=30]
(obj23) to[bend right=30] (obj24) to[bend left=30] (obj22);

     \draw[fill=black,fill opacity=0.1] (objin11) to[bend left=40]
(objin2) to[bend left=40] (objin3) to[bend right=100,looseness=2]
(objin11);

     \draw[fill=black,fill opacity=0.1] (objin5) to[bend left=40]
(objin7) to[bend left=40] (objin9) to[bend right=100,looseness=2]
(objin5);

    \end{tikzpicture}
  \caption{Two $2$-pseudo-rotationally invariant non-crossing
    partitions
  in $\mNCDPlus[5][3]$}
\label{fig:26}
\end{figure}
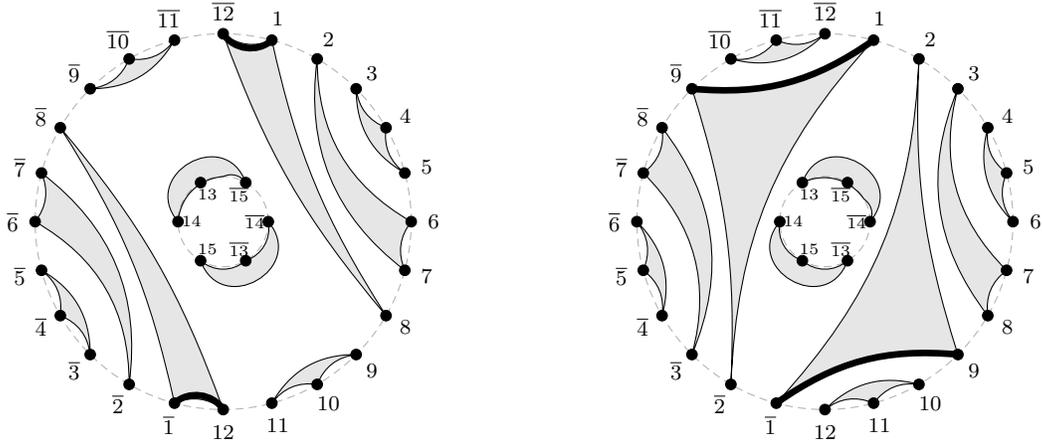

\medskip
\noindent
{\sc Construction 2.} 
We start with a positive
non-crossing partition of type~$B$ of $\{1,2,\dots,\break m(n-1),\overline
1,\overline2,\dots,\overline{m(n-1)}\}$ without zero block,
all of whose blocks have size divisible by~$m$. We remark that this
partition has order $m(n-1)-1$ as type~$B$ non-crossing partition; see~\Cref{lem:N-2B}.
We then add the elements $\{m(n-1)+1,\dots,mn,\overline{m(n-1)+1},\dots,\overline{mn}\}$
of the inner circle, which are partitioned into two blocks of size~$m$ according to
the rule~(D3) in \Cref{sec:realD}. Thereby we create a
positive $m$-divisible non-crossing partition of type~$D$ which has order
$m(n-1)-1$ (see \Cref{lem:N-2D}). 
Thus, this partition is $2$-pseudo-rotationally
invariant.
We illustrate this construction in the case where $m=3$ and $n=5$.
We start with the type~$B$ positive non-crossing partition 
\begin{equation*}
\{\{1,8,\overline{12}\},\{2,6,7\},\{3,4,5\},\{9,10,11\},
\{\overline1,\overline8,{12}\},\{\overline2,\overline6,\overline7\},
\{\overline3,\overline4,\overline5\},\{\overline9,\overline{10},\overline{11}\}\}
\end{equation*}
of $\{1,2,\dots,\overline{12}\}$ (see \Cref{fig:11b}).
Then we add the elements
$\{13,14,15,\overline{13},\overline{14},\overline{15}\}$ on the inner
circle, partitioned into the two blocks $\{\overline{15},13,14\}$
and $\{{15},\overline{13},\overline{14}\}$
according to Condition~(D3) in the definition of $\mNCD$. See
the left part of \Cref{fig:26}. On the right, the result of application
of~$\rotD$ (according to Case~(3) of \Cref{prop:2D}) is shown.

\begin{figure}
  \centering
  \begin{tikzpicture}[scale=1]
    \polygon{(-4,0)}{obj}{30}{2.5}
      {1,2,3,4,5,6,7,8,9,10,11,12,13,14,15,\overline{1},\overline{2},\overline{3},\overline{4},\overline{5},\overline{6},\overline{7},\overline{8},\overline{9},\overline{10},\overline{11},\overline{12},\overline{13},\overline{14},\overline{15}}

    \polygoninner{(-4,0)}{objin}{10}{0.9}
      {18,17,16,\overline{20},\overline{19},\overline{18},\overline{17},\overline{16},20,19}

      \draw[line width=2.5pt,black] (obj1) to[bend left=50] (obj30);
      \draw[line width=2.5pt,black] (obj16) to[bend left=50] (obj15);

     \draw[fill=black,fill opacity=0.1] (obj2) to[bend right=30]
(obj3) to[bend right=30]
(obj4) to[bend right=30] (obj5) to[bend right=30] (obj6) to[bend
left=30] (obj2);
     \draw[fill=black,fill opacity=0.1] (obj9) to[bend right=30]
(obj10) to[bend right=30]
(obj11) to[bend right=30] (obj12) to[bend right=30] (obj13) to[bend
left=30] (obj9);
     \draw[fill=black,fill opacity=0.1] (obj1) to[bend right=30]
(obj7) to[bend right=5] (objin1) to[bend right=100, looseness=2]
(objin10) to[bend right=10] (obj30) to[bend right=30] (obj1);
     \draw[fill=black,fill opacity=0.1] (obj8) to[bend right=30]
(obj14) to[bend right=5] (objin4) to[bend right=100, looseness=2]
(objin3) to[bend right=100, looseness=2] (objin2) to[bend left=10]
(obj8);

     \draw[fill=black,fill opacity=0.1] (obj17) to[bend right=30]
(obj18) to[bend right=30]
(obj19) to[bend right=30] (obj20) to[bend right=30] (obj21) to[bend
left=30] (obj17);
     \draw[fill=black,fill opacity=0.1] (obj24) to[bend right=30]
(obj25) to[bend right=30]
(obj26) to[bend right=30] (obj27) to[bend right=30] (obj28) to[bend
left=30] (obj24);
     \draw[fill=black,fill opacity=0.1] (obj16) to[bend right=30]
(obj22) to[bend right=5] (objin6) to[bend right=100, looseness=2]
(objin5) to[bend right=10] (obj15) to[bend right=30] (obj16);
     \draw[fill=black,fill opacity=0.1] (obj23) to[bend right=30]
(obj29) to[bend right=5] (objin9) to[bend right=100, looseness=2]
(objin8) to[bend right=100, looseness=2] (objin7) to[bend left=10]
(obj23);

    \end{tikzpicture}
  \caption{A $4$-pseudo-rotationally invariant non-crossing
    partition
  in $\mNCDPlus[4][5]$}
\label{fig:27}
\end{figure}
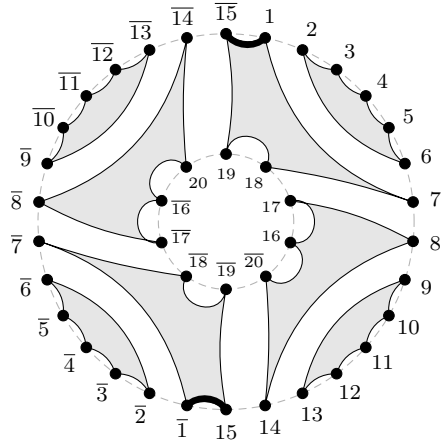

\medskip
\noindent
{\sc Construction 3.}
According to \Cref{lem:Const3} below, this construction only applies if $\rRr\mid n$
and $\rRr \mid(m+1)$.

We start with a block $B$ of size divisible by~$m$ 
that contains $\overline{m(n-1)}$ as only negative element on the
outer circle, and that contains~$1$ and possibly further elements;
if it contains 
elements from the inner circle then, according to rule~(D1) in 
\Cref{sec:realD}, among these there 
must be $mn-1$ or $\overline{mn-1}$). To this block, we add more
blocks of size divisible by~$m$ in a non-crossing fashion such that,
together, they form a non-crossing partition of successive elements
on the outer circle starting with
$\overline{m(n-1)},1, \dots$, and of~$x$ successive elements on the inner circle.
(According to rule~(D1), the latter must end in $mn-1$ or $\overline{mn-1}$
when read counter-clockwise.)
We call this (partial) non-crossing partition the \defn{fundamental seed}.
We place $\rRr -1$ \emph{modified\/} copies along the annulus, which,
together with the fundamental seed, cover the elements
$\{\overline{m(n-1)},1,2,\dots,m(n-1)-1\}$ on the outer circle and the
elements $\{\mp mn,\pm(m(n-1)+1),\pm(m(n-1)+2),\dots,\pm(mn-1)\}$
on the inner circle.
The modification consists in
deleting the ``first" element from $B$ on the outer circle
(that is, $\overline{m(n-1)}$) and replacing it by 
an element on
the inner circle. A copy of the obtained partition is then formed by
replacing $i$ by~$-i$, for all~$i$. Together, a
positive $m$-divisible non-crossing partition of type~$D$ which is
$\rRr $-pseudo-rotationally invariant is obtained.
We illustrate this construction in the case where $\rRr=2$, $m=5$, and $n=4$;
cf.\ \Cref{fig:27}.
We start with the block $\{\overline{15},1,7,18,19\}$ and augment it to
the partial non-crossing partition 
$$
\{\{\overline{15},1,7,18,19\},\{2,3,4,5,6\}\}.
$$
A modified copy of this fundamental seed is formed,
$$
\{\{8,14,\overline{20},16,17\},\{9,10,11,12,13\}\},
$$
and added to it. Finally, a copy of the obtained partition is formed by
replacing $i$ by~$-i$, for all~$i$. Together, a 
positive $m$-divisible non-crossing partition of type~$D$ is obtained;
see again \Cref{fig:27}. This non-crossing partition is indeed
$4$-pseudo-rotationally invariant, that is, it is invariant under
the action of~$\rotD^7$. 

\begin{lemma} \label{lem:Const3}
Let $\rRr\mid\big(m(n-1)-1\big)$.
Construction~{\em3} is only applicable if 
$\rRr\mid n$ and $\rRr \mid(m+1)$.
Moreover, the number of elements in the fundamental seed equals $mn/\rRr$.
\end{lemma}

\begin{proof}
The two conditions are in fact equivalent. Indeed, if $\rRr\mid n$, then
$\rRr\mid(m+1)$ follows from $\rRr\mid \big(m(n-1)-1\big)$. Conversely, 
if $\rRr\mid(m+1)$, then, again due to $\rRr\mid \big(m(n-1)-1\big)$,
we infer $\rRr\mid mn$. Since $\rRr$ and $m$ must be mutually prime because of
$\rRr\mid(m+1)$, it follows that $\rRr\mid n$.

In order to show that $\rRr\mid(m+1)$,
we write $x$ for the number of elements on the inner circle that are
contained in the fundamental seed according to Construction~3. 
Since the $\rRr-1$ modified copies
of the fundamental seed that are placed around the annulus in
Construction~3 have one more element on the inner circle than the
fundamental seed itself, and since the fundamental seed
together with its $\rRr-1$ copies covers one half of the inner circle, we have
$$
x+(\rRr-1)(x+1)=m,
$$
or, equivalently,
\begin{equation} \label{eq:rotrel1} 
\rRr(x+1)=m+1.
\end{equation}
This implies that $\rRr\mid (m+1)$. 

\medskip
The second part of the assertion of the lemma is obvious since the
fundamental seed together with its $\rRr-1$ copies covers half of
the annulus, that is, $mn$ elements.
\end{proof}

\begin{theorem} \label{lem:allD}
Let $m,n,\rrr$ be positive integers with $\rrr \ge2$. Then
all $\rrr $-pseudo-rotationally invariant positive non-crossing
partitions of\/ $\{1,2,\dots,mn,\overline1,\overline2,\dots,\overline{mn}\}$
of type~$D$ 
are obtained by starting with a non-crossing partition
obtained by Construction~{\em1}, {\em2}, or~{\em3},
and applying the pseudo-rotation~$\rotD$ repeatedly to it.

More precisely, a partition arising from Construction~{\em1} is 
$2\rRr$-pseudo-rotationally invariant. Furthermore, a partition arising
from Construction~{\em2} is $2$-pseudo-rotationally invariant.
Finally, a partition arising from Construction~{\em3} is
$2\rRr$-pseudo-rotationally invariant if $n/\rRr$ is even, and it is
$\rRr$-pseudo-rotationally invariant if $n/\rRr$ is odd.
\end{theorem}

The proof of this theorem is given in \Cref{app:inv-D}.

\subsection{Enumeration of pseudo-rotationally invariant elements}
\label{sec:rotenumD}

With the characterisation of pseudo-rotationally invariant elements
of $\mNCDPlus$ at hand, we can now embark on
the enumeration of such elements. 
\Cref{thm:enumD-1} gives the total number of elements of
$\mNCDPlus$ that arise from Constructions~1 and~2, while
\Cref{thm:enumD-2} gives the total number of elements of
$\mNCDPlus$ that arise from Constructions~3.
\Cref{thm:enumD-3} and \Cref{thm:enumD-4} provide the cardinalities
of the subsets of those elements all of whose block sizes are~$m$.

\begin{proposition} \label{thm:enumD-1}
The total number of elements of $\mNCDPlus$ that 
arise from Constructions~{\em1} and~{\em2} in \Cref{sec:rotD} equals
\begin{equation} \label{eq:C1-2a} 
\binom {(m+1)(n-1)-1}{n-1}.
\end{equation}
If $\rrr$ is even, then the
number of elements of $\mNCDPlus$ that are $\rrr$-pseudo-rotationally
invariant and arise from Constructions~{\em1} and~{\em2} in \Cref{sec:rotD} equals
\begin{equation} \label{eq:C1-2b} 
\binom {\fl{2((m+1)(n-1)-1)/\rrr }}{\fl{2(n-1)/\rrr }}.
\end{equation}
\end{proposition} 

\begin{proof}
As we already argued in \Cref{lem:N-2D} and \Cref{lem:allD},
the elements of\break $\mNCDPlus$ that are $\rrr$-pseudo-rotationally
invariant and arise from Constructions~1 and~2 (the latter is only
possible for $\rrr=1$ and $\rrr=2$) are in bijection with 
elements of $\mNCBPlus[n-1]$ that are $\rrr$-rotationally
invariant. This is seen by starting with an element of
$\mNCDPlus$ and deleting the inner circle, thereby obtaining
an element of $\mNCBPlus[n-1]$. Clearly, this map is reversible.
If $\rrr=1$, then, by~\eqref{eq:multichains-pos-B-1} with $n$
replaced by~$n-1$, this number equals~\eqref{eq:C1-2a}.
If $\rrr$ is even, then, by~\eqref{eq:multichains-pos-B} with $n$
replaced by~$n-1$ and $\nu=\rrr/2$, this number equals~\eqref{eq:C1-2b}.
\end{proof}

\begin{proposition} \label{thm:enumD-2}
The total number of elements of $\mNCDPlus$ that 
arise from Construction~{\em3} in \Cref{sec:rotD} equals
\begin{equation} \label{eq:enumConst3} 
2\binom {((m+1)(n-1)-\rRr)/\rRr}{n/\rRr}.
\end{equation}
\end{proposition} 

The proof of this proposition is postponed to \Cref{app:GF-D-inv}.

\begin{proposition} \label{thm:enumD-3}
The total number of elements of $\mNCDPlus$ that 
arise from Constructions~{\em1} and~{\em2} in \Cref{sec:rotD}
with all block sizes equal to~$m$ is
\begin{equation} \label{eq:C2-2a} 
\binom {m(n-1)-1}{n-1}.
\end{equation}
If $\rrr$ is even, then the
number of elements of $\mNCDPlus$ that are $\rrr$-pseudo-rotationally
invariant and arise from Constructions~{\em1} and~{\em2} in
\Cref{sec:rotD}, and all of whose blocks have size~$m$, equals
\begin{equation} \label{eq:C2-2b} 
\binom {\fl{2(m(n-1)-1)/\rrr }}{\fl{2(n-1)/\rrr }}.
\end{equation}
\end{proposition} 

\begin{proof}
As we argued before,
the elements of $\mNCDPlus$ that are $\rrr$-pseudo-rotationally
invariant and arise from Constructions~1 and~2 (the latter is only
possible for $\rrr=1$ and $\rrr=2$) are in bijection with 
elements of $\mNCBPlus[n-1]$ that are $\rrr$-rotationally
invariant. 
If $\rrr=1$, then, by~\eqref{eq:multichains-si-pos-B-1} with $b=n$ and
subsequent replacement of~$n$ by~$n-1$, this number equals~\eqref{eq:C2-2a}.
If $\rrr$ is even, then, by~\eqref{eq:multichains-si-pos-B} with $b=\fl{n/\nu}$
and subsequent replacement of~$n$
by~$n-1$ and setting $\nu=\rrr/2$, this number equals~\eqref{eq:C2-2b}.
\end{proof}

\begin{proposition} \label{thm:enumD-4}
The total number of elements of $\mNCDPlus$ that 
arise from Construction~{\em3} in \Cref{sec:rotD} 
and all of whose blocks have size~$m$ equals
\begin{equation} \label{eq:enumConst3-m} 
2\binom {(m(n-1)-1)/\rRr}{n/\rRr}.
\end{equation}
\end{proposition} 

The proof of this proposition is also postponed to \Cref{app:GF-D-inv}.

\subsection{Cyclic sieving}
\label{sec:sievD}

With the enumeration results in \Cref{sec:enumD,sec:rotenumD}, we are
now ready to derive our cyclic sieving results in type~$D$; see
\Cref{thm:3-D} below.
\Cref{lem:1-D,lem:2-D} prepare for the proof of this theorem.

\begin{lemma} \label{lem:1-D}
Let $m$ and $n$ be positive integers.
Then the expression
\begin{equation}
\frac {[2m(n-1)+n-2]_{q}} {[n]_{q}}
\begin{bmatrix} (m+1)(n-1)-1\\n-1\end{bmatrix}_{q^2}
\label{eq:SD-D}
\end{equation}
is a polynomial in $q$ with non-negative integer coefficients.
\end{lemma}

\begin{proof}
We have
\begin{multline*}
\frac {[2m(n-1)+n-2]_{q}} {[n]_{q}}
\begin{bmatrix} (m+1)(n-1)-1\\n-1\end{bmatrix}_{q^2}
=
\begin{bmatrix} (m+1)(n-1)-1\\n-1\end{bmatrix}_{q^2}\\
+q^{n}\left(1+q^n\right)
\begin{bmatrix} (m+1)(n-1)-1\\n\end{bmatrix}_{q^2}.
\end{multline*}
Since $q$-binomial coefficients are polynomials in~$q$ with non-negative
coefficients, the claim is obviously true.
\end{proof}

\begin{lemma} \label{lem:2-D}
Let $m$ and $n$ be integers.
Furthermore, let $\si$ be a positive integer such that $\si\mid (2m(n-1)-2)$.
Writing $\rrr =2(m(n-1)-1)/\si$ and 
$\om_\rrr =e^{2\pi i\si/(2m(n-1)-2)}=e^{2\pi i/\rrr }$, we have
\begin{align}
\notag
\frac {[2m(n-1)+n-2]_{q}} {[n]_{q}}
&\begin{bmatrix} (m+1)(n-1)-1\\n-1\end{bmatrix}_{q^2}
\Bigg\vert_{q=\om_\rrr }\\
&=\begin{cases} 
\frac {2m(n-1)+n-2} {n}
\binom {\fl{2((m+1)(n-1)-1)/\rrr }} {\fl{2(n-1)/\rrr }},&\text{if $\rrr$ is
  even and $\rrr\mid n$},\\
\binom {\fl{2((m+1)(n-1)-1)/\rrr }} {\fl{2(n-1)/\rrr }},&\text{if $\rrr$ is
  even and $\rrr\nmid n$},\\
\frac {2m(n-1)+n-2} {n}
\binom {\fl{((m+1)(n-1)-1)/\rrr }} {\fl{(n-1)/\rrr }},&\text{if $\rrr$ is
  odd and $\rrr\mid n$},\\
\binom {\fl{((m+1)(n-1)-1)/\rrr }} {\fl{(n-1)/\rrr }},&\text{if $\rrr$ is
  odd and $\rrr\nmid n$},
\end{cases}
\label{eq:SI-D}\\
\notag
\frac {[2m(n-1)-n]_{q}} {[n]_{q}}
&\begin{bmatrix} m(n-1)-1\\n-1\end{bmatrix}_{q^2}
\Bigg\vert_{q=\om_\rrr }\\
&=\begin{cases} 
\frac {(2m(n-1)-n)} {n}
\binom {m(n-1)-1}{n-1},&\text{if $\rrr=2$ and $n$ is even,}\\
\binom {m(n-1)-1}{n-1},
&\text{if $\rrr=2$ and $n$ is odd,}\\
2\binom {2(m(n-1)-1)/\rrr}{2n/\rrr},&\text{if $\rrr$ is even, $\rrr\ge4$,
and $\rrr\mid n$,}\\
\frac {2m(n-1)-n} {n}
\binom {{m(n-1)-1}} {n-1},&\text{if $\rrr=1$},\\
2\binom {{(m(n-1)-1)/\rrr }} {n/\rrr },&\text{if $\rrr$ is
  odd, $\rrr\ge3$, and $\rrr\mid n$},\\
0,&\text{otherwise.}
\end{cases}
\label{eq:SJ-D}
\end{align}
\end{lemma}

Due to its technical nature, the proof of this lemma is deferred
to the end of \Cref{app:unnec}.

\begin{theorem} \label{thm:3-D}
Let $m$ and $n$ be positive integers.
Furthermore, 
let $C$ be the cyclic group of pseudo-rotations of the annulus
with
$\{1,2,\dots,m(n-1),\overline1,\overline2,\dots,\overline{m(n-1)}\}$
on the outer circle and
$\{m(n-1)+1,\dots,mn,\overline{m(n-1)+1},\dots,\overline{mn}\}$
on the inner circle
generated by~$\rotD$, viewed as a group of order~$2m(n-1)-2$. Then
the triple $(M,P,C)$ exhibits the cyclic sieving phenomenon
for the following choices of sets $M$ and polynomials $P$:

\begin{enumerate}
\item[\em(1)]
$M=\mNCDPlus$, and
$$P(q)=\frac {[2m(n-1)+n-2]_{q}} {[n]_{q}}
\begin{bmatrix} (m+1)(n-1)-1\\n-1\end{bmatrix}_{q^2};$$
\item[\em(2)]
$M$ consists of the elements of
$\mNCDPlus$ all of whose blocks have size~$m$, and
$$P(q)=\frac {[2m(n-1)-n]_{q}} {[n]_{q}}
\begin{bmatrix} m(n-1)-1\\n-1\end{bmatrix}_{q^2}.$$
\end{enumerate}
\end{theorem}

\begin{proof}
The polynomials $P(q)$ in the assertion of the theorem are indeed
polynomials with non-negative coefficients due to \Cref{lem:1-D}.

Now, given a positive integer $\rrr$ with $\rrr\mid (2m(n-1)-2)$, we must show  
\begin{equation} \label{eq:P(q)-D} 
P(q)\big\vert_{q=\om_\rrr }=
\text{number of elements of $\mNCDPlus$ which are invariant under
  $\rotD^{(2m(n-1)-2)/\rrr}$},
\end{equation}
where $\om_\rrr=e^{2\pi i/\rrr}$.

\medskip
We begin with Item~(1).
Let first $\rrr$ be even. 
By~\eqref{eq:C1-2b}, there are
\begin{equation} \label{eq:Constr1-2A} 
\binom {\fl{2((m+1)(n-1)-1)/\rrr }}{\fl{2(n-1)/\rrr }}
\end{equation}
elements of $\mNCDPlus$ that are $\rrr$-pseudo-rotationally
invariant and do not contain a bridging block (and hence arise from
Constructions~1 and~2 in \Cref{sec:rotD}). 

Moreover, by \Cref{lem:allD}, there exist elements 
of $\mNCDPlus$ that are $\rrr$-pseudo-rotationally
invariant and \emph{do} contain a bridging block (and hence arise from
Construction~3 in \Cref{sec:rotD}) only if $n/(\rrr/2)$ is even or if
$n/\rrr$ is odd. In both cases, it follows that $\rrr\mid n$, and,
due to \Cref{thm:enumD-2} with $\rRr=\rrr/2$, there are then
\begin{equation} \label{eq:Constr1-2B} 
2\binom {(2(m+1)(n-1)-\rrr)/\rrr}{2n/\rrr}
\end{equation}
such elements of $\mNCDPlus$ that arise from Construction~3.
The sum of~\eqref{eq:Constr1-2A} and~\eqref{eq:Constr1-2B} equals the
first alternative on the right-hand side of~\eqref{eq:SI-D}, which confirms~\eqref{eq:P(q)-D} in this case.

On the other hand, if $\rrr\nmid n$, then there are no
$\rrr$-pseudo-rotationally invariant elements of $\mNCDPlus$
that arise from Construction~3. Consequently, the expression in~\eqref{eq:Constr1-2A} is already equal to the total number of
$\rrr$-pseudo-rotationally invariant elements, which agrees with the
second alternative on the right-hand side of~\eqref{eq:SI-D}.
This confirms~\eqref{eq:P(q)-D} in this other case.

\medskip
Now let $\rrr$ be odd. 
By an argument that is analogous to the one
around~\eqref{eq:rotB/2}, 
we infer that
a $\rrr$-pseudo-rotationally invariant element of $\mNCDPlus$
without a bridging block (and thus arising from Construction~1 or~2 in
\Cref{sec:rotD})
is automatically also $2\rrr$-pseudo-rotationally invariant.
By~\eqref{eq:C1-2b} with $\rrr$ replaced by~$2\rrr$, there are
\begin{equation} \label{eq:Constr1-2C} 
\binom {\fl{((m+1)(n-1)-1)/\rrr }}{\fl{(n-1)/\rrr }}
\end{equation}
elements of $\mNCDPlus$ that are $2\rrr$-pseudo-rotationally
invariant (and also $\rrr$-pseudo-rota\-tionally invariant) 
and do not contain a bridging block.

Moreover, by \Cref{lem:allD}, there exist elements 
of $\mNCDPlus$ that are $\rrr$-pseudo-rotationally
invariant and \emph{do} contain a bridging block (and hence arise from
Construction~3 in \Cref{sec:rotD}) only if $\rrr\mid n$, and,
due to \Cref{thm:enumD-2} with $\rRr=\rrr$, there are then
\begin{equation} \label{eq:Constr1-2D} 
2\binom {((m+1)(n-1)-\rrr)/\rrr}{n/\rrr}
\end{equation}
such elements of $\mNCDPlus$ that arise from Construction~3.
The sum of~\eqref{eq:Constr1-2C} and~\eqref{eq:Constr1-2D} equals the
third alternative on the right-hand side of~\eqref{eq:SI-D}, which confirms~\eqref{eq:P(q)-D} in this case.

On the other hand, if $\rrr\nmid n$, then there are no
$\rrr$-pseudo-rotationally invariant elements of $\mNCDPlus$
that arise from Construction~3. Consequently, the expression in~\eqref{eq:Constr1-2C} is already equal to the total number of
$\rrr$-pseudo-rotationally invariant elements, which agrees with the
fourth alternative on the right-hand side of~\eqref{eq:SI-D}.
This confirms~\eqref{eq:P(q)-D} in this other case.

\medskip
We turn to Item (2). We will be brief here since the arguments 
are very similar to those above addressing Item~(1).
Again, let first $\rrr$ be even.

If $\rrr=2$ and $n$ is even, then the assertion follows immediately from~\eqref{eq:C2-2b} with $\rrr=2$, and from \Cref{lem:allD} and
\Cref{thm:enumD-4} with $\rRr=1$
and comparison with the first
alternative on the right-hand side of~\eqref{eq:SJ-D}.

If $\rrr=2$ and $n$ is odd, then, according to \Cref{lem:allD} with $\rRr=1$, 
elements of $\mNCDPlus$ that are $2$-pseudo-rotationally
invariant and all of whose blocks have size~$m$ cannot arise from
Construction~3.
The assertion follows from~\eqref{eq:C2-2b} with $\rrr=2$ and comparison with the second
alternative on the right-hand side of~\eqref{eq:SJ-D}.

If $\rrr$ is even and $\rrr\ge4$, then an element of $\mNCDPlus$
which is $\rrr $-pseudo-rotationally invariant and of which all blocks
have size~$m$ cannot arise from Construction~1 since a zero block
must contain all elements of the inner circle and therefore has more
than $2m$~elements. It also cannot arise from Construction~2 since,
according to \Cref{lem:allD}, these elements are
$2$-pseudo-rotationally invariant (but not pseudo-rotationally
invariant of higher order). 
Consequently, such an element must arise from Construction~3, either with
$\rRr=\rrr/2$ and where $n/(\rrr/2)$ is even, or with $\rRr=\rrr$ and 
where $n/\rrr$ is odd. In both cases, it follows that $\rrr\mid n$.
The number of these elements is then given by
\Cref{thm:enumD-4} with $\rRr=\rrr/2$ which agrees with the third
alternative on the right-hand side of~\eqref{eq:SJ-D}.

\medskip
Now let $\rrr$ be odd. If $\rrr=1$, comparison of~\eqref{eq:C2-2a} 
and the fourth alternative on the right-hand side of~\eqref{eq:SJ-D} 
confirms~\eqref{eq:P(q)-D} in this case.

On the other hand, if $\rrr\ge3$, then, as above in the case of even
$\rrr\ge4$, $\rrr $-pseudo-rotationally invariant elements of $\mNCDPlus$
all of whose blocks
have size~$m$ cannot arise from Construction~1 or~2.
Consequently, such elements must arise from Construction~3
with $\rRr=\rrr$. 
The number of these elements is given by
\Cref{thm:enumD-4} with $\rRr=\rrr$, which agrees with the fifth
alternative on the right-hand side of~\eqref{eq:SJ-D}.

\medskip
In all other cases, there are no positive $m$-divisible non-crossing
partitions which are $\rrr $-pseudo-rotationally invariant.

\medskip
This completes the proof of the theorem.
\end{proof}

\begin{corollary} \label{cor:CS-D}
  The conclusion of \Cref{thm:CS} holds for type $D_n$.
\end{corollary}

\begin{proof}
The cyclic sieving phenomenon for
  $\Big(\mNCPlus[D_n][m+1],\ \mCatplus[m+1](D_n;q),\ C\Big)$
follows directly from \Cref{thm:3-D}(1) with~$m$ replaced by~$m+1$.

The image of the
embedding of $\mNCPlus[D_n][m]$
in $\mNCPlus[D_n][m+1]$ as described in
\Cref{prop:PositiveKrewAlt} is given by all tuples
$(w_0,w_1,\dots,w_{m+1})$ with $w_0=\ep$. By
\Cref{rem:id-D}, these tuples correspond to non-crossing set partitions
in $\mNCD[n][m+1]$ all of whose non-zero blocks have size~$m+1$. 
The cyclic sieving phenomenon for
$\Big(\mNCPlus[D_n][m],\ \mCatplus(D_n;q),\ \widetilde C\Big)$
thus follows from \Cref{thm:3-D}(2) with $m$ replaced by~$m+1$.
\end{proof}

\section[Positive \texorpdfstring{$m$}{m}-divisible non-crossing partitions in dihedral and exceptional types]{Positive \texorpdfstring{$m$}{m}-divisible non-crossing partitions in dihedral\\ and exceptional types}
\label{sec:typeExc}

In this section, we study the action of the two positive Kreweras maps in the dihedral and the exceptional types.
The main tool will be a general approach for counting (positive) $m$-divisible non-crossing partitions that are invariant under powers of the positive Kreweras maps.
We then apply this approach to each type individually to derive the proposed cyclic sieving phenomena.

\medskip

It is well known that any two Coxeter elements~$c$ and~$c'$ are conjugate, and that this conjugation induces a bijection between $\mNCPlus$ and ${\NCsymb^{(m)}_+(W,c')}$.
This bijection is moreover $\Krewplusbar$-equivariant.
(Recall that $\Krewplusbar=\Krewplusbar^{(m)}$ has been
defined in \Cref{sec:poskrewmap2}.)
It therefore suffices to prove the proposed cyclic sieving phenomena in \Cref{thm:CS} for the \defn{bipartite Coxeter element}~$c = c_A c_B$ given by the decomposition $\reflS = A \sqcup B$ such that all simple generators in~$A$ pairwise commute, all simple generators in~$B$ pairwise commute, and where we write $c_A$ and $c_B$ for the
products of the generators in~$A$ and in~$B$.
This has the computational advantage that we have
\begin{equation}
  \wwo(\c) =  \underbrace{c_A\ c_B\ c_A\ \cdots}_h. \label{eq:wocword}
\end{equation}
That is, $\wwo(\c)$ is formed, as a word, of~$h$ alternating copies of~$\c_A$ and~$\c_B$ where we recall that~$h$ is the Coxeter number.
We thus do not need to consider the situation ``up to commutations''.
The word $\invc^m\invc^L$ then consists of~$m+1$ concatenated copies of~$\invc$ with the last~$n$ letters removed, and this word equals
\begin{equation}
  \invc^m\invc^L =  \underbrace{c_A\ c_B\ c_A\ \cdots}_{(m+1)h-2}.
\end{equation}
We refer to \Cref{ex:biginvexample} for several examples and note that $\invc^{m-1}\invc^L$ consists of $mh-2$ alternating copies of~$\c_A$ and~$\c_B$.

This additional regularity (that is in general missing for arbitrary Coxeter elements) allows us to describe subwords that are invariant under powers of the positive Kreweras maps.

\medskip

Let $(m+1)h-2 = k \cdot \rrr$, where we assume~$\rrr$ to be even if~$\psi \equiv\one$.
(The situation $mh-2 = k \cdot \rrr$ is obtained by simply replacing~$m$ by~$m-1$.)
We always consider the application of $(\Krewplusbar)^k$ which is a $(\rrr/2)$-fold symmetry if $\psi\equiv\one$ and an~$\rrr$-fold symmetry if $\psi\not\equiv\one$ (except in dimension~$2$ with $m=1$, compare \Cref{thm:order}).
It follows that the orbit of a letter~$\r$ inside $\invc^m\invc^L$ contains $\rrr/2$ letters if and only if either $\psi\equiv\one$ or~$\rrr$ is even with~$\psi(r) = r$.
This yields the following first important restriction concerning the symmetries we need to consider.

\begin{lemma} \label{lem:possiblefoldings}
  Let $\r_1\cdots\r_{n'}$ be a subword of\/ $\invc^m\invc^L$ which is invariant under the application of $(\Krewplusbar)^k$.
  Then $\rrr$ divides~$2n'$.
\end{lemma}
\begin{proof}
  As the subword is invariant when applying $(\Krewplusbar)^k$, it must be the union of orbits of letters of this operation.
  The statement follows as we have observed that every orbit contains either $\rrr/2$ or $\rrr$ many letters.
\end{proof}

When counting positive $m$-divisible non-crossing partitions that are invariant under powers of either of the two positive Kreweras maps, we only consider subwords of $\invc^m\invc^L$ that are reduced factorisations of non-crossing partitions.
We therefore only consider subwords of length at most~$n$.
In other words, we only need to consider parameters~$\rrr$ dividing $2n'$ for some $n' = \lenR(w) \leq n$ with $w \in \NC$.

\medskip

After we have seen that we need to understand only a small number of symmetries, we next discuss how to find the appropriate numerical parameters for the situation that one considers for a $\rrr$- or $(\rrr/2)$-fold symmetry, respectively.
To this end, we make use of the following lemma, explaining for which parameters~$m$ we have that~$\rrr$ divides~$(m+1)h-2$.

\begin{lemma}
\label{lem:numbercrunching}
  Let $\alpha,\beta,\gamma$ be coprime natural numbers, that is, $\gcd(\alpha,\beta,\gamma) = 1$, and let $\mathcal{Y}$ be the set of integers~$Y$ such that $\alpha$ divides $\beta\cdot Y + \gamma$.
  Then
  \begin{enumerate}
    \item[\em(1)] $\mathcal{Y} = \emptyset$ if $\gcd(\alpha,\beta)>1$, and
    \item[\em(2)] $\mathcal{Y} = \{ Y \mid Y \equiv Y_0\pmod{\alpha} \}$ for a fixed $Y_0$ if $\gcd(\alpha,\beta)=1$.
  \end{enumerate}
\end{lemma}
\begin{proof}
  Suppose $Y$ is an integer such that~$\alpha$ divides $\beta\cdot Y + \gamma$, or, phrased differently, $\beta\cdot Y + \gamma\equiv0\pmod\alpha$.

  If $\gcd(\alpha,\beta)>1$, then that common divisor of $\alpha$ and~$\beta$ does not divide $\gamma$ (by coprimeness of all three), so no solution exists.
  Otherwise $\beta$ has a multiplicative inverse modulo~$\alpha$, say $\beta^{-1}$ with $\beta^{-1}\beta\equiv1\pmod\alpha$.
  Multiplication of both sides of the congruence by $\beta^{-1}$ gives
  $Y \equiv -\beta^{-1}\gamma\pmod\alpha$,
  and hence all solutions form the single residue class
  $\{Y_0 + k\alpha\mid k\in\mathbb{Z}\}$, with $Y_0\equiv-\beta^{-1}\gamma\pmod\alpha$.
\end{proof}

One now applies this lemma to the situation $mh+(h-2) = k\cdot\rrr$, where we want to determine all possible parameters~$m$ for a given parameter~$\rrr$.
First, we set
\[
  y = \begin{cases}
        2, &\text{ if~$\rrr$ and~$h$ are even}, \\
        1, &\text{ otherwise.}
      \end{cases}
\]
\Cref{lem:numbercrunching} then says that there exists an~$m$ such that~$\rrr$ divides $hm+(h-2)$ if and only if~$\rrr/y$ and~$h/y$ are coprime.
In this case,~$m$ is given by the fixed 
residue class~$m'$ modulo $\rrr/y$, $0 \leq m' < \rrr/y$, given by
\begin{equation}
  m' = \tfrac{2-h}{y}(h/y)^{-1} \pmod{\rrr/y},
\end{equation}
and we moreover obtain
\begin{equation}
  \tfrac{y}{\rrr}\big((m+1)h-2\big) = ah + \tfrac{y}{\rrr}\big((m'+1)h-2\big)
\end{equation}
for the parameter~$a$ given by
\[
  m = a\tfrac{\rrr}{y} + m',
\]
and integral $\frac{y}{\rrr}\big((m'+1)h-2\big)$.
We thus define in this case the reflection~$\rfix$ inside $\invc$ to be the letter in position
\begin{equation}
  \ifix = \tfrac{n}{2}\cdot\tfrac{y}{\rrr}\big((m'+1)h-2\big)+1, \label{eq:ifix}
\end{equation}
and $\invc^{L'}$ to be the subword of $\invc$ consisting of the first $\ifix-1$ letters (so that $\rfix$ is the first letter in $\invc$ not contained in $\invc^{L'}$).
This yields the following lemma.

\begin{lemma}
\label{lem:missingletter}
  We have:
  \begin{enumerate}
    \item[\em(i)] if~$h$ is even, the length of $\invc^{L'}$ is a multiple of~$n$;
    \item[\em(ii)] if~$h$ is odd, the length of $\invc^{L'}$ is a multiple of~$n/2$.
  \end{enumerate}
\end{lemma}
In the second case in the lemma, we recall in \Cref{lem:comphn} that~$n$ is indeed even.

\begin{proof}[Proof of \Cref{lem:missingletter}]
  If~$h$ is even, we consider the two situations that $\rrr$ is even or odd.
  If~$\rrr$ is even then $y=2$ and
  \[
    \tfrac{n}{2}\cdot\tfrac{y}{\rrr}\big((m'+1)h-2\big) = n \cdot \tfrac{1}{\rrr}\big((m'+1)h-2\big),
  \]
  and~$\rrr$ divides $\big((m'+1)h-2\big)$.
  If~$\rrr$ is odd then $y=1$ and
  \[
    \tfrac{n}{2}\cdot\tfrac{y}{\rrr}\big((m'+1)h-2\big) = n \cdot \tfrac{1}{\rrr}\big((m'+1)h/2-1\big),
  \]
  and~$\rrr$ divides $\big((m'+1)h/2-1\big)$.

  If~$h$ is odd, then $y=1$ and
  \[
    \tfrac{n}{2}\cdot\tfrac{y}{\rrr}\big((m'+1)h-2\big) = \tfrac{n}{2} \cdot \tfrac{1}{\rrr}\big((m'+1)h-2\big)
  \]
  and~$\rrr$ divides $\big((m'+1)h-2\big)$.
\end{proof}

We refer the reader to \Cref{ex:biginvexample} were we give several examples of the above computations.
Next, we use this description of~$m$ as a residue class modulo~$\rrr/y$ to obtain another 
important restriction.
Let~$\r$ be a letter in $\invc^m\invc^L$ and let $\r_1\cdots\r_{\rrr'}$ be the subword of $\invc^m\invc^L$ obtained from~$\r$ under the application of~$(\Krewplusbar)^k$ (in particular, $\rrr' \in \{\rrr/2,\rrr\}$).
We then have the following lemma.

\begin{lemma}
\label{lem:sequenceindependence}
  We have that
  \begin{enumerate}
    \item[\em(1)] the sequence $r_1,r_2,\ldots,r_{\rrr'}$ of reflections in the subword $\r_1\cdots\r_{\rrr'}$, and
    \item[\em(2)] the order in which the letters $\r_1,\r_2,\ldots,\r_{\rrr'}$ are visited when applying $(\Krewplusbar)^k$
  \end{enumerate}
   are independent of the parameter~$m$.
\end{lemma}
\begin{proof}
  The above considerations imply that both the sequence $r_1,r_2,\ldots,r_{\rrr'}$ of reflections and the order in which the letters are visited do not change when passing between the parameters $m = a\frac{\rrr}{y}+m'$ and $m+\frac{\rrr}{y} = m + (a+1)\frac{\rrr}{y}+m'$.
\end{proof}

From now on, we distinguish several disjoint cases which in total cover all possible situations:
\begin{enumerate}
  \item The Coxeter number~$h$ is odd. \label{case:hodd}
  \item The Coxeter number~$h$ is even and \label{case:heven}
  \begin{enumerate}
    \item $\psi \equiv \one$, that is, $\wo s \wo = s$ for all $s \in \reflS$, \label{case:psione}
    \item $\psi \not\equiv \one$ and $\rrr$ is even,\label{case:psinotoneeven}
    \item $\psi \not\equiv \one$ and $\rrr$ is odd.\label{case:psinotoneodd}
  \end{enumerate}
\end{enumerate}

Here, we observe that it is well-known that $\psi \equiv \one$ implies that~$h$ is even, as seen in the following auxiliary lemma.

\begin{lemma}
\label{lem:comphn}
  We have $h$ is odd if and only if $c_B = \psi(c_A)$, and~$h$ is even if and only if $\psi(c_A) = c_A, \psi(c_B) = c_B$.
  In particular,~$n$ is even if~$h$ is odd.
\end{lemma}
\begin{proof}
  This follows from \Cref{eq:wocword} together with the observation that~$c_A$,~$c_B$, and~$\wo$ are involutions.
  As $\psi(c_A) = c_B$ implies that $|A| = |B|$ and $\reflS = A \sqcup B$, we deduce (the well-known property) that~$n = |\reflS|$ is even in this case.
\end{proof}

This lemma implies that
\begin{equation}
  (\c_A\c_B)^h = \wwo(\c)\psi(\wwo(\c)), \label{eq:cAcB}
\end{equation}
without commutations in Case~\eqref{case:psione}, respectively in
Case~\eqref{case:hodd} with the additional assumptions that $\c_B =
\psi(\c_A)$ and $\c = \c_A\c_B$.
On the other hand,~\eqref{eq:cAcB} cannot hold in
Cases~\eqref{case:psinotoneeven} and~\eqref{case:psinotoneodd} as in
both we have $c_A = \psi(c_A)$ inside~$W$ but $\c_A \neq \psi(\c_A)$
as words.

\medskip

Before continuing, we look at explicit examples of all these four cases.

\begin{example}
\label{ex:biginvexample}
  To visualise the previous analysis, we consider all four situations.
  \begin{enumerate}[(i)]
    \item In type~$A_4$, we are in Case~\eqref{case:hodd}.
      In this situation with $m=3$, we have
      \[
        \invc^3\invc^L \longleftrightarrow \underset{\bf 1}{\c_A}\ \underset{12}{\c_B}\ \underset{5}{\c_A}\ \underset{\bf 16}{\c_B}\ \underset{9}{\c_A}\ |\ \underset{2}{\c_B}\ \underset{\bf 13}{\c_A}\ \underset{6}{\c_B}\ \underset{17}{\c_A}\ \underset{\bf 10}{\c_B}\ |\ \underset{3}{\c_A}\ \underset{14}{\c_B}\ \underset{\bf 7}{\c_A}\ \underset{18}{\c_B}\ \underset{11}{\c_A}\ |\ \underset{\bf 4}{\c_B}\ \underset{15}{\c_A}\ \underset{8}{\c_B}
      \]
      where the numbers below all copies of~$\c_A$ and of~$\c_B$ indicate the order in which the orbit of one letter in the first copy of~$\c_A$ is visited by applying~$\Krewplusbar$.
      One can thus consider $k \in \{1,2,3,6,9,18\}$ as $(m+1)h-2 = 18$.
      For example, the numbers $1,16,13,10,7,4$ (highlighted above) sit in
      every $k$\th\ alternating copy of~$\c_A$ and~$\c_B$, here with $k=3$.

    \item In type~$B_3$, we are in Case~\eqref{case:psione}.
      In this situation with $m=2$, we have
      \[
        \invc^3\invc^L \longleftrightarrow \underset{\bf 1}{\c_A}\ \c_B\ \underset{4}{\c_A}\ \c_B\ \underset{\bf 7}{\c_A}\ \c_B\ |\ \underset{2}{\c_A}\ \c_B\ \underset{\bf 5}{\c_A}\ \c_B\ \underset{8}{\c_A}\ \c_B\ |\ \underset{\bf 3}{\c_A}\ \c_B\ \underset{6}{\c_A}\ \c_B
      \]
      where the numbers below all copies of~$\c_A$ indicate the order in which the orbit of one letter in the first copy of~$\c_A$ is visited by applying $\Krewplusbar$.
      One can thus consider $k \in \{1,2,4\}$ as $(m+1)h-2 = 8$.
      For example, the numbers $1,7,5,3$ sit in every $(2k)$\th\ alternating copy
      of~$\c_A$ and~$\c_B$, here with $k=2$.

      \item\label{ex:biginvexampleiii} In type~$A_3$, we are in Case~\eqref{case:psinotoneeven} or in Case~\eqref{case:psinotoneodd}.
      In this situation with $m=4$, we have
      \begin{align*}
        \invc^3\invc^L
        &\longleftrightarrow \underset{\substack{\bf 1 \\ \bf 10}}{\c_A}\ \c_B \ \underset{\substack{6 \\ 15}}{\c_A}\ \c_B \ |\ \underset{\substack{2 \\ 11}}{\c_A}\ \c_B \ \underset{\substack{\bf 7 \\ \bf 16}}{\c_A}\ \c_B \ |\ \underset{\substack{3 \\ 12}}{\c_A}\ \c_B \ \underset{\substack{8 \\ 17}}{\c_A}\ \c_B \ |\ \underset{\substack{\bf 4 \\ \bf 13}}{\c_A}\ \c_B \ \underset{\substack{9 \\ 18}}{\c_A}\ \c_B \ |\ \underset{\substack{5 \\ 14}}{\c_A}\ \c_B \\
        &\longleftrightarrow \underset{\substack{\bf 1 \\ 10}}{\c_A}\ \c_B \ \underset{\substack{6 \\ \bf 15}}{\c_A}\ \c_B \ |\ \underset{\substack{2 \\ \bf 11}}{\c_A}\ \c_B \ \underset{\substack{\bf 7 \\ 16}}{\c_A}\ \c_B \ |\ \underset{\substack{\bf 3 \\ 12}}{\c_A}\ \c_B \ \underset{\substack{8 \\ \bf 17}}{\c_A}\ \c_B \ |\ \underset{\substack{4 \\ \bf 13}}{\c_A}\ \c_B \ \underset{\substack{\bf 9 \\ 18}}{\c_A}\ \c_B \ |\ \underset{\substack{\bf 5 \\ 14}}{\c_A}\ \c_B
      \end{align*}
      where the numbers below all copies of~$\c_A$ indicate the order in which the orbit of one letter in the first copy of~$\c_A$ is visited by applying $\Krewplusbar$, with the top letter number indicating this letter, and the bottom number indicating the letter after applying the involution~$\psi$.
      One can thus consider $k \in \{1,2,3,6,9,18\}$ as $(m+1)h-2 = 18$.
      Thus, we are in Case~\eqref{case:psinotoneeven} for $k \in \{1,3,9\}$ and in Case~\eqref{case:psinotoneodd} for $k \in \{2,6,18\}$.
      For example, the numbers $1,10,7,16,4,13$ come in pairs $\big(\r,\psi(\r)\big)$, and the pairs sit in every $(2k)$\th\ alternating copy of~$\c_A$ and~$\c_B$, here with $k=3$.
      Similarly, the numbers $1,15,11,7,3,17,13,9,5$ sit in every $k$\th\ alternating copy of~$\c_A$ and~$\c_B$, here with $k=2$.
      Moreover, this word is travelled by~$(\Krewplusbar)^k$ in the order
      \[
        1,8,6,4,2,9,7,5,3,
      \]
      so the letters in positions $1,4,5,8,9$ are copies of~$\r$ and the letters in positions $2,3,6,7$ are copies of~$\psi(\r)$, as desired.
      The other possible~$k$'s are analogous.
\end{enumerate}
\end{example}

Given the various cases, we can directly further describe the orbits in \Cref{lem:sequenceindependence}.

\begin{lemma}
\label{lem:sequenceindependence2}
  In the situation of \Cref{lem:sequenceindependence}, the order in which letters are visited by~$\Krewplusbar$ is given by cyclically visiting
  \begin{enumerate}
    \setlength\itemindent{30pt}
    \item[\sc Case~\eqref{case:hodd}:]
       every~$h$\th\ letter,
    \item[\sc Case~\eqref{case:psione}:]
       every~$(h/2)$\th\ letter,
    \item[\sc Case~\eqref{case:psinotoneeven}:] 
       every~$(h/2)$\th\ letter if $\psi(r) = r$ or consecutive pair of letters if $\psi(r) \neq r$,
    \item[\sc Case~\eqref{case:psinotoneodd}:] 
       every~$h$\th\ letter.
  \end{enumerate}
  Moreover, in
    \begin{enumerate}
    \setlength\itemindent{100pt}
    \item[\sc Case~\eqref{case:psinotoneeven} with $\psi(r) \neq r$:]
       first, the $h/2$ letters $\r_1,\r_3,\ldots,\r_{h-1}$ are visited, and then the $h/2$ letters $\r_2,\r_4,\ldots,\r_h$ are visited, in this order, where $\psi(r_{2k-1}) = r_{2k}$ for $1 \leq k \leq h/2$,
    \item[\sc Case~\eqref{case:psinotoneodd} with $\psi(r) = r$:]
       the $h/2$ letters $\r_1,\r_2,\ldots,\r_{h/2}$ are visited, in this order.
  \end{enumerate}
\end{lemma}
\begin{proof}
  This is a direct consequence of the way the letters in $\invc^m\invc^L$ are visited under the application of~$\Krewplusbar$.
  In particular observe that we have seen in the discussion after \Cref{lem:numbercrunching} that~$h$ respectively $h/2$ are indeed coprime with $\rrr'$.
  This can be easily rechecked in all cases in \Cref{ex:biginvexample}.
  Observe in type~$A_3$ --- which corresponds to Case~\eqref{case:psinotoneeven} --- that the word is $1,10,7,16,4,13$, and the involution~$\psi$ combines $1,10$, $7,16$, and $4,13$, respectively.
  Since~$h/2 = 2$, the order is $1,4,7$ followed by $10,13,16$, as expected.
\end{proof}

The following lemma describes the orbits of a letter inside $\invc^m\invc^L$ as subwords of $(m+1)h-2$ alternating copies of~$\c_A$ and~$\c_B$ in the four cases.

\begin{lemma}\label{lem:invariantsitting}
  Let~$\r_1,\dots,\r_{\rrr'}$ be a subword of\/ $\invc^m\invc^L$.
  Then this subword is an orbit of the action of $(\Krewplusbar)^k$ if and only if
  \begin{enumerate}
    \setlength\itemindent{30pt}
    \item[\sc Case~\eqref{case:hodd}:]
      its length~$\rrr'$ equals~$\rrr$ and it is invariant under the cyclic shift sending a letter~$\s$ in the $j$\th\ alternating copy of~$\c_A$ or~$\c_B$ to the letter $\s$ {\em(}if~$k$ is even{\em)} or $\psi(\s)$ {\em(}if~$k$ is odd{\em)} in the $(j+k)$\th\ alternating copy of~$\c_A$ or~$\c_B$;
    \item[\sc Case~\eqref{case:psione}:]
      its length~$\rrr'$ equals~$\rrr/2$ and it is invariant under the cyclic shift sending a letter~$\s$ in the $j$\th\ alternating copy of~$\c_A$ or~$\c_B$ to the letter $\s$ in the $(j+2k)$\th\ alternating copy of~$\c_A$ or~$\c_B$;
    \item[\sc Case~\eqref{case:psinotoneeven}:] 
      its length~$\rrr'$ equals~$\rrr$ if $\psi(\r_1) \neq \r_1$ or $\rrr' = \rrr/2$ if $\psi(\r_1) = \r_1$, and it is invariant under the cyclic shift sending a letter~$\s$ in the $j$\th\ alternating copy of~$\c_A$ or~$\c_B$ to the letter $\s$ in the $(j+2k)$\th\ alternating copy of~$\c_A$ or~$\c_B$ and under the involution sending this letter~$\s$ to the letter~$\psi(\s)$ in the same copy of~$\c_A$ or of~$\c_B$;
    \item[\sc Case~\eqref{case:psinotoneodd}:] 
      its length~$\rrr'$ equals~$\rrr$ and it is invariant under the cyclic shift sending a letter~$\s$ in the $j$\th\ alternating copy of~$\c_A$ or~$\c_B$ to the letter $\s$ in the $(j+k)$\th\ alternating copy of~$\c_A$ or~$\c_B$ after the following replacement: cyclically visit every $h$\th\ letter starting with the leftmost.
      Leave the first $(\rrr+1)/2$ unchanged while replace the later $(\rrr-1)/2$ by $\psi(\r)$ in the same alternating copy of~$\c_A$ or~$\c_B$.
      Finally replace every letter smaller~$\r$ with $\r \lec r_1$ by $\psi(\r)$ in the same alternating copy of~$\c_A$ or~$\c_B$.
  \end{enumerate}
\end{lemma}
\begin{proof}
  Recall that $\invc^m\invc^L$ is given by $(m+1)h-2$ alternating copies of~$\c_A$ and~$\c_B$.
  Let now~$\r$ be a letter of $\invc^m\invc^L$ corresponding to some letter~$\s$ inside some copy of~$\c_A$.
  (The case that $\s$ is inside some copy of~$\c_B$ is completely analogous.)
  As we have seen in \Cref{sec:poskrewmap2}, the orbit of~$\r$ under $\Krewplusbar$ is given in this case by all letters~$\r'$ corresponding to all letters~$\s$ and to all letters~$\psi(\s)$ inside all copies of~$\c_A$ and of~$\c_B$.
  The four described cases are then obtained by a direct analysis of the order in which the letters in this orbit of~$\r$ under $\Krewplusbar$ are visited as given in \Cref{lem:sequenceindependence}.

  The only difficult situation is the replacement at the end of Item~(iv):
  we have seen in \Cref{lem:sequenceindependence} that the letters are visited in the cyclic order $\r_1,\r_{h+1},\ldots,r_{(\rrr'-1)h+1}$.
  As~$\rrr' = \rrr$ is odd, we have that~$\rrr$ and~$h$ are coprime by \Cref{lem:numbercrunching}.
  One can therefore iterate through the letters $\r_1,\ldots,\r_{\rrr'}$ each once in this order.
  If~$\r_1$ is the very first letter inside $\invc$, then this iteration first visits the $(\rrr+1)/2$ letters corresponding again to the letter~$\r_1$ and then visits the $(\rrr-1)/2$ letters corresponding to~$\psi(\r_1)$.
  This is captured by the first replacement above.
  If~$\r_1$ is instead any letter inside $\invc$, then the resulting word is correct exactly in the positions starting at~$\r_1$, while the letters left of~$\r_1$ inside~$\invc$ have again to be replaced.
\end{proof}

To make it explicit, we record from \Cref{lem:invariantsitting} the possible orbit sizes in the various cases.

\begin{corollary}
  In the situation of \Cref{lem:invariantsitting}, the orbit sizes are given by
  \begin{enumerate}
    \item[\em(i)] $\rrr$ if we are in Case~\eqref{case:hodd}, in Case~\eqref{case:psinotoneeven} with $\psi(\r_1) \neq \r_1$, or in Case~\eqref{case:psinotoneodd}, and
    \item[\em(ii)] $\rrr/2$ if we are in Case~\eqref{case:psione}, or in Case~\eqref{case:psinotoneeven} with $\psi(\r_1) = \r_1$.
  \end{enumerate}
\end{corollary}

We next describe how to decompose a subword of $\invc^m\invc^L$ which is invariant under $(\Krewplusbar)^k$ into orbits, sharpening in particular \Cref{lem:possiblefoldings}.

\begin{lemma}\label{lem:orbitdecomposition}
  Let~$\r_1,\dots,\r_{n'}$ be a subword of\/ $\invc^m\invc^L$ which is invariant under the action of\/ $(\Krewplusbar)^k$.
  The orbit decomposition of\/ $\r_1,\dots,\r_{n'}$ is then given as follows:
  \begin{enumerate}
    \setlength\itemindent{30pt}
    \item[\sc Case~\eqref{case:hodd}:]
      $\rrr$ divides~$n'$ and the orbits consist of those letters that are multiples of $k'$~positions apart, where~$k'$ is given by~$n' = \rrr\cdot k'$.
    \item[\sc Case~\eqref{case:psione}:]
      $\rrr/2$ divides~$n'$ and the orbits consist of those letters that are multiples of $k'$~positions apart, where~$k'$ is given by~$n' = \rrr/2 \cdot k'$.
    \item[\sc Case~\eqref{case:psinotoneeven}:]
      divide the word $\r_1,\ldots\r_{n'}$ into subwords $\r'_1,\r'_2,\ldots,\r'_{n''}$ and $\r''_1,\r''_2,\ldots,\r''_{n'''}$ such that the former have the property that $\psi(\r'_i) \neq \r'_i$ and the latter have the property $\psi(\r''_i) = \r''_i$.
      Then~$\rrr$ divides~$n''$ and the orbits of size~$\rrr$ consist of those letters in $\r'_1,\r'_2,\ldots,\r'_{n''}$ that are multiples of $k'$~positions apart, where~$k'$ is given by~$n'' = \rrr\cdot k'$, and~$\rrr/2$ divides~$n'''$ and the orbits of size~$\rrr/2$ consist of those letters in $\r''_1,\r''_2,\ldots,\r''_{n'''}$ that are multiples of $k''$~positions apart, where~$k''$ is given by~$n''' = \rrr/2 \cdot k''$.
    \item[\sc Case~\eqref{case:psinotoneodd}:]
      $\rrr$ divides~$n'$ and the orbits consist of those letters that are multiples of $k'$~positions apart, where~$k'$ is given by~$n' = \rrr\cdot k'$.
  \end{enumerate}
\end{lemma}
\begin{proof}
  This is a direct consequence of the description of the orbits in \Cref{lem:invariantsitting} as there we have seen that the letters in an orbit are equally far apart inside $\invc^m\invc^L$.
\end{proof}

Given \Cref{lem:invariantsitting,lem:orbitdecomposition}, we obtain a purely numerical (and computationally tract\-able) way to give a necessary and sufficient condition on a reduced factorisation $\r_1\cdots\r_{n'}$ of a non-crossing partition $w = r_1\cdots r_{n'} \in \NC$ to correspond to an invariant positive $m$-divisible non-crossing partition for some~$m$.
We then later explicitly give, as a polynomial in the parameter~$m$ (or rather its appropriate residue class), the number of subwords of $\invc^m\invc^L$ for $\c = \c_A \c_B$ that are invariant under the positive Kreweras map that spell out this word.
To this end, we assume
\begin{itemize}
  \item in Case~\eqref{case:hodd} that $\c_B = \psi(\c_A)$ and
  \item in Case~\eqref{case:heven} that the letters~$\s$ and~$\psi(\s)$ inside $\c_A$ and inside~$\c_B$ are equal or next to each other.
\end{itemize}
The situation is almost the same in all cases, but with the slight modifications already present in the previous lemmas.
We need one more auxiliary lemma to nicely describe Case~\eqref{case:psinotoneeven}.

\begin{lemma}
\label{lem:auxpsinotoneeven}
  Let~$h$ be even, let~$\r$ be a letter in~$\invc$, and let~$\s$ be the corresponding letter in the~$h$ alternating copies of~$\c_A$ and~$\c_B$.
  Let moreover $\r'$ be the letter in~$\invc$ given by~$\psi(\r)$ with corresponding letter~$\s'$.
  Then~$\s' = \psi(\s)$ in the same copy of~$\c_A$ or of~$\c_B$ as~$\s$.
\end{lemma}
\begin{proof}
  This is a well-known fact that holds for bipartite Coxeter elements with~$h$ even:
  write $\wo = (c_Ac_B)^{h/2} = s_1\cdots s_N$, and let $r = s_1\cdots s_is_{i+1}s_i\cdots s_1$.
  As $\psi(r) = \wo r \wo$, we obtain $\psi(r) = \psi(s_1)\cdots \psi(s_i)\psi(s_{i+1})\psi(s_i)\cdots \psi(s_1)$, implying the statement because $\psi(c_A) = c_A$ and $\psi(c_B) = c_B$ from \Cref{lem:comphn}.
\end{proof}

This lemma enables us in Case~\eqref{case:psinotoneeven} in \Cref{lem:invariantsitting} to not consider all letters~$\r$ and~$\psi(\r)$, but to restrict attention to a consistent choice of one of the two.
To this end, let~$\r_1,\dots,\r_{n'}$ be a subword of $\invc^m\invc^L$ in Case~\eqref{case:psinotoneeven} with $\psi(\r_1) \neq \r_1$.
Then we have seen in \Cref{lem:sequenceindependence2} that this is an orbit of the action of $(\Krewplusbar)^k$ if and only if $n'=\rrr$, $\psi(\r_{2i}) = \r_{2i+1}$, and the subword $\r_1,\r_3,\dots,\r_{n'-1}$ is invariant under the cyclic shift sending a letter~$\s$ in the $j$\th\ alternating copy of~$\c_A$ or~$\c_B$ to the letter $\s$ in the $(j+2k)$\th\ alternating copy of~$\c_A$ or~$\c_B$.
We refer to \Cref{ex:biginvexample}\eqref{ex:biginvexampleiii} for comparison.

\medskip

For the following proposition, recall the discussion after \Cref{lem:numbercrunching} and in particular the definition of~$y$ and of~$m'$.

\begin{proposition}
\label{prop:invariantsittingdetails}
  Let $\r_1\cdots\r_{n'}$ be a reduced $\reflR$-word for $w = r_1\cdots r_{n'} \in \NC$, and let $i_1,\ldots,i_{n'}$ be the sequence of indices of the letters  $\r_1,\ldots,\r_{n'}$ inside $\invc$.
  Let~$\rrr$ be such that~$\rrr/y$ and $h/y$ are coprime and let {\em$m \equiv m'$ (mod $\rrr/y$)}, where we assume in Case~\eqref{case:psione} that~$h$ is even.
  Then there exists a subword of $\invc^m\invc^L$ which is invariant under the $\rrr$-fold application of~$\Krewplusbar$ and which spells out $\r_1\cdots\r_{n'}$ if and only if the cyclic distances
  \[
  \begin{array}{rlll}
    a_1         &\equiv& i_{k'+1}-i_1 & (\text{\em mod }N) \\
    a_2         &\equiv& i_{k'+2}-i_2 & (\text{\em mod }N) \\
                &\vdots \\
    a_{n'-k'}   &\equiv& i_{n'}-i_{n'-k'} & (\text{\em mod }N) \\
    a_{n'-k'+1} &\equiv& i_{1}-i_{n'-k'+1}+n & (\text{\em mod }N) \\
                &\vdots \\
    a_{n'}      &\equiv& i_{k'}-i_{n'}+n & (\text{\em mod }N) 
  \end{array}
  \]
  inside $\invc$ satisfy
  \[
    a_1\equiv a_2 \equiv \dots \equiv a_{n'} \equiv 0 \begin{array}{ll}
                                          \pmod \rrr    & \text{ in Case~\eqref{case:hodd}},\\
                                          \pmod   {\rrr/2}  & \text{ in Case~\eqref{case:heven}},
                                        \end{array}
  \]
  where in
  \begin{enumerate}
    \setlength\itemindent{30pt}
    \setlength\itemsep{5pt}

    \item[\sc Case~\eqref{case:hodd}:]
      $\rrr$ divides~$n'$, with $k'$ defined by $n' = \rrr \cdot k'$;

    \item[\sc Case~\eqref{case:psione}:]
      $\rrr/2$ divides~$n'$, with $k'$ defined by $n' = \rrr/2 \cdot k'$;

    \item[\sc Case~\eqref{case:psinotoneodd}:]
      $\rrr$ divides~$n'$, with $k'$ defined by $n' = \rrr \cdot k'$, and after one has taken into account the modification described in \Cref{lem:invariantsitting};

    \item[\sc Case~\eqref{case:psinotoneeven}:]
      $\rrr/2$ divides~$n''$, with $k''$ defined by $n'' = \rrr/2 \cdot k''$, and where first, for every letter $\r_i$, also $\psi(\r_i)$ must be in $\r_1\cdots\r_{n'}$, and one uses the subword of\/ $\r_1\cdots\r_{n'}$ given by $\r'_1\cdots\r'_{n''}$ after \Cref{lem:auxpsinotoneeven}. In this case, let $i'_1,\ldots,i'_{n''}$ be the sequence of indices of the letters $\r'_1,\ldots,\r'_{n''}$ inside $\invc$.
  \end{enumerate}
\end{proposition}
\begin{proof}
  We prove Case~\eqref{case:psione}, the others are completely analogous, with the appropriate modifications in the last two cases.

  First assume that there is a $(\rrr/2)$-fold invariant subword of $\invc^m\invc^L$ spelling out $\r_1\cdots\r_{n'}$.
  This means that this subword is invariant under the action of~$(\Krewplusbar)^k$.

  We have seen in \Cref{lem:orbitdecomposition} that orbits of the reflections $r_1,\dots,r_{n'}$ in this word under $(\Krewplusbar)^k$ must be all letters that are $k'$~positions apart for~$k'$ given by $n' = \rrr/2 \cdot k'$.
  The letters that are $k'$~positions apart must thus be spread equally inside the word $\invc^m\invc^L$, implying the numerical condition
  \[
    a_1 = a_2 = \dots = a_{n'},
  \]
  where the shift by~$n$ in $a_{n'-k'+1},\ldots,a_{n'}$ comes from the fact that these distances have passed the end of the word $\invc^m\invc^L$ and there are exactly~$n$ letters missing in the final $\invc^L$.
  Moreover, the congruence condition
  \[
    a_1 \equiv a_2 \equiv \dots \equiv a_{n'} \equiv 0\pmod{\rrr/2}
  \]
  means that all letters that are $k'$~positions apart are indeed in the same orbit of~$\Krewplusbar$, as desired.

  \smallskip

  On the other hand, if these numerical conditions are satisfied for a reduced $\reflR$-word $\r_1\cdots\r_{n'}$ of some $w \in \NC$, then $\r_1\cdots\r_{n'}$ fits properly into~$\invc^m\invc^L$ as described in \Cref{lem:invariantsitting}.
\end{proof}

Using this proposition, we now have a numerical algorithm to determine which $w \in \NC$ and which reduced $\reflR$-words of~$w$ possibly contribute to the counting formula of $\rrr$- respectively $(\rrr/2)$-fold symmetric positive $m$-divisible non-crossing partitions.
It is now only left to determine to which~$m$ they contribute and how much they contribute.

Recall from the paragraph below \Cref{ex:AA} that, for all $m\ge1$,
both versions of the positive Kreweras maps $\Krewplustilde^{(m)}$
and $\Krewplus^{(m)}$ can be considered as the action of~$\Krewplusbar$ on the positive elements in $\mNCnabla$,
respectively in $\mNCdelta$.

\begin{theorem}\label{thm:countingbinomial}
  Let $(W,\reflS)$ be a finite Coxeter system, let~$\c = \c_A\c_B$ be a fixed word for the bipartite Coxeter element $c = c_Ac_B \in W$ such that $\c_B = \psi(\c_A)$ if we are in Case~\eqref{case:hodd} and such that the letters~$\s$ and~$\psi(\s)$ inside $\c_A$ and inside~$\c_B$ are equal or next to each other in Case~\eqref{case:heven}.
  Let $(m+1)h-2 = k \cdot \rrr$ such that $\rrr$ is even in Case~\eqref{case:psione}.
  Let $a = \lfloor m/\frac{\rrr}{y} \rfloor$, and let~$\ifix$ be as defined in~\eqref{eq:ifix}.
  Then the number of positive $(m+1)$-divisible non-crossing partitions $\mNCPlus[W][m+1]$ that are invariant under the $\rrr$-fold application of\/ $\Krewplustilde$ is given by the sum
      \begin{equation}
        \sum \binom{a+\asc(\r_1,\ldots,\r_{k'},\rfix)}{k'},
      \end{equation}
      which ranges over all factorisations $\r_1\cdots \r_{n'}$ of all elements $w \in \NC$ that satisfy the properties in \Cref{prop:invariantsittingdetails}, and where in
  \begin{enumerate}
    \setlength\itemindent{30pt}
    \setlength\itemsep{10pt}

    \item[\sc Case~\eqref{case:hodd}:]
      $\rrr$ divides~$n'$ and~$k'$ is given by $n' = \rrr \cdot k'$;

    \item[\sc Case~\eqref{case:psione}:]
      $\rrr/2$ divides~$n'$ and~$k'$ is given by $n' = \rrr/2 \cdot k'$;

    \item[\sc Case~\eqref{case:psinotoneodd}:]
      $\rrr$ divides~$n'$ and $k'$ is given by $n' = \rrr \cdot k'$;
      furthermore, $\r'_1,\ldots,\r'_{k'}$ is the sequence obtained from $\r_1,\ldots,\r_{k'}$ inside $\invc$ via the modification described in \Cref{lem:invariantsitting};

    \item[\sc Case~\eqref{case:psinotoneeven}:]
      $\rrr/2$ divides~$n''$ where $\r'_1\cdots\r'_{n''}$ is the subword of $\r_1\cdots\r_{n'}$ given in \Cref{prop:invariantsittingdetails}, and $k''$ is given by $n'' = \rrr \cdot k''$.
  \end{enumerate}
\end{theorem}
\begin{proof}
  This is a direct consequence of the analysis in \Cref{prop:invariantsittingdetails}, together with the same counting argument as used for \Cref{cor:fusscatcounting}.
\end{proof}

\begin{corollary}
  In the situation of \Cref{thm:countingbinomial}, the number of positive $m$-divisible non-crossing partitions $\mNCPlus[W][m]$ that are invariant under the $k$-fold application of\/ $\Krewplus$ is given by the same counting formula where one only considers factorisations $\r_1\cdots\r_n$ of the Coxeter element~$c = c_Lc_R \in W$.
\end{corollary}

To finish these counting considerations, we recall the steps to count positive $m$-divisible non-crossing partitions that are invariant under a $\rrr$- respectively $(\rrr/2)$-fold symmetry, and compare it to the evaluation of the $q$-extension of the positive Fu\ss--Catalan numbers at the corresponding primitive root of unity:
\begin{enumerate}
  \item\label{item:1} Compute all reduced $\reflR$-words of the Coxeter element $c = c_A c_B$ using the transitivity of the Hurwitz action (see~\cite[Prop.~1.6.1]{BesDAA}).
  \begin{enumerate}
    \item In the case of~$\Krewplustilde$, one then computes as well all prefixes of these reduced $\reflR$-words to obtain all reduced factorisations of elements in $\NC$.
  \end{enumerate}
  \item Represent all reduced words computed in \Cref{item:1} as indices of positions of subwords of~$\invc$.
  \item Use \Cref{prop:invariantsittingdetails} to determine which factorisations contribute to the enumeration. Disregard all factorisations not satisfying these properties.
  \item Use \Cref{thm:countingbinomial} to determine the contribution of each factorisation as a polynomial in the parameter~$a$.
  \item Evaluate $\mCatplus(W;q)$ or $\mCatplus[m+1](W;q)$ (depending on whether one considers $\Krewplus$ or $\Krewplustilde$) at the root~$\zeta^{k}$ where $\zeta$ is a primitive $\big((m+1)h-2\big)$\th\ root of unity and express the result in terms of the parameter~$a$.
  \item Compare the two results in the previous two steps to verify the cyclic sieving phenomenon in this situation.
\end{enumerate}

We use the described machinery to prove \Cref{thm:CS} for dihedral and exceptional types and record this here for reference.

\begin{theorem} \label{thm:cycexc}
  The conclusion of \Cref{thm:CS} holds for dihedral and exceptional types.
\end{theorem}

The concrete calculations are postponed to \Cref{app:exc}.

\section{Further open problems}
\label{sec:open}

We use this final section to make explicit some more open problems
raised by our work.

\subsection{Chain enumeration}

The attentive reader may have observed that our enumerative treatment 
for types $B_n$ and $D_n$ does not include any results on chain
enumeration, as opposed to type $A_{n-1}$ (cf.\ \Cref{sec:enumA}).
It may entirely be possible that formulae analogous to the one in
\Cref{thm:countingA} exist also for types $B_n$ and $D_n$. However, we
have not been able to push our approach through in these cases.

\begin{openproblem}
  Find formulae for the number of multichains in $\mNCBPlus$ and/or
  $\mNCDPlus$ where the ranks of the elements of the chain are
  prescribed, and/or even the block structure of the bottom element
  of the chain.
\end{openproblem}

More generally, the question of formulae for chain enumeration of
pseudo-rotationally invariant elements arises for all types.
Here, we have not even been successful in type~$A_{n-1}$. In all
(classical) types, the reason is that the positive Kreweras map does
not preserve the order relation, but the characterisation of
pseudo-rotationally invariant partitions (see
\Cref{sec:rotA,sec:rotB,sec:rotD}) involves repeated application of
the positive Kreweras map in an essential way.

\begin{openproblem}
  Find formulae for the number of multichains in the subposets of
  $\mNCAPlus$, $\mNCBPlus$, and/or
  $\mNCDPlus$ consisting of $\rrr$-pseudo-rotationally invariant elements,
  where the ranks of the elements of the chain are
  prescribed, and/or even the block structure of the bottom element
  of the chain.
\end{openproblem}

If there is a positive solution of the above problem in type~$D$, then
we may also ask for more general cyclic sieving phenomena in this type.

\begin{openproblem}
  Prove a cyclic sieving phenomenon for the subset of elements of
  $\mNCDPlus$ where their rank is prescribed, and, more generally,
  where their block structure is prescribed.
\end{openproblem}

Closely related to chain enumeration in a poset is its M\"obius
function; see~\cite{MR0595933}.
\Cref{cor:moebiusA} leads us right-away to the next open problem.

\begin{openproblem}
Determine the M\"obius function for $\mNCBPlus$ and/or
  $\mNCDPlus$. More generally,
determine the M\"obius function for the subposets of $\mNCAPlus$,\break $\mNCBPlus$, and/or
  $\mNCDPlus$ consisting of $\rrr$-pseudo-rotationally invariant elements.
\end{openproblem}

\subsection{Block enumeration}

In~\cite{ReSoAA}, Reiner and Sommers proved cyclic sieving
phenomena for $m$-divisible non-crossing partitions with prescribed
block structure. Their uniform definition of ``block structure" is
also applicable in exceptional types. This suggests the following 
problems.

\begin{question}
  Does the positive Kreweras map preserve the block structure, as
  defined in~\cite{ReSoAA}, of positive $m$-divisible non-crossing partitions?
\end{question}

\begin{openproblem}
In case of an affirmative answer to the above question,
prove a cyclic sieving phenomenon for \emph{positive} $m$-divisible non-crossing
partitions in exceptional types with prescribed block structure.
\end{openproblem}

\subsection{A positive Panyushev map on non-nesting partitions}

Recall that $\Delta \subseteq \Phi^+$ denotes the simple and positive roots for the Coxeter systems $(W,\reflS)$.
For crystallographic Coxeter systems, define the \defn{root poset} $\rootPoset(W)$ to be the partial order on $\Phi^+$ induced by the cover relations given by $\alpha \prec \beta$ if $\beta - \alpha \in \Delta$.
The \defn{non-nesting partitions} associated with~$W$ were defined by Panyushev~\cite{Pan2008} as the set of all antichains in $\rootPoset$,
\[
  \NN := \big\{ \antichain \subseteq \rootPoset \mid \antichain \text{ antichain} \big\}.
\]
As for non-crossing partitions, there is also a positive variant of non-nesting partitions given by
\[
  \NNplus := \big\{ \antichain \subseteq \rootPoset\setminus \Delta \mid \antichain \text{ antichain} \big\}.
\]
These (positive) non-nesting partitions turn out to share many enumerative properties with non-crossing partitions; see~\cite{Arm2006} for a detailed treatment thereof. In particular, we have
\[
  \big| \NN \big| = \big| \NC \big| \text{ and } \big| \NNplus \big| = \big| \NCPlus \big|.
\]
Obviously, these equalities suggest to construct bijections, and indeed multiple bijections were constructed in various types~\cite{CDDSY2007,FG2009,RS2010}, the only uniformly described bijection was given in~\cite{AST2010}.
One particular avenue to follow --- as explained by Armstrong, Thomas,
and the second author in~\cite{AST2010} --- are two maps on $\NC$ and
on $\NN$ given by the Kreweras map $\Krew : \NC \longrightarrow \NC$
and the \defn{Panyushev map} $\Pan : \NN \longrightarrow \NN$ which is
given by mapping an antichain $\antichain$ to the collection of
minimal elements in the root poset among all elements that are not
below $\antichain$.
(This map has in fact been earlier defined by Cameron and
Fon-Der-Flaass~\cite{CaFlAA} for general posets.\footnote{Still, we prefer to
call this map on non-nesting partitions the ``Panyushev map" since
it was Panyushev who studied this map for root posets.})
In~\cite{AST2010}, a uniformly described bijection between non-crossing and non-nesting partitions is constructed sending the Kreweras map to the Panyushev map. This bijection is uniquely determined (but actually strongly overdetermined) by certain inductive properties on parabolic subgroups. The proof of its existence is given by explicitly constructing it in the various types. Observe that this bijection does not send positive non-crossing partitions to positive non-nesting partitions. Moreover, as mentioned above for the Kreweras map, the Panyushev map does not stabilise positive non-nesting partitions either. But again, there is a similarly defined \defn{positive Panyushev map} $\Panplus : \NNplus \longrightarrow \NNplus$ given by the same rule as for $\Pan$ but considered on the root poset $\rootPoset$ with the simple roots $\Delta$ removed.

Before raising two questions on the connections between positive non-crossing and positive non-nesting partitions, we copy two conjectures on the maps $\Pan$ and $\Panplus$ by Panyushev~\cite[Conj.~2.1 and~2.3]{Pan2008}, the first of which was proven in~\cite{AST2010} by constructing an equivariant bijection between $\NN$ and $\NC$ as briefly mentioned above.

\begin{theorem}[\sc Armstrong, Stump, Thomas, conjectured by Panyushev]
  Let~$W$ be a finite, crystallographic Coxeter group with Coxeter number~$h$, and let $\Pan$ be the Panyushev map on $\NN$. The order of\/ $\Pan$ is given by
  $$
    \order(\Pan) =
    \begin{cases}
      (m+1)h, & \text{if } \psi \nequiv \one, \\
      (m+1)h/2, & \text{if } \psi \equiv \one.
    \end{cases}
  $$
\end{theorem}
\begin{conjecture}[\sc Panyushev]
  Under the same assumptions as in the previous theorem with $\Panplus$ being the positive Panyushev map on $\NNplus$, the order of\/ $\Panplus$ is given by
  $$
    \order(\Panplus) =
    \begin{cases}
      (m+1)h - 2, & \text{if } \psi|_{\reflR \setminus \reflS} \nequiv \one, \\
      (m+1)h/2 - 1, & \text{if } \psi|_{\reflR \setminus \reflS} \equiv \one.
    \end{cases}
  $$  
\end{conjecture}

(Here, we have corrected a subtle mistake in the original conjecture in that we added the restriction of~$\psi$ to non-simple reflections in the above case distinction.
The only difference occurs in type $A_2$, where $\psi$ interchanges the two simple reflections $s_1,s_2$ while fixing the third reflection $s_1s_2s_1 = s_2s_1s_2$.) The combinatorial models in types~$A$ and~$B$ for the bijection in~\cite{AST2010} immediately give that, in these types, the triple
$$\Big( \NNplus, \Cat_+(W;q), \langle\Panplus\rangle \Big)$$
exhibits the cyclic sieving phenomenon. 
This holds as well in the exceptional types,
which was checked using a computer,
thus leaving only type $D$ open. 
Together with \Cref{thm:order} that now raises 
the following two questions.

\begin{question}
  Is there an equivariant bijection between $\NNplus$ and $\NCPlus$ for the Kreweras map respectively the Panyushev map, similar to the bijection between $\NN$ and $\NC$?
\end{question}

We remark that there are as well $m$-extensions $\mNN$ and $\mNNplus$ of (positive) non-nesting partitions, cf.\ \cite[Cor.~1.3]{Ath2004}.
No extension of the Panyushev map in that direction is known so far.
In particular, we completely lack an understanding of uniform connections between $m$-divisible non-crossing partitions and the $m$-extension of non-nesting partitions.

\begin{openproblem}
  Construct a generalised (positive) Panyushev map on $\mNN$ and on $\mNNplus$ which generalises the (positive) Panyushev map and at the same time provides the same orbit structure as the generalised (positive) Kreweras map on $\mNC$ and on $\mNCPlus$.
\end{openproblem}

\bibliographystyle{amsalpha}
\bibliography{positive}

\appendix

\section{Proofs: The order of the positive Kreweras map}
\label{app:order}

This appendix is devoted to the proofs of \Cref{lem:N-2,,lem:N-2B,,lem:N-2D} describing the order of the positive Kreweras map in classical types.

\subsection{The order of the positive Kreweras map in
  type $A_{n-1}$: Proof of \Cref{lem:N-2}}
\label{app:order-A}

The proof proceeds in several steps. Let $\pi$ be some non-crossing
partition of~$\{1,2,\dots,N\}$.

{
\medskip
\hangindent2\parindent\hangafter1
{\sc Step 1.} We argue that singleton blocks are not ``important", in
  the sense that it suffices to prove the claim for non-crossing
  partitions without singleton blocks.

\hangindent2\parindent\hangafter1
{\sc Step 2.} 
We prove that a block of~$\pi$ of the form $\{N,1,2,\dots,k\}$
returns to its position after $N-2$ applications of~$\rotA$.

\hangindent2\parindent\hangafter1
{\sc Step 3.} We prove that a block of~$\pi$ 
containing $N$ and $1$ but not~$N-1$
returns to its position after $N-2$ applications of~$\rotA$.

\hangindent2\parindent\hangafter1
{\sc Step 4.} We argue that Steps~1--3 suffice to complete the proof.
\par
}

We now enter in the details of each step.

\medskip\noindent
{\sc Step 1.} If we concentrate ourselves on a particular singleton
block and what it does under successive applications of the
pseudo-rotation~$\rotA$, then we see that $\rotA$ always rotates
this block by $1$~unit in clockwise direction,
except that the singleton block $\{N-1\}$ is mapped
to $\{2\}$, that is, it is --- so-to-speak --- rotated by $3$~units
in clockwise direction.
In particular, after $N-2$ applications of~$\rotA$, 
a singleton block returns to its original position. 

Let us suppose for the moment that the assertion of the
lemma is already established for the case of non-crossing partitions
without singleton blocks.
We furthermore suppose that $\pi$ contains exactly $k$ singleton blocks.
It is straightforward to see that, if we remove the singleton blocks
and consider the remainder of~$\pi$ as a non-crossing partition on
the remaining elements, the pseudo-rotation~$\rotA$ acts
as the identity if $\{N-1\}$ is a
singleton block, and otherwise 
in the way illustrated in \Cref{fig:7}. In particular, if we denote the
non-crossing partition arising from $\pi'$ by ignoring its singleton
blocks, then, according to our assumption, 
after $N-k-2$ applications of~$\rotA$ on $\pi'$ we return
back to $\pi'$. This implies that $N-2$ applications
of~$\rotA$ to~$\pi$ produce a non-crossing partition $\hat\pi$ 
in which the singleton blocks are in their original place, and which
is identical with $\pi'$ after removal of the singleton blocks. 
Since there is a unique way to reinsert the singleton blocks into
$\pi'$ such that they sit in their original position, this shows
that, under the above assumption, the map~$\rotA^{N-2}$ acts
as the identity on~$\pi$.

\medskip
From now on, we assume that $\pi$ does not contain any singleton blocks.

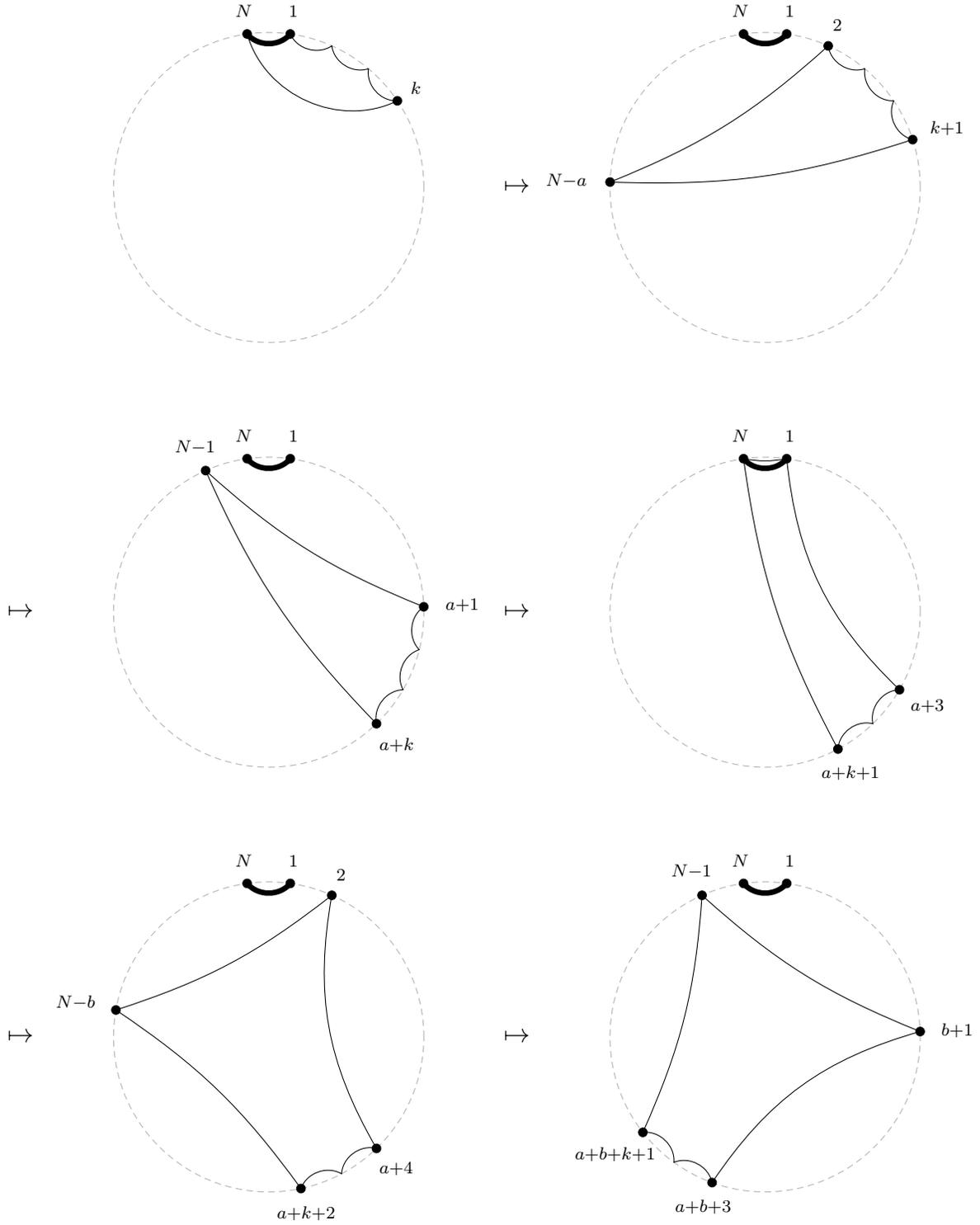
\begin{figure}
\begin{center}
  \begin{tikzpicture}[scale=1]
    \polygonlabel{(-4,0)}{obj}{45}{2.5}
      {1,,,,,,k,,,,,,,,,,,,,,,,,,,,,,,,,,,,,,,,,,,,,N,}

    \draw[line width=2.5pt,black] (obj1) to[bend left=50] (obj44);
    \draw (obj1) to[bend right=50]
          (obj3) to[bend right=50]
          (obj5) to[bend right=50]
          (obj7) to[bend left=50]
          (obj44);

    \node[inner sep=0pt] at (0,0) {$\mapsto$};

    \polygonlabel{( 4,0)}{obj}{45}{2.5}
      {1,,2,,,,,,k+1,,,,,,,,,,,,,,,,,,,,,,,,,N-a,,,,,,,,,,N,}

    \draw[line width=2.5pt,black] (obj1) to[bend left=50] (obj44);
    \draw (obj3)  to[bend right=50]
          (obj5)  to[bend right=50]
          (obj7)  to[bend right=50]
          (obj9)  to[bend left =10]
          (obj34) to[bend right=10]
          (obj3);

    \node[inner sep=0pt] at (-8,-6.85) {$\mapsto$};

    \polygonlabel{(-4,-6.85)}{obj}{45}{2.5}
      {1,,,,,,,,,,a+1,,,,,,a+k,,,,,,,,,,,,,,,,,,,,,,,,,N-1,,N,}

    \draw[line width=2.5pt,black] (obj1) to[bend left=50] (obj44);
    \draw (obj11)  to[bend right=50]
          (obj13)  to[bend right=50]
          (obj15)  to[bend right=50]
          (obj17)  to[bend left =10]
          (obj42) to[bend right=10]
          (obj11);

        \node[inner sep=0pt] at (0,-6.85) {$\mapsto$};

    \polygonlabel{( 4,-6.85)}{obj}{45}{2.5}
      {1,,,,,,,,,,,,,,a+3,,,,a+k+1,,,,,,,,,,,,,,,,,,,,,,,,,N,}

    \draw[line width=2.5pt,black] (obj1) to[bend left=50] (obj44);
    \draw (obj1)  to[bend right=20]
          (obj15)  to[bend right=50]
          (obj17)  to[bend right=50]
          (obj19)  to[bend left =10]
          (obj44) to[bend right=10]
          (obj1);

    \node[inner sep=0pt] at (-8,-13.7) {$\mapsto$};

    \polygonlabel{(-4,-13.7)}{obj}{45}{2.5}
      {1,,2,,,,,,,,,,,,,,a+4,,,,a+k+2,,,,,,,,,,,,,,N-b,,,,,,,,,N,}

    \draw[line width=2.5pt,black] (obj1) to[bend left=50] (obj44);
    \draw (obj3)  to[bend right=20]
          (obj17) to[bend right=50]
          (obj19) to[bend right=50]
          (obj21) to[bend right=10]
          (obj35) to[bend right=10]
          (obj3);

  \node[inner sep=0pt] at (0,-13.7) {$\mapsto$};

  \polygonlabel{( 4,-13.7)}{obj}{45}{2.5}
      {1,,,,,,,,,,b+1,,,,,,,,,,,,,,a+b+3,,,,a+b+k+1,,,,,,,,,,,,,N-1,,N,}

    \draw[line width=2.5pt,black] (obj1) to[bend left=50] (obj44);
    \draw (obj11)  to[bend right=20]
          (obj25) to[bend right=50]
          (obj27) to[bend right=50]
          (obj29) to[bend right=10]
          (obj42) to[bend right=10]
          (obj11);

\end{tikzpicture}
\end{center}
\caption{Pseudo-rotation of a block of successive elements}
\label{fig:30a}
\end{figure}

\begin{figure}
\begin{center}
\begin{tikzpicture}

    \node[inner sep=0pt] at (-8,-21) {$\mapsto$};

    \polygonlabel{(-4,-21)}{obj}{45}{2.5}
      {1,,,,,,,,,,,,,,,,,,,,,,,,,,a+b+4,,,,a+b+k+2,,,,,,,,,,,,,N,}

    \draw[line width=2.5pt,black] (obj1) to[bend left=50] (obj44);
    \draw (obj1)  to[bend left =10]
          (obj27) to[bend right=50]
          (obj29) to[bend right=50]
          (obj31) to[bend right=10]
          (obj44);

    \node[inner sep=0pt] at (0,-21) {$\mapsto$};

    \polygonlabel{( 4,-21)}{obj}{45}{2.5}
      {1,,2,,,,,,,,,,,,,,,,,,,,,,,,,,,,N-c-k+1,,,,,,N-c,,,,,,,N,}

    \draw[line width=2.5pt,black] (obj1) to[bend left=50] (obj44);
    \draw (obj3)  to[bend left =10]
          (obj31) to[bend right=50]
          (obj33) to[bend right=50]
          (obj35) to[bend right=50]
          (obj37) to[bend right=10]
          (obj3);1

    \node[inner sep=0pt] at (-8,-28) {$\mapsto$};

    \polygonlabel{(-4,-28)}{obj}{45}{2.5}
      {1,,,,,,,c+1,,,,,,,,,,,,,,,,,,,,,,,,,,,,N-k,,,,,,N-1,,N,}

    \draw[line width=2.5pt,black] (obj1) to[bend left=50] (obj44);
    \draw (obj8)  to[bend left =10]
          (obj36) to[bend right=50]
          (obj38) to[bend right=50]
          (obj40) to[bend right=50]
          (obj42) to[bend right=10]
          (obj8);

    \node[inner sep=0pt] at (0,-28) {$\mapsto$};

    \polygonlabel{( 4,-28)}{obj}{45}{2.5}
      {1,,,,,,,,,,,,,,,,,,,,,,,,,,,,,,,,,,,,,N-k+1,,,,,,N,}

    \draw[line width=2.5pt,black] (obj1) to[bend left=50] (obj44);
    \draw (obj1)  to[bend left =50]
          (obj38) to[bend right=50]
          (obj40) to[bend right=50]
          (obj42) to[bend right=50]
          (obj44);

  \end{tikzpicture}
\end{center}
\caption{Pseudo-rotation of a block of successive elements}
\label{fig:30b}
\end{figure}
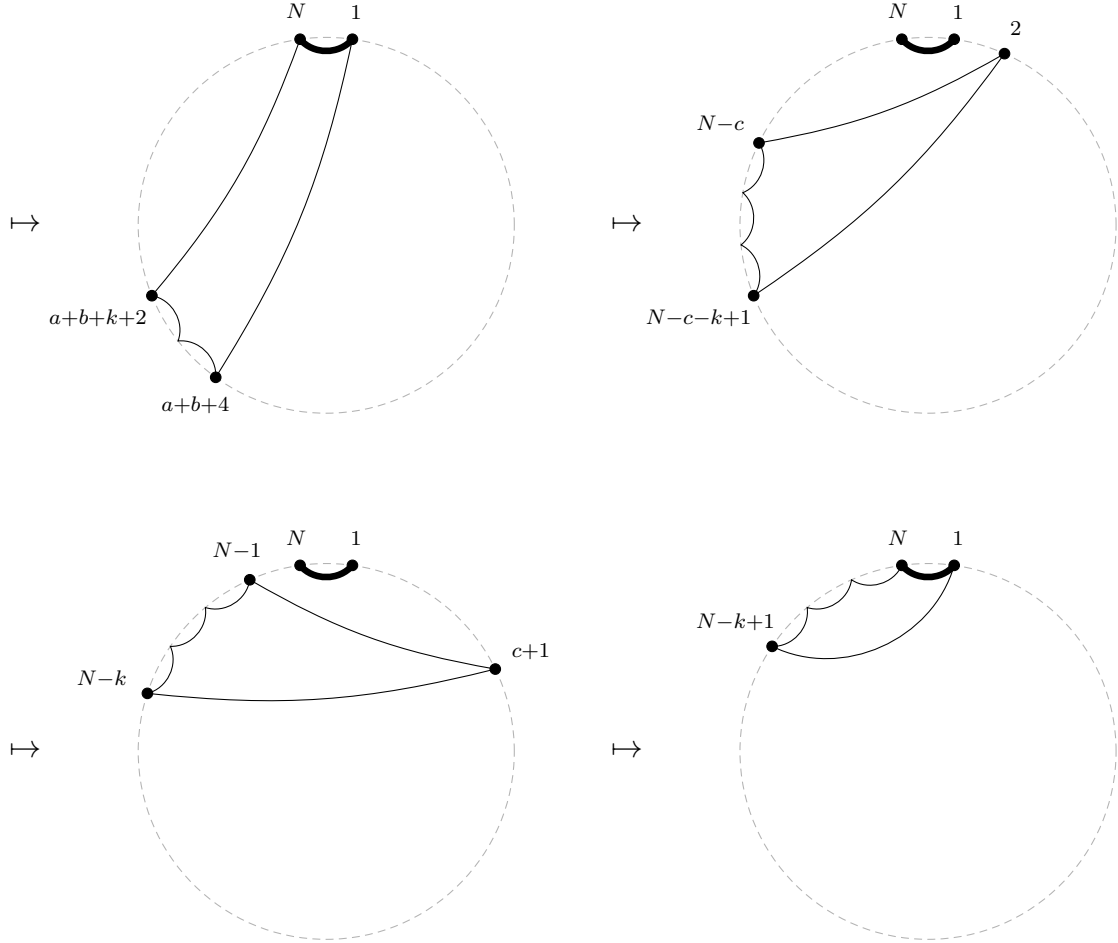

\medskip\noindent
{\sc Step 2.} Let $\{N,1,2,\dots,k\}$ be a block of~$\pi$. 
See \Cref{fig:30a} for an illustration. The first
application of~$\rotA$ maps this block to
$\{N-a,2,\dots,k+1\}$, for a suitable positive integer~$a$; see
the second image in \Cref{fig:30a}. Each of the next $a-1$ applications
of~$\rotA$ rotates this block by one unit in clockwise direction, so that we obtain
$\{N-1,a+1,\dots,a+k\}$. See the third image
in \Cref{fig:30a} for the resulting non-crossing partition. If we
we now apply $\rotA$ another time, then the block on which
we are concentrating ourselves becomes $\{N,1,a+3,\dots,a+k+1\}$;
see the fourth image in \Cref{fig:30a}.
It should be noted that $a+3$ is not an error: one element in the
maximal sequence of successive elements has been ``lost'' under this
operation. Another application of~$\rotA$ then turns this
block into $\{N-b,2,a+4,\dots,a+k+2\}$, for a suitable positive
integer~$b$. See the fifth image in \Cref{fig:30a}.

There follow $b-1$ applications of~$\rotA$ which act as
ordinary rotations on our block of interest. Our block becomes
$\{N-1,b+1,a+b+3,\dots,a+b+k+1\}$; see the sixth image in
\Cref{fig:30a}. The next application of~$\rotA$ has the
effect that it ``removes'' $b+1$ from this block. More precisely, the
block becomes $\{N,1,a+b+4,\dots,a+b+k+2\}$; see the first image in
\Cref{fig:30b}.

Comparison with the fourth image in \Cref{fig:30a} shows that, except
for the location of the maximal sequence of successive elements, we are in the
same situation again. In other words, we may now run through the
situations illustrated by the fourth, fifth, sixth image in
\Cref{fig:30a}, and the first image
in \Cref{fig:30b} several times, until the ``added'' element ($N-b$ in
the fifth image in \Cref{fig:30a}) eventually
``joins'' the maximal sequence of successive
elements (i.e., $N-b=a+k+3$). Once this happens, we are in the
situation illustrated by the second image in \Cref{fig:30b}. 

Referring to this image, $c-1$ further applications
of~$\rotA$ lead to a block of the form $\{N-k,\dots,N-1,c+1\}$;
see the third image in \Cref{fig:30b}.
One more application of~$\rotA$ turns this block into
$\{N-k+1,\dots,N,1\}$; see the last image in \Cref{fig:30b}. 
(One ``loses'' $c+1$ under this operation.)
Finally, the next $k-1$ applications of~$\rotA$ act as
ordinary rotations on our block, so that we return to the situation
illustrated by the first image in \Cref{fig:30a}. 

It can be checked that, in total, we performed $N-2$ applications of~$\rotA$. 

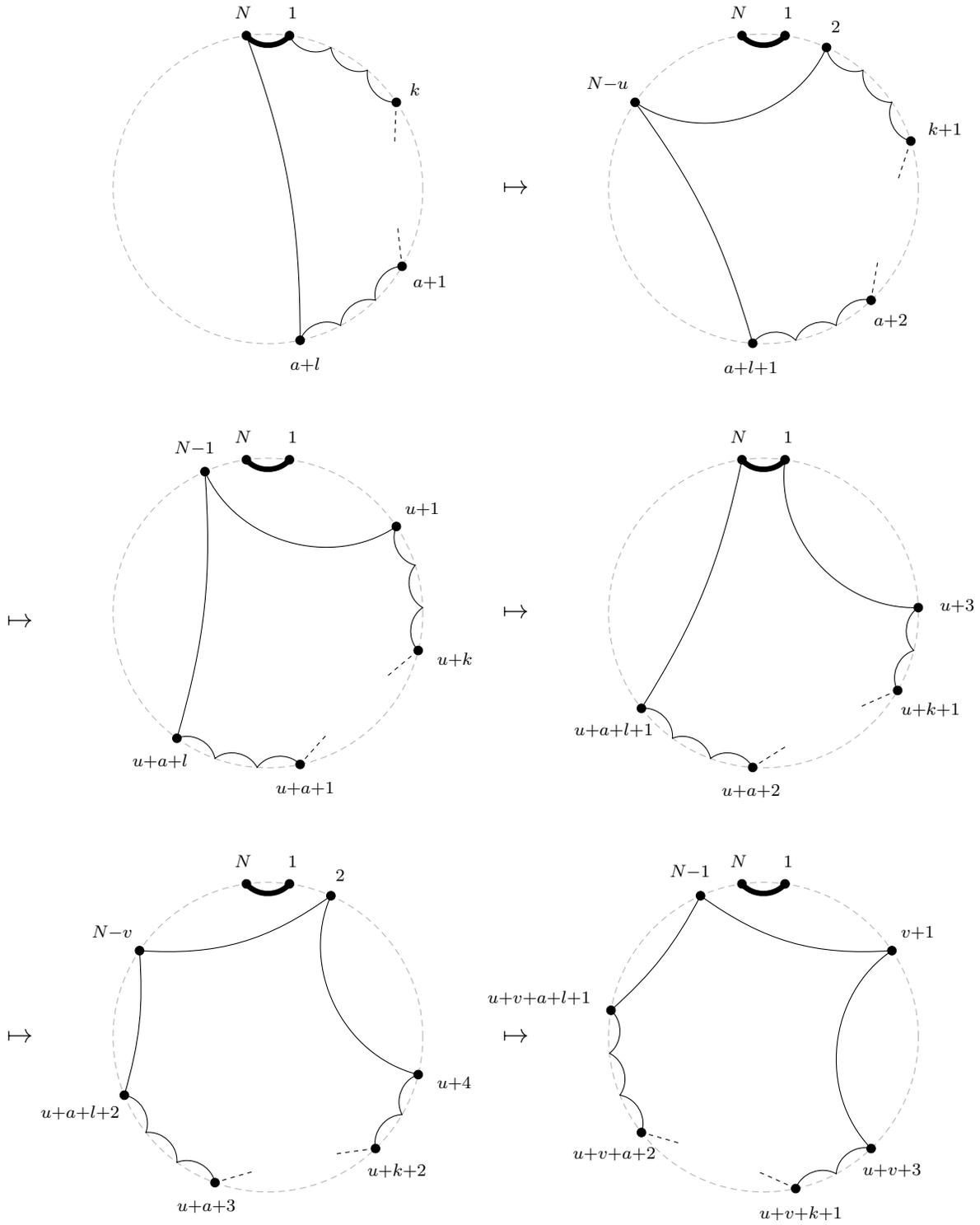
\begin{figure}
  \begin{center}
  \begin{tikzpicture}[scale=1]
    \polygonlabel{(-4,0)}{obj}{45}{2.5}
      {1,,,,,,k,,,,,,,,a+1,,,,,,a+l,,,,,,,,,,,,,,,,,,,,,,,N,}

    \draw[line width=2.5pt,black] (obj1) to[bend left=50] (obj44);
    \draw (obj1) to[bend right=50]
          (obj3) to[bend right=50]
          (obj5) to[bend right=50]
          (obj7);
    \draw (obj15) to[bend right=50]
          (obj17) to[bend right=50]
          (obj19) to[bend right=50]
          (obj21) to[bend right=10]
          (obj44);

    \draw[dash pattern=on 2pt off 2pt on 2pt off 2pt on 2pt off 2pt on
2pt off 2pt on 2pt off 40pt on 2pt off 2pt on 2pt off 2pt on 2pt off
2pt on 2pt off 2pt] (obj7) to[bend right=5] (obj15);

    \node[inner sep=0pt] at (0,0) {$\mapsto$};

    \polygonlabel{( 4,0)}{obj}{45}{2.5}
      {1,,2,,,,,,k+1,,,,,,,,a+2,,,,,,a+l+1,,,,,,,,,,,,,,,N-u,,,,,,N,}

    \draw[line width=2.5pt,black] (obj1) to[bend left=50] (obj44);
    \draw (obj38) to[bend right=50]
          (obj3) to[bend right=50]
          (obj5) to[bend right=50]
          (obj7) to[bend right=50]
          (obj9);
    \draw (obj17) to[bend right=50]
          (obj19) to[bend right=50]
          (obj21) to[bend right=50]
          (obj23) to[bend right=10]
          (obj38);

    \draw[dash pattern=on 2pt off 2pt on 2pt off 2pt on 2pt off 2pt on
2pt off 2pt on 2pt off 40pt on 2pt off 2pt on 2pt off 2pt on 2pt off
2pt on 2pt off 2pt] (obj9) to[bend right=5] (obj17);

    \node[inner sep=0pt] at (-8,-7) {$\mapsto$};

    \polygonlabel{(-4,-6.85)}{obj}{45}{2.5}
      {1,,,,,,u+1,,,,,,u+k,,,,,,,,u+a+1,,,,,,u+a+l,,,,,,,,,,,,,,,N-1,,N,}

    \draw[line width=2.5pt,black] (obj1) to[bend left=50] (obj44);
    \draw (obj42) to[bend right=50]
          (obj7) to[bend right=50]
          (obj9) to[bend right=50]
          (obj11) to[bend right=50]
          (obj13);
    \draw (obj21) to[bend right=50]
          (obj23) to[bend right=50]
          (obj25) to[bend right=50]
          (obj27) to[bend right=10]
          (obj42);

    \draw[dash pattern=on 2pt off 2pt on 2pt off 2pt on 2pt off 2pt on
2pt off 2pt on 2pt off 40pt on 2pt off 2pt on 2pt off 2pt on 2pt off
2pt on 2pt off 2pt] (obj13) to[bend right=5] (obj21);

    \node[inner sep=0pt] at (0,-6.85) {$\mapsto$};

    \polygonlabel{( 4,-6.85)}{obj}{45}{2.5}
      {1,,,,,,,,,,u+3,,,,u+k+1,,,,,,,,u+a+2,,,,,,u+a+l+1,,,,,,,,,,,,,,,N,}

    \draw[line width=2.5pt,black] (obj1) to[bend left=50] (obj44);
    \draw (obj1)  to[bend right=50]
          (obj11) to[bend right=50]
          (obj13) to[bend right=50]
          (obj15);
    \draw (obj23) to[bend right=50]
          (obj25) to[bend right=50]
          (obj27) to[bend right=50]
          (obj29) to[bend right=10]
          (obj44);

    \draw[dash pattern=on 2pt off 2pt on 2pt off 2pt on 2pt off 2pt on
2pt off 2pt on 2pt off 40pt on 2pt off 2pt on 2pt off 2pt on 2pt off
2pt on 2pt off 2pt] (obj15) to[bend right=5] (obj23);

    \node[inner sep=0pt] at (-8,-13.7) {$\mapsto$};

    \polygonlabel{(-4,-13.7)}{obj}{45}{2.5}
      {1,,2,,,,,,,,,,u+4,,,,u+k+2,,,,,,,,u+a+3,,,,,,u+a+l+2,,,,,,,N-v,,,,,,N,}

    \draw[line width=2.5pt,black] (obj1) to[bend left=50] (obj44);
    \draw (obj3)  to[bend right=50]
          (obj13) to[bend right=50]
          (obj15) to[bend right=50]
          (obj17);
    \draw (obj25) to[bend right=50]
          (obj27) to[bend right=50]
          (obj29) to[bend right=50]
          (obj31) to[bend right=10]
          (obj38) to[bend right=20]
          (obj3);

    \draw[dash pattern=on 2pt off 2pt on 2pt off 2pt on 2pt off 2pt on
2pt off 2pt on 2pt off 40pt on 2pt off 2pt on 2pt off 2pt on 2pt off
2pt on 2pt off 2pt] (obj17) to[bend right=5] (obj25);

    \node[inner sep=0pt] at (0,-13.7) {$\mapsto$};

    \polygonlabel{( 4,-13.7)}{obj}{45}{2.5}
      {1,,,,,,v+1,,,,,,,,,,u+v+3,,,,u+v+k+1,,,,,,,,u+v+a+2,,,,,,u+v+a+l+1,,,,,,,N-1,,N,}

    \draw[line width=2.5pt,black] (obj1) to[bend left=50] (obj44);
    \draw (obj7)  to[bend right=50]
          (obj17) to[bend right=50]
          (obj19) to[bend right=50]
          (obj21);
    \draw (obj29) to[bend right=50]
          (obj31) to[bend right=50]
          (obj33) to[bend right=50]
          (obj35) to[bend right=10]
          (obj42) to[bend right=20]
          (obj7);

    \draw[dash pattern=on 2pt off 2pt on 2pt off 2pt on 2pt off 2pt on
2pt off 2pt on 2pt off 40pt on 2pt off 2pt on 2pt off 2pt on 2pt off
2pt on 2pt off 2pt] (obj21) to[bend right=5] (obj29);

\end{tikzpicture}
\end{center}
\caption{Pseudo-rotation of a block}
\label{fig:31a}
\end{figure}

\begin{figure}
\begin{center}
\begin{tikzpicture}
    \node[inner sep=0pt] at (-8,-21) {$\mapsto$};

    \polygonlabel{(-4,-21)}{obj}{45}{2.5}
      {1,,,,,,,,,,,,,,,,,,u+v+4,,,,u+v+k+2,,,,,,,,u+v+a+3,,,,,,u+v+a+l+2,,,,,,,N,}

    \draw[line width=2.5pt,black] (obj1) to[bend left=50] (obj44);
    \draw (obj1)  to[bend right=10]
          (obj19) to[bend right=50]
          (obj21) to[bend right=50]
          (obj23);
    \draw (obj31) to[bend right=50]
          (obj33) to[bend right=50]
          (obj35) to[bend right=50]
          (obj37) to[bend right=30]
          (obj44);

    \draw[dash pattern=on 2pt off 2pt on 2pt off 2pt on 2pt off 2pt on
2pt off 2pt on 2pt off 40pt on 2pt off 2pt on 2pt off 2pt on 2pt off
2pt on 2pt off 2pt] (obj23) to[bend right=5] (obj31);

    \node[inner sep=0pt] at ( 0,-21) {$\mapsto$};

    \polygonlabel{( 4,-21)}{obj}{45}{2.5}
      {1,,2,,,,,,,,,,,,z+2,,,,z+k,,,,,,,,z+a+1,,,,,,,,z+a+l+1,,,,,,,,,N,}

    \draw[line width=2.5pt,black] (obj1) to[bend left=50] (obj44);
    \draw (obj3)  to[bend right=10]
          (obj15) to[bend right=50]
          (obj17) to[bend right=50]
          (obj19);
    \draw (obj27) to[bend right=50]
          (obj29) to[bend right=50]
          (obj31) to[bend right=50]
          (obj33) to[bend right=50]
          (obj35) to[bend right=30]
          (obj3);

    \draw[dash pattern=on 2pt off 2pt on 2pt off 2pt on 2pt off 2pt on
2pt off 2pt on 2pt off 40pt on 2pt off 2pt on 2pt off 2pt on 2pt off
2pt on 2pt off 2pt] (obj19) to[bend right=5] (obj27);

    \node[inner sep=0pt] at (-8,-28) {$\mapsto$};

    \polygonlabel{(-4,-28)}{obj}{45}{2.5}
      {1,,,,,,,,,N-z-a-l,,,,,,,,,,,,N-a-l,,,,N-a-l+k-2,,,,,,,,N-l-1\!\!\!\!,,,,,,,,N-1,,N,}

    \draw[line width=2.5pt,black] (obj1) to[bend left=50] (obj44);
    \draw (obj10)  to[bend right=10]
          (obj22) to[bend right=50]
          (obj24) to[bend right=50]
          (obj26);
    \draw (obj34) to[bend right=50]
          (obj36) to[bend right=50]
          (obj38) to[bend right=50]
          (obj40) to[bend right=50]
          (obj42) to[bend right=30]
          (obj10);

    \draw[dash pattern=on 2pt off 2pt on 2pt off 2pt on 2pt off 2pt on
2pt off 2pt on 2pt off 40pt on 2pt off 2pt on 2pt off 2pt on 2pt off
2pt on 2pt off 2pt] (obj26) to[bend right=5] (obj34);

    \node[inner sep=0pt] at ( 0,-28) {$\mapsto$};

    \polygonlabel{( 4,-28)}{obj}{45}{2.5}
      {1,,,,,,,,,,,,,,,,,,,,,,,N-a-l+1,,,,N-a-l+k-1,,,,,,,,N-l,,,,,,,,N,}

    \draw[line width=2.5pt,black] (obj1) to[bend left=50] (obj44);
    \draw (obj1)  to[bend right=10]
          (obj24) to[bend right=50]
          (obj26) to[bend right=50]
          (obj28);
    \draw (obj36) to[bend right=50]
          (obj38) to[bend right=50]
          (obj40) to[bend right=50]
          (obj42) to[bend right=50]
          (obj44);

    \draw[dash pattern=on 2pt off 2pt on 2pt off 2pt on 2pt off 2pt on
2pt off 2pt on 2pt off 40pt on 2pt off 2pt on 2pt off 2pt on 2pt off
2pt on 2pt off 2pt] (obj28) to[bend right=5] (obj36);

    \node[inner sep=0pt] at (-8,-35) {$\mapsto$};

    \polygonlabel{(-4,-35)}{obj}{45}{2.5}
      {1,,,,,,,,l+1,,,,,,,,,,,,,,,,,,,,,,,N-a+1,,,,N-a+k-1,,,,,,,,N,}

    \draw[line width=2.5pt,black] (obj1) to[bend left=50] (obj44);
    \draw (obj1) to[bend right=50]
          (obj3) to[bend right=50]
          (obj5) to[bend right=50]
          (obj7) to[bend right=50]
          (obj9) to[bend right=10]
          (obj32) to[bend right=50]
          (obj34) to[bend right=50]
          (obj36);

    \draw[dash pattern=on 2pt off 2pt on 2pt off 2pt on 2pt off 2pt on
2pt off 2pt on 2pt off 40pt on 2pt off 2pt on 2pt off 2pt on 2pt off
2pt on 2pt off 2pt] (obj36) to[bend right=5] (obj44);

\end{tikzpicture}
\end{center}
\caption{Pseudo-rotation of a block}
\label{fig:31b}
\end{figure}
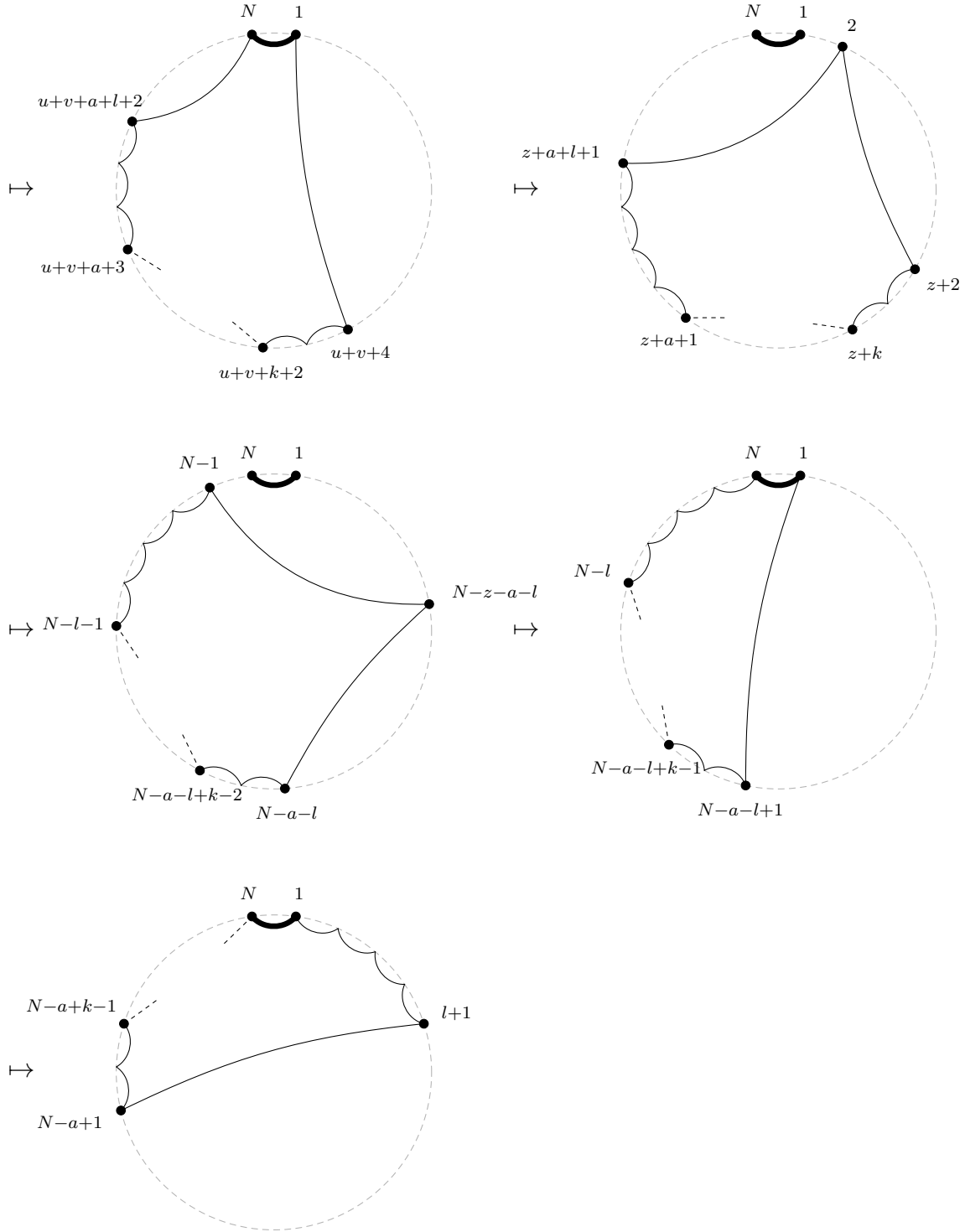

\medskip\noindent
{\sc Step 3.} Now we consider the case of a block that contains
more than one maximal sequence of successive elements.
Let the block contain
$N,1,2,\dots,k$, and let the maximal sequence of successive elements
(cyclically) preceding~$N$ be $a+1,\dots,a+l$. The first image
in \Cref{fig:31a} provides an illustration.

We now apply $\rotA$. Thereby, our block turns into
$\{N-u,2,\dots,k+1,*,a+2,\dots,a+l+1\}$;
see the second image in \Cref{fig:31a}. Here, the symbol $*$ indicates
further possible elements in the block lying between $k+1$ and $a+2$.
The next $u-1$ applications of~$\rotA$ act as ordinary
rotations on our block; see the third image in \Cref{fig:31a}.

By a further application of $\rotA$, the maximal sequence of
$k$~successive elements in the block is shortened by one element.
The result is that our block has become $\{N,1,\break u+3,\dots,u+k+1,*,
u+a+2,\dots,u+a+l+1\}$; see the fourth image in \Cref{fig:31a}.

In a sense, we already know from Step~2 what happens next: the next
application of~$\rotA$ removes $N$ and $1$ from our block,
rotates the other elements by one unit, and adds $N-v$ and $2$ to it,
for a suitable positive integer~$v$; see the fifth image in
\Cref{fig:31a}. There follow $v-1$ applications of~$\rotA$
which act as ordinary rotations on our block; see the sixth image in
\Cref{fig:31a}. The following application of~$\rotA$ then
turns our block into the form $\{N,1,u+v+4,\dots,u+v+k+2,*,
u+v+a+3,\dots,u+v+a+l+2\}$; see the first image in \Cref{fig:31b},
a form we have already seen before as comparison with the fourth image
in \Cref{fig:31a} reveals. Thus, this procedure may repeat itself
several times until, eventually, we reach the situation where our
block becomes $\{2,z+2,\dots,z+k,*,z+a+1,\dots,z+a+l+1\}$; see the second image
in \Cref{fig:31b}. It should be noted that the maximal sequence of
successive elements consisting originally of $l$~elements now has
been enlarged to the $l+1$ elements $z+a+1,\dots,z+a+l+1$.

Now follow $N-z-a-l-2$ applications of~$\rotA$, each of them
acting as ordinary rotations on our block of interest; see the third
image in \Cref{fig:31b}. The next application of~$\rotA$\break
makes us ``lose'' the ``isolated'' element $N-z-a-l$, so that we
obtain the block\break $\{N-a-l+1,\dots,N-a-l+k-1,*,N-l,\dots,N,1\}$;
see the fourth image in \Cref{fig:31b}. Finally, $l$ further
applications of~$\rotA$ let us arrive at
$\{N,1,\dots,l+1,N-a+1,\dots,N-a+k-1,*\}$; see the last image in
\Cref{fig:31b}. 

At this point, let us summarise what happened. Comparison of the
first image in \Cref{fig:31a} and the last image in \Cref{fig:31b}
shows that the (cyclically) maximal sequence of successive elements 
that originally contained $N$ and $1$ at its beginning
(namely $N,1,\dots,k$ in the first image in \Cref{fig:31a}) got rotated
further but ``lost'' its first two elements.
(It became $N-a+1,\dots,N-a+k-1$; see the last image in
\Cref{fig:31b}.) By contrast, the maximal
sequence of successive elements that preceded that sequence
(namely $a+1,\dots,a+l$ in the first image in \Cref{fig:31a}) now got
rotated to the top. It gained two elements so that it now contains 
$N$ and $1$ at its beginning. (It became $N,1,\dots,l+1$; see the last
image in \Cref{fig:31b}.) 

This process now will repeat itself. Every time a maximal sequence
of successive elements is rotated away from $\{N,1\}$, it will
``lose'' two elements, and every time a maximal sequence of successive
elements is moved into $\{N,1\}$ it will gain two elements.
In particular, such a maximal sequence is almost always rotated
ordinarily. At one point, however, it ``loses'' its first element,
and at a later point it ``gains'' a new last element. These two events
together explain an extra ``jump'' of~$1$. A further such ``jump''
happens when the maximal sequence is moved into $\{N,1\}$.
Thus, after $N-2$ applications of~$\rotA$, our block
will be back in its original position.

\medskip\noindent
{\sc Step 4.} We have shown that, for a non-crossing partition~$\pi$
without singleton block, a block containing a maximal sequence of
(cyclically) successive elements $N,1,\dots,k$ will be back in its
original position after $N-2$ applications of the
pseudo-rotation~$\rotA$. If we now consider any other block,
then we may apply $\rotA$ as many times as it is necessary
to move that block into a position where it contains a (cyclically) maximal
sequence of successive elements beginning with $N$ and $1$. Then the
above arguments apply. This shows that also any other block returns
back to its original position after $N-2$ applications of~$\rotA$. 

Since, in Step~1, we have shown that it suffices to prove the
assertion for non-crossing partitions without singleton blocks, this
completes the proof of the theorem.
\qed

\subsection{The order of the positive Kreweras map in
  type~$B_n$: Proof of \Cref{lem:N-2B}}
\label{app:order-B}

This proof is rather similar in spirit to the one of \Cref{lem:N-2}.
We shall therefore be somewhat briefer here.

Again, we proceed in several steps. 

{
\medskip
\hangindent2\parindent\hangafter1
{\sc Step 1.} We argue that singleton blocks are not ``important", in
  the sense that it suffices to prove the claim for non-crossing
  partitions without singleton blocks.

\hangindent2\parindent\hangafter1
{\sc Step 2.} 
We prove that a block of~$\pi$ of the form 
$\{\overline N,1,2,\dots,k\}$, with $k\le N$,
becomes its negative, $\{N,\overline 1,\overline 2,\dots,\overline k\}$,
after $N-1$ applications of~$\rotB$.

\hangindent2\parindent\hangafter1
{\sc Step 3.} We prove that a block of~$\pi$ 
containing $\overline N$ and $1$ but not~$\overline {N-1}$
becomes its negative after $N-1$ applications of~$\rotB$.

\hangindent2\parindent\hangafter1
{\sc Step 4.} We argue that Steps~1--3 suffice to complete the proof.
\par
}

\medskip
We now discuss each step somewhat more in detail.

\medskip\noindent
{\sc Step 1.} The arguments here are completely analogous to the ones
in Step~1 of the preceding proof of \Cref{lem:N-2}. We leave the details to the
reader.

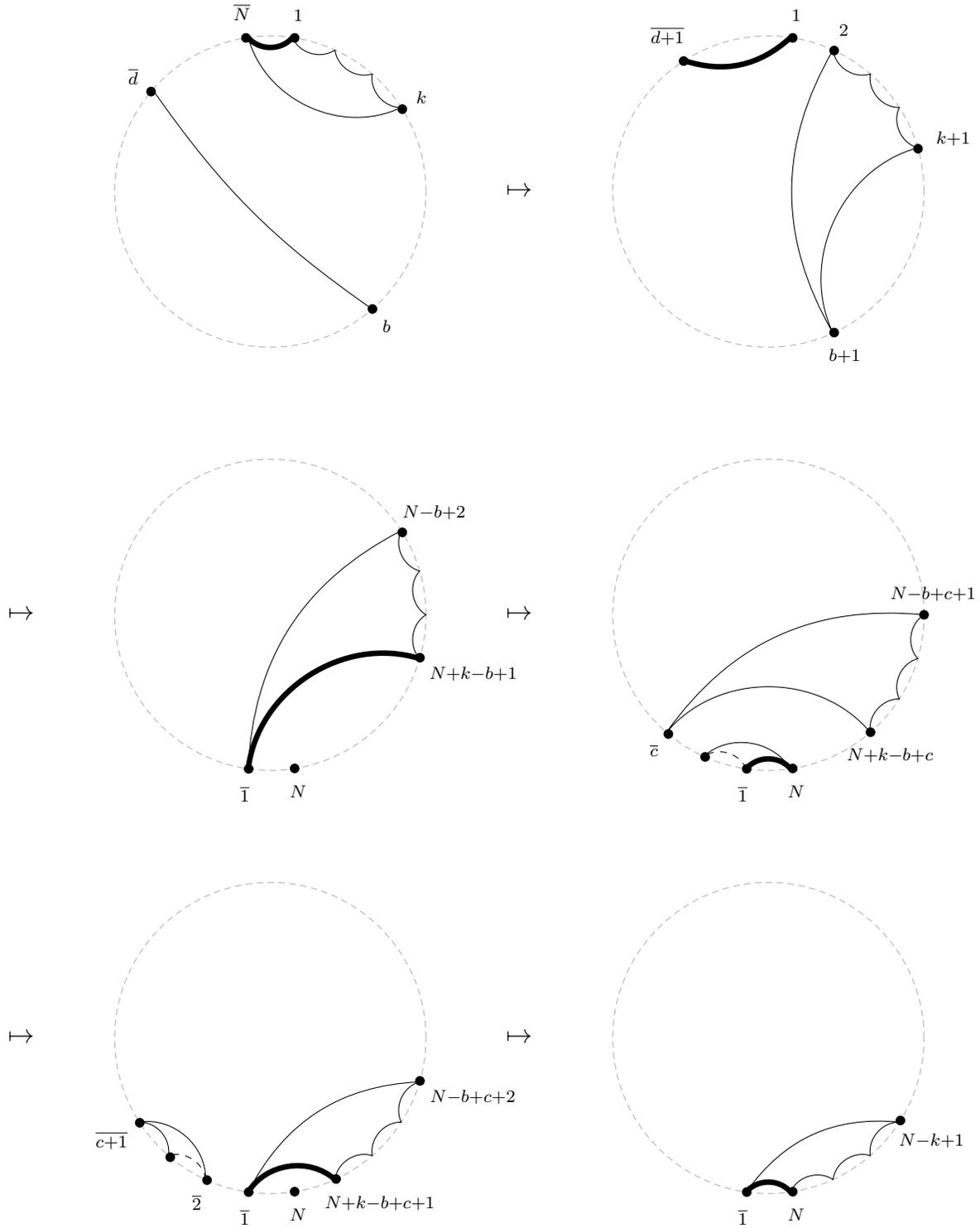
\begin{figure}
\begin{center}
    \begin{tikzpicture}[scale=1]
    \polygonlabel{(-4,0)}{obj}{44}{2.5}
      {1,,,,,,k,,,,,,,,,,b,,,,,,,,,,,,,,,,,,,,,\overline{d},,,,,\overline{N},}

    \draw[line width=2.5pt,black] (obj1) to[bend left=50] (obj43);
    \draw (obj1) to[bend right=50]
          (obj3) to[bend right=50]
          (obj5) to[bend right=50]
          (obj7) to[bend left=50]
          (obj43);

    \draw (obj17) to [bend left=10] (obj38);

    \node[inner sep=0pt] at (0,0) {$\mapsto$};

    \polygonlabel{( 4,0)}{obj}{44}{2.5}
      {1,,2,,,,,,k+1,,,,,,,,,,b+1,,,,,,,,,,,,,,,,,,,,,\overline{d+1},,,,,}

    \draw[line width=2.5pt,black] (obj1) to[bend left=30] (obj40);
    \draw (obj3)  to[bend right=50]
          (obj5)  to[bend right=50]
          (obj7)  to[bend right=50]
          (obj9)  to[bend right=50]
          (obj19) to[bend left =30]
          (obj3);

    \node[inner sep=0pt] at (-8,-6.8) {$\mapsto$};

    \polygonlabel{(-4,-6.8)}{obj}{44}{2.5}
      {,,,,,,N-b+2,,,,,,N+k-b+1,,,,,,,,N,,\overline{1},,,,,,,,,,,,,,,,,,,,,,}

    \draw[line width=2.5pt,black] (obj13) to[bend right=50] (obj23);
    \draw (obj7)  to[bend right=50]
          (obj9)  to[bend right=50]
          (obj11)  to[bend right=50]
          (obj13)  to[bend right=50]
          (obj23) to[bend left =30]
          (obj7);

    \node[inner sep=0pt] at (0,-6.8) {$\mapsto$};

    \polygonlabel{( 4,-6.8)}{obj}{44}{2.5}
      {,,,,,,,,,,\hspace{-25pt}\raise20pt\hbox{$\scriptstyle N-b+c+1$},,,,,,N+k-b+c,,,,N,,\overline{1},,{\quad },,\overline{c},,,,,,,,,,,,,,,,,}

    \draw[line width=2.5pt,black] (obj21) to[bend right=50] (obj23);
    \draw (obj11)  to[bend right=50]
          (obj13)  to[bend right=50]
          (obj15)  to[bend right=50]
          (obj17)  to[bend right=50]
          (obj27) to[bend left =30]
          (obj11);

    \draw (obj21) to[bend right=50]
          (obj23);
    \draw (obj25) to[bend left =50]
          (obj21);
    \draw[dashed] (obj23) to[bend right=50]
          (obj25);

    \node[inner sep=0pt] at (-8,-13.6) {$\mapsto$};

    \polygonlabel{(-4,-13.6)}{obj}{44}{2.5}
      {,,,,,,,,,,,,N-b+c+2,,,,,,\hspace*{30pt}N+k-b+c+1,,N,,\overline{1},,\overline{2},,{\quad },,\overline{c+1},,,,,,,,,,,,,,}

    \draw[line width=2.5pt,black] (obj19) to[bend right=50] (obj23);
    \draw (obj13) to[bend right=50]
          (obj15) to[bend right=50]
          (obj17) to[bend right=50]
          (obj19) to[bend right=50]
          (obj23) to[bend left =30]
          (obj13);

    \draw[dashed] (obj25) to[bend right=50]
          (obj27);
    \draw (obj27) to[bend right=50]
          (obj29) to[bend left =50]
          (obj25);

  \node[inner sep=0pt] at (0,-13.6) {$\mapsto$};

  \polygonlabel{( 4,-13.6)}{obj}{44}{2.5}
      {,,,,,,,,,,,,,,N-k+1,,,,,,N,,\overline{1},,,,,,,,,,,,,,,,,,,,,}

    \draw[line width=2.5pt,black] (obj21) to[bend right=50] (obj23);
    \draw (obj15)  to[bend right=50]
          (obj17)  to[bend right=50]
          (obj19)  to[bend right=50]
          (obj21)  to[bend right=50]
          (obj23) to[bend left =30]
          (obj15);

\end{tikzpicture}
\end{center}
\caption{Pseudo-rotation of a block of successive elements}
\label{fig:32a}
\end{figure}

\medskip\noindent
{\sc Step 2.} Let $\{\overline N,1,2,\dots,k\}$ be a block of~$\pi$. 
There are two cases depending on whether or not there is a block
other than the above block and its ``companion''
$\{N,\overline 1,\overline 2,\dots,\overline k\}$  
containing both positive and negative elements. They are illustrated in
\Cref{fig:32a} respectively in \Cref{fig:32b}.

Let us start with the case where there is such an additional block
containing both positive and negative elements; see the first image in
\Cref{fig:32a}. 
The first application of~$\rotB$ maps our block to
$\{2,\dots,k+1,b+1\}$, for a suitable positive integer~$b$; see
the second image in \Cref{fig:32a}. Each of the next $N-b$ applications
of~$\rotB$ rotates this block by one unit, so that we obtain
$\{N-b+2,\dots,N-b+k+1,\overline1\}$. See the third image
in \Cref{fig:32a} for the resulting non-crossing partition. 

There will be some further applications of~$\rotB$ which rotate
this block ordinarily, until a situation is reached which is indicated
in the fourth image of \Cref{fig:32a} (which assumes that $c-1$ applications
of~$\rotB$ have taken place, with $c\ge2$): there is a block
containing~$N$ and otherwise only negative elements, and
the connection between $\overline c$ and $N+k-b+c$ is the closest to
this block which is between a negative and a positive element.
The next application of~$\rotB$ then turns our
block of interest into $\{N-b+c+2,\dots,N+k-b+c+1,\overline{1}\}$;
see the fifth image in \Cref{fig:32a}.

Comparison with the third image in \Cref{fig:32a} shows that, except
for the location of the maximal sequence of 
(cyclically) successive elements, we are in the
same situation again. In other words, we may now run through the
situations illustrated by the third, fourth, and fifth image in
\Cref{fig:32a} several times, until our maximal sequence of successive
elements eventually reaches~$N$ and thereby joins with~$\overline1$ to form
the block $\{N-k+1,\dots,N,\overline1\}$; see
the sixth image in \Cref{fig:32a}. 

Finally, another $k-1$ applications of~$\rotB$ rotate this block
to $\{N,\overline1,\overline2,\dots,\overline k\}$, the negative of the
block with which we started. Thus, we are back
to the situation illustrated by the first image in \Cref{fig:32a}. 

It can be checked that, in total, we performed $N-1$ applications of~$\rotB$. 

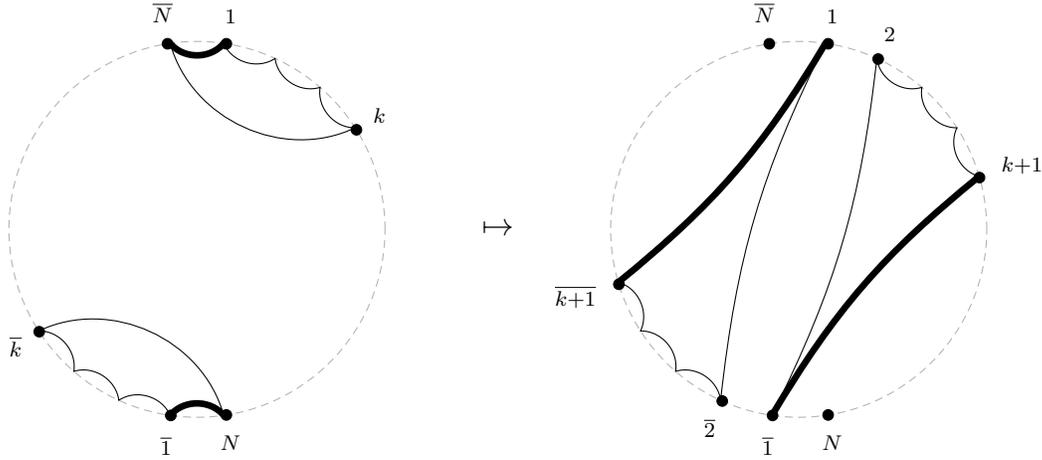
\begin{figure}
\begin{center}
    \begin{tikzpicture}[scale=1]
    \polygonlabel{(-4,0)}{obj}{44}{2.5}
      {1,,,,,,k,,,,,,,,,,,,,,N,,\overline{1},,,,,,\overline{k},,,,,,,,,,,,,,\overline{N},}

    \draw[line width=2.5pt,black] (obj1) to[bend left=50] (obj43);
    \draw[line width=2.5pt,black] (obj21) to[bend right=50] (obj23);
    \draw (obj1) to[bend right=50]
          (obj3) to[bend right=50]
          (obj5) to[bend right=50]
          (obj7) to[bend left=50]
          (obj43);

    \draw (obj21) to[bend right=50]
          (obj23) to[bend right=50]
          (obj25) to[bend right=50]
          (obj27) to[bend right=50]
          (obj29) to[bend left=50]
          (obj21);

    \node[inner sep=0pt] at (0,0) {$\mapsto$};

    \polygonlabel{( 4,0)}{obj}{44}{2.5}
      {1,,2,,,,,,k+1,,,,,,,,,,,,N,,\overline{1},,\overline{2},,,,,,\overline{k+1},,,,,,,,,,,,\overline{N},}

    \draw[line width=2.5pt,black] (obj1) to[bend left=10] (obj31);
    \draw[line width=2.5pt,black] (obj23) to[bend left=10] (obj9);
    \draw (obj23) to[bend right=10]
          (obj3) to[bend right=50]
          (obj5) to[bend right=50]
          (obj7) to[bend right=50]
          (obj9);

    \draw (obj1) to[bend right=10]
          (obj25) to[bend right=50]
          (obj27) to[bend right=50]
          (obj29) to[bend right=50]
          (obj31);

\end{tikzpicture}
\end{center}
\caption{Pseudo-rotation of a block of successive elements}
\label{fig:32b}
\end{figure}

\medskip
The second case, namely when there is no additional block containing
both positive and negative elements, can be quickly reduced to the
previous case; see \Cref{fig:32b}. Namely, the first application
of~$\rotB$ maps our block to $\{2,3,\dots,k+1,\overline1\}$.
This is the same situation as the one which we encountered in the
third image in \Cref{fig:32a} (namely the one with $b=N$). Thus,
the subsequent arguments from the previous case apply again and
show that also here $N-1$ applications of~$\rotB$ map
the pair of blocks $\{\overline N,1,\dots,k\}$ and
$\{N,\overline1,\dots,\overline k\}$ to itself.

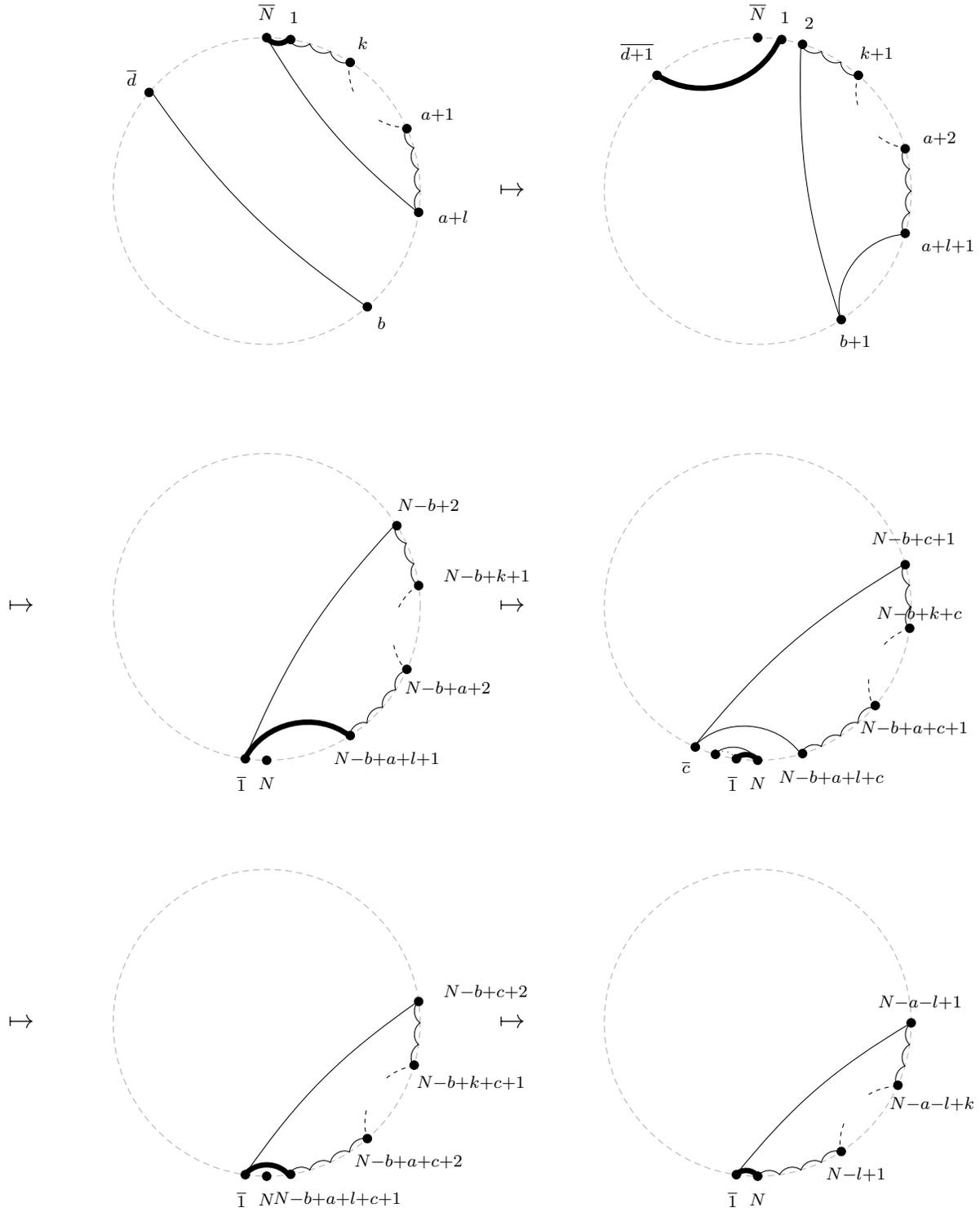
\begin{figure}
\begin{center}
    \begin{tikzpicture}[scale=1]
    \polygonlabel{(-4,0)}{obj}{44}{2.5}
      {1,,,k,,,,a+1,,,,a+l,,,,,b,,,,,,,,,,,,,,,,,,,,,\overline{d},,,,,,\overline{N}}

    \draw[line width=2.5pt,black] (obj1) to[bend left=50] (obj44);
    \draw (obj1) to[bend right=50]
          (obj2) to[bend right=50]
          (obj3) to[bend right=50]
          (obj4);
    \draw (obj8) to[bend right=50]
          (obj9) to[bend right=50]
          (obj10) to[bend right=50]
          (obj11) to[bend right=50]
          (obj12);

    \draw[dash pattern=on 2pt off 2pt on 2pt off 2pt on 2pt off 2pt on
2pt off 18pt on 2pt off 2pt on 2pt off 2pt] (obj4) to[bend right=50]
(obj8);

    \draw (obj12) to [bend left=10] (obj44);
    \draw (obj17) to [bend left=10] (obj38);

    \node[inner sep=0pt] at (0,0) {$\mapsto$};

    \polygonlabel{( 4,0)}{obj}{44}{2.5}
      {1,2,,,k+1,,,,a+2,,,,a+l+1,,,,,b+1,,,,,,,,,,,,,,,,,,,,,\overline{d+1},,,,,\overline{N}}

    \draw[line width=2.5pt,black] (obj1) to[bend left=50] (obj39);
    \draw (obj2) to[bend right=50]
          (obj3) to[bend right=50]
          (obj4) to[bend right=50]
          (obj5);
    \draw (obj9) to[bend right=50]
          (obj10) to[bend right=50]
          (obj11) to[bend right=50]
          (obj12) to[bend right=50]
          (obj13) to[bend right=50]
          (obj18);

    \draw[dash pattern=on 2pt off 2pt on 2pt off 2pt on 2pt off 2pt on
2pt off 18pt on 2pt off 2pt on 2pt off 2pt] (obj5) to[bend right=50]
(obj9);

    \draw (obj2) to [bend right=10] (obj18);
    \draw (obj1) to [bend left=50] (obj39);

    \node[inner sep=0pt] at (-8,-6.78) {$\mapsto$};

    \polygonlabel{(-4,-6.78)}{obj}{44}{2.5}
      {,,,,,,N-b+2,,,N-b+k+1,,,,N-b+a+2,,,,\hspace{20pt}N-b+a+l+1,,,,N,\overline{1},,,,,,,,,,,,,,,,,,,,,}

    \draw[line width=2.5pt,black] (obj18) to[bend right=50] (obj23);
    \draw (obj7) to[bend right=50]
          (obj8) to[bend right=50]
          (obj9) to[bend right=50]
          (obj10);
    \draw (obj14) to[bend right=50]
          (obj15) to[bend right=50]
          (obj16) to[bend right=50]
          (obj17) to[bend right=50]
          (obj18) to[bend right=50]
          (obj23);

    \draw[dash pattern=on 2pt off 2pt on 2pt off 2pt on 2pt off 2pt on
2pt off 18pt on 2pt off 2pt on 2pt off 2pt] (obj10) to[bend right=50]
(obj14);

    \draw (obj7) to [bend right=10] (obj23);

    \node[inner sep=0pt] at (0,-6.78) {$\mapsto$};

    \polygonlabel{( 4,-6.78)}{obj}{44}{2.5}
      {,,,,,,,,\hspace{-25pt}\raise15pt\hbox{$\scriptstyle N-b+c+1$},,,\hspace{-25pt}\raise20pt\hbox{$\scriptstyle N-b+k+c$},,,,\hspace{12pt}N-b+a+c+1,,,,\hspace{20pt}N-b+a+l+c,,N,\overline{1},\hspace{1pt},\overline{c},,,,,,,,,,,,,,,,,,,}

    \draw[line width=2.5pt,black] (obj22) to[bend right=50] (obj23);
    \draw (obj9) to[bend right=50]
          (obj10) to[bend right=50]
          (obj11) to[bend right=50]
          (obj12);
    \draw (obj16) to[bend right=50]
          (obj17) to[bend right=50]
          (obj18) to[bend right=50]
          (obj19) to[bend right=50]
          (obj20) to[bend right=50]
          (obj25);

    \draw (obj22) to[bend right=50]
          (obj23);
    \draw[dotted] (obj23) to[bend right=50]
          (obj24);
    \draw (obj24) to[bend left=50]
          (obj22);

    \draw[dash pattern=on 2pt off 2pt on 2pt off 2pt on 2pt off 2pt on
2pt off 18pt on 2pt off 2pt on 2pt off 2pt] (obj12) to[bend right=50]
(obj16);

    \draw (obj9) to [bend right=10] (obj25);

    \node[inner sep=0pt] at (-8,-13.56) {$\mapsto$};

    \polygonlabel{(-4,-13.56)}{obj}{44}{2.5}
      {,,,,,,,,,N-b+c+2,,,N-b+k+c+1,,,,\hspace{20pt}N-b+a+c+2,,,,\hspace{40pt}N-b+a+l+c+1,N,\overline{1},,,,,,,,,,,,,,,,,,,,,}

    \draw[line width=2.5pt,black] (obj21) to[bend right=50] (obj23);
    \draw (obj10) to[bend right=50]
          (obj11) to[bend right=50]
          (obj12) to[bend right=50]
          (obj13);
    \draw (obj17) to[bend right=50]
          (obj18) to[bend right=50]
          (obj19) to[bend right=50]
          (obj20) to[bend right=50]
          (obj21) to[bend right=50]
          (obj23);

    \draw[dash pattern=on 2pt off 2pt on 2pt off 2pt on 2pt off 2pt on
2pt off 18pt on 2pt off 2pt on 2pt off 2pt] (obj13) to[bend right=50]
(obj17);

    \draw (obj10) to [bend right=10] (obj23);

  \node[inner sep=0pt] at (0,-13.56) {$\mapsto$};

  \polygonlabel{( 4,-13.56)}{obj}{44}{2.5}
      {,,,,,,,,,,\hspace{-25pt}\raise20pt\hbox{$\scriptstyle N-a-l+1$},,,\hspace{-5pt}N-a-l+k,,,,N-l+1,,,,N,\overline{1},,,,,,,,,,,,,,,,,,,,,}

    \draw[line width=2.5pt,black] (obj22) to[bend right=50] (obj23);
    \draw (obj11) to[bend right=50]
          (obj12) to[bend right=50]
          (obj13) to[bend right=50]
          (obj14);
    \draw (obj18) to[bend right=50]
          (obj19) to[bend right=50]
          (obj20) to[bend right=50]
          (obj21) to[bend right=50]
          (obj22) to[bend right=50]
          (obj23);

    \draw[dash pattern=on 2pt off 2pt on 2pt off 2pt on 2pt off 2pt on
2pt off 18pt on 2pt off 2pt on 2pt off 2pt] (obj14) to[bend right=50]
(obj18);

    \draw (obj11) to [bend right=10] (obj23);

\end{tikzpicture}
\end{center}
\caption{Pseudo-rotation of a block}
\label{fig:33a}
\end{figure}

\begin{figure}
\begin{center}
    \begin{tikzpicture}[scale=1]
    \polygonlabel{(-4,0)}{obj}{44}{2.5}
      {,,,,,,,,,,,,,,N-a,,,N-a+k-1,,,,N,\overline{1},,,,\overline{l},,,,,,,,,,,,,,,,,}

    \draw[line width=2.5pt,black] (obj22) to[bend right=50] (obj23);
    \draw (obj15) to[bend right=50]
          (obj16) to[bend right=50]
          (obj17) to[bend right=50]
          (obj18);
    \draw (obj22) to[bend right=50]
          (obj23) to[bend right=50]
          (obj24) to[bend right=50]
          (obj25) to[bend right=50]
          (obj26) to[bend right=50]
          (obj27);

    \draw[dash pattern=on 2pt off 2pt on 2pt off 2pt on 2pt off 2pt on
2pt off 18pt on 2pt off 2pt on 2pt off 2pt] (obj18) to[bend right=50]
(obj22);

    \draw (obj15) to [bend right=10] (obj27);

    \node[inner sep=0pt] at (0,0) {$\mapsto$};

    \polygonlabel{( 4,0)}{obj}{44}{2.5}
      {,,,,,,,,,,,,,,,,,N-a+e,,,\hspace{40pt}N-a+k-1+e,N,\overline{1},\hspace*{1pt},\overline{e},,,,,\overline{l+e},,,,,,,,,,,,,,}

    \draw[line width=2.5pt,black] (obj22) to[bend right=50] (obj23);
    \draw (obj18) to[bend right=50]
          (obj19) to[bend right=50]
          (obj20) to[bend right=50]
          (obj21);
    \draw (obj25) to[bend right=50]
          (obj26) to[bend right=50]
          (obj27) to[bend right=50]
          (obj28) to[bend right=50]
          (obj29) to[bend right=50]
          (obj30);

    \draw[dash pattern=on 2pt off 2pt on 2pt off 2pt on 2pt off 2pt on
2pt off 18pt on 2pt off 2pt on 2pt off 2pt] (obj21) to[bend right=50]
(obj25);

    \draw (obj18) to [bend right=10] (obj30);
    \draw[dotted] (obj23) to[bend right=50]
          (obj24);
    \draw (obj24) to[bend left=50]
          (obj22);

    \node[inner sep=0pt] at (-8,-7) {$\mapsto$};

    \polygonlabel{(-4,-7)}{obj}{44}{2.5}
      {,,,,,,,,,,,,,,,,,N-a+e+1,,,\hspace{30pt}N-a+k+e,N,\overline{1},,,,\overline{e+2},,,,\overline{l+e+1},,,,,,,,,,,,,}

    \draw (obj18) to[bend right=50]
          (obj19) to[bend right=50]
          (obj20) to[bend right=50]
          (obj21);
    \draw (obj23) to[bend right=50]
          (obj27) to[bend right=50]
          (obj28) to[bend right=50]
          (obj29) to[bend right=50]
          (obj30) to[bend right=50]
          (obj31);
    \draw (obj18) to[bend right=20] (obj31);

    \draw[dash pattern=on 2pt off 2pt on 2pt off 2pt on 2pt off 2pt on
2pt off 20pt] (obj21) to[bend right=50,looseness=3] (obj23);
    \draw[line width=2.5pt,dash pattern=on 2pt off 2pt on 2pt off 2pt
on 2pt off 2pt on 2pt off 20pt] (obj23) to[bend left=50,looseness=3]
(obj21);

\end{tikzpicture}
\end{center}
\caption{Pseudo-rotation of a block}
\label{fig:33b}
\end{figure}
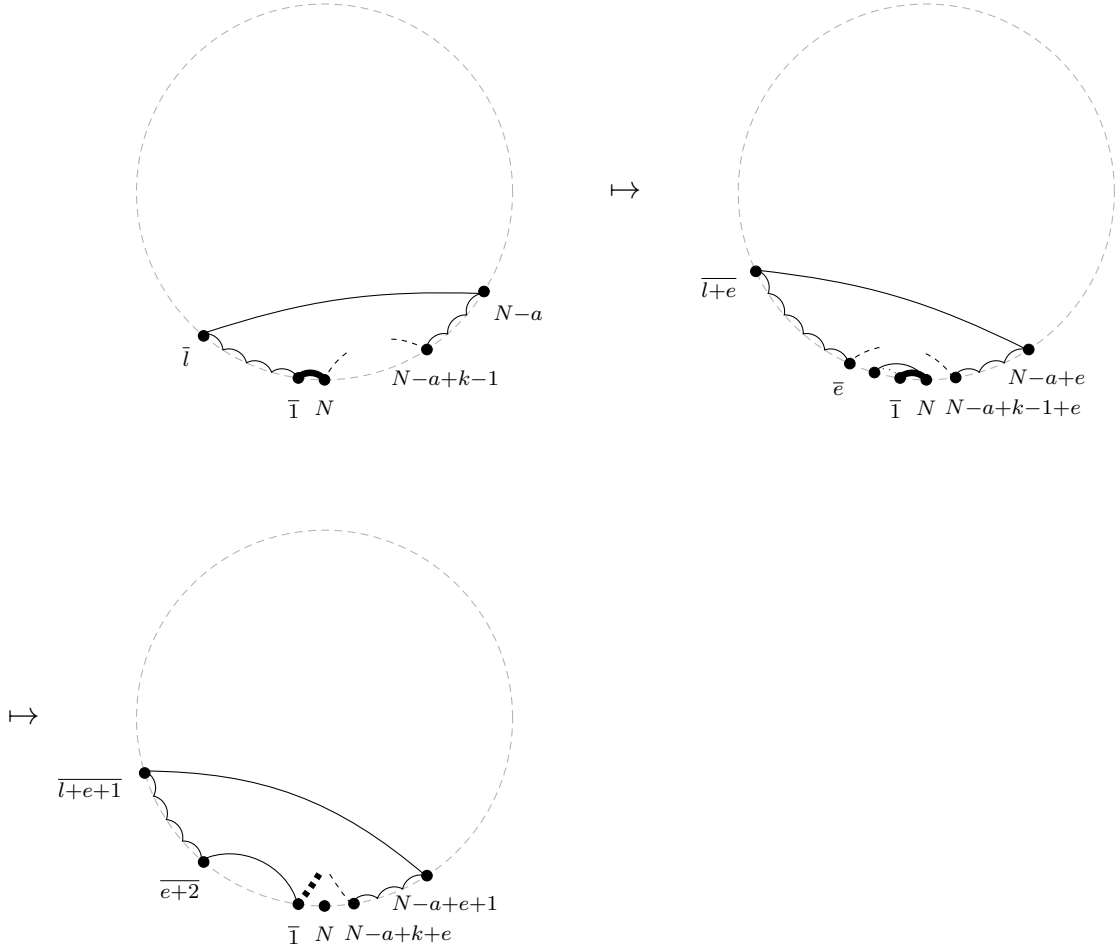

\medskip\noindent
{\sc Step 3.} Now we consider the case of a block that contains
more than one maximal sequence of successive elements.
Let the block contain $\overline N,1,2,\dots,k$ but not $\overline{N-1}$.
In principle, we distinguish between several cases depending on
whether further maximal sequences of successive elements contain positive
or negative elements. We shall see, however, that it suffices to look
at the following situation.

Let the maximal sequence of successive elements
(cyclically) preceding~$\overline N$ be $a+1,\dots,\break a+l$ (consisting of positive 
numbers). The first image in \Cref{fig:33a} provides an illustration.

The first application of~$\rotB$ maps this block to
$\{2,\dots,k+1,*,a+2,\dots,a+l+1,b+1\}$; see the second image
in \Cref{fig:33a}. (We tacitly include the Case~(2) of the map~$\rotB$
by the choice of~$b=N$ and the identification $b+1=N+1\equiv \overline1$.)
Here, as in the proof of \Cref{lem:N-2}, the symbol $*$ indicates
further possible elements in the block lying between $k+1$ and $a+2$.

The next applications of $\rotB$ act as ordinary
rotations on our block. Thereby, the element~$b+1$ will be rotated
into $\overline1$ (cf.\ the third image in \Cref{fig:33a}) and
eventually into the situation illustrated in the fourth image in
\Cref{fig:33a}. Here, the connection between $\overline c$ and
$N-b+a+l+c$ is supposed to be the closest to the block containing
$N$ and $\overline1$ among the blocks containing a positive and
a negative element. 

The next application of $\rotB$ brings us into the situation
illustrated in the fifth image in \Cref{fig:33a}. Comparison with
the third image shows that we shall now (possibly) revisit such
situations several times until our maximal sequence of $l$~successive 
elements reaches~$N$ and joins with~$\overline1$. See the sixth image
in \Cref{fig:33a}. 

Further applications of $\rotB$ map our block to
$\{N-a,\dots,N-a+k-1,*,N,\overline1,\dots,\overline l\}$; see
the first image in \Cref{fig:33b}.

\medskip
At this point, we should pause for a moment and make an
intermediate summary of what we achieved so far: we started with a
block containing two (or more) maximal sequences of
(cyclically) successive numbers,
one maximal sequence starting with $\overline N,1,\dots$, the one
preceding $\overline N$ being one consisting entirely of positive
numbers; see the first image in \Cref{fig:33a}. If we regard the
first image in \Cref{fig:33b}, and take into account that also the
``negative'' of the considered block is a block in the type~$B$
non-crossing partition, then we have now obtained a block with
one maximal sequence of successive elements starting with
$\overline N,1,\dots$, however with the preceding maximal sequence
of successive elements consisting entirely of \defn{negative} elements.
Thus, if we continue our analysis, then we will also have covered that
case.

\medskip
If we now continue applying the pseudo-rotation $\rotB$, then
we know what is going to happen: a few applications will act as ordinary
rotation by one unit on our block, until we reach the situation where
the connection between the maximal sequence of successive negative
numbers (the ``leading'' maximal sequence) 
and the preceding maximal sequence of successive numbers will be
next to a block containing $N,\overline1$, and otherwise only negative
numbers; see the second image in \Cref{fig:33b}. The next application
of~$\rotB$ makes the ``leading'' maximal sequence lose its
last element, connects what remains from it to $\overline1$, which
itself is connected to the ``next'' maximal sequence; see the third
image in \Cref{fig:33b}. Further applications of~$\rotB$ let
one potentially run through situations illustrated in the second and
third images in \Cref{fig:33b} several times, until the ``next'' maximal
sequence reaches $N$, thereby also joining with $\overline1$.
Thus, we have again reached a situation as in the first image in
\Cref{fig:33b}, with now two maximal sequences of successive numbers
on the negative side. 

\medskip
We may now summarise: whenever a maximal sequence of (cyclically)
successive elements in a block is in the position $\{\overline N,1,\dots,k\}$, 
then, when $\rotB$ is applied, it gets rotated by one unit,
it ``loses'' its last element, and in turn it ``gains'' an element in 
front of it (cf.\ the first and second images in \Cref{fig:33a}).
Subsequent applications of~$\rotB$ act as ordinary rotations
on this maximal sequence plus ``front runner'', until the front runner
moves beyond~$\overline1$. Upon further applications of~$\rotB$,
the ``front runner'' is eventually lost and replaced by~$\overline1$,
and this may happen several times (cf.\ the third, fourth, and fifth
images in \Cref{fig:33a}), until our maximal sequence has reached~$N$ and
thereby joins with $\overline1$. 

The short (and rough) version of this process is 
that our maximal sequence of successive elements gets rotated by one
unit by each application of~$\rotB$, and, on the way, ``loses''
its last element and eventually ``gains'' the element in front of it.
This implies that after $N-1$ applications of~$\rotB$ a block
is mapped to its ``mirror image'' consisting of all negatives of its
elements. Since we started with a type~$B$ non-crossing partition,
this shows that $\rotB^{N-1}$ acts as the identity on the pair
consisting of the block and its ``mirror image".

\medskip\noindent
{\sc Step 4.} So far, we have shown that, for a non-crossing partition~$\pi$
without singleton block, a block containing a maximal sequence of
(cyclically) successive elements $\overline N,1,\dots,k$ is mapped to
its negative or to itself (if it is a zero block) by~$\rotB^{N-1}$.
If we now consider any other block,
then we may apply $\rotB$ as many times as it is necessary
to move that block into a position where it contains a (cyclically) maximal
sequence of successive elements beginning with $\overline N$ and $1$. Then the
above arguments apply. This shows that also any other block is mapped to its
negative or to itself by~$N-1$ applications of~$\rotB$. 

Since, in Step~1, we have shown that it suffices to prove the
assertion for non-crossing partitions without singleton blocks, this
completes the proof of the theorem.
\qed

\subsection{The order of the positive Kreweras map in
  type~$D_n$: Proof of \Cref{lem:N-2D}}
\label{app:order-D}

As earlier, it is not difficult to see that it suffices
to prove the assertion of the theorem for non-crossing partitions
without singleton blocks.
So, let $\pi$ be a partition in $\mNCDPlus$ without singleton blocks.
We distinguish between the following three cases:

{
\medskip
\hangindent2\parindent\hangafter1
{\sc Case 1.} $\pi$ contains a zero block;

\hangindent2\parindent\hangafter1
{\sc Case 2.} $\pi$ contains no zero block and no bridging block, so that
the inner circle is split into two blocks;

\hangindent2\parindent\hangafter1
{\sc Case 3.} $\pi$ contains at least one bridging block.
\par
}

\medskip
We now discuss each case in more detail.

\medskip\noindent
{\sc Case 1.} When we apply $\rotD$ repeatedly
to a partition~$\pi$ in $\mNCDPlus$
with a zero block, then it is always Case~(1) of
\Cref{prop:2D} which occurs. If we ignore the elements of~$\pi$ on
the inner circle, then $\pi$ reduces to a partition $\hat\pi$ in
$\mNCBPlus[n-1]$, and the action of~$\rotD$ on~$\pi$ behaves like the action
of~$\rotB$ on $\hat\pi$ as described in Case~(1) of \Cref{prop:2B}. 
\Cref{lem:N-2B} with $N=m(n-1)$ says that $\rotB^{m(n-1)-1}$ acts on
$\hat\pi$ as the identity. Consequently, $\rotD^{m(n-1)-1}$ acts on
$\pi$ as the identity. 

\medskip\noindent
{\sc Case 2.} When we apply $\rotD$ repeatedly
to a partition~$\pi$ in $\mNCDPlus$
without a zero block and without any bridging block, 
then it is the Cases~(1) and~(3) of
\Cref{prop:2D} which occur. If we ignore the elements of~$\pi$ on
the inner circle, then $\pi$ reduces to a partition $\hat\pi$ in
$\mNCBPlus[n-1]$, and the action of~$\rotD$ on~$\pi$ behaves like the action
of~$\rotB$ on $\hat\pi$ as described in \Cref{prop:2B}, with Case~(1)
corresponding to Case~(1) of \Cref{prop:2D}, and Case~(2)
corresponding to Case~(3) of \Cref{prop:2D}. 
Here again, \Cref{lem:N-2B} with $N=m(n-1)$ says that $\rotB^{m(n-1)-1}$ acts on
$\hat\pi$ as the identity, and hence $\rotD^{m(n-1)-1}$ acts on
$\pi$ as the identity. 

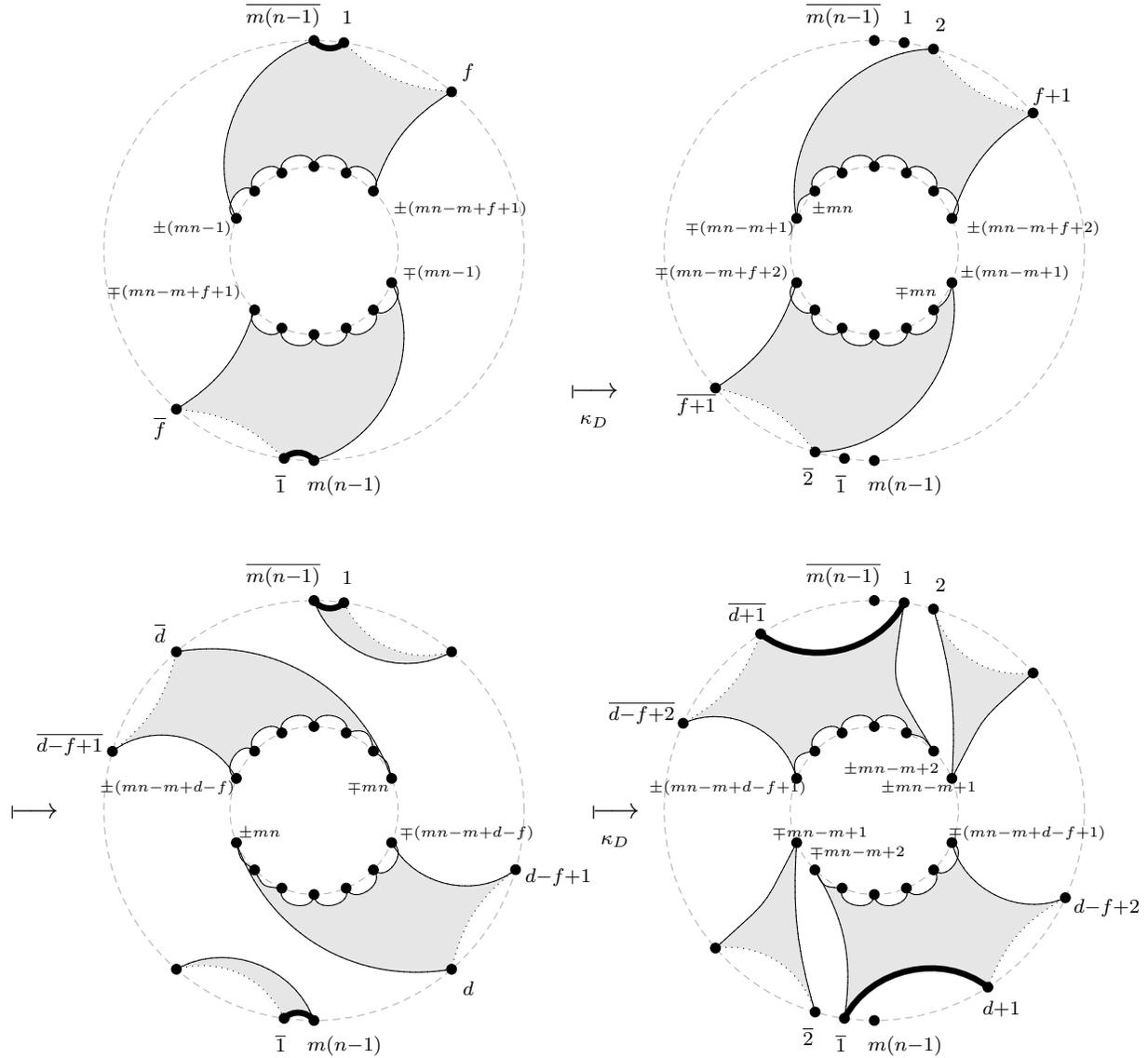
\begin{figure}
\begin{center}
  \begin{tikzpicture}[scale=1]
    \polygon{(-4,0)}{obj}{44}{3}
       {1,,,,f,,,,,,,,,,,,,,,,,\hspace*{25pt}m(n-1),\overline{1},,,,\overline{f},,,,,,,,,,,,,,,,,\overline{m(n-1)}\hspace*{25pt}}

    \polygoninner{(-4,0)}{objin}{16}{1.2}
      {\ ,\hspace*{85pt}\pm{(mn-m+f+1)},,,\hspace*{60pt}\mp(mn-1),\ ,\ ,\ ,\ ,\mp{(mn-m+f+1)}\hspace*{80pt},,,\pm{(mn-1)}\hspace*{55pt},\ ,\ ,\ }

    \draw[line width=2.5pt,black] (obj1) to[bend left=50] (obj44);
    \draw[line width=2.5pt,black] (obj23) to[bend left=50] (obj22);



\draw[dotted, fill=black, fill opacity=0.1] (obj5) to[bend right=20]
    (objin2) to[bend right=100,looseness=1.5] (objin1)
    to[bend right=100,looseness=1.5] (objin16)
    to[bend right=100,looseness=1.5] (objin15)
    to[bend right=100,looseness=1.5] (objin14)
    to[bend right=100,looseness=1.5] (objin13)
    to[bend
left=50] (obj44) to[bend right=50] (obj1) to[bend right=20] (obj5);
     \draw[] (obj5) to[bend right=20] (objin2);
     \draw[] (objin13) to[bend left=50] (obj44);
     \draw[] (objin2) to[bend right=100,looseness=1.5] (objin1);
     \draw[] (objin1) to[bend right=100,looseness=1.5] (objin16);
     \draw[] (objin16) to[bend right=100,looseness=1.5] (objin15);
     \draw[] (objin15) to[bend right=100,looseness=1.5] (objin14);
     \draw[] (objin14) to[bend right=100,looseness=1.5] (objin13);
     \draw[dotted, fill=black, fill opacity=0.1] (obj27) to[bend
       right=20] (objin10)
     to[bend right=100,looseness=1.5] (objin9)
     to[bend right=100,looseness=1.5] (objin8)
     to[bend right=100,looseness=1.5] (objin7)
     to[bend right=100,looseness=1.5] (objin6)
     to[bend right=100,looseness=1.5] (objin5)
     to[bend
left=50] (obj22) to[bend right=50] (obj23) to[bend right=20] (obj27);
     \draw[] (obj27) to[bend right=20] (objin10);
     \draw[] (objin5) to[bend left=50] (obj22);
     \draw[] (objin10) to[bend right=100,looseness=1.5] (objin9);
     \draw[] (objin9) to[bend right=100,looseness=1.5] (objin8);
     \draw[] (objin8) to[bend right=100,looseness=1.5] (objin7);
     \draw[] (objin7) to[bend right=100,looseness=1.5] (objin6);
     \draw[] (objin6) to[bend right=100,looseness=1.5] (objin5);

  \newcommand{\longrightmapsto}{\shortmid\!\longrightarrow}
    \node[inner sep=0pt] (M3) at (0,-2) {$\longrightmapsto$};
    \node[inner sep=0pt, below=0.2 of M3] {$\scriptstyle{\rotD}$};

    \polygon{(4,0)}{obj}{44}{3}
       {1,2,,,,f+1,,,,,,,,,,,,,,,,\hspace*{25pt}m(n-1),\overline{1},\overline{2},,,,\overline{f+1},,,,,,,,,,,,,,,,\overline{m(n-1)}\hspace*{25pt}}

    \polygoninner{(4,0)}{objin}{16}{1.2}
      {\ ,\ ,\hspace*{85pt}\pm{(mn-m+f+2)},,\hspace*{70pt}\pm(mn-m+1),\hspace*{0pt}\mp{mn},\ ,\ ,\ ,\ ,\mp{(mn-m+f+2)}\hspace*{80pt},,\mp(mn-m+1)\hspace*{65pt},\pm{mn}\hspace*{0pt},\ ,\ }




     \draw[dotted, fill=black, fill opacity=0.1] (obj6) to[bend
       right=20] (objin3)
     to[bend right=100,looseness=1.5] (objin2)
     to[bend right=100,looseness=1.5] (objin1)
     to[bend right=100,looseness=1.5] (objin16)
     to[bend right=100,looseness=1.5] (objin15)
     to[bend right=100,looseness=1.5] (objin14)
     to[bend right=30,looseness=1.5] (objin13) to[bend
left=50] (obj2) to[bend right=20] (obj6);
     \draw[] (obj6) to[bend right=20] (objin3);
     \draw[] (objin13) to[bend left=50] (obj2);
     \draw[] (objin14) to[bend right=30,looseness=1.5] (objin13);
     \draw[] (objin15) to[bend right=100,looseness=1.5] (objin14);
     \draw[] (objin16) to[bend right=100,looseness=1.5] (objin15);
     \draw[] (objin1) to[bend right=100,looseness=1.5] (objin16);
     \draw[] (objin2) to[bend right=100,looseness=1.5] (objin1);
     \draw[] (objin3) to[bend right=100,looseness=1.5] (objin2);
     \draw[dotted, fill=black, fill opacity=0.1] (obj28) to[bend
       right=20] (objin11)
     to[bend right=100,looseness=1.5] (objin10)
     to[bend right=100,looseness=1.5] (objin9)
     to[bend right=100,looseness=1.5] (objin8)
     to[bend right=100,looseness=1.5] (objin7)
     to[bend right=100,looseness=1.5] (objin6)
     to[bend
right=30] (objin5) to[bend left=50] (obj24) to[bend right=20] (obj28);
     \draw[] (obj28) to[bend right=20] (objin11);
     \draw[] (objin5) to[bend left=50] (obj24);
     \draw[] (objin6) to[bend right=30] (objin5);
     \draw[] (objin7) to[bend right=100,looseness=1.5] (objin6);
     \draw[] (objin8) to[bend right=100,looseness=1.5] (objin7);
     \draw[] (objin9) to[bend right=100,looseness=1.5] (objin8);
     \draw[] (objin10) to[bend right=100,looseness=1.5] (objin9);
     \draw[] (objin11) to[bend right=100,looseness=1.5] (objin10);

    \node[inner sep=0pt] (M3) at (-8.0,-8) {$\longrightmapsto$};

    \polygon{(-4,-8)}{obj}{44}{3}
       {1,,,,\ ,,,,,,,,\hspace*{15pt}d-f+1,,,,d,,,,,\hspace*{25pt}m(n-1),\overline{1},,,,\ ,,,,,,,,\overline{d-f+1}\hspace*{15pt},,,,\overline{d},,,,,\overline{m(n-1)}\hspace*{25pt}}

    \polygoninner{(-4,-8)}{objin}{16}{1.2}
      {\ ,\ ,\hspace*{0pt}\mp{mn},,\hspace*{80pt}\mp(mn-m+d-f),\ ,\ ,\ ,\ ,\ ,\pm{mn}\hspace*{00pt},,\pm(mn-m+d-f)\hspace*{75pt},\ ,\ ,\ }



    \draw[line width=2.5pt,black] (obj1) to[bend left=50] (obj44);
    \draw[line width=2.5pt,black] (obj23) to[bend left=50] (obj22);
    \draw (obj44) to[bend right=50] (obj5);
    \draw (obj22) to[bend right=50] (obj27);

     \draw[dotted, fill=black, fill opacity=0.1] (obj17) to[bend
       left=40] (objin11)
     to[bend right=20,looseness=1.5] (objin10)
     to[bend right=30,looseness=1.5] (objin9)
     to[bend right=100,looseness=1.5] (objin8)
     to[bend right=100,looseness=1.5] (objin7)
     to[bend right=100,looseness=1.5] (objin6)
     to[bend right=100,looseness=1.5] (objin5) to[bend
right=50] (obj13) to[bend right=20] (obj17);
     \draw[dotted, fill=black, fill opacity=0.1] (obj1) to[bend
       right=40] (obj5) to[bend left=50] (obj44)
     to[bend right=40] (obj1);
     \draw[] (obj17) to[bend left=40] (objin11);
     \draw[] (objin5) to[bend right=50] (obj13);
     \draw[] (objin6) to[bend right=100,looseness=1.5] (objin5);
     \draw[] (objin7) to[bend right=100,looseness=1.5] (objin6);
     \draw[] (objin8) to[bend right=100,looseness=1.5] (objin7);
     \draw[] (objin9) to[bend right=100,looseness=1.5] (objin8);
     \draw[] (objin10) to[bend right=30,looseness=1.5] (objin9);
     \draw[] (objin11) to[bend right=20,looseness=1.5] (objin10);
     \draw[dotted, fill=black, fill opacity=0.1] (obj39) to[bend
       left=40] (objin3)
     to[bend right=20,looseness=1.5] (objin2)
     to[bend right=50,looseness=1.5] (objin1)
     to[bend right=100,looseness=1.5] (objin16)
     to[bend right=100,looseness=1.5] (objin15)
     to[bend right=100,looseness=1.5] (objin14)
     to[bend right=100,looseness=1.5] (objin13)
     to[bend right=50] (obj35) to[bend right=20] (obj39);
     \draw[dotted, fill=black, fill opacity=0.1] (obj23) to[bend
       right=40] (obj27) to[bend left=50] (obj22)
     to[bend right=40] (obj23);
     \draw[] (obj39) to[bend left=40] (objin3);
     \draw[] (objin13) to[bend right=50] (obj35);
     \draw[] (objin14) to[bend right=100,looseness=1.5] (objin13);
     \draw[] (objin15) to[bend right=100,looseness=1.5] (objin14);
     \draw[] (objin16) to[bend right=100,looseness=1.5] (objin15);
     \draw[] (objin1) to[bend right=100,looseness=1.5] (objin16);
     \draw[] (objin2) to[bend right=50,looseness=1.5] (objin1);
     \draw[] (objin3) to[bend right=20,looseness=1.5] (objin2);

    \node[inner sep=0pt] (M3) at (.3,-8) {$\longrightmapsto$};
    \node[inner sep=0pt, below=0.2 of M3] {$\scriptstyle{\rotD}$};

    \polygon{(4,-8)}{obj}{44}{3}
       {1,2,,,,\ ,,,,,,,,\hspace*{15pt}d-f+2,,,,d+1,,,,\hspace*{25pt}m(n-1),\overline{1},\overline{2},,,,\ ,,,,,,,,\overline{d-f+2}\hspace*{15pt},,,,\overline{d+1},,,,\overline{m(n-1)}\hspace*{25pt}}

    \polygoninner{(4,-8)}{objin}{16}{1.2}
      {\ ,\pm{mn-m+2}\hspace*{20pt},\hspace*{0pt}\pm{mn-m+1},,\hspace*{80pt}\mp(mn-m+d-f+1),\ ,\ ,\ ,\ ,\hspace*{20pt}\mp{mn-m+2},\mp{mn-m+1}\hspace*{00pt},,\pm(mn-m+d-f+1)\hspace*{75pt},\ ,\ ,\ }



    \draw[line width=2.5pt,black] (obj1) to[bend left=50] (obj40);
    \draw[line width=2.5pt,black] (obj23) to[bend left=50] (obj18);
    \draw (obj2) to[bend left=10] (objin3);
    \draw (obj6) to[bend right=10,looseness=1.5] (objin3);
    \draw (obj24) to[bend left=10] (objin11);
    \draw (obj28) to[bend right=10,looseness=1.5] (objin11);

     \draw[dotted, fill=black, fill opacity=0.1] (obj23) to[bend
       right=20] (objin10)
     to[bend right=30,looseness=1.5] (objin9)
     to[bend right=100,looseness=1.5] (objin8)
     to[bend right=100,looseness=1.5] (objin7)
     to[bend right=100,looseness=1.5] (objin6)
     to[bend right=100,looseness=1.5] (objin5) to[bend
       right=50] (obj14) to[bend right=20] (obj18)
     to[bend right=50] (obj23);
     \draw[dotted, fill=black, fill opacity=0.1] (obj2) to[bend
       right=40] (obj6)
          to[bend right=10,looseness=1.5] (objin3)
     to[bend right=10] (obj2);
     \draw[] (obj23) to[bend right=20] (objin10);
     \draw[] (objin5) to[bend right=50] (obj14);
     \draw[] (objin6) to[bend right=100,looseness=1.5] (objin5);
     \draw[] (objin7) to[bend right=100,looseness=1.5] (objin6);
     \draw[] (objin8) to[bend right=100,looseness=1.5] (objin7);
     \draw[] (objin9) to[bend right=100,looseness=1.5] (objin8);
     \draw[] (objin10) to[bend right=30,looseness=1.5] (objin9);
     \draw[dotted, fill=black, fill opacity=0.1] (obj40) to[bend
       right=50] (obj1)
     to[bend right=20,looseness=1.5] (objin2)
     to[bend right=20,looseness=1.5] (objin1)
     to[bend right=100,looseness=1.5] (objin16)
     to[bend right=100,looseness=1.5] (objin15)
     to[bend right=100,looseness=1.5] (objin14)
     to[bend right=50,looseness=1.5] (objin13)
     to[bend right=50] (obj36) to[bend right=20] (obj40);
     \draw[dotted, fill=black, fill opacity=0.1] (obj24) to[bend
       right=40] (obj28)
               to[bend right=10,looseness=1.5] (objin11)
     to[bend right=10] (obj24);
     \draw[] (obj1) to[bend right=20,looseness=1.5] (objin2);
     \draw[] (objin13) to[bend right=50] (obj36);
     \draw[] (objin14) to[bend right=50,looseness=1.5] (objin13);
     \draw[] (objin15) to[bend right=100,looseness=1.5] (objin14);
     \draw[] (objin16) to[bend right=100,looseness=1.5] (objin15);
     \draw[] (objin1) to[bend right=100,looseness=1.5] (objin16);
     \draw[] (objin2) to[bend right=20,looseness=1.5] (objin1);

    \end{tikzpicture}
\end{center}
\caption{Movements of a bridging block under the pseudo-rotation $\rotD$}
\label{fig:B3}
\end{figure}

\medskip\noindent
{\sc Case 3.} 
Let $\pi$ be an element of $\mNCDPlus$ with a bridging block. 
By repeated application of~$\rotD$, we may move this block into a
position where its ``first" element on the outer circle (in clockwise
direction) is $\overline{m(n-1)}$, its next element is $1$, and its
``last" element on the outer circle is $f$, say. The successor of~$f$
in the block then is $\pm(mn-m+f+1)$,
an element of the inner circle (to achieve this,
$mn-m+f+1$ might have to be reduced modulo~$m$),
and the predecessor of $\overline{m(n-1)}$ is $\pm(mn-1)$, also an
element of the inner circle. Let us call this block~$B$, that is,
$B=\{\overline{m(n-1)},1,*,f,\pm(mn-m+f+1),\dots,\pm(mn-1)\}$.
Here, as before, the symbol $*$ indicates
further possible elements in the block lying between $1$ and $f$,
while the dots $\dots$ indicate that all possible elements lying
between $\pm(mn-m+f+1)$ and $\pm(mn-1)$ are included in the block.
See the first image in \Cref{fig:B3}.
In a degenerate case, it may also be that $B$
contains no element from the inner circle. Nevertheless, in such a
case we consider it still as a bridging block since another
application of~$\rotD$ will make it
a bridging block (via Case~(2) of \Cref{prop:2D}).

We claim that, after $m(n-1)-1$ applications of~$\rotD$, the block~$B$
is moved to the ``opposite side" of the annulus; more precisely, it
becomes the block $\tilde B$, say, which consists of the elements
of~$B$ on the outer circle with opposite sign, and of the elements
of~$B$ on the inner circle if $n$ is odd, and otherwise of the
elements of~$B$ on the inner circle with opposite sign.

In order to see this, we notice that, for the first application
of~$\rotD$, we are in Case~(2) of \Cref{prop:2D}, with $a=\pm(mn-1)$.
This first application of~$\rotD$ then removes $\overline{m(n-1)}$ from
the block, adds $\mp(mn-m+1)$ to it, and replaces $i$ by~$i+1$ and
$\overline i$ by~$\overline{i+1}$ for $i\in B$ respectively $\overline
i\in B$.
See the second image in \Cref{fig:B3}.
From here on, application of~$\rotD$ acts as ordinary
rotation, that is, elements on the outer circle are shifted by one
unit in clockwise direction, and elements on the inner circle are shifted by one
unit in counter-clockwise direction, until a situation is reached
where Case~(2) of \Cref{prop:2D} applies again
and involves our block of interest.
More precisely, this situation occurs when our block has been (ordinarily)
rotated $d-f-1$ times, for some~$d>f+1$, the block containing $\overline1$
contains~$m(n-1)$ and otherwise only negative elements, and there is no
other block containing positive and negative elements that separates
the latter block from our block of interest. The third image in
\Cref{fig:B3} provides a schematic illustration of that situation.
In particular, our block has been (pseudo-)rotated to
$\{d-f+1,*,d,\pm mn,\dots,\mp(mn-m+d-f)\}$, where, again,
$mn-m+d-f$ might have to be reduced modulo~$m$.
It should be observed that the fact that there are no blocks
separating the block containing~$\overline{1}$ and our block
of interest implies that the number of elements lying on the
outer circle between $d$ and $m(n-1)$ must be divisible by~$m$.
Hence, we have $d\equiv-1$~(mod~$m$).

As we already mentioned, for the next application of~$\rotD$ we are
in Case~(2) of \Cref{prop:2D}. It makes our block of interest
``lose"~$\pm mn$, ``gain"~$\overline{1}$, and otherwise $i$ is replaced
by~$i+1$, respectively $\overline i$ by $\overline{i+1}$, for
elements $i$ respectively~$\overline i$ in the block. That is, our
block has become
$\{d-f+2,*,d+1,\overline{1},\pm(mn-m+2),\dots,\mp(mn-m+d-f+1)\}$.
See the fourth image in \Cref{fig:B3}.

At this point, we should pause for a moment and make an
intermediate summary of what happened so far: we started with a
block containing $\overline{m(n-1)}$ and~$1$, some further
positive elements on the outer circle with maximal element~$f$,
and some successive elements on the inner circle ending
in~$\pm(mn-1)$ (when read in counter-clockwise direction).
In the first application of~$\rotD$ we ``lost" $\overline{m(n-1)}$
on the outer circle and we ``gained" $\mp(mn-m+1)$ on the
inner circle, while all other elements were rotated by one unit,
in clockwise direction on the outer circle and in
counter-clockwise direction on the inner circle.
Subsequently, this (new) block was ordinarily rotated (in the
same sense concerning the different directions on the
outer circle and the inner circle). In the last application
of~$\rotD$, we ``gained"~$\overline{1}$ but ``lost" the ``first" element
on the inner circle (read in counter-clockwise direction),
while all other elements were rotated by one unit.
This implies in particular that the successive elements
on the inner circle made --- so-to-speak ---
a further ``jump" by one unit in addition to the ``ordinary"
rotations.

From here on, further applications of~$\rotD$ let 
the ``first" and ``last" elements on outer and inner
circle of the block ``move forward" by one unit; the map~$\rotD$
either acts by ordinary rotation or via Case~(1) of
\Cref{prop:2D}. Both act by ordinary rotation on the inner circle,
while in the latter case the action is equivalent to an action
of~$\rotB$ on the outer circle (ignoring the inner circle). 
Since we have shown in \Cref{lem:N-2B} that the action of~$\rotB$ on
positive type~$B$ non-crossing partitions of
$\{1,2,\dots,m(n-1),\overline1,\overline2,\dots,\overline{m(n-1)}\}$
has order~$m(n-1)-1$, the above arguments show that the elements on
the outer circle of block~$B$ will be ``moved forward" --- in clockwise
direction --- by~$m(n-1)$
units, that is, to their negatives, while the elements on the inner
circle will be ``moved forward" --- in counter-clockwise direction --- by
$m(n-1)$ units. The latter implies that they are mapped to the
elements on the inner circle of~$B$ if $n$~is odd, and otherwise to
their negatives. This establishes the claim.

\medskip
We are now in the position to conclude the proof. We have shown that\break
$m(n-1)-1$ applications of~$\rotD$ map the elements on the outer
circle of a bridging block to their negatives, and the elements on the
inner circle to themselves if $n$~is odd, while to their negatives if
$n$~is even. Moreover, the other blocks (which must lie on the outer
circle since the inner circle has only $2m$~elements and all block
sizes are divisible by~$m$) are moved by~$\rotB$-operations when
$\rotD$~is applied. Consequently, by \Cref{lem:N-2B}, $m(n-1)-1$
applications of~$\rotD$ map them to their negatives. Thus, if $n$~is
even, \emph{all\/} blocks are mapped to their negatives, which means
that we obtained the original partition since it was invariant under
substitution of~$i$ by~$-i$, for all~$i$, from the very beginning. On the other hand, if
$n$~is odd, then we need another $m(n-1)-1$ applications of~$\rotD$ to
get back to our original partition.
\qed

\section{Proofs: Enumeration results}
\label{app:GF}

This appendix is devoted to the proofs of \Cref{thm:countingA,,thm:1-B,,cor:2-D} in which we enumerate multichains of positive $m$-divisible non-crossing set partitions in classical types according to their block structure, and to the proofs of \Cref{thm:2,,thm:2-B} and \Cref{thm:enumD-2,,thm:enumD-4} in which we enumerate pseudo-rotationally invariant positive $m$-divisible non-crossing set partitions of classical types.

\subsection{Preparations and auxiliary results}
\label{app:GF-aux}

We use the generating function approach from~\cite{KratCG}.
We start by recalling the two forms of Lagrange inversion
(cf.\ \cite[Thm.~1.9b]{HenrAA}, \cite[Thm.~5.4.2]{StanBI}.

\begin{lemma}[\sc Lagrange inversion] \label{lem:TA}
Let $f(z)$ be a formal power series with $f(0)=0$ and $f'(0)\ne0$,
and let $F(z)$ be its compositional inverse. Then, for all
integers~$a$ and~$b$,
\begin{equation} \label{eq:CA} 
\coef{z^0}z^aF^{b-1}(z)=\coef{z^0}z^bf^{a-1}(z)f'(z)
\end{equation}
and
\begin{equation} \label{eq:CB} 
a\coef{z^0}z^aF^{b}(z)=-b\coef{z^0}z^bf^{a}(z).
\end{equation}
More generally, if $g(z)$ is a Laurent series with only finitely many terms
containing negative powers of~$z$, then
\begin{equation} \label{eq:CBa} 
\coef{z^{-1}}g'(z)F^{-b}(z)=b\coef{z^b}g(f(z)).
\end{equation}
\end{lemma}

Let $x_i,y_i$, $i=1,2,\dots$, be variables, and $x_0=y_0=1$.
For a positive $m$-divisible non-crossing partition~$\pi$ let the weight
$w_{\mathbf y}^{(m)}(\pi)$ be defined by
$$w_{\mathbf y}^{(m)}(\pi)=\prod _{i=1} ^{\infty}
y_i^{\#(\text {special blocks of $\pi$ of size
$mi$})}x_i^{\#(\text {non-special blocks of $\pi$ of size
$mi$})}.$$
Recall that the special block of a positive $m$-divisible non-crossing
partition~$\pi$ is the one which contains both $1$ and $mn$, so that,
clearly, in $w_{\mathbf y}^{(m)}(\pi)$ there appears exactly one variable
$y_i$, and it does so with exponent $1$.
For example, the weight $w_{\mathbf y}^{(3)}(\,.\,)$ of the
positive $m$-divisible non-crossing partition on the left of 
\Cref{fig:1} is $y_2x_1^6$, while the weight of the
non-crossing partition on the right of the same figure is
$y_1x_1^5x_2$.

Sometimes, we shall consider the special case where $y_i$ is
set equal to $x_i$ for all~$i$. We abbreviate this specialised weight by
$w^{(m)}$, that is,
$$w^{(m)}(\pi)=\prod _{i=1} ^{\infty}x_i^{\#(\text {blocks of $\pi$ of size
$mi$})}.$$
For example, the specialised weight $w^{(m)}(\,\cdot\,)$ of the two
positive $m$-divisible non-crossing partitions in
\Cref{fig:1} is $x_1^6x_2$.

We write $\vert\pi\vert$ for the \defn{size} of~$\pi$, that is, 
for the number of elements of the set that is partitioned by~$\pi$.

First we quote an auxiliary result from~\cite{KratCG},
providing a functional equation for the generating function
for \emph{all}
$m$-divisible non-crossing partitions
with respect to the specialised weight $w^{(m)}$.
For the sake of completeness, we also sketch its proof.

\begin{lemma} \label{lem:TB}
Let $C^{(m)}(z)$ be the generating function $\sum _{\pi}
^{}w^{(m)}(\pi)z^{\vert \pi\vert}$, where the sum is over all
$m$-divisible non-crossing partitions. Then
\begin{equation} \label{eq:CC} 
C^{(m)}(z)=\sum _{i=0} ^{\infty}x_iz^{mi}\left(C^{(m)}(z)\right)^{mi}.
\end{equation}
\end{lemma}

\begin{figure}
  \begin{tikzpicture}[scale=1]
    \polygon{(-3.5,0)}{obj}{45}{2.5}
      {1,,,,,,\hspace{5pt},,,,,,,,,\hspace{5pt},,,,\hspace{4pt},,,,,\hspace{4pt},,,,\hspace{4pt},,,,,,,\hspace{15pt},,,,,,\hspace{4pt},,,}

    \draw[dash pattern=on 2pt off 2pt on 2pt off 2pt on 2pt off 2pt on
2pt off 2pt on 2pt off 2pt on 2pt off 46pt on 2pt off 2pt on 2pt off
2pt on 2pt off 2pt on 2pt off 2pt] (obj7) to[bend right=50] (obj16);

    \draw (obj16) to[bend right=50]
          (obj20) to[bend right=50]
          (obj25) to[bend right=50]
          (obj29) to[bend right=50]
          (obj36) to[bend right=50]
          (obj42) to[bend right=50]
          (obj1) to[bend right=50]
          (obj7);

    \draw[dotted, fill=black, fill opacity=0.1] (obj2)  to[bend
right=50] (obj6) to[bend right=15] (obj2);
    \draw[dotted, fill=black, fill opacity=0.1] (obj17) to[bend
right=50] (obj19) to[bend right=10] (obj17);
    \draw[dotted, fill=black, fill opacity=0.1] (obj21) to[bend
right=50] (obj24) to[bend right=13] (obj21);
    \draw[dotted, fill=black, fill opacity=0.1] (obj26) to[bend
right=50] (obj28) to[bend right=10] (obj26);
    \draw[dotted, fill=black, fill opacity=0.1] (obj30) to[bend
right=50] (obj35) to[bend right=17] (obj30);
    \draw[dotted, fill=black, fill opacity=0.1] (obj37) to[bend
right=50] (obj41) to[bend right=15] (obj37);
    \draw[dotted, fill=black, fill opacity=0.1] (obj43) to[bend
right=50] (obj45) to[bend right=10] (obj43);

    \end{tikzpicture}
\caption{Illustration of the decomposition explaining~\eqref{eq:CC}:
between the arcs of the block containing~$1$ there are (smaller) non-crossing
partitions, indicated by shaded ``half disks".}
\label{fig:40}
\end{figure}
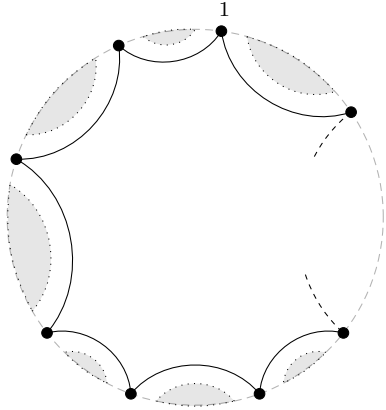

\begin{proof}The equation is best explained by a picture; see
\Cref{fig:40}.
Given an $m$-divisible non-crossing partition~$\pi$, let $B$ be the
block containing the element~$1$. This block has a size that is
divisible by~$m$, say~$mi$. The (weight) contribution of that block
is then $x_iz^{mi}$.
Between the ``arcs" of that block~$B$ we find smaller $m$-divisible
non-crossing partitions, each one contributing $C^{(m)}(z)$ to
the term. Since the block~$B$ has $mi$~arcs, this explains the summand
in~\eqref{eq:CC} by standard generating function calculus.
The equation follows once one observes that the
summand for $i=0$ corresponds to the (empty) partition of~$\emptyset$.
\end{proof}

From now on, 
\begin{equation} \label{eq:CD} 
F(z)=\frac {z} {\sum _{i=0} ^{\infty}x_iz^{mi}},
\end{equation}
and $f(z)$ denotes its compositional inverse.

Slightly rewriting \eqref{eq:CC}, we see that it is equivalent to
$$z=\frac {zC^{(m)}(z)} {\sum _{i=0} ^{\infty}x_i(zC^{(m)}(z))^{im}}=
F(zC^{(m)}(z)).$$
Thus, we have $zC^{(m)}(z)=f(z)$.

We quote an easy corollary from~\cite{KratCG}.

\begin{corollary} \label{cor:TC}
The generating function $\sum _{\pi}
^{}w^{(m)}(\pi)$, where the sum is over all
$m$-divisible non-crossing partitions of\/ $\{1,2,\dots,mn\}$, equals
\begin{equation} \label{eq:CE} 
\frac {1} {mn+1}\coef{z^0}zF^{-mn-1}(z),
\end{equation}
where $F(z)$ is defined by \eqref{eq:CD}.
\end{corollary}

\begin{proof}
By \Cref{lem:TB}, the generating function in question is
$$
\coef{z^{mn}}C^{(m)}(z)=\coef{z^0}z^{-mn-1}f(z).
$$
The assertion now follows from the application of the Lagrange
inversion formula in the form~\eqref{eq:CB}.
\end{proof}

\subsection{Enumeration of positive $m$-divisible non-crossing partitions
in type $A$}
\label{app:GF-A}

We turn to the subset of all \emph{positive} $m$-divisible non-crossing 
partitions.

\begin{lemma} \label{lem:TBa}
Let $C^{(m)}_+(z)$ be the generating function $\sum _{\pi}
^{}w_{\mathbf y}^{(m)}(\pi)z^{\vert \pi\vert}$, where the sum is over all
positive $m$-divisible non-crossing partitions. Then
\begin{equation} \label{eq:CCa} 
C^{(m)}_+(z)=\sum _{i=1} ^{\infty}y_iz^{mi}\left(C^{(m)}(z)\right)^{mi-1}.
\end{equation}
\end{lemma}

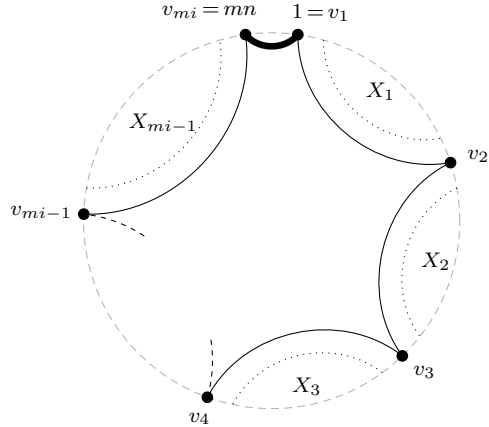
\begin{figure}
  \centering
  \begin{tikzpicture}[scale=1]
    \polygon{(0,0)}{obj}{45}{2.5}
      {\hspace{15pt}1\,=\,v_1,,,,,,,,\hspace{5pt}v_2,,,,,,,,\hspace{5pt}v_3,,,,,,,,v_4,,,,,,,,,v_{mi-1}\hspace{15pt},,,,,,,,,,v_{mi}\,{}={}\,mn\hspace{25pt},}
     \draw[line width=2.5pt,black] (obj1) to[bend left=50] (obj44);
     \draw (obj1) to[bend right=50] (obj9) to[bend right=50] (obj17) to[bend right=50] (obj25);
     \draw (obj34) to[bend right=50] (obj44) to[bend right=50] (obj1);
     \draw[black, dash pattern=on 2pt off 2pt on 2pt off 2pt on 2pt off 2pt on 2pt off 2pt on 2pt off 2pt on 2pt off 48pt on 2pt off 2pt on 2pt off 2pt on 2pt off 2pt on 2pt off 2pt] (obj25) to[bend right=50] (obj34);
     \draw[dotted] (obj2)  to[bend right=50] (obj8);
     \draw[dotted] (obj10) to[bend right=50] (obj16);
     \draw[dotted] (obj18) to[bend right=50] (obj24);
     \draw[dotted] (obj35) to[bend right=50] (obj43);
     \node at ($0.9*(obj5)$) {\tiny$X_1$};
     \node at ($0.9*(obj13)$) {\tiny$X_2$};
     \node at ($0.9*(obj21)$) {\tiny$X_3$};
     \node at ($0.775*(obj39)$) {\tiny$X_{mi-1}$};
    \end{tikzpicture}
  \caption{The decomposition of a positive $m$-divisible non-crossing partition}
\label{fig:2}
\end{figure}

\begin{proof}
The vertices $1$ and $mn$ 
must be in one of the blocks, say in a block of size~$mi$. Then,
if $v_1=1,v_2,\dots,v_{mi}=mn$ are all the vertices in this block in
clockwise order, then the vertices between $v_j$ and $v_{j+1}$,
$j=1,2,\dots,mi-1$ are
involved in a smaller (arbitrary, not necessarily positive) 
$m$-divisible non-crossing partition (see
\Cref{fig:2}). 
Equation~\eqref{eq:CCa} then follows by standard generating function calculus,
the block containing $1$ and $mn$ contributing the term
$y_iz^{mi}$, and the ``small" $m$-divisible non-crossing partitions
between $v_j$ and $v_{j+1}$ each contributing a term
$C^{(m)}(z)$, $j=1,2,\dots,mi-1$.
\end{proof}

The analogue of \Cref{cor:TC} for positive $m$-divisible
non-crossing partitions reads as follows.

\begin{corollary} \label{cor:TCa}
The generating function $\sum _{\pi}
^{}w_{\mathbf y}^{(m)}(\pi)$, where the sum is over all
positive $m$-divisible non-crossing partitions of\/ $\{1,2,\dots,mn\}$, equals
\begin{equation} \label{eq:CEa} 
\sum_{i=1}^\infty \frac {mi-1} {mn-1}y_i\coef{z^0}z^{mi-1}F^{-mn+1}(z),
\end{equation}
where $F(z)$ is defined by~\eqref{eq:CD}.
\end{corollary}

\begin{proof}
Because of $C^{(m)}(z)=f(z)/z$, we know from~\eqref{eq:CCa} 
that the desired generating function 
equals 
\begin{equation} \label{eq:CEb} 
\coef{z^{0}}\sum _{i=1} ^{\infty}y_iz^{-mn+1}f^{mi-1}(z).
\end{equation}
Hence, the claim follows from~\eqref{eq:CB} with
$a=mi-1$ and $b=-mn+1$.
\end{proof}

\begin{proposition} \label{prop:TG}
Let $m,n,l$ be positive integers with $l\ge2$. 
Furthermore, let $s'_2,s_3,\dots,s_l$ be
non-negative integers with $s'_2+s_3+\dots+s_l=n-1$. 
The generating function
\begin{equation} \label{eq:DD} 
{\sum }{}^{\displaystyle\prime}
w_{\mathbf y}^{(m)}(\pi_1),
\end{equation}
where the sum is over all multichains
$\pi_1\le\pi_2\le\dots\le\pi_{l-1}$ of positive $m$-divisible non-crossing
partitions of\/ $\{1,2,\dots,mn\}$, where the rank of~$\pi_i$ is
$s'_2+s_3+\cdots +s_i$, $i=2,3,\dots,\break l-1$, 
is given by
\begin{multline} \label{eq:DE}
\frac {1} {mn-1}
\binom {mn-1}{s_3}\cdots\binom {mn-1}{s_{l}}\\
\times
\sum _{k=0} ^{s_3+\dots+s_l}\!\!
(-1)^{k+s_3+\dots+s_l}
\binom {s_3+\dots+s_l}k 
\sum _{a=1} ^{\infty}
\coef{z^0}(ma-1)y_az^{{ma}+k-1}F^{-{mn}-k+1}(z).
\end{multline}
If $l=2$, empty sums have to interpreted as $0$, and empty products
have to be interpreted as $1$.
\end{proposition}

\begin{proof}
We prove the assertion by induction on $l$. If $l=2$, the assertion
is true due to~\eqref{eq:CEa} and our conventions for $l=2$. Let us now
assume that the assertion is true for multichains consisting of~$l-2$
elements and consider the multichain
$\pi_1\le\pi_2\le\dots\le\pi_{l-1}$. The induction hypothesis implies that
the generating function 
$${\sum }
w^{(m)}_{\mathbf y}(\pi_2),$$
where the sum is over all multichains
$\pi_2\le\dots\le\pi_{l-1}$ of positive $m$-divisible non-crossing
partitions of $\{1,2,\dots,mn\}$, where the rank of~$\pi_i$ is
$(s'_2+s_3)+s_4+\cdots +s_i$, $i=3,\dots,l-1$, 
is given by
\begin{multline} \label{eq:DF}
\frac {1} {mn-1}\binom {mn-1}{s_4}\cdots\binom {mn-1}{s_{l}}\\
\times
\sum _{k=0} ^{s_4+\dots+s_l}\!\!
(-1)^{k+s_4+\dots+s_l}
\binom {s_4+\dots+s_l}k 
\sum _{a=1} ^{\infty}
\coef{z^0}(ma-1)y_az^{{ma}+k-1}F^{-{mn} -k+1}(z).
\end{multline}

Next we model the relation
$\pi_1\le\pi_2$ by appropriate replacements
of the variables $x_b$ and $y_a$ in the generating function~\eqref{eq:DF}.
The partition $\pi_1$ arises from $\pi_2$ by replacing non-special blocks
of~$\pi_2$ of size~$mb$ by $m$-divisible non-crossing partitions of
$mb$~elements, and by replacing the special block, say of size~$ma$, 
by a \emph{positive} $m$-divisible non-crossing partition of
$ma$~elements. Hence, to model this in the
generating function, we will replace in~\eqref{eq:DF} the variable $x_b$ by the
coefficient of~$u^{mb}$ in $tC^{(m)}(u)$, and we will replace
the variable $y_a$ by the coefficient of~$u^{ma}$ in $tC^{(m)}_+(u)$.
(Recall that the series $C^{(m)}(u)$ is
the generating function for \emph{all\/} 
$m$-divisible non-crossing partitions, while $C^{(m)}_+(u)$ is the
one for \emph{positive} $m$-divisible non-crossing partitions.)
Here, the variable $t$ keeps
track of the number of blocks of~$\pi_2$ and, thus, by the equation
$$n-(\text{number of blocks of $\pi_2$})=(\text{rank of $\pi_2$})
=s'_2,$$
respectively equivalently,
$$(\text{number of blocks of $\pi_2$})
=s_3+s_4+\dots+s_l+1,$$
of the rank of $\pi_2$.
From the result we will then extract the coefficient of
$t^{s_3+s_4+\dots+ s_l+1}$ to obtain the generating function~\eqref{eq:DD}. 

Next,
we compute the effect of
substitution of $\coef{u^{mb}}tC^{(m)}(u)$ for $x_{b}$,
$b=1,2,\dots$, in the series $F(z)$ (recall its definition in \Cref{eq:CD}):
\begin{align} 
\notag
z\(1+
\sum _{b=1} ^{\infty}z^{mb}\coef{u^{mb}}tC^{(m)}(u)\)^{-1}&=
z\(1+t(C^{(m)}(z)-1)\)^{-1}\\
&=
z\(1+t\(\frac {f(z)} {z}-1\)\)^{-1}.
\label{eq:CHa}
\end{align}
On the other hand, we need to compute the effect of
substitution of $\coef{u^{ma}}tC^{(m)}_+(u)=
\coef{u^{ma}}t\sum _{\ell=1} ^{\infty}y_\ell uf(u)^{m\ell-1}$
(see~\eqref{eq:CEa}, respectively~\eqref{eq:CEb}) for
$y_a$ in $\sum _{a=1} ^{\infty}(ma-1)y_az^{{ma} -1}$. We obtain
\begin{align} \notag
\sum _{a=1} ^{\infty}(ma-1)z^{{ma} -1}
&\left(\coef{u^{{ma}}}
t\sum _{\ell=1} ^{\infty}y_\ell uf(u)^{m\ell-1}\right)\\
\notag
&=
\sum _{a=1} ^{\infty}\left(\frac {d} {dz}-\frac {1}
     {z}\right)\big(z^{{ma} }\big) \cdot
\left(\coef{u^{{ma}}}
t\sum _{\ell=1} ^{\infty}y_\ell uf(u)^{m\ell-1}\right)\\
\notag
&=
t\left(\frac {d} {dz}\sum _{\ell=1} ^{\infty}
y_\ell zf^{m\ell-1}(z)-
\frac {1} {z}\sum _{\ell=1} ^{\infty} y_\ell zf^{m\ell-1}(z)\right)\\
&=
t
\sum _{\ell=1} ^{\infty}
(m\ell-1)y_\ell zf'(z)f^{m\ell-2}(z).
\label{eq:DG} 
\end{align}

Using the substitutions \eqref{eq:CHa} and \eqref{eq:DG}
in~\eqref{eq:DF}, subsequent extraction of the
coefficient of $t^{s_3+s_4+\dots+ s_l+1}$ leads to 
\begin{align}
\notag
&\coef{t^{s_3+s_4+\dots+s_l+1}}
\frac {1} {mn-1}
\binom {mn-1}{s_4}\cdots\binom {mn-1}{s_{l}}
\sum _{k=0} ^{s_4+\dots+s_l}\!\!
(-1)^{k+s_4+\dots+s_l}
\binom {s_4+\dots+s_l}k \\
\notag
&\kern2cm
\cdot 
t\sum _{\ell=1} ^{\infty}
\coef{z^0}(m\ell-1)y_\ell zf'(z) f^{{m\ell} -2}(z)z^{-{mn} +1}
\(1+t\(\frac {f(z)} {z}-1\)\)^{{mn} +k-1}\\
\notag
&=\frac {1} {mn-1}
\binom {mn-1}{s_4}\cdots\binom {mn-1}{s_{l}}
\sum _{k=0} ^{s_4+\dots+s_l}\!\!
(-1)^{k+s_4+\dots+s_l}
\binom {s_4+\dots+s_l}k \\
\notag
&\kern.5cm
\cdot 
\sum _{\ell=1} ^{\infty}
\coef{z^0}(m\ell-1)y_\ell f'(z) f^{{m\ell} -2}(z)z^{-{mn} +2}\\
\notag
&\kern5.5cm
\cdot 
\binom {{mn} +k-1}{s_3+s_4+\dots+s_l}
\(\frac {f(z)} {z}-1\)^{s_3+s_4+\dots+s_l}\\
\notag
&=\frac {1} {mn-1}
\binom {mn-1}{s_4}\cdots\binom {mn-1}{s_{l}}
\sum _{k=0} ^{s_4+\dots+s_l}\!\!
(-1)^{k}
\binom {{mn} +k-1}{s_3+s_4+\dots+s_l}
\binom {s_4+\dots+s_l}k \\
\notag
&\kern2cm
\cdot 
\sum _{\ell=1} ^{\infty}
\sum _{j=0} ^{s_3+s_4+\dots+s_l}(-1)^{s_3+j}
\binom {s_3+s_4+\dots+s_l}j\\
&\kern6cm
\cdot 
\coef{z^0}(m\ell-1)y_\ell f'(z) f^{{m\ell}+j -2}(z)z^{-{mn}-j +2}
\label{eq:DH}
\end{align}
The sums over $k$ and $j$ have
become completely independent, and the sum over $k$ can be
evaluated by means of the Chu--Vandermonde summation formula 
(see e.g.\ \cite[Sec.~5.1, Eq.~(5.27)]{GrKPAA}). Namely, we have
\begin{align}
\notag
&\sum _{k=0} ^{s_4+\dots+s_l}\!\!
(-1)^k
\binom {mn+k-1}{s_3+s_4+\dots+s_l}
\binom {s_4+\dots+s_l}{k} \\
\notag
&\qquad 
=(-1)^{mn-s_3-s_4-\dots-s_l-1}
\sum _{k=0} ^{s_4+\dots+s_l}\!\!
\binom {-s_3-s_4-\dots-s_l-1}{mn+k-s_3-s_4-\dots-s_l-1}
\binom {s_4+\dots+s_l}{s_4+\dots+s_l-k} \\
\notag
&\qquad 
=(-1)^{mn-s_3-s_4-\dots-s_l-1}
\binom {-s_3-1}{mn-s_3-1}\\
&\qquad 
=(-1)^{s_4+\dots+s_l}
\binom {mn-1}{mn-s_3-1}.
\label{eq:CIa}
\end{align}
If we substitute this in \eqref{eq:DH} and use the Lagrange inversion
formula~\eqref{eq:CA} from\break \Cref{lem:TA}  with $a= {m\ell}+j-1$ and 
$b=-{mn}
-j+2$, we obtain exactly~\eqref{eq:DE}.
\end{proof}

\begin{proof}[Proof of \Cref{thm:countingA}]
We use \Cref{prop:TG} with $s'_2=s_1+s_2$. We have
$$n-(\text{number of blocks of $\pi_1$})=(\text{rank of $\pi_1$})
=s_1,$$
or, equivalently,
$$\text{number of blocks of $\pi_1$}
=s_2+s_3+\dots+s_l+1.$$
Hence, we must replace both $x_b$ and $y_b$ by~$tx_b$
in~\eqref{eq:DE} and extract the coefficient of 
$$t^{s_2+s_3+\dots+s_l+1}x_1^{b_1}x_2^{b_2}\cdots x_n^{b_n}.$$
(In the sequel, we write $t^{s_2+s_3+\dots+s_l+1}\mathbf x^{\mathbf b}$ for
short for the above monomial.) 
To begin with, we perform the substitution $y_b\to x_b$ in the
sum over $a$ in~\eqref{eq:DE} and simplify:
\begin{align*} 
\sum _{a=1} ^{\infty}
\coef{z^0}&(ma-1)x_az^{{ma}+k-1}F^{-{mn}-k+1}(z)\\
&=\coef{z^0}z^{k+1}F^{-{mn}-k-1}(z)
\frac {\sum _{a=1} ^{\infty}
(ma-1)x_az^{{ma}}}
{\left(1+\sum _{a=1} ^{\infty}
x_az^{ma}\right)^2}\\
&=\coef{z^0}z^{k+1}F^{-{mn}-k-1}(z)
\left(-F'(z)+z^{-2}F^2(z)\right)\\
&=\coef{z^0}\left(-z^{k+1}F'(z)F^{-{mn}-k-1}(z)
+z^{k-1}F^{-mn-k+1}(z)\right)\\
&=\coef{z^0}\left(z^{k+1}\frac {1} {mn+k}\frac {d} {dz}F^{-{mn}-k}(z)
+z^{k-1}F^{-mn-k+1}(z)\right)\\
&=\coef{z^0}\left(-\frac {k} {mn+k}z^{k}F^{-{mn}-k}(z)
+z^{k-1}F^{-mn-k+1}(z)\right).
\end{align*}
This is now used in~\eqref{eq:DE}. Subsequently, as we already said,
we replace $x_b$ by~$tx_b$ for all~$b$ and extract the coefficient of 
$t^{s_2+s_3+\dots+s_l+1}\mathbf x^{\mathbf b}$. This leads to
\begin{align*}
&\frac {1} {mn-1}
\binom {mn-1}{s_3}\cdots\binom {mn-1}{s_{l}}
\sum _{k=0} ^{s_3+\dots+s_l}\!\!
(-1)^{k+s_3+s_4+\dots+s_l}
\binom {s_3+s_4+\dots+s_l}{k} \\
&\kern2cm
\cdot
\coef{t^{s_2+s_3+\dots+s_l+1}\mathbf x^{\mathbf b}z^{mn}}
\Bigg(-\frac {k} {mn+k}\(1+t
\sum _{i=1} ^{\infty}x_iz^{mi}\)^{mn+k}\\
&\kern8.5cm
+\(1+t
\sum _{i=1} ^{\infty}x_iz^{mi}\)^{mn+k-1}\Bigg)\\
&=
\frac {1} {mn-1}
\binom {mn-1}{s_3}\cdots\binom {mn-1}{s_{l}}
\sum _{k=0} ^{s_3+\dots+s_l}\!\!
(-1)^{k+s_3+s_4+\dots+s_l}
\binom {s_3+s_4+\dots+s_l}{k} \\
&\kern1cm
\cdot\Bigg(-\frac {k} {mn+k}\binom {mn+k}{s_2+s_3+\dots+s_l+1}
\coef{\mathbf x^{\mathbf b}z^{mn}}
\(\sum _{i=1} ^{\infty}x_iz^{mi}\)^{s_2+s_3+\dots+s_l+1}\\
&\kern3.5cm
+\binom {mn+k-1}{s_2+s_3+\dots+s_l+1}
\coef{\mathbf x^{\mathbf b}z^{mn}}
\(\sum _{i=1} ^{\infty}x_iz^{mi}\)^{s_2+s_3+\dots+s_l+1}
\Bigg)\\
&=\frac {1} {mn-1}
\binom {mn-1}{s_3}\cdots\binom {mn-1}{s_{l}}
\sum _{k=0} ^{s_3+\dots+s_l}\!\!
(-1)^{k+s_3+s_4+\dots+s_l}
\binom {s_3+s_4+\dots+s_l}{k} \\
&\kern1cm
\cdot
\frac {-k+(mn+k-s_2-s_3-\dots-s_l-1)} {s_2+s_3+\dots+s_l+1}
\binom {mn+k-1}{s_2+s_3+\dots+s_l}\\
&\kern7cm
\cdot
\coef{\mathbf x^{\mathbf b}z^{mn}}
\(\sum _{i=1} ^{\infty}x_iz^{mi}\)^{s_2+s_3+\dots+s_l+1}.
\end{align*}
The sum over $k$ is completely analogous to the sum over $k$ in the
proof of \Cref{prop:TG} and, hence, can as well be evaluated by
means of the Chu--Vandermonde summation formula. Repeating the
arguments from there, we arrive at the expression
\begin{multline} \label{eq:x^b}
\frac {mn-s_2-s_3-\dots-s_l-1} {(mn-1)(s_2+s_3+\dots+s_l+1)}
\binom {mn-1}{s_2}\binom {mn-1}{s_3}\cdots\binom {mn-1}{s_{l}}\\
\times
\coef{\mathbf x^{\mathbf b}z^{mn}}
\(\sum _{i=1} ^{\infty}x_iz^{mi}\)^{s_2+s_3+\dots+s_l+1}.
\end{multline}
Noting that the conditions in the statement of the theorem say that
$$b_1+2b_2+3b_3+\dots+nb_n=n$$
and
$$
s_1+b_1+b_2+\dots+b_n=n,
$$
and, hence,
$$b_1+b_2+\dots+b_n=s_2+s_3+\dots+s_l+1,$$
we arrive at~\eqref{eq:multichains-bi-si} without any difficulty.
\end{proof}

\subsection{Enumeration of pseudo-rotationally invariant
  positive $m$-divisible non-cros\-sing partitions
in type $A$}
\label{app:GF-A-inv}

Again, we use a generating function approach as in \Cref{app:GF-A}.
When we consider $\rrr $-pseudo-rotationally invariant non-crossing
partitions, we need to introduce a ``simplified" weight, which
we denote by $w_{\rrr }^{(m)}$, and which is defined by
$$w^{(m)}_\rrr (\pi)=\prod _{i=1} ^{\infty}
x_i^{\#(\text {non-central blocks of $\pi$ of size
$mi$})/\rrr }.$$
For example, the weight $w_{4}^{(3)}(\,.\,)$ of both non-crossing partitions
in \Cref{fig:3,fig:4} is $x_1^2$, 
while the weight $w_{2}^{(3)}(\,.\,)$ of both partitions
in \Cref{fig:5,fig:6} is $x_1^4x_2$.

\begin{lemma} \label{lem:TF-ma=2}
Let $\rrr ,m,n$ be positive integers with $\rrr \ge2$, $\rrr \mid (mn-2)$.
The generating function $\sum _{\pi}
^{}w_{\rrr }^{(m)}(\pi)$, where the sum is over all positive
$m$-divisible non-crossing partitions of\/ $\{1,2,\dots,mn\}$
which are invariant under the
$\rrr $-pseudo-rotation~$\rotA^{(mn-2)/\rrr }$ and have a central block of
size~$2$ equals
\begin{equation} \label{eq:DA-ma=2} 
\coef{z^0}F^{-(mn-2)/\rrr }(z),
\end{equation}
where $F(z)$ is defined by \eqref{eq:CD} and
$f(z)$ is the compositional inverse of~$F(z)$, as
before. 
\end{lemma}

\begin{remark}
Since $m$ divides the sizes of all blocks, $m$-divisible non-crossing
partitions as described in the lemma only exist for $m=1$ and
$m=2$. Indeed, the formula~\eqref{eq:DA-ma=2} produces $0$ for $m\ge3$.
\end{remark}

\begin{proof}[Proof of \Cref{lem:TF-ma=2}]
We call a subpartition of a non-crossing partition a \defn{connected
component\/} if it is a partition of consecutive elements, of
$\{a,a+1,\dots,b\}$ say, and $a$ and $b$ are in the same block.
In Construction~1 in \Cref{sec:rotA} with $N=mn$ and a central block
of size~2, $\{2,3,\dots,\frac {mn-2} {\rrr }+1\}$ is
a union of blocks. Likewise, $\{mn-\frac {mn-2} {\rrr },\dots,mn-2,mn-1\}$ is
a union of blocks. These blocks are organised in connected components.
Let the connected component next to $mn$ cover the (consecutive)
elements $\{a,\dots,mn-2,mn-1\}$. The first application of~$\rotA$
moves this connected component ``into $\{mn,1\}$", while a block
$\{a+1,2\}$ is created covering this moved component. The next
$mn-a-2$ applications of~$\rotA$ rotate this arc-like block
ordinarily, until a further application of~$\rotA$ --- so-to-speak ---
maps it back to $\{mn,1\}$, while the connected component now covers
the elements $\{2,3,\dots,mn-a+1\}$. 

What the above arguments show is that the set of non-crossing
partitions that one obtains from a particular partition, $\hat\pi$~say, of
$\{mn-\frac {mn-2} {\rrr },\dots,mn-2,mn-1\}$ in Construction~1 and
subsequent applications of~$\rotA$ will be the same as from the
partition where one has circularly permuted the connected components
of~$\hat\pi$ and subsequent applications of~$\rotA$. Consequently,
using elementary generating function calculus, if
$G(z)$ denotes the generating function $\sum _{\ga}
^{}w^{(m)}(\ga)$, where the sum is over all possible connected
components, the generating function that we want to compute equals
\begin{equation} \label{eq:G(z)}
 \frac {mn-2} {\rrr }\coef{z^{(mn-2) /{\rrr }}}
\sum_{k=1}^\infty \frac {1} {k}G^k(z)=
- \frac {mn-2} {\rrr }\coef{z^{(mn-2) /{\rrr }}}\log \big(1-G(z)\big).
\end{equation}

In order to find an expression for $G(z)$, we observe that this can be
achieved by following the proof of \Cref{lem:TBa} (with the role of
the arc covering the connected component being played by the arc
$\{1,mn\}$  there). The result is
$$
G(z)=\sum _{i=1} ^{\infty}x_iz^{mi}\left(C^{(m)}(z)\right)^{mi-1}
=1-\frac {1} {C^{(m)}(z)}=1-\frac {z} {f(z)},
$$
where we used \eqref{eq:CC} and the definition of $f(z)$ to find the right-hand side.
We substitute this in~\eqref{eq:G(z)} to find that the generating
function that we are interested in is
\begin{align*}
- \frac {mn-2} {\rrr }\coef{z^{(mn-2) /{\rrr }}}\log \frac {z} {f(z)}
&=
\coef{z^{(mn-2) /{\rrr }}}z\frac {d} {dz}\log \frac {f(z)} z\\
&=\coef{z^{(mn-2) /{\rrr }}}
\frac {zf'(z)-f(z)} { f(z)}\\
&=\coef{z^0} z^{1-(mn-2) /{\rrr }}f^{-1}(z)f'(z).
\end{align*}
By Lagrange inversion in the form~\eqref{eq:CA}, this expression can
be transformed into the one in~\eqref{eq:DA-ma=2}.
\end{proof}

\begin{lemma} \label{lem:TF}
Let $\rrr ,m,n,a$ be positive integers with $\rrr \ge2$, $\rrr \mid (mn-2)$, 
and $a\equiv n\ (\text{\rm mod}\ \rrr )$.
The generating function $\sum _{\pi}
^{}w_{\rrr }^{(m)}(\pi)$, where the sum is over all positive
$m$-divisible non-crossing partitions of\/ $\{1,2,\dots,mn\}$
which are invariant under the
$\rrr $-pseudo-rotation~$\rotA^{(mn-2)/\rrr }$ and have a central block of
size~$ma$ equals
\begin{equation} \label{eq:DA} 
\frac {mn-2} {ma-2}\coef{z^{(mn-2)/\rrr }}f^{(ma-2)/\rrr }(z)=
\coef{z^0}z^{(ma-2)/\rrr }F^{-(mn-2)/\rrr }(z),
\end{equation}
where $F(z)$ is defined by \eqref{eq:CD} and
$f(z)$ is the compositional inverse of~$F(z)$, as
before. In the special case where $ma=2$, the left-hand side
of~\eqref{eq:DA} must be ignored.
\end{lemma}

\begin{remark}
The condition $a\equiv n\ (\text{\rm mod}\ \rrr )$ is actually redundant,
but is recorded in the statement of the lemma for the convenience of the
reader. More precisely, according to \Cref{lem:allA}, the non-crossing
partitions in the statement of the lemma arise exclusively from Construction~1.
In that construction, the number of the
non-central blocks must be divisible by~$\rrr$. Consequently, the number
of elements contained in non-central blocks is divisible by~$m\rrr$.
Thus, we have $mn=ma+m\rrr B$, for some~$B$. Clearly, this implies
$a\equiv n\ (\text{\rm mod}\ \rrr )$.
\end{remark}

\begin{proof}[Proof of \Cref{lem:TF}]
We have already treated the case where $ma=2$ in \Cref{lem:TF-ma=2} above.
So, let us assume that $ma\ne2$ from now on.

The equality in \eqref{eq:DA} is a direct consequence of the Lagrange
inversion formula~\eqref{eq:CB} with $a$ replaced by $(ma-2)/\rrr $ and 
$b=-(mn-2)/\rrr $.
Hence, it suffices to derive one of the two formulae.

\Cref{lem:allA}(1) says that --- essentially --- enumerating
positive $m$-divisible non-crossing partitions of $\{1,2,\dots,mn\}$ 
which are invariant
under the $\rrr $-pseudo
rotation~$\rotA^{(mn-2)/\rrr }$ is equivalent to enumerating (ordinary)
non-crossing partitions of $mn-2$ elements,
with all block sizes except the one of the central block multiples
of~$m$, which are invariant under the
$\rrr $-rotation~$\rot^{(mn-2)/\rrr }$.

\begin{figure}
  \begin{tikzpicture}[scale=1]
    \polygon{(0,0)}{obj}{28}{2.5}
      {2,3,4,5,6,7,8,9,10,11,12,13,14,15,16,17,18,19,20,21,22,23,24,25,26,27,28,29}
     \draw[fill=black,fill opacity=0.1] (obj1) to[bend right=30]
(obj8) to[bend right=30] (obj15) to[bend right=30] (obj22) to[bend
right=30] (obj1);
     \draw[fill=black,fill opacity=0.1] (obj2) to[bend right=30]
(obj6) to[bend right=30] (obj7) to[bend left=30] (obj2);
     \draw[fill=black,fill opacity=0.1] (obj3) to[bend right=30]
(obj4) to[bend right=30] (obj5) to[bend left=30] (obj3);
     \draw[fill=black,fill opacity=0.1] (obj9) to[bend right=30]
(obj13) to[bend right=30] (obj14) to[bend left=30] (obj9);
     \draw[fill=black,fill opacity=0.1] (obj10) to[bend right=30]
(obj11) to[bend right=30] (obj12) to[bend left=30] (obj10);
     \draw[fill=black,fill opacity=0.1] (obj16) to[bend right=30]
(obj20) to[bend right=30] (obj21) to[bend left=30] (obj16);
     \draw[fill=black,fill opacity=0.1] (obj17) to[bend right=30]
(obj18) to[bend right=30] (obj19) to[bend left=30] (obj17);
     \draw[fill=black,fill opacity=0.1] (obj23) to[bend right=30]
(obj27) to[bend right=30] (obj28) to[bend left=30] (obj23);
     \draw[fill=black,fill opacity=0.1] (obj24) to[bend right=30]
(obj25) to[bend right=30] (obj26) to[bend left=30] (obj24);
    \end{tikzpicture}
\caption{A $4$-rotationally invariant non-crossing partition of
$\{2,3,\dots,29\}$}
\label{fig:16d}
\end{figure}
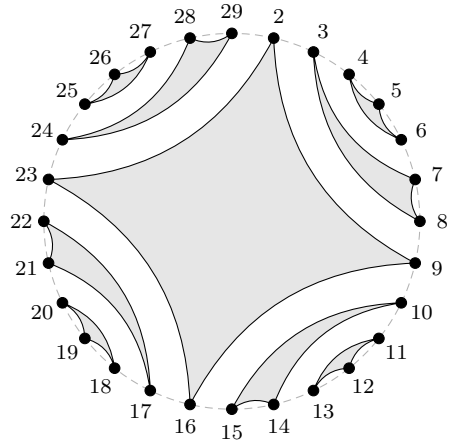

For the following considerations it will be useful to simultaneously
look at \Cref{fig:16d}
which shows a non-crossing partition
of $\{2,3,\dots,29\}$ which is invariant under the
$4$-pseudo-rotation~$\rot^{7}$ (in the sense described in Construction~1,
that is,
identifying $30$ and $2$) and which has a central block of size $4=3\cdot 2-2$.
This is the non-crossing partition which underlies the partitions in
\Cref{fig:3,fig:4} through Construction~1.

We want to count non-crossing partitions of $\{2,3,\dots,mn-1\}$
with all block sizes except the one of the central block multiples
of~$m$, the central block containing the element~$2$ and having size $ma-2$,
which remain invariant under a $\rrr $-rotation. The insertion
of~$1$ and~$mn$ and subsequent action
of~$\rotA$ which is described in \Cref{lem:allA}
to obtain \emph{all\/} positive $m$-divisible non-crossing partitions
of $\{1,2,\dots,mn\}$ then implies that the result has to
be multiplied by $\frac {(mn-2)/\rrr } {(ma-2)/\rrr }=\frac {mn-2} {ma-2}$.
(In particular, we divide by $(ma-2)/\rrr $ in order to avoid
overcounting: any element of the central block can be moved to the
element~$2$ by repeated application of the rotation~$\rot$!)

Since the partition is invariant under the $\rrr $-rotation~$\rot^{(mn-2)/\rrr }$, 
it suffices to concentrate on the ``fundamental" part of
the partition, which we may choose as the restriction of the partition
to $\mathcal F=\{2,3,\dots,\frac {mn-2} {\rrr }+1\}$.

Now let $s_1=2,s_2,\dots,s_{(ma-2)/\rrr }$ be the elements of the central block
which are at the same time contained in the fundamental section
$\mathcal F$.
The restriction of our original partition to the vertices in between
$s_i$ and $s_{i+1}$, $i=1,2,\dots,\frac {ma-2} {\rrr }$ (with the convention that
$s_{(ma-2)/\rrr +1}=s_1+\frac {mn-2} {\rrr }=\frac {mn-2} {\rrr }+2$) is an 
$m$-divisible non-crossing partition on $s_{i+1}-s_i-1$
vertices. That is, the restriction to the fundamental section
$\mathcal F$
consists of $(ma-2)/\rrr $ $m$-divisible non-crossing partitions 
``interleaved" by elements of the central block.
From \Cref{lem:TB} and the subsequent arguments we know that
the generating function for $m$-divisible non-crossing partitions is
given by $f(z)/z$. Then, by elementary
generating function calculus, we infer that
the generating function $\sum _{\pi}
^{}w_{\rrr }^{(m)}(\pi)$ of the statement in the lemma is equal to
$$\frac {mn-2} {ma-2}\coef{z^{(mn-2)/\rrr }}f^{(ma-2)/\rrr }(z).$$
This is exactly~\eqref{eq:DA}.
\end{proof}

\begin{lemma} \label{lem:TFa}
Let $m$ and $n$ be positive integers with $2\mid mn$.
The generating function $\sum _{\pi}
^{}w_{2}^{(m)}(\pi)$, where the sum is over all positive
$m$-divisible non-crossing partitions of\/ $\{1,2,\break\dots,mn\}$
without central block
which are invariant under the $2$-pseudo-rotation~$\rotA^{(mn-2)/2}$ equals
\begin{equation} \label{eq:DAa} 
\sum_{a=1}^\infty 
\frac {mn-2} {2}x_a\coef{z^{(mn-2)/2}}f^{ma-1}(z)=
\sum_{a=1}^\infty 
(ma-1)x_a\coef{z^{-(ma-1)}}F^{-(mn-2)/2}(z),
\end{equation}
where $F(z)$ is defined by~\eqref{eq:CD} and
$f(z)$ is the compositional inverse of~$F(z)$, as
before. 
\end{lemma}

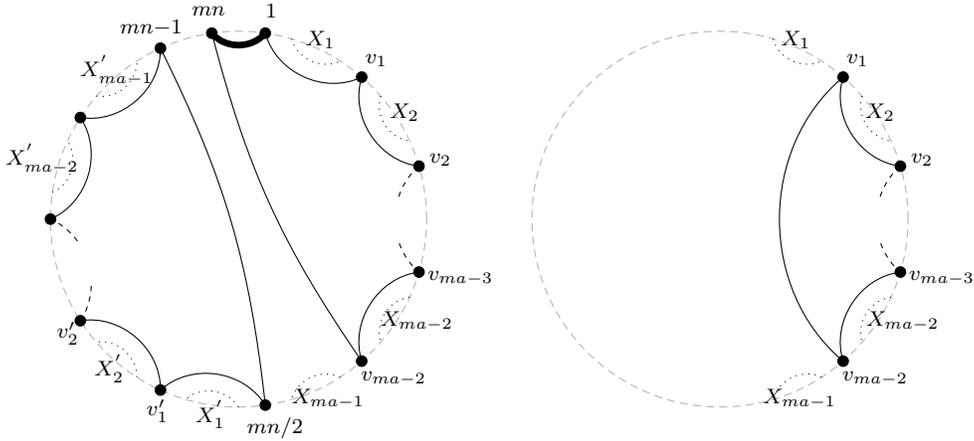
\begin{figure}
  \centering
  \begin{tikzpicture}[scale=1]
    \polygon{(4.4,0)}{obj}{44}{2.5}
      {,,,,v_1,,,,v_2,,,,\hspace{15pt}v_{ma-3},,,,\hspace{13pt}v_{ma-2},,,,,,,,,,,,,,,,,,,,,,,,,,,}

    \draw[dash pattern=on 2pt off 2pt on 2pt off 2pt on 2pt off 2pt on 2pt off 18pt on 2pt off 2pt on 2pt off 2pt] (obj9) to[bend right=50] (obj13);
    
    \draw (obj13) to[bend right=50]
          (obj17) to[bend left=50]
          (obj5)  to[bend right=50]
          (obj9);

    \draw[dotted] (obj2)  to[bend right=50] (obj4);
    \node at ($($1.05*($(obj3)+(4,0)$)$)-(4.5,0)$) {\tiny$X_1$};

    \draw[dotted] (obj6)  to[bend right=50] (obj8);
    \node at ($($1.05*($(obj7)+(4,0)$)$)-(4.5,0)$) {\tiny$X_2$};

    \draw[dotted] (obj14)  to[bend right=50] (obj16);
    \node at ($($1.15*($(obj15)+(4,0)$)$)-(5.25,-0.2)$) {\tiny$X_{ma-2}$};

    \draw[dotted] (obj18)  to[bend right=50] (obj20);
    \node at ($($1.05*($(obj19)+(4,0)$)$)-(4.45,0)$) {\tiny$X_{ma-1}$};

    \polygon{(-2,0)}{obj}{44}{2.5}
      {\hspace{2pt}1,,,,v_1,,,,v_2,,,,\hspace{15pt}v_{ma-3},,,,\hspace{13pt}v_{ma-2},,,,\hspace{5pt}mn/2,,,,\hspace{5pt}v_1',,,,\hspace{5pt}v_2',,,,\hspace{5pt},,,,\hspace{5pt},,,,mn-1,,mn,}

    \draw[line width=2.5pt,black] (obj1) to[bend left=50] (obj43);

    \draw[dash pattern=on 2pt off 2pt on 2pt off 2pt on 2pt off 2pt on 2pt off 18pt on 2pt off 2pt on 2pt off 2pt] (obj9) to[bend right=50] (obj13);
    
    \draw (obj13) to[bend right=50]
          (obj17) to[bend left=10]
          (obj43) to[bend right=50]
          (obj1)  to[bend right=50]
          (obj5)  to[bend right=50]
          (obj9);

    \draw[dotted] (obj2)  to[bend right=50] (obj4);
    \node at ($($1.05*($(obj3)-(4,0)$)$)+(4.3,0)$) {\tiny$X_1$};

    \draw[dotted] (obj6)  to[bend right=50] (obj8);
    \node at ($($1.05*($(obj7)-(4,0)$)$)+(4.3,0)$) {\tiny$X_2$};

    \draw[dotted] (obj14)  to[bend right=50] (obj16);
    \node at ($($1.15*($(obj15)-(4,0)$)$)+(4.85,0.2)$) {\tiny$X_{ma-2}$};

    \draw[dotted] (obj18)  to[bend right=50] (obj20);
    \node at ($($1.05*($(obj19)-(4,0)$)$)+(4.4,0)$) {\tiny$X_{ma-1}$};

    \draw[dash pattern=on 2pt off 2pt on 2pt off 2pt on 2pt off 2pt on 2pt off 18pt on 2pt off 2pt on 2pt off 2pt] (obj29) to[bend right=50] (obj33);
    
    \draw (obj33) to[bend right=50]
          (obj37) to[bend right=50]
          (obj41) to[bend left=10]
          (obj21) to[bend right=50]
          (obj25) to[bend right=50]
          (obj29);

    \draw[dotted] (obj22)  to[bend right=50] (obj24);
    \node at ($($1.05*($(obj23)-(4,0)$)$)+(4.3,0)$) {\tiny$X^{'}_{1}$};

    \draw[dotted] (obj26)  to[bend right=50] (obj28);
    \node at ($($1.05*($(obj27)-(4,0)$)$)+(4.3,0)$) {\tiny$X_{2}^{'}$};

    \draw[dotted] (obj34)  to[bend right=50] (obj36);
    \node at ($($1.15*($(obj35)-(4,0)$)$)+(5.05,0)$) {\tiny$X_{ma-2}^{'}$};

    \draw[dotted] (obj38)  to[bend right=50] (obj40);
    \node at ($($1.05*($(obj39)-(4,0)$)$)+(4.4,0)$) {\tiny$X_{ma-1}^{'}$};

  \end{tikzpicture}
  \caption{The decomposition of a $2$-pseudo-rotationally
    invariant non-crossing partition in 
$\mNCAPlus$ without central block}
\label{fig:9}
\end{figure}

\begin{proof}
According to \Cref{lem:allA}(2), all 
$2$-pseudo-rotationally invariant positive $m$-di\-vi\-si\-ble non-crossing
  partitions of $\{1,2,\dots,mn\}$ 
are obtained by starting with a non-crossing partition
obtained by Construction~2, $\pi$ say,
and applying the pseudo-rotation~$\rotA$ repeatedly to it.
Since, by \Cref{lem:allA}(1), $\rrr$-pseudo-rotational invariant 
positive non-crossing partitions necessarily have a central block if
$\rrr\ge3$, the fact that $\pi$ does not contain a central block implies that
the partitions $\rotA^i(\pi)$, $i=0,1,\dots,\frac {mn-2}2-1$,
are all different.
Therefore, what we need to do is to count the number of
non-crossing partitions of $\{1,2,\dots,mn\}$ which can be obtained
by Construction~2, and then multiply the result by $(mn-2)/2$.

If we fix the size of the block containing $mn$ and $1$ in 
Construction~2 to be $ma$, then we claim that
the generating function $\sum _{\pi}
^{}w_{2}^{(m)}(\pi)$ for the non-crossing partitions arising from
Construction~2 with the size of the block containing $mn$ and $1$
being equal to $ma$ is given by 
\begin{equation} \label{eq:decomp9}
x_az^{ma}\left(C^{(m)}(z)\right)^{ma-1}
=x_azf^{ma-1}(z),
\end{equation}
where $C^{(m)}(z)$ is the generating function for $m$-divisible
non-crossing partitions from\break \Cref{lem:TB}. To see this,
the reader must consult \Cref{fig:9}: the left part
of the figure shows a schematic
illustration of Construction~2 with $N=mn$ and $a$ replaced by~$ma$,
while the right part extracts the relevant
portions that explain the above expression. Here, $X_1,X_2,\dots,X_{ma-1}$
symbolise $m$-divisible non-crossing partitions that may be inserted
between $1,v_1,\dots,v_{ma-2}$ and after $v_{am-2}$, while
$X_i'=R^{(mn-2)/2}X_i$ and $v_i'=R^{(mn-2)/2}v_i$. The generating
function for each of the $X_i$'s is given by $C^{(m)}(z)$.
By elementary generating function
calculus, this gives the expression in~\eqref{eq:decomp9}.

In view of this expression, the 
number we want to compute is 
$$
\frac {mn-2} {2}x_a\coef{z^{mn/2}}zf^{ma-1}(z)=
\frac {mn-2} {2}x_a\coef{z^{(mn-2)/2}}f^{ma-1}(z).
$$
This number has to be summed over all possible~$a$'s. 
This leads directly to the left-hand side expression in~\eqref{eq:DAa}.
The right-hand side expression is then a consequence of the
Lagrange inversion formula~\eqref{eq:CB}.
\end{proof}

\begin{proof}[Proof of \Cref{thm:2}]
We distinguish two cases depending on whether $a>0$ 
or $a=0$, that is, whether the choice of~$b_k$'s implies the existence
of a central block or not.

If $a>0$, that is, $n>\rrr \sum_{j=1}^njb_j$, then,
using the right-hand side expression in~\eqref{eq:DA}
and~\eqref{eq:CD},
we see that the number which we want to compute is given by
\begin{align*}
\coef{\mathbf x^{\mathbf b}z^0}z^{(ma-2)/\rrr }F^{-(mn-2)/\rrr }(z)
&=
\coef{\mathbf x^{\mathbf b}z^0}z^{(ma-mn)/\rrr }\(1+\sum_{\ell=1}^\infty x_\ell
z^{m\ell}\)^{(mn-2)/\rrr }\\
&=\coef{\mathbf x^{\mathbf b}z^{m(n-a)/\rrr }}
\sum_{k=0}^{(mn-2)/\rrr }\binom {(mn-2)/\rrr }k
\(\sum_{\ell=1}^\infty x_\ell z^{m\ell}\)^k\\
&=
\binom {(mn-2)/\rrr }{b_1+b_2+\dots+b_n}
\binom {b_1+b_2+\dots+b_n}{b_1,b_2,\dots,b_n},
\end{align*}
where the short notation $\mathbf x^{\mathbf b}$ stands again for
$x_1^{b_1}x_2^{b_2}\cdots x_n^{b_n}$.
This is exactly the asserted expression.

On the other hand, if $\rrr =2$ and $n=2\sum_{j=1}^njb_j$, then,
using the right-hand side expression in~\eqref{eq:DAa}
and~\eqref{eq:CD}, 
we see that the number which we want to compute is given by
\begin{align*}
\sum_{a=1}^n
(ma-1)&\coef{x_a^{-1}\mathbf x^{\mathbf b}z^{-(ma-1)}}F^{-(mn-2)/2}(z)\\
&=\sum_{a=1}^n
(ma-1)\coef{x_a^{-1}\mathbf x^{\mathbf b}z^{\frac {mn} {2}-ma}}
\(1+\sum_{\ell=1}^\infty x_\ell z^{m\ell}\)^{(mn-2)/2}\\
&=\sum_{a=1}^n 
(ma-1)
\binom {(mn-2)/2}{b_1+b_2+\dots+b_n-1}
\binom {b_1+b_2+\dots+b_n-1}
{b_1,\dots,b_{a-1},b_a-1,b_{a+1},\dots,b_n}\\
&=
\binom {(mn-2)/2}{b_1+b_2+\dots+b_n}
\binom {b_1+b_2+\dots+b_n}
{b_1,b_2,\dots,b_n}
\sum_{a=1}^n 
\frac {(ma-1)b_a}
{\frac {mn} {2}-b_1-b_2-\dots-b_n}.
\end{align*}
Since $\sum_{a=1}^n {(ma-1)b_a}=\frac {mn}
{2}-b_1-b_2-\dots-b_n$,
this proves the formula in~\eqref{eq:multichains-bi-si-pos} also
in this case.
 
According to \Cref{lem:10} and \Cref{lem:allA}, in all other cases
there are no $\rrr $-pseudo-rotationally invariant positive $m$-divisible
non-crossing partitions.
\end{proof}

\begin{proof}[Proof of \Cref{cor:4}]
The expression~\eqref{eq:multichains-si-pos-a} could be obtained
from~\eqref{eq:multichains-bi-si-pos} by summing over all possible~$b_i$'s,
$i=1,2,\dots,n$, with $a=n-\rrr (b_1+2b_2+\dots+nb_n)$ and 
$b_1+b_2+\dots+b_n=b$. It is however simpler to start again with
the right-hand side of~\eqref{eq:DA} (if $a>0$), respectively
the right-hand side of~\eqref{eq:DAa} (if $a=0$), 
set all~$x_j$'s (except $x_0$
which, by definition, equals~$1$) equal to $t$, 
and finally extract
the coefficient of~$t^{b}$. 
Indeed, since in
$\rrr $-pseudo-rotationally invariant $m$-divisible non-crossing partitions
the number of non-central blocks of size~$mi$ must be a multiple of~$\rrr $,
implying that the size of the central block must be $ma=mn-m\rrr B$, for some
positive integer $B$, we have
$a\equiv n$~(mod~$\rrr $).

Hence, if $a>0$, from \eqref{eq:DA} we obtain
\begin{align*}
\coef{t^{b}z^0}
&z^{(ma-2)/\rrr }F^{-(mn-2)/\rrr }(z)\Big\vert_{x_1=x_2=\dots=t}\\
&=
\coef{t^{b}z^0}
z^{(ma-mn)/\rrr }\(1+t\sum_{\ell=1}^\infty 
z^{m\ell}\)^{(mn-2)/\rrr }\\
&=\coef{t^{b}z^{(mn-ma)/\rrr }}
\sum_{k=0}^{(mn-2)/\rrr }\binom {(mn-2)/\rrr }k
t^k\(\frac {z^m} {1-z^m}\)^k\\
&=\coef{z^{m\frac {n-a} {\rrr }-mb}}
\binom {(mn-2)/\rrr }{b}
\(\frac {1} {1-z^m}\)^{b}\\
&=
\binom {(mn-2)/\rrr }{b}
\binom {{(n-a-\rrr )/\rrr }}{b-1},
\end{align*}
which is exactly~\eqref{eq:multichains-si-pos-a}.

On the other hand, if $a=0$ (and hence $n\equiv0$~(mod~$2$) and $\rrr=2$), 
from~\eqref{eq:DAa} we obtain
\begin{align*}
\coef{t^{b}z^0}
&\sum_{\al=1}^\infty
(m\al-1)x_\al z^{m\al-1}F^{-(mn-2)/2}(z)
\Big\vert_{x_1=x_2=\dots=t}\\
&=
\coef{t^{b}z^0}\sum_{\al=1}^\infty
(m\al-1)t z^{m\al-\frac {mn} {2}}
\(1+t\sum_{\ell=1}^\infty 
z^{m\ell}\)^{(mn-2)/2 }\\
&=\sum_{\al=1}^\infty (m\al-1)
\coef{t^{b-1}z^{\frac {mn} {2}-m\al}}
\sum_{k=0}^{(mn-2)/2 }\binom {(mn-2)/2 }k
t^k\(\frac {z^m} {1-z^m}\)^k\\
&=\sum_{\al=1}^\infty (m\al-1)
\coef{z^{\frac {mn} {2}-m\al-m(b-1)}}
\binom {(mn-2)/2 }{b-1}
\(\frac {1} {1-z^m}\)^{b-1}\\
&=\sum_{\al=1}^\infty (m\al-1)
\binom {(mn-2)/2 }{b-1}
\binom {{\frac {n} {2}-\al-1 }}{b-2}\\
&=\binom {(mn-2)/2 }{b-1}
\sum_{\al=1}^\infty \left(m\left(\al-\tfrac {n} {2}\right)
  +\left(m\tfrac {n} {2}-1\right)\right)
\binom {{\frac {n} {2}-\al-1 }}{b-2}\\
&=\binom {(mn-2)/2 }{b-1}
\sum_{\al=1}^\infty 
\left(-m(b-1)\binom {{\frac {n} {2}-\al }}{b-1}
+\left(m\tfrac {n} {2}-1\right)\binom {{\frac {n} {2}-\al-1 }}{b-2}\right)\\
&=\binom {(mn-2)/2 }{b-1}
\left(-m(b-1)\binom {{\frac {n} {2}}}{b}
+\left(m\tfrac {n} {2}-1\right)\binom {{\frac {n} {2}-1 }}{b-1}\right)\\
&=\binom {(mn-2)/2 }{b-1}\binom {{\frac {n} {2}-1 }}{b-1}
\frac {m\frac {n} {2}-b} {b}\\
&=\binom {(mn-2)/2 }{b}\binom {{\frac {n} {2}-1 }}{b-1},
\end{align*}
which is exactly \eqref{eq:multichains-si-pos-a} with $a=0$.

\medskip
Finally, by summing the expression~\eqref{eq:multichains-si-pos-a} over all $a\equiv n$~(mod~$\rrr $),
we obtain~\eqref{eq:multichains-si-pos}.
\end{proof}

\subsection{Enumeration of  positive $m$-divisible non-cros\-sing partitions
in type $B$}
\label{app:GF-B}

As in \Cref{app:GF-A},
we use the generating function approach from~\cite{KratCG}.
Let $x_i$, $i=1,2,\dots$, be variables, and $x_0=1$.
For a positive $m$-divisible non-crossing partition~$\pi$ of type~$B$ let the weight
$w^{(m)}(\pi)$ be defined by
$$w^{(m)}(\pi)=\prod _{i=1} ^{\infty}
x_i^{\#(\text {non-zero blocks of $\pi$ of size
$mi$})/2}.$$
For example, the weight $w^{(3)}(\,.\,)$ of the
positive $m$-divisible non-crossing partitions of type~$B$ in
\Cref{fig:10b} is $x_1^3$, while the weight of the
non-crossing partitions in \Cref{fig:11b} is
$x_1^4$.

\begin{lemma} \label{lem:TF-ma=1}
Let $n$ be a positive integer.
The generating function $\sum _{\pi}
^{}w^{(1)}(\pi)$, where the sum is over all positive
non-crossing partitions of\/
$\{1,2,\dots,n,\overline1,\overline2,\dots,\overline n\}$ of type~$B$
which have a zero block of size~$2$ equals
\begin{equation} \label{eq:DA-ma=1} 
\coef{z^0}F^{-(n-1)}(z),
\end{equation}
where $F(z)$ is defined by~\eqref{eq:CD} and
$f(z)$ is the compositional inverse of~$F(z)$, as
before. 
\end{lemma}

\begin{proof}
Let $\pi$ be a positive non-crossing partition as described in the
statement of the lemma, and let the zero block be $\{a,\overline a\}$.
We may map $\pi$ to the pair $(a,\si)$, where $\si$ is the non-crossing
partition we see between $\overline {a+1}$ and $a-1$, considered as a
non-crossing partition of $\{1,2,\dots,n-1\}$, by appropriate
relabelling. Thus, we have defined a mapping from the non-crossing
partitions of the statement of the lemma and pairs $(a,\si)$, where
$1\le a\le n$ and $\si$ is a non-crossing partition of
$\{1,2,\dots,n-1\}$. However, this mapping, while injective, is not
surjective due to the condition in~\Cref{prop:1B} that the
element~$1$ must be in a block that also contains negative elements.
To be precise, given a non-crossing partition~$\si$, there are
$n-\#(\text{blocks of }\si)$ numbers $a$ such that there is a
corresponding non-crossing partition of type~$B$ as in the statement
of the lemma under this mapping. This observation can be modelled by
generating functions as follows: the generating function for the
non-crossing partitions in the statement of the lemma is given by
$$
\coef{z^{n-1}}
\left(\left(n-\tfrac {d} {dt}\right)C^{(1)}(z)\Big\vert_{x_i\to tx_i}
\right)\bigg\vert_{t=1}.
$$
Here, the replacement of the $x_i$'s is restricted to the $x_i$'s with
$i\ge1$ ($x_0$ being~$1$ by definition from the outset).
We use again that $C^{(1)}(z)=f(z)/z$ and subsequently apply Lagrange
inversion in the form~\eqref{eq:CB}. This leads to
\begin{align*}
n\coef{z^{n}}f(z)-\coef{z^n}
&\left(\tfrac {d} {dt}f(z)\Big\vert_{x_i\to tx_i}
\right)\bigg\vert_{t=1}
=
n\coef{z^{0}}z^{-n}f(z)-\coef{z^0}z^{-n}
\left(\tfrac {d} {dt}f(z)\Big\vert_{x_i\to tx_i}
\right)\bigg\vert_{t=1}\\
&=\coef{z^{0}}zF^{-n}(z)-\frac {1} {n}\coef{z^0}z
\left(\tfrac {d} {dt}F^{-n}(z)\Big\vert_{x_i\to tx_i}
\right)\bigg\vert_{t=1}\\
&=\coef{z^{0}}zF^{-n}(z)-\frac {1} {n}\coef{z^0}z^{1-n}
\tfrac {d} {dt}\left(1+t\sum_{i=1}^\infty x_iz^i\right)^n
\Bigg\vert_{t=1}\\
&=\coef{z^{0}}zF^{-n}(z)-\coef{z^0}z^{1-n}
\left(\sum_{i=1}^\infty x_iz^i\right)
\left(1+\sum_{i=1}^\infty x_iz^i\right)^{n-1}\\
&=\coef{z^{0}}zF^{-n}(z)-\coef{z^0}
\left(\frac {z} {F(z)}-1\right)
F^{-(n-1)}(z)\\
&=\coef{z^{0}}F^{-(n-1)}(z),
\end{align*}
as desired.
\end{proof}

\begin{lemma} \label{lem:TF-B-1}
Let $m,n,a$ be positive integers. 
The generating function $\sum _{\pi}
^{}w^{(m)}(\pi)$, where the sum is over all positive
$m$-divisible non-crossing partitions of\/ 
$\{1,2,\dots,mn,\overline1,\overline2,\dots,\break\overline{mn}\}$
of type~$B$ which have a zero block of
size~$2ma$ equals
\begin{equation} \label{eq:DA-B-1} 
\frac {mn-1} {ma-1}\coef{z^{mn-1}}f^{ma-1}(z)=
\coef{z^0}z^{ma-1}F^{-(mn-1)}(z),
\end{equation}
where $F(z)$ is defined by~\eqref{eq:CD} and
$f(z)$ is the compositional inverse of~$F(z)$, as
before. In the special case where $m=a=1$, the left-hand side
of~\eqref{eq:DA-B-1} must be ignored.
\end{lemma}

\begin{proof}
We have already treated the case where $m=a=1$ in \Cref{lem:TF-ma=1} above.
So, let us assume that $ma\ne1$ from now on.

Given an element of $\mNCBPlus$, by applying the
pseudo-rotation~$\rotB$ repeatedly to it, we will achieve to obtain
a non-crossing
partition in which the zero block contains $\overline{mn}$ {\em and\/}~$1$.
Another application of~$\rotB$ then just normally rotates the zero block,
and thus yields a non-crossing
partition in which the zero block contains $1$ {\em and\/}~$2$.
Hence, we may obtain all elements of $\mNCBPlus$ with a zero
block via the Construction in \Cref{sec:rotB} with $\rRr=1$. 
More precisely, we obtain all elements of $\mNCBPlus$ with a
zero block of size~$2ma$ by
starting with an ordinary $m$-divisible non-crossing
partition of
$\{2,3,\dots,mn,\overline2,\overline3,\dots,\overline{mn}\}$
containing a zero block of size~$2ma-2$ and that remains invariant under the map
that exchanges $i$ and $\overline i$ for all~$i$, then adding $1$ and
$\overline1$ to the zero block,
and subsequently applying the pseudo-rotation~$\rotB$ repeatedly.

The remaining arguments are
completely analogous to the proof of \Cref{lem:TF} and are therefore
left to the reader.
\end{proof}

\begin{lemma} \label{lem:TFa-B}
Let $m$ and $n$ be positive integers.
The generating function $\sum _{\pi}
^{}w^{(m)}(\pi)$, where the sum is over all positive
$m$-divisible non-crossing partitions of 
$\{1,2,\dots,mn,\overline1,\overline2,\break\dots,\overline{mn}\}$
of type~$B$ without zero block equals
\begin{equation} \label{eq:DAa-B} 
\sum_{a=1}^\infty 
(mn-1)x_a\coef{z^{mn-1}}f^{ma-1}(z)=
\sum_{a=1}^\infty 
(ma-1)\,x_a\coef{z^{-(ma-1)}}F^{-(mn-1)}(z),
\end{equation}
where $F(z)$ is defined by~\eqref{eq:CD} and
$f(z)$ is the compositional inverse of~$F(z)$, as
before. 
\end{lemma}

\begin{proof}
We claim that all positive $m$-divisible non-crossing
partitions of $\{1,2,\dots,mn,\overline1,\overline2,\break\dots,\overline{mn}\}$ 
of type~$B$ without zero block
can be obtained by starting with a non-crossing partition
that has a block containing~$1$ and $\overline2$, and otherwise only
negative elements,
and then applying the pseudo-rotation~$\rotB$ repeatedly to this partition.
Indeed, the claim follows immediately by a careful reading of the proof of
\Cref{lem:ord=N-1}.

Consequently, what we need to do is to count the number of
non-crossing partitions of 
$\{1,2,\dots,mn,\overline1,\overline2,\dots,\overline{mn}\}$ 
as described above, and then multiply the result by~$mn-1$.

If we fix the size of the block containing $1$ and $\overline2$ in
the above claim to be $ma$, then, by elementary generating function
calculus,  we obtain that
the generating function $\sum _{\pi}
^{}w^{(m)}(\pi)$ for these non-crossing partitions 
is given by 
$$
x_az^{ma}\left(C^{(m)}(z)\right)^{ma-1}
=x_azf^{ma-1}(z),
$$
where $C^{(m)}(z)$ is the generating function for $m$-divisible
non-crossing partitions from\break \Cref{lem:TB}. Hence, the 
number we want to compute is 
$$
(mn-1)x_a\coef{z^{mn}}zf^{ma-1}(z)=
(mn-1)x_a\coef{z^{mn-1}}f^{ma-1}(z).
$$
This number has to be summed over all possible~$a$'s. 
This leads directly to the left-hand side expression in~\eqref{eq:DAa-B}.
The right-hand side expression is then a consequence of the
Lagrange inversion formula~\eqref{eq:CB}.
\end{proof}

\begin{proof}[Proof of \Cref{thm:1-B}]
We distinguish two cases depending on whether $a>0$ 
or $a=0$, that is, whether the choice of~$b_k$'s implies the existence
of a zero block or not.

If $a>0$, that is, $n> \sum_{j=1}^njb_j$, then,
using the right-hand side expression in~\eqref{eq:DA-B-1}
and~\eqref{eq:CD},
we see that the number which we want to compute is given by
\begin{align*}
\coef{\mathbf x^{\mathbf b}z^0}z^{ma-1 }F^{-(mn-1)}(z)
&=
\coef{\mathbf x^{\mathbf b}z^0}z^{ma-mn}\(1+\sum_{\ell=1}^\infty x_\ell
z^{m\ell}\)^{mn-1}\\
&=\coef{\mathbf x^{\mathbf b}z^{m(n-a)}}
\sum_{k=0}^{mn-1}\binom {mn-1}k
\(\sum_{\ell=1}^\infty x_\ell z^{m\ell}\)^k\\
&=
\binom {mn-1}{b_1+b_2+\dots+b_n}
\binom {b_1+b_2+\dots+b_n}{b_1,b_2,\dots,b_n},
\end{align*}
where the short notation $\mathbf x^{\mathbf b}$ stands again for
$x_1^{b_1}x_2^{b_2}\cdots x_n^{b_n}$.
This is exactly the asserted expression.

On the other hand, if $a=0$ and hence $n=\sum_{j=1}^njb_j$, then,
using the right-hand side expression in~\eqref{eq:DAa-B}
and~\eqref{eq:CD}, 
we see that the number which we want to compute is given by
\begin{align*}
\sum_{a=1}^n
(ma-1)&\coef{x_a^{-1}\mathbf x^{\mathbf b}z^{-(ma-1)}}F^{-(mn-1)}(z)\\
&=\sum_{a=1}^n
(ma-1)\coef{x_a^{-1}\mathbf x^{\mathbf b}z^{mn-ma}}
\(1+\sum_{\ell=1}^\infty x_\ell z^{m\ell}\)^{mn-1}\\
&=\sum_{a=1}^n 
(ma-1)
\binom {mn-1}{b_1+b_2+\dots+b_n-1}
\binom {b_1+b_2+\dots+b_n-1}
{b_1,\dots,b_{a-1},b_a-1,b_{a+1},\dots,b_n}\\
&=
\binom {mn-1}{b_1+b_2+\dots+b_n}
\binom {b_1+b_2+\dots+b_n}
{b_1,b_2,\dots,b_n}
\sum_{a=1}^n 
\frac {(ma-1)b_a}
{mn-b_1-b_2-\dots-b_n}.
\end{align*}
Since $\sum_{a=1}^n {(ma-1)b_a}=mn
-b_1-b_2-\dots-b_n$,
this proves the formula in~\eqref{eq:block-B}.
\end{proof}

\begin{proof}[Proof of \Cref{cor:2-B}]
The expression~\eqref{eq:multichains-si-pos-a-B-1} could be obtained
from~\eqref{eq:block-B} by summing over all possible~$b_i$'s,
$i=1,2,\dots,n$, with $a=n- (b_1+2b_2+\dots+nb_n)$ and 
$b_1+b_2+\dots+b_n=b$. It is however simpler to start again with
the right-hand sides of~\eqref{eq:DA-B-1} respectively~\eqref{eq:DAa-B},
set all~$x_j$'s equal to~$t$ (except $x_0$
which, by definition, equals~$1$), 
and finally extract
the coefficient of~$t^{b}$. 

For the case where $a>0$, we would start with the right-hand side
of~\eqref{eq:DA-B-1} and then perform the above described steps.
A more general calculation is done later in the proof of
\Cref{cor:4-B}. More precisely, if we put $\rRr=1$ there, then
the calculation there reduces to the one which is relevant here.
The result is~\eqref{eq:multichains-si-pos-a-B} with $\rRr=1$, which
indeed agrees with~\eqref{eq:multichains-si-pos-a-B-1}.

On the other hand, if $a=0$, then from~\eqref{eq:DAa-B} we get
\begin{align*}
\coef{t^{b}z^0}
&\sum_{\al=1}^\infty 
(m\al-1)\,x_\al\coef{z^{-(m\al-1)}}F^{-(mn-1)}(z)\Big\vert_{x_1=x_2=\dots=t}\\
&=
\coef{t^{b}z^0}
\sum_{\al=1}^\infty (m\al-1)t
z^{m\al-mn}\(1+t\sum_{\ell=1}^\infty 
z^{m\ell}\)^{mn-1}\\
&=\coef{t^{b-1}z^{m(n-1)}}
\left(\frac {m} {(1-z^m)^2}-\frac {1} {1-z^m}\right)
\sum_{k=0}^{mn-1}\binom {mn-1}k
t^k\(\frac {z^m} {1-z^m}\)^k\\
&=\coef{z^{m(n-b)}}
\binom {mn-1}{b-1}
\left(
m\(\frac {1} {1-z^m}\)^{b+1}
-\(\frac {1} {1-z^m}\)^{b}
\right)\\
&=
\binom {mn-1}{b-1}
\left(
m\binom {n}{b}
-\binom {n-1}{b-1}
\right)\\
&=
\binom {mn-1 }{b-1}
\frac {(n-1)!} {b!\,(n-b)!}(mn-b),
\end{align*}
which can be simplified to~\eqref{eq:multichains-si-pos-a-B-1} with $a=0$.

By summing the expression~\eqref{eq:multichains-si-pos-a-B-1} over all
possible~$a$, we obtain~\eqref{eq:multichains-si-pos-B-1}. 
\end{proof}

\subsection{Enumeration of pseudo-rotationally invariant
  positive $m$-divisible non-cros\-sing partitions
in type $B$}
\label{app:GF-B-inv}

Again, we use a generating function approach as in \Cref{app:GF-A}.
Abusing notation from \Cref{app:GF-A-inv},
we introduce the weight function $w_{\rRr }^{(m)}$, defined by
$$w^{(m)}_\rRr (\pi)=\prod _{i=1} ^{\infty}
x_i^{\#(\text {non-zero blocks of~$\pi$ of size
$mi$})/2\rRr }.$$
For example, the weight $w_{2}^{(3)}(\,.\,)$ of both non-crossing partitions
in \Cref{fig:34,fig:35} is $x_1^2$.

\begin{lemma} \label{lem:TF-B}
Let $\rRr ,m,n,a$ be positive integers with $\rRr \ge1$, $\rRr \mid (mn-1)$, 
and $a\equiv n\ (\text{\rm mod}\ \rRr )$.
The generating function $\sum _{\pi}
^{}w_{\rRr }^{(m)}(\pi)$, where the sum is over all positive
$m$-divisible non-crossing partitions of\/ 
$\{1,2,\dots,mn,\overline1,\overline2,\dots,\overline{mn}\}$
of type~$B$ which are invariant under the
$2\rRr $-pseudo-rotation~$\rotB^{(mn-1)/\rRr }$ and have a zero block of
size~$ma$ equals
\begin{equation} \label{eq:DA-B} 
\frac {mn-1} {ma-1}\coef{z^{(mn-1)/\rRr }}f^{(ma-1)/\rRr }(z)=
\coef{z^0}z^{(ma-1)/\rRr }F^{-(mn-1)/\rRr }(z),
\end{equation}
where $F(z)$ is defined by~\eqref{eq:CD} and
$f(z)$ is the compositional inverse of~$F(z)$, as
before. In the special case where $m=a=1$, the left-hand side
of~\eqref{eq:DA-B} must be ignored.
\end{lemma}

\begin{remark}
Similarly to \Cref{lem:TF}, also here the 
condition $a\equiv n\ (\text{\rm mod}\ \rRr )$ is actually redundant.
\end{remark}

\begin{proof}[Proof of \Cref{lem:TF-B}]
Using \Cref{lem:allB},
this is completely analogous to the proofs of \Cref{lem:TF,lem:TF-ma=2} and 
is therefore left to the reader.
\end{proof}

\begin{proof}[Proof of \Cref{thm:2-B}]
Using the right-hand side expression in~\eqref{eq:DA-B}
and~\eqref{eq:CD},
we see that the number which we want to compute is given by
\begin{align*}
\coef{\mathbf x^{\mathbf b}z^0}z^{(ma-1)/\rRr }F^{-(mn-1)/\rRr }(z)
&=
\coef{\mathbf x^{\mathbf b}z^0}z^{(ma-mn)/\rRr }\(1+\sum_{\ell=1}^\infty x_\ell
z^{m\ell}\)^{(mn-1)/\rRr }\\
&=\coef{\mathbf x^{\mathbf b}z^{m(n-a)/\rRr }}
\sum_{k=0}^{(mn-1)/\rRr }\binom {(mn-1)/\rRr }k
\(\sum_{\ell=1}^\infty x_\ell z^{m\ell}\)^k\\
&=
\binom {(mn-1)/\rRr }{b_1+b_2+\dots+b_n}
\binom {b_1+b_2+\dots+b_n}{b_1,b_2,\dots,b_n},
\end{align*}
where the short notation $\mathbf x^{\mathbf b}$ stands again for
$x_1^{b_1}x_2^{b_2}\cdots x_n^{b_n}$.
This is exactly the asserted expression.
\end{proof}

\begin{proof}[Proof of \Cref{cor:4-B}]
The expression~\eqref{eq:multichains-si-pos-a-B} could be obtained
from~\eqref{eq:multichains-bi-si-pos-B} by summing over all possible~$b_i$'s,
$i=1,2,\dots,n$, with $a=n-\rRr (b_1+2b_2+\dots+nb_n)$ and 
$b_1+b_2+\dots+b_n=b$. It is however simpler to start again with
the right-hand side of~\eqref{eq:DA-B},
set all~$x_j$'s equal to $t$ (except $x_0$
which, by definition, equals~$1$), 
and finally extract
the coefficient of~$t^{b}$. 
Indeed, since in
$2\rRr $-pseudo-rotationally invariant $m$-divisible non-crossing partitions
of type~$B$
the number of non-zero blocks of size~$mi$ must be a multiple of~$2\rRr $,
implying that the size of the zero block must be $2ma=2mn-2m\rRr B$, for some
positive integer $B$, we have
$a\equiv n$~(mod~$\rRr $).
Hence, we obtain
\begin{align*}
\coef{t^{b}z^0}
&z^{(ma-1)/\rRr }F^{-(mn-1)/\rRr }(z)\Big\vert_{x_1=x_2=\dots=t}\\
&=
\coef{t^{b}z^0}
z^{(ma-mn)/\rRr }\(1+t\sum_{\ell=1}^\infty 
z^{m\ell}\)^{(mn-1)/\rRr }\\
&=\coef{t^{b}z^{(mn-ma)/\rRr }}
\sum_{k=0}^{(mn-1)/\rRr }\binom {(mn-1)/\rRr }k
t^k\(\frac {z^m} {1-z^m}\)^k\\
&=\coef{z^{m\frac {n-a} {\rRr }-mb}}
\binom {(mn-1)/\rRr }{b}
\(\frac {1} {1-z^m}\)^{b}\\
&=
\binom {(mn-1)/\rRr }{b}
\binom {{(n-a-\rRr )/\rRr }}{b-1},
\end{align*}
which is exactly \eqref{eq:multichains-si-pos-a-B}.

By summing the expression \eqref{eq:multichains-si-pos-a-B} over all
possible~$a$, that is, over all~$a$ with $a\equiv n$~(mod~$\rRr $),
we obtain~\eqref{eq:multichains-si-pos-B}. 
\end{proof}

\subsection{Enumeration of pseudo-rotationally invariant
  positive $m$-divisible non-cros\-sing partitions
in type $D$}
\label{app:GF-D-inv}


\begin{proof}[Proof of \Cref{thm:enumD-2}]
We assume that the reader recalls the definition of ``fundamental seed"
given in Construction~3 in \Cref{sec:rotD}.

Let $B_1$ be the block containing $\overline{m(n-1)}$ and $1$.
We assume that it contains $i_1$ elements of the inner circle, $i_1\ge0$.
We ignore the choice of signs of these elements for the moment, but put on record that,
in the end, the result must be multiplied by~$2$ since there are two
possible choices for these signs.

Likewise, let $B_2,B_3,\dots,B_k$ be
the further blocks of the fundamental seed that contain elements
from the outer \emph{and\/} the inner circle. For $j=2,3,\dots,k$,
we assume that block
$B_j$ contains $i_j$ elements of the inner circle, $i_j\ge1$.
For $j=1,2,\dots,k$, we let the block size of~$B_j$ 
equal~$ma_j$.

As in the proof of \Cref{lem:Const3}, we write $x$ for the number 
of elements on the inner circle that are
contained in the fundamental seed according to Construction~3,
and we write $my$ for the total number of elements in this
fundamental seed.

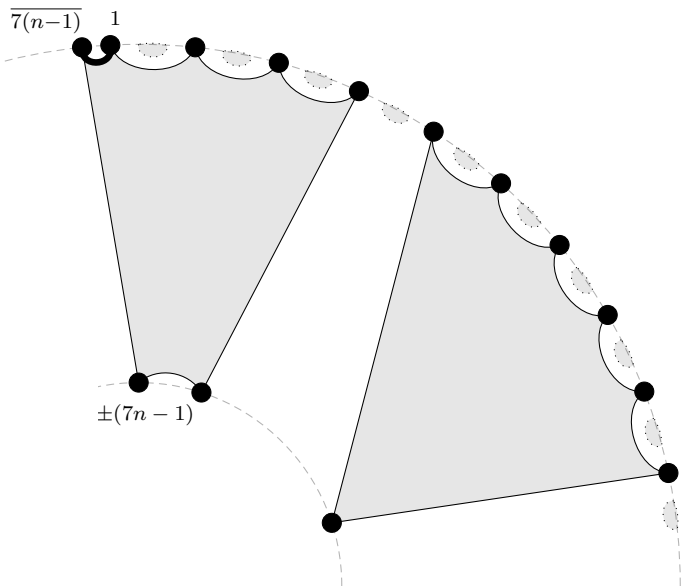
\begin{figure}
\renewcommand{\polygon}[5]{ 
  \pgfmathsetmacro{\angle}{360/#3}
  \foreach \t in {1,...,#3} {
    \coordinate (#2\t) at ($#1+(90-\t*\angle:#4)$);
  }
  \draw[thin,black,fill=white,opacity=0.3,densely dashed] #1 circle (#4);
  \setcounter{intege}{1}
  \pgfmathsetcounter{intege}{1}
  \foreach \object in {#5}{
    \ifthenelse{\not\equal{\object}{}}{
      \filldraw[black] ($#1+($(90-\theintege*\angle:#4)$)$) circle(2pt);
    }{
    }
    \node[inner sep=0pt] at ($#1+($1.05*(90-\theintege*\angle:#4)$)$)
{$\scriptstyle\object$};
    \pgfmathsetcounter{intege}{\theintege+1}
    \setcounter{intege}{\theintege}
  }
}
\vspace*{-50pt}
\begin{center}
\hspace*{65pt}
  \begin{tikzpicture}[scale=1.8]
      \begin{scope}
      \clip (-1,-0) rectangle (6,6);
       \polygon{(0,0)}{obj}{120}{4}
         {,\hspace{5pt},,,\hspace{5pt},,,\hspace{5pt},,,\hspace{5pt},,,\hspace{5pt},,,\hspace{5pt},,,\hspace{5pt},,,\hspace{5pt},,,\hspace{5pt},,,,,,,,,,,,,,,,,,,,,,,,,,,,,,,,,,,,,,,,,,,,,,,,,,,,,,,,,,,,,,,,,,,,,,,,,,,,,,,,,,,,,,,,,,,,\overline{7(n-1)}\hspace*{25pt},\hspace*{5pt}1,}
      \end{scope}
      \begin{scope}
      \clip (-0.3,0) rectangle (6,6);
      \polygon{(0,0)}{objin}{20}{1.5}
        {,,,\hspace{5pt},,,,,,,,,,,,,,,,\hspace{5pt},\hspace{5pt}\raise-30pt\hbox{\kern-50pt\tiny $\pm(7n-1)$\hfill}}{0.65}
      \end{scope}

         \draw[line width=2.5pt,black] (obj119) to[bend left=100,
looseness=2] (obj118);

       \draw[fill=black,fill opacity=0.1] (obj119) to[bend right=70]
             (obj2) to[bend right=70]
             (obj5) to[bend right=70]
             (obj8) to
             (objin1) to[bend right=50]
             (objin20) to
             (obj118) to[bend right=100, looseness=2]
             (obj119);

       \draw[fill=black,fill opacity=0.1]
             (obj11) to[bend right=70]
             (obj14) to[bend right=70]
             (obj17) to[bend right=70]
             (obj20) to[bend right=70]
             (obj23) to[bend right=70]
             (obj26) to
             (objin4) to
             (obj11);

      \draw[dotted, fill=black, fill opacity=0.1] (obj120)  to[bend
right=100, looseness=1.5] (obj1) to[bend right=10] (obj120);
      \draw[dotted, fill=black, fill opacity=0.1] (obj3)  to[bend
right=100, looseness=1.5] (obj4) to[bend right=15] (obj3);
      \draw[dotted, fill=black, fill opacity=0.1] (obj6)  to[bend
right=100, looseness=1.5] (obj7) to[bend right=15] (obj6);
      \draw[dotted, fill=black, fill opacity=0.1] (obj9)  to[bend
right=100, looseness=1.5] (obj10) to[bend right=15] (obj9);
      \draw[dotted, fill=black, fill opacity=0.1] (obj12)  to[bend
right=100, looseness=1.5] (obj13) to[bend right=15] (obj12);
      \draw[dotted, fill=black, fill opacity=0.1] (obj15)  to[bend
right=100, looseness=1.5] (obj16) to[bend right=15] (obj15);
      \draw[dotted, fill=black, fill opacity=0.1] (obj18)  to[bend
right=100, looseness=1.5] (obj19) to[bend right=15] (obj18);
      \draw[dotted, fill=black, fill opacity=0.1] (obj21)  to[bend
right=100, looseness=1.5] (obj22) to[bend right=15] (obj21);
      \draw[dotted, fill=black, fill opacity=0.1] (obj24)  to[bend
right=100, looseness=1.5] (obj25) to[bend right=15] (obj24);
      \draw[dotted, fill=black, fill opacity=0.1] (obj27)  to[bend
right=100, looseness=1.5] (obj28) to[bend right=15] (obj27);
  \end{tikzpicture}
\end{center}
  \caption{A fundamental seed with two bridging blocks}
\label{fig:bridging}
\end{figure}

The generating function $\sum z^{|\sigma|_o}$, with the sum running over all
$m$-divisible non-crossing partitions~$\sigma$ defining a fundamental seed
with $k$ bridging blocks, is given by 
\begin{equation} \label{eq:GFConstr3} 
\underset{a_1,\dots,a_k\ge1}
{\underset{i_1\ge0,\,i_2,\dots,i_k\ge1}
{\sum_{i_1+\dots +i_k=x}}}z\,(zC(z))^{ma_1-i_1-1}
(zC(z))^{ma_2-i_2}\cdots (zC(z))^{ma_k-i_k},
\end{equation}
where $C(z)$ is short for our earlier series $C^{(m)}(z)$
in which all~$x_i$'s are set equal to~1, and
where $|\sigma|_o$ denotes the number of
elements on the outer circle that are
covered by~$\sigma$. Indeed, by definition
(see \Cref{lem:TB}) $C(z)=C^{(m)}(z)$ (with $x_i=1$ for all~$i$)
is the generating function
for $m$-divisible non-crossing partitions. Between any two vertices
of a bridging block that lie on the outer circle we may place an
$m$-divisible non-crossing partition, as well as between the last
vertex on the outer circle of a bridging block and the first
vertex of the next bridging block; the only exception is that
$\overline{m(n-1)}$ and~1, the first two elements on the outer circle
of the first bridging block, $B_1$, must be successive, and hence
nothing can be placed between these two. Since block~$B_j$ has
$ma_j-i_j$ elements on the outer circle, this explains the summand
in~\eqref{eq:GFConstr3}. \Cref{fig:bridging} gives a schematic
illustration of a seed with two bridging blocks, in which
$m=7$, $a_1=1$, $i_1=2$, $a_2=1$, $i_2=1$,
$x=i_1+i_2=3$, and $y=a_1+a_2=2$. The shaded ``half disks"
indicate smaller non-crossing partitions that are placed between
vertices on the outer circle.

Recalling our earlier notation $f(z)=zC(z)$ and summing over
$a_1,a_2,\dots,a_k$, we obtain
$$
{\underset{i_1\ge0,\,i_2,\dots,i_k\ge1}
{\sum_{i_1+\dots +i_k=x}}}z\,f^{mk-i_1-\dots-i_k-1}(z)
\big(1-f^m(z)\big)^{-k}
=
{\underset{i_1\ge0,\,i_2,\dots,i_k\ge1}
{\sum_{i_1+\dots +i_k=x}}}z\,f^{mk-x-1}(z)
\big(1-f^m(z)\big)^{-k}.
$$
It is an easy combinatorial exercise to see that the number of
tuples $(i_1,i_2,\dots,i_k)$ summing to~$x$, where $i_1$ is
non-negative and all other~$i_j$'s are positive, equals
$\binom {x}{k-1}$. Thus we obtain
$$
\binom x{k-1}
z\,f^{mk-x-1}(z)
\big(1-f^m(z)\big)^{-k}
$$
for our generating function.
However, since we have to apply $\rotD$ multiple times, we have to
multiply the result by $(m(n-1)-1)/\rRr$, and we have to divide by~$k$
since any of the $k$ bridging blocks could have been the ``first"
block in a fundamental seed. Subsequently, we must sum the
result over all possible~$k$. In other words, we must compute
\begin{align*}
\sum_{k\ge1}
\frac {(m(n-1)-1)/\rRr} {k}&\binom x{k-1}
z\,f^{mk-x-1}(z)
\big(1-f^m(z)\big)^{-k}\\
&=
\sum_{k\ge1}
\frac {my-x-1} {x+1}\binom {x+1}{k}
z\,f^{mk-x-1}(z)
\big(1-f^m(z)\big)^{-k}\\
&=
\frac {my-x-1} {x+1}
\left(
z\,f^{-x-1}(z)\left(1+\frac {f^m(z)} {1-f^m(z)}\right)^{x+1}
-
z\,f^{-x-1}(z)\right)\\
&=
\frac {my-x-1} {x+1}
z\,f^{-x-1}(z)
\left(
\left( {1-f^m(z)}\right)^{-x-1}
-1
\right).
\end{align*}
Here we used \eqref{eq:rotrel1} to obtain the second line.

In the last expression, we want to extract the coefficient of $z^{my-x}$.
Equivalently, we want to compute
$$
\frac {my-x-1} {x+1}
\coef{z^{my-x-1}}
f^{-x-1}(z)
\left(
\left( {1-f^m(z)}\right)^{-x-1}
-1
\right).
$$
By using the Lagrange inversion formula in the form~\eqref{eq:CBa},
with $b=my-x-1$ and $g(z)=z^{-x-1}\big(( {1-z^m})^{-x-1}-1\big)$, 
this is the same as
\begin{multline*}
\frac {my-x-1} {x+1}
\frac {1} {my-x-1}\coef{z^{-1}}
F^{-my+x+1}(z)
\frac {d} {dz}\left(z^{-x-1}
\left(
\left( {1-z^m}\right)^{-x-1}
-1
\right)\right)\\
=
\coef{z^{-1}}
F^{-my+x+1}(z)
z^{-x-2}
\left(1
+((m+1)z^{m}-1)
\left( {1-z^m}\right)^{-x-2}
\right).
\end{multline*}
Since we specialised all $x_i$'s to~$1$, the series $F(z)$ equals
$z(1-z^m)$. If we substitute this in the above expression, then we get
\begin{align*}
\coef{z^{-1}}
z^{-my-1}&(1-z^m)^{-my+x+1}
\left(1
+((m+1)z^{m}-1)
\left( {1-z^m}\right)^{-x-2}
\right)\\
&=
\coef{z^{my}}
(1-z^m)^{-my+x+1}
+((m+1)z^{m}-1)
\left( {1-z^m}\right)^{-my-1}\\
&=\binom {my-x-2+y}y+(m+1)\binom {my+y-1}{y-1}-\binom {my+y}y\\
&=\binom {((m+1)(n-1)-\rRr)/\rRr}{n/\rRr}.
\end{align*}
This is the desired result since, as we mentioned at the beginning of
the proof, we must multiply the counting result by~$2$.
\end{proof}

\begin{proof}[Proof of \Cref{thm:enumD-4}]
We proceed as in the proof of \Cref{thm:enumD-2}.
In particular, we first concentrate on fundamental seeds with
$k$ bridging blocks. Also here, we ignore the signs of the elements of
the inner circle for the moment, so that, in the end, we must multiply
the result by~$2$.
The corresponding generating function is given
by~\eqref{eq:GFConstr3}, except that there is no summation over~$a_i$'s,
but rather all~$a_i$'s are equal to~$1$ since here all block sizes are
equal to~$m$. In other words, the expression that we must consider here is
$$
{\underset{i_1\ge0,\,i_2,\dots,i_k\ge1}
{\sum_{i_1+\dots +i_k=x}}}z\,(zC(z))^{m-i_1-1}
(zC(z))^{m-i_2}\cdots (zC(z))^{m-i_k}.
$$
Again using our earlier notation $f(z)=zC(z)$, this can be rewritten as
$$
{\underset{i_1\ge0,\,i_2,\dots,i_k\ge1}
{\sum_{i_1+\dots +i_k=x}}}z\,f^{mk-i_1-\dots-i_k-1}(z)
=
{\underset{i_1\ge0,\,i_2,\dots,i_k\ge1}
{\sum_{i_1+\dots +i_k=x}}}z\,f^{mk-x-1}(z).
$$
As before, the sum over the $i_j$'s is not difficult to carry out, and
we obtain
$$
\binom x{k-1}
z\,f^{mk-x-1}(z)
$$
for our generating function.
This expression must be multiplied by
$(m(n-1)-1)/\rRr$, divided by~$k$, and then summed over all~$k$.
Hence, we must compute
\begin{align*}
\sum_{k\ge1}
\frac {(m(n-1)-1)/\rRr} {k}\binom x{k-1}&
z\,f^{mk-x-1}(z)
\\
&=
\sum_{k\ge1}
\frac {my-x-1} {x+1}\binom {x+1}{k}
z\,f^{mk-x-1}(z)\\
&=
\frac {my-x-1} {x+1}
\left(
z\,f^{-x-1}(z)(1+f^m(z))^{x+1}
-
z\,f^{-x-1}(z)\right)\\
&=
\frac {my-x-1} {x+1}
z\,f^{-x-1}(z)
\left(
\left( {1+f^m(z)}\right)^{x+1}
-1
\right).
\end{align*}
Again we used \eqref{eq:rotrel1} to obtain the second line.

In the last expression, we want to extract the coefficient of $z^{my-x}$.
Equivalently, we want to compute
$$
\frac {my-x-1} {x+1}
\coef{z^{my-x-1}}
f^{-x-1}(z)
\left(
\left( {1+f^m(z)}\right)^{x+1}
-1
\right).
$$
By using the Lagrange inversion formula in the form~\eqref{eq:CBa},
with $b=my-x-1$ and $g(z)=z^{-x-1}\big(( {1+z^m})^{x+1}-1\big)$, 
this is the same as
\begin{multline*}
\frac {my-x-1} {x+1}
\frac {1} {my-x-1}\coef{z^{-1}}
F^{-my+x+1}(z)
\frac {d} {dz}\left(z^{-x-1}
\left(
\left( {1+z^m}\right)^{x+1}
-1
\right)\right)\\
=
\coef{z^{-1}}
F^{-my+x+1}(z)
z^{-x-2}
\left(1
+((m-1)z^{m}-1)
\left( {1+z^m}\right)^{x}
\right).
\end{multline*}
Since here all block sizes are~$m$, we must specialise $x_0=1$ (as
always), $x_1=1$, and $x_i=0$ for $i\ge2$.
With these choices, the series $F(z)$ equals
$z/(1+z^m)$. If we substitute this in the above expression, then we get
\begin{align*}
\coef{z^{-1}}
z^{-my-1}&(1+z^m)^{my-x-1}
\left(1+((m-1)z^{m}-1)
\left( {1+z^m}\right)^{x}
\right)\\
&=
\coef{z^{my}}
(1+z^m)^{my-x-1}
+((m-1)z^{m}-1)
\left( {1+z^m}\right)^{my-1}\\
&=\binom {my-x-1}y+(m-1)\binom {my-1}{y-1}-\binom {my-1}y\\
&=\binom {(m(n-1)-1)/\rRr}{n/\rRr}.
\end{align*}
This is the desired result since, as we mentioned at the beginning of
the proof, we must multiply the counting result by~$2$.
\end{proof}

\section{Proofs: Characterising pseudo-rotationally invariant positive non-crossing partitions}
\label{app:inv}

This appendix is devoted to the proofs of \Cref{lem:allA,,lem:allB,,lem:allD} in which we characterise pseudo-rotationally invariant positive non-crossing set partitions in classical types.

\subsection{Characterisation of pseudo-rotationally invariant positive
  non-crossing\break partitions in type~$A_{n-1}$}
\label{app:inv-A}

Before we are able to prove \Cref{lem:allA}, we need to first establish some
auxiliary results which provide a detailed examination how
blocks are moved by successive applications of the
pseudo-rotation~$\rotA$. In order to have a convenient language to
carry this analysis out, we need some more definitions.

%

Recall from \Cref{sec:rotenumA} that,
for a $\rrr $-pseudo-rotationally invariant
element~$\pi$ of $\NCAPlus$
(that is, $\pi$ is invariant with respect to~$\rotA^{(N-2)/\rrr }$), 
a block $B$ of~$\pi$ is called
\defn{$\rotA$-central block\/} (or simply \defn{central block} if $\rotA$
is clear from the context)
if $B=\rotA^{(N-2)/\rrr }(B)$.
(How $\rotA$ acts on a block~$B$ should be intuitively clear from
the definition of~$\rotA$, cf.\ \Cref{fig:7,fig:1,fig:8}.)
For example, both non-crossing partitions in \Cref{fig:3,fig:4} 
have central blocks: in the partition of \Cref{fig:3}
the central block is $\{1,2,9,16,23,30\}$, while in 
the partition of \Cref{fig:4}
the central block is $\{3,4,11,18,25,26\}$. On the other hand,
neither of the non-crossing partitions in \Cref{fig:5,fig:6} 
has a central block.
While a singleton block can never be a central block, it may
happen that --- slightly counter-intuitively --- $\{N,1\}$ is a central block. 
An example with $N=32$ and $\rrr=3$ is shown in \Cref{fig:22},
where indeed
$\rotA^{(32-2)/3}(\(1,32\))=\rotA^{10}(\(1,32\))=\(1,32\)$.

Furthermore, we define a \defn{cyclic interval\/} of $\{1,2,\dots,N\}$
to be an interval of the form $\{e,e+1,\dots,f-1,f\}$, where all
elements of the interval are taken modulo~$N$. Thus, for $N=12$,
$\{3,4,5,6\}$, $\{9,10,11\}$, as well as $\{10,11,12,1,2\}$ are cyclic
intervals.
Given an element~$\pi$ of $\NCAPlus$,
the \defn{width} of a block $B$ of~$\pi$, denoted by~$\width(B)$,
is defined as the minimal
cardinality of a cyclic interval containing~$B$. 
Such a minimal-cardinality cyclic interval will also be called
a \defn{support\/} for~$B$. (With this definition, a support of a
block does not need to be unique.)
A block $B$ in a $\rrr$-pseudo-rotationally invariant element of
$\NCAPlus$ is called
\defn{fat\/} if $\width(B)>\frac {N-2} {\rrr }+1$.
See \Cref{fig:4}, where $\{3,4,11,18,25,26\}$ is a fat block
of width~$24>\frac {30-2}4+1=8$.

The key for the proof of \Cref{lem:allA} is to show that, for
$\rrr\ge3$, every $\rrr$-pseudo-rotationally invariant element of
$\NCAPlus$ has a central block. We are going to
finally achieve this in \Cref{lem:10}. In turn, that proposition needs
another fact that we shall prove in \Cref{lem:fat}:
for every $\rrr $-pseudo-rotationally invariant element of 
$\NCAPlus$, a fat block is automatically a central block.
In the proof of \Cref{lem:fat}, we need certain properties of
the pseudo-rotation~$\rotA$ which we state and prove
separately in the next two lemmas, before we embark on
\Cref{lem:fat} and \Cref{lem:10} themselves and their proofs.

\begin{lemma} \label{lem:rotungl}
Let $B$ be a block of $\pi\in \NCAPlus$
which has a support not containing
$1$ and~$N$. Then $\width\big(\rotA^j(B)\big)\le
\width(B)$ for any~$j$.
\end{lemma}

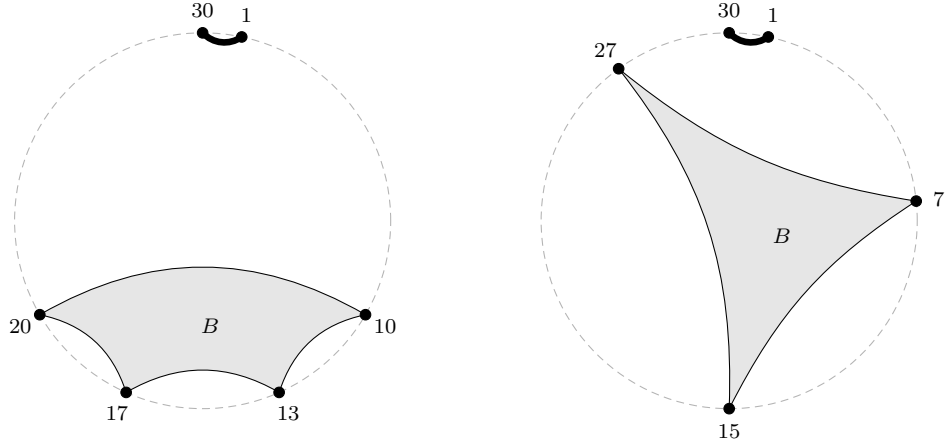
\begin{figure}
\begin{center}
  \begin{tikzpicture}[scale=1]
    \polygon{(-3.5,0)}{objleft}{30}{2.5}
      {1,,,,,,,,,10,,,13,,,,17,,,20,,,,,,,,,,30}
     \draw[fill=black,fill opacity=0.1] (objleft10) to[bend right=30]
(objleft13) to[bend right=30] (objleft17) to[bend right=30]
(objleft20) to[bend left=30] (objleft10);
    \draw[line width=2.5pt,black] (objleft1) to[bend left=40] (objleft30);
    \node at (-3.4,-1.4) {\tiny$B$};

    \polygon{(3.5,0)}{objright}{30}{2.5}
      {1,,,,,,7,,,,,,,,15,,,,,,,,,,,,27,,,30}
     \draw[fill=black,fill opacity=0.1] (objright7) to[bend right=15]
(objright15) to[bend right=20] (objright27) to[bend right=15]
(objright7);
    \draw[line width=2.5pt,black] (objright1) to[bend left=40] (objright30);
        \node at (4.2,-0.2) {\tiny$B$};
    \end{tikzpicture}
\end{center}
\caption{Example (on the left) and non-example (on the right) of a
  block~$B$ in \Cref{lem:rotungl}}
\label{fig:B1}
\end{figure}

\begin{remark}
On the left of \Cref{fig:B1}, there is an example of a block whose
only support is $[10,11,\dots,20]$, which does not contain $1$ and~$N=30$.
On the other hand, the right of \Cref{fig:B1} shows a block whose only
support is $[27,28,29,30,1,2,\dots,15]$, which does contain $1$ and $N=30$.
\end{remark}

\begin{proof}[Proof of \Cref{lem:rotungl}]
Clearly, the case of $|B|=1$ being obvious, we may assume
$\width(B)>1$. 
Without loss of generality we may restrict our attention to
the case where the mentioned support of~$B$ has the form
$C=[e,e+1,\dots,N-1]$; see the left image in 
\Cref{fig:7} for a schematic illustration. After application
of~$\rotA$, the block~$\rotA(B)$ 
is contained in the cyclic interval $\{e+2,\dots,1\}$,
while
$e+1$ and $2$ are both contained in a different block
``overarching"~$\rotA(B)$; see the right image in
\Cref{fig:7}. During the next
$N-e-2$ applications of~$\rotA$, this overarching
block persists, ensuring that
$\width\big(\rotA^j(B)\big)\le \width(B)$ for all~$j$
with $0\le j\le N-e-1$. A further application of~$\rotA$
``restores" $B$, i.e., $\rotA^{N-e}(B)$ is 
a(n ordinarily) rotated version of~$B$. (That we get indeed a rotated
version follows from 
\Cref{lem:N-2} and the observation that the next applications
of~$\rotA$ will (ordinarily) rotate our block.)
Thus we have
$\width\big(\rotA^j(B)\big)=\width(B)$ for all~$j$
with $N-e \le j\le N-2$. This concludes the proof of the lemma.
\end{proof}

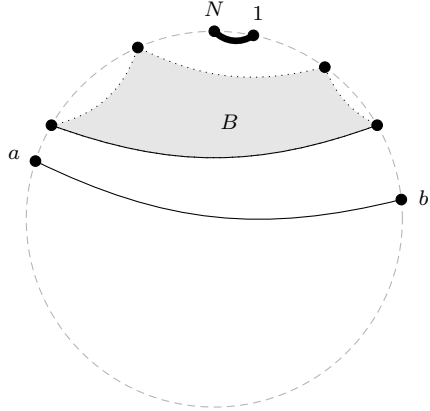
\begin{figure}
\begin{center}
  \begin{tikzpicture}[scale=1]
    \polygon{(-3.5,0)}{objleft}{30}{2.5}
      {1,,\hspace{1pt},,\hspace{1pt},,b,,,,,,,,,,,,,,,,,a,\hspace{1pt},,,\hspace{1pt},,N}
     \draw[dotted, fill=black,fill opacity=0.1] (objleft3) to[bend
right=20] (objleft5) to[bend left=20] (objleft25) to[bend right=30]
(objleft28) to[bend right=20] (objleft3);

     \draw[] (objleft5) to[bend left=20] (objleft25);
     \draw[] (objleft7) to[bend left=20] (objleft24);
     \node at (-3.3,1.3) {\tiny$B$};
    \draw[line width=2.5pt,black] (objleft1) to[bend left=40] (objleft30);
    \end{tikzpicture}
\end{center}
\caption{Illustration of a block $B$ ``covered" as described in \Cref{lem:Bogen}}
\label{fig:B2}
\end{figure}

\begin{lemma} \label{lem:Bogen}
Let $B$ be a block in a non-crossing partition $\pi\in \NCAPlus$ 
which is ``covered" by two elements $a$ and $b$
(that is, $a$ and $b$ belong to a block $B'$ different from~$B$) with
$a<N$, $b>1$, and $b+N-a\le N/2$.
{\em(}See \Cref{fig:B2} for a schematic
illustration of such a configuration.{\em)}
Then
$\width\big(\rotA^j(B)\big)\le b+N-a-1$ for all~$j$.
\end{lemma}

\begin{proof}
The argument is similar to the ones in the proof of the previous
lemma.
Without loss of generality we may assume that $a$ and $b$ are
(cyclically) successive elements in~$B'$. Thus, the arc $a-b$
of the block $B'$ is ``overarching"~$B$. 
During the first
$N-a-1$ applications of~$\rotA$, this overarching
block persists, ensuring that
$\width\big(\rotA^j(B)\big)\le b+N-a-1$ for all~$j$
with $0\le j\le N-a-1$. 
(It is here, where we use the condition $b+N-a\le N/2$,
otherwise we could not be sure to draw this conclusion for
the width of the pseudo-rotations of~$B$.)
A further application of~$\rotA$
produces the block~$\rotA^{N-a}(B)$ which arises from
$\rotA^{N-a-1}(B)$ by adding one to all elements of the latter
block (modulo~$N$), removing~$1$ and inserting an element~$e$
with $e\le b+N-a$ to it.
Clearly, for this new block we have also
$\width\big(\rotA^{N-a}(B)\big)=e-2+1\le b+N-a-1$.
In order to now complete the proof, we appeal to \Cref{lem:rotungl}.
\end{proof}

We are now in the position for proving that fat blocks are at the same
time central.

\begin{lemma} \label{lem:fat}
A fat block $B$ in a $\rrr $-pseudo-rotationally invariant positive
non-crossing partition is automatically a central block.
\end{lemma}

\begin{proof}
Let $\pi$ be a $\rrr $-pseudo-rotationally invariant element of 
$\NCAPlus$, and let $B$ be a fat block of~$\pi$.
By way of contradiction, let us assume that $B$ is not a central block,
that is, $\rotA^{(N-2)/\rrr }(B)\ne B$. Then there must
exist another block of~$\pi$, say $B'$ with $B'\ne B$, for which
$\rotA^{(N-2)/\rrr }(B')=B$.

We distinguish between three cases (cf.\ \Cref{fig:B4,fig:B5,fig:B6}):

\begin{figure}
\begin{center}
  \begin{tikzpicture}[scale=1]
    \polygon{(-3.5,0)}{objleft}{32}{2.5}
      {1,,,,5,,7,,,10,,,,,,16,,,,,,,,,,,,,,,,32}
     \draw[fill=black,fill opacity=0.1] (objleft5) to[bend right=30]
(objleft7) to[bend right=30] (objleft10) to[bend right=30]
(objleft16) to[bend left=30] (objleft5);
    \draw[line width=2.5pt,black] (objleft1) to[bend left=40] (objleft32);
    \node at (-2.0,0) {\tiny$B$};
    \end{tikzpicture}
\end{center}
\caption{A block $B$ with support $[5,\dots,16]$ not containing 1
and $N=32$, and with width $12\le N/2=16$}
\label{fig:B4}
\end{figure}

\begin{figure}
\begin{center}
  \begin{tikzpicture}[scale=1]
    \polygon{(-3.5,0)}{objleft}{32}{2.5}
      {1,,,,5,,,,,,11,,,14,,,17,,,,,22,,,,,,,,,,32}
     \draw[fill=black,fill opacity=0.1] (objleft5) to[bend right=30]
(objleft11) to[bend right=30] (objleft14) to[bend right=30]
(objleft17) to[bend right=30]
(objleft22) to[bend left=30] (objleft5);
    \draw[line width=2.5pt,black] (objleft1) to[bend left=40] (objleft32);
    \node at (-3.0,-0.5) {\tiny$B$};
    \end{tikzpicture}
\end{center}
\caption{A block $B$ with support $[5,\dots,22]$ not containing 1
and $N=32$, and with width $18>N/2=16$}
\label{fig:B5}
\end{figure}
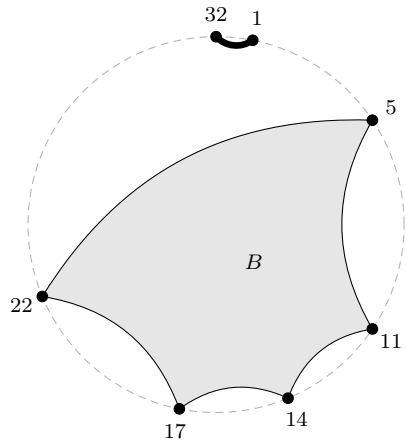

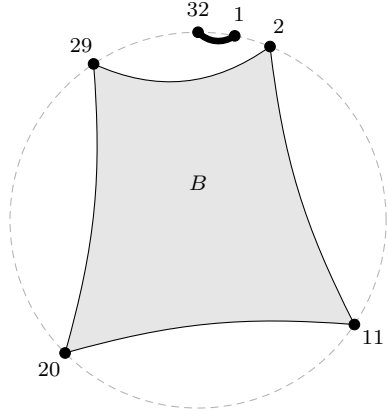
\begin{figure}
\begin{center}
  \begin{tikzpicture}[scale=1]
    \polygon{(-3.5,0)}{objleft}{32}{2.5}
      {1,2,,,,,,,,,11,,,,,,,,,20,,,,,,,,,29,,,32}
     \draw[fill=black,fill opacity=0.1] (objleft2) to[bend right=10]
(objleft11) to[bend right=10] (objleft20) to[bend right=10]
(objleft29) to[bend right=30] (objleft2);
    \draw[line width=2.5pt,black] (objleft1) to[bend left=40] (objleft32);
    \node at (-3.5,0.5) {\tiny$B$};
    \end{tikzpicture}
\end{center}
\caption{A block $B$ with supports $[11,\dots,32,1,2]$,
  $[20,\dots,32,1,\dots,11]$, and $[29,\dots,32,1,\dots,20]$, all of them
containing 1 and $N=32$}
\label{fig:B6}
\end{figure}

\medskip

\begin{enumerate}[\labelsep10pt]
\item[{\sc Case 1.}] $B$ has a support not containing $N$ and $1$, and
$\width(B)\le N/2$.
\item[{\sc Case 2.}] $B$ has a support not containing $N$ and $1$, and
$\width(B)> N/2$.
\item[{\sc Case 3.}] All supports of $B$ contain $N$ and $1$.
\end{enumerate}

\medskip
\noindent
{\sc Case 1.}
Since $\pi$ has a central block if and only if $\rotA^{j}(\pi)$
has a central block for any~$j$, we may without loss of generality assume
that $B$ is embedded in $C=\{e,e+1,\dots,N-1\}$ with $|C|=\width(B)$,
that is, $N-e=\width(B)$. It should be
noted that $e$ and $N-1$ must be elements of the block $B$. 

We must split our argument again, depending on whether
$B'$ is contained in $C$ or not.

If $B'$ is contained in $C$ then $\width(B')<\width(B)$.
This is absurd since, by \Cref{lem:rotungl}, we have
$\width\big(\rotA^j(B')\big)\le \width(B')$ for all~$j$.

On the other hand, if $B'$ is not contained in $C$, then 
it must actually be disjoint from $C$. In particular, $\max B'<e$.
If $N\notin B'$, then $\rotA^j(B')$ is a(n ordinarily)
rotated version of~$B'$ for $1\le j\le (N-2)/\rrr $. Thus,
$$\max \rotA^{(N-2)/\rrr }(B')=\max B'+\frac {N-2} {\rrr }
<e+\frac {N-2} {\rrr }<N-1,
$$
the last inequality following from the assumption that $B$ is fat.
If $N\in B'$, then\break $\max\rotA(B')=
e+1$. Further $\frac {N-2} {\rrr }-1$ 
applications of~$\rotA$ now (ordinarily)
rotate~$\rotA(B')$. Consequently,
$$\max \rotA^{(N-2)/\rrr }(B')=e+1+\left(\frac {N-2} {\rrr }-1\right)
=e+\frac {N-2} {\rrr }<N-1.
$$
This shows that in both cases
$\rotA^{(N-2)/\rrr }(B')$ cannot contain $N-1$,
a contradiction.

\medskip
\noindent
{\sc Case 2.} Again, we argue differently depending on whether
$B'$ is contained in the support of~$B$ or not.

If $B'$ is contained in the support of~$B$, then
$\width(B')<\width(B)$.
Again, this is absurd since, by \Cref{lem:rotungl}, we have
$\width\big(\rotA^j(B')\big)\le \width(B')$ for all~$j$.

On the other hand, if $B'$ is not contained in the support of~$B$, then 
we may apply \Cref{lem:Bogen} with $a$ the largest element of~$B$ ---
that is, $a=N-1$ ---
and $b$ the smallest element of~$B$
--- that is, $b=e$. The conclusion is that
$\width\big(\rotA^j(B')\big)<N/2<\width(B)$, a contradiction.

\medskip
\noindent
{\sc Case 3.} If $B'$ is contained in the support of~$B$, then
$\width(B')<\width(B)$. \Cref{lem:Bogen} applies with
with $a$ the largest element of~$B$
and $b$ the smallest element of~$B$. The conclusion is 
as in the above case and leads to a contradiction.

Finally, let $B'$ be not contained in the support of~$B$.
If $B'$ is not fat, then we know that $\width(B')<\width(B)$.
By \Cref{lem:rotungl}, this is absurd. On the other hand, if
$B'$ is fat, then $B'$ fits into Case~1 or Case~2, depending on its
width. In both cases, by what we have already established we see
that $B'$ must be a central block. But this means that
$\rotA^{(N-2)/\rrr }(B')=B'\ne B$, which is again absurd.

This completes the proof of the lemma.
\end{proof}

\begin{proposition} \label{lem:10}
If $\rrr \ge3$,
there is no $\rrr $-pseudo-rotationally invariant positive
non-crossing partition without central block.
\end{proposition}

\begin{proof}
Let $\pi$ be a $\rrr $-pseudo-rotationally invariant element of 
$\NCAPlus$. Arguing by contradiction,
we assume that $\pi$ contains no central block. Then,
according to \Cref{lem:fat}, it is at the same time without
a fat block,
and the same is true for any pseudo-rotation
of~$\pi$, that is, $\rotA^j(\pi)$ does not contain a
fat block for any~$j$.

From now on we write $M$ for $(N-2)/\rrr $. 
We claim that, in the orbit of~$\pi$ under the action of
the pseudo-rotation~$\rotA$, we find a
non-crossing partition in which the block containing $N$ and $1$
has a maximal element~$k$ with $k\le M$, and the block containing
$N-1$ has a support of the form $[N-a,N-a+1,\dots,N-1]$ with
$a\le M+1$. See the first image of \Cref{fig:inv_has_central}
for a schematic illustration of this block structure.

To see this, we proceed as follows. It is obvious that by sufficiently
many applications of~$\rotA$ to~$\pi$ we may obtain a non-crossing
partition in which the block containing $N$ and $1$ has the
claimed form. (Here we also use that we cannot have any fat blocks.) 
However, at this point it is not at all obvious that
the block containing $N-1$, $B'$ say, has the claimed form; it may rather
be a block with a support containing $N$ and $1$; see \Cref{fig:fat}
for an illustration (at this point, $e$, $f$, and $\widehat b$ should
be ignored). Moreover, $B'$ may be ``covered'' by another
block $B''$ with a support containing $N$ and $1$. Let $\widehat B$
be the block with maximal width with a support containing $N$ and~$1$.
We let the support of~$\widehat B$ be $[e,e+1,\dots,f]$.
See again \Cref{fig:fat} for an illustration of that situation.
(In the figure, we have $\widehat B\ne B'$. However, it might well be
that $\widehat B=B'$.)

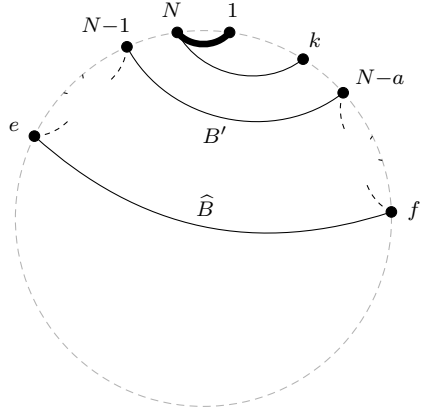
\begin{figure}
\begin{center}
  \begin{tikzpicture}[scale=1]
     \polygon{(0,0)}{obj}{45}{2.5}
       {1,,,k,,\hspace{15pt}N-a,,,,,f,,,,,,,,,,,,,,,,,,,,,,,,,,e,,,,,N-1\hspace{9pt},,N\hspace{2pt},}

     \draw[line width=2.5pt,black] (obj1) to[bend left=50] (obj44);

     \draw (obj1)  to[bend left=50]
           (obj44) to[bend right=50]
           (obj4);

     \draw (obj6) to[bend left=50]
           (obj42);

     \draw[dash pattern=on 2pt off 2pt on 2pt off 2pt on 2pt off 2pt
on 2pt off 18pt on 2pt off 2pt on 2pt off 2pt] (obj6)  to[bend
right=50] (obj9);

     \draw[dash pattern=on 2pt off 2pt on 2pt off 2pt on 2pt off 2pt
on 2pt off 18pt on 2pt off 2pt on 2pt off 2pt] (obj42)  to[bend
left=50] (obj39);

     \draw (obj37)  to[bend right=30] (obj11);

     \draw[dash pattern=on 2pt off 2pt on 2pt off 2pt on 2pt off 2pt
on 2pt off 18pt on 2pt off 2pt on 2pt off 2pt] (obj37)  to[bend
right=50] (obj40);

     \draw[dash pattern=on 2pt off 2pt on 2pt off 2pt on 2pt off 2pt
on 2pt off 18pt on 2pt off 2pt on 2pt off 2pt] (obj11)  to[bend
left=50] (obj8);

    \node at ($($0.07*($(obj1)$)$)$) {\tiny$\widehat B$};
    \node at ($($0.45*($(obj1)$)$)$) {\tiny$B'$};
    \end{tikzpicture}
\end{center}
\caption{The situation at the beginning of the proof of \Cref{lem:10}}
\label{fig:fat}
\end{figure}

According to our assumptions, the block $\widehat B$ cannot be fat,
hence $\width(\widehat B)\le M+1$.
Furthermore, all blocks which are disjoint
from $\widehat B$ cannot be fat either. 

Now we apply $\rotA$ $N-e$ times. From the definition of the
pseudo-rotation~$\rotA$ applied to our particular situation, 
we may see that $\widehat B$ will be ``moved into $N$", in the sense that 
$\rotA^{N-e}\big(\widehat B\big)$ will have a support of the form
$[N,1,\dots,l]$ with $l\le f+N-e$, 
while the blocks originally outside the support
of~$\widehat B$ will just be (ordinarily) rotated $N-e$ times.
Thus, we have indeed reached the claimed situation, illustrated
in the first image of \Cref{fig:inv_has_central}.
Here we announce the award of a
published version of this article,
autographed by~C. To receive 
the award, a 
5-pseudo-rotationally invariant positive non-crossing partition
in $\mNCAPlus[8][3]$ different from the one-block-partition
has to be sent to the authors.

In view
of the fact that $\pi$ contains a central block if and only if
$\rotA^{j}(\pi)$ has a central block for any~$j$, 
we now assume without loss of generality that $\pi$
contains a block $B$ which has a support of the form
$[N,1,\dots,k]$, with $k\le M$, and the
block $B'$ containing
$N-1$ has a support of the form $[N-a,N-a+1,\dots,N-1]$ with
$a\le M+1$. 
The first image in \Cref{fig:inv_has_central} provides a schematic
sketch of the combinatorial structure of~$\pi$. 
In the figure, the block $B'$ containing $N-1$ is indicated as well.
It is indeed disjoint from $B$ since
$$
k\le M<N-M-1\le N-a.
$$

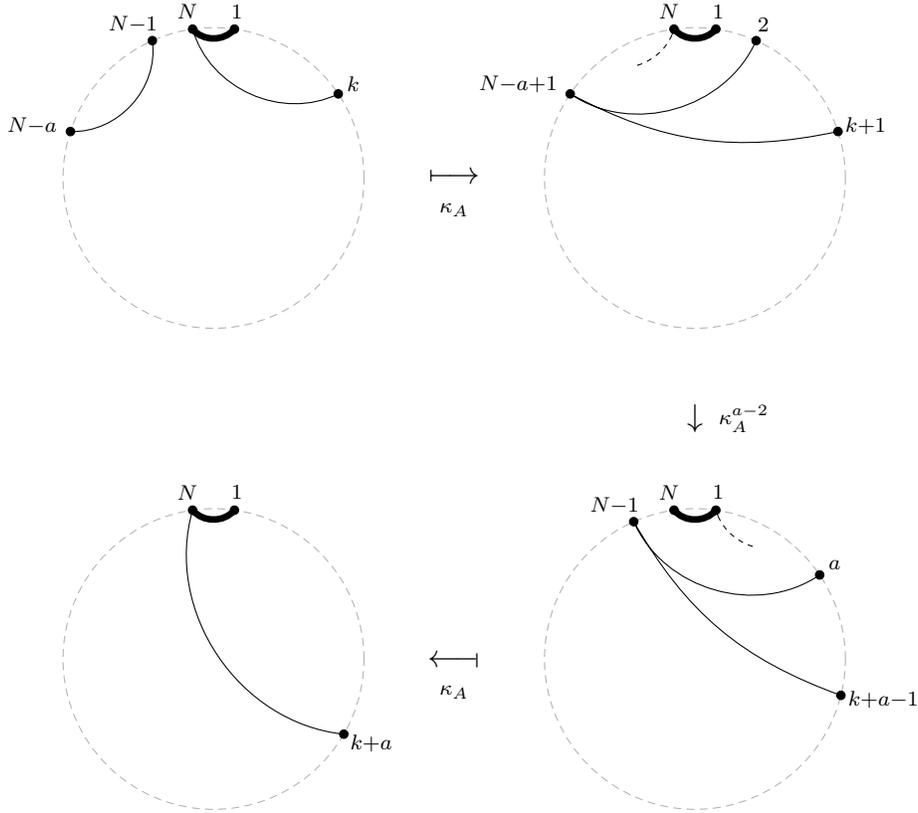
\begin{figure}
  \centering
  \newcommand{\longleftmapsto}{\longleftarrow\!\shortmid}
  \begin{tikzpicture}[scale=0.8]
    \polygon{(0,0)}{obj}{45}{2.5}
      {1,,,,,,k,,,,,,,,,,,,,,,,,,,,,,,,,,,,,N-a\hspace{16pt},,,,,,N-1\hspace{9pt},,N\hspace{2pt},}

    \draw[line width=2.5pt,black] (obj1) to[bend left=50] (obj44);

    \draw (obj1)  to[bend left=50]
          (obj44) to[bend right=50]
          (obj7);

    \draw (obj36) to[bend right=50]
          (obj42);

    \node[inner sep=0pt] (M1) at (4,0) {$\longmapsto$};
    \node[inner sep=0pt, below=0.25 of M1] {$\scriptstyle{\rotA}$};

    \polygon{(8,0)}{obj}{45}{2.5}
      {1,,2,,,,,,\hspace{8pt}k+1,,,,,,,,,,,,,,,,,,,,,,,,,,,,,N-a+1\hspace{28pt},,,,,,N\hspace{2pt},}

    \draw[line width=2.5pt,black] (obj1) to[bend left=50] (obj44);

    \draw[dash pattern=on 0pt off 15pt on 2pt off 2pt on 2pt off 2pt
on 2pt off 2pt on 2pt off 2pt on 2pt] (obj40) to[bend right=50]
(obj44);
    \draw (obj44) to[bend right=50]
          (obj1);

    \draw (obj9)  to[bend left=20]
          (obj38) to[bend right=50]
          (obj3);

    \node[inner sep=0pt] (M2) at (8,-4) {$\downarrow$};
    \node[inner sep=0pt, right=0.2 of M2]
{$\scriptstyle{\rotA^{a-2}}$};

    \polygon{(8,-8)}{obj}{45}{2.5}
      {1,,,,,,a,,,,,,\hspace{19pt}k+a-1,,,,,,,,,,,,,,,,,,,,,,,,,,,,,N-1\hspace{9pt},,N\hspace{2pt},}

    \draw[line width=2.5pt,black] (obj1) to[bend left=50] (obj44);

    \draw (obj44) to[bend right=50]
          (obj1);
    \draw[dash pattern=on 0pt off 15pt on 2pt off 2pt on 2pt off 2pt
on 2pt off 2pt on 2pt off 2pt on 2pt] (obj5) to[bend left=50] (obj1);

    \draw (obj13)  to[bend left=20]
          (obj42) to[bend right=50]
          (obj7);

    \node[inner sep=0pt] (M3) at (4,-8) {$\longleftmapsto$};
    \node[inner sep=0pt, below=0.2 of M3] {$\scriptstyle{\rotA}$};

    \polygon{(0,-8)}{obj}{45}{2.5}
      {1,,,,,,,,,,,,,,\hspace{9pt}k+a,,,,,,,,,,,,,,,,,,,,,,,,,,,,,N\hspace{2pt},}

    \draw[line width=2.5pt,black] (obj1) to[bend left=50] (obj44);

    \draw (obj1)  to[bend left=50]
          (obj44) to[bend right=50]
          (obj15);
    \end{tikzpicture}
  \caption{A schematic illustration of the construction in \Cref{lem:10}}
\label{fig:inv_has_central}
\end{figure}

If $B=\{N,1\}$ then it is a central block as is easy to see directly.
Hence, we may restrict ourselves to the case where $k\ge2$, or, 
in other words, to
the case where $B$ contains at least three elements.

We now investigate what happens when we apply 
${\rotA}$ repeatedly to~$\pi$.
In~$\rotA(\pi)$, we find the block~$\rotA(B)$
that contains $2,k+1,N-a+1$, and we find the block $\rotA(B')$
that contains $N$ and~$1$; see the second image in
\Cref{fig:inv_has_central}. 

The next $a-2$ applications of~$\rotA$ (ordinarily) rotate
$\rotA(B')$, until we arrive at $\rotA^{a-1}(\pi)$,
in which $\rotA^{a-1}(B)$ is a block containing
$a,k+a-1,N-1$; see the third image in \Cref{fig:inv_has_central}.
At no point do we arrive at a structure that would be in agreement
with the structure at the very beginning; compare the second and
third image in \Cref{fig:inv_has_central} with the first image.

If $a=M+1$, then we have already obtained
$\rotA^{a-1}(\pi)=\rotA^{M}(\pi)=\rotA^{(N-2)/\rrr }(\pi)$, 
which should be the same as~$\pi$.
However, comparison of the first and third image in
\Cref{fig:inv_has_central} shows that we should have $(N-a)\in
\rotA^{M}(B)$, which is impossible since
\begin{equation} \label{eq:imposs} 
N-a\ge N-M-1>2M\ge k+a-1
\end{equation}
for $\rrr \ge3$. Thus, we have reached a contradiction.

If $a\le M$, then we have to apply $\rotA$ 
another time. We arrive at~$\rotA^a(\pi)$, in which the
block~$\rotA^a(B)$ contains $N,1,k+a$; see the fourth
image in \Cref{fig:inv_has_central}.
The structure of the obtained non-crossing partition is the
same as in the beginning --- cf.\ the first image in
\Cref{fig:inv_has_central} --- except that the maximum element of
the block containing $N,1$ (namely $k+a$) is larger than at
the beginning (where it was $k$). 
Since this block cannot be a fat block according to our
assumptions, we have $k+a\le M$.
We now run through the
arguments of this proof another time, starting with
$\rotA^a(\pi)$ instead of~$\pi$. After finitely
many rounds we arrive at a contradiction.

This completes the proof of the lemma.
\end{proof}

\begin{proof}[Proof of \Cref{lem:allA}]
It is not difficult to see from the analysis of what
happens to a block under repeated application of the
pseudo-rotation~$\rotA$ in Step~3 of the proof of \Cref{lem:N-2}
(given in \Cref{app:order-A}) that Constructions~1
and~2 from \Cref{sec:rotA}
do indeed yield $\rrr $-pseudo-rotationally invariant positive
non-crossing partitions.

\medskip
For the converse, that is, to show that 
$\rrr $-pseudo-rotationally invariant positive
non-crossing partitions cannot arise in a different
way, we discuss the two cases separately.

\medskip\noindent
(1) Let $\rrr \ge3$ and $\pi$ be a $\rrr $-pseudo-rotationally invariant 
non-crossing partition in $\NCAPlus$. 
By \Cref{lem:10}, we know that $\pi$ must contain a central block, $Z$ say.
If $|Z|=2$, then it is not difficult to see directly that $\pi$
must arise from the corresponding (degenerate) case of Construction~1.

From now on we let $|Z|\ge3$.
We now consider $\rotA^j(\pi)$ for $j=0,1,\dots$. At some point,
$\rotA^j(Z)$ must contain $N$ and $1$. If in all such situations
also $N-1$ is contained in $Z$, then $Z=\{1,2,\dots,N\}$ and there
is nothing to prove. Otherwise, for some~$j$, $\rotA^j(Z)$ will
contain $N$ and $1$ but not $N-1$. We write $\pi'$ for
$\rotA^j(\pi)$ and $Z'$ for~$\rotA^j(Z)$. 
By the definition of the pseudo-rotation~$\rotA$, we see that the elements
$N-\frac {N-2} {\rrr }k$ must also be in $Z'$ for $k=1,2,\dots,\rrr $.
In particular, we have $2\in Z'$.

Let $B$ a block of $\pi'$
different from~$Z'$. Its support does not contain $N$ and~$1$,
and, by the last observation above, we have $\width(B)<(N-2)/\rrr $.
Hence, by the proof of \Cref{lem:rotungl}, the pseudo-rotation~$\rotA$
acts as ordinary rotation, except while moving across $N$ and~$1$.
In particular, $\rotA^{(N-2)/\rrr }$ acts by
(ordinarily) rotating $B$ by $\frac {N-2} {\rrr }$ units, except when we move
across $N$ and~$1$ where $B$ is rotated by $\frac {N-2}
{\rrr }+2$ units. Thus we have shown that $Z'$ arises from Construction~1,
and all other elements in the orbit of~$\pi$ arise by applying the
pseudo-rotation~$\rotA$ several times.

\medskip\noindent
(2) Let $\pi$ be a $2$-pseudo-rotationally invariant 
non-crossing partition in $\NCAPlus$. 
If it contains a central block, then the arguments from Case~(1) apply
also here and show that $\pi$ arises from Construction~1 possibly 
followed by several applications of~$\rotA$.

If $\pi$ does not contain a central block, then we run through the
arguments of the proof of \Cref{lem:10}. There, the restriction of
the proposition that $\rrr \ge3$ enters only at exactly one point, namely
in the chain of inequalities~\eqref{eq:imposs}. So, if
we want to find a $2$-pseudo-rotationally invariant
non-crossing partitions in $\NCAPlus$, then it
is found at this point of the proof of \Cref{lem:10}.
Thus, in the construction there (cf.\ the first and the third image
in \Cref{fig:inv_has_central}) we must have $a=M+1=\frac {N-2}
{2}+1=\frac {N} {2}$. However,
examining the arguments in the proof of \Cref{lem:10}
leading from the first to the third image in
\Cref{fig:inv_has_central}, we see that these describe exactly
how to generate an element of $\NCAPlus$ by
means of Construction~2.

\medskip
This completes the proof.
\end{proof}

\subsection{Characterisation of pseudo-rotationally invariant positive
  non-crossing\break partitions in type~$B_n$}
\label{app:inv-B}

In the proof of \Cref{lem:ord=N-1} below, 
we make use of the following notions, modifying notions
from the previous subsection.
Given $N$, we define a \defn{cyclic interval\/} of 
$\{1,2,\dots,N,\overline1,\overline2,\dots,\overline{N}\}$
to be an interval of the form $\{e,e+1,\dots,f-1,f\}$, where 
two successive elements in this set are successive in
$\{1,2,\dots,N,\overline1,\overline2,\dots,\overline{N}\}$ or
are equal to $\overline{N},1$. Thus, for $N=12$,
the sets $\{3,4,5,6\}$, $\{10,11,12,\overline1\}$, 
as well as the set 
$\{\overline{10},\overline{11},\overline{12},1,2\}$ are cyclic
intervals.
Given an element~$\pi$ of $\NCBPlus$,
the \defn{width} of a block $B$ of~$\pi$, denoted by~$\width(B)$,
is defined as the minimal
cardinality of a cyclic interval containing~$B$. 
Such a minimal-cardinality cyclic interval will also be called
a \defn{support\/} for~$B$. (With this definition, a support of a
block need not be unique; it is unique however for non-zero blocks.)

\begin{lemma} \label{lem:ord=N-1}
Let $\pi$ be a non-crossing partition in $\NCBPlus$ without zero block.
Then the order of~$\pi$ under the action of~$\rotB$ is~$N-1$.
\end{lemma}

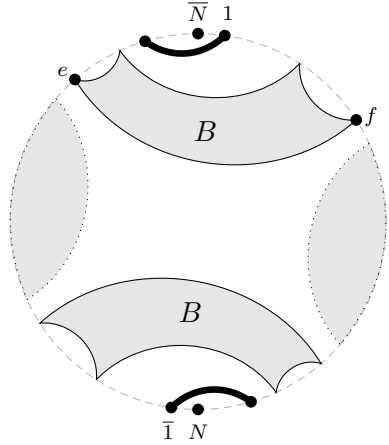
\begin{figure}
\begin{center}
  \begin{tikzpicture}[scale=1]
    \polygon{(-3.5,0)}{obj}{44}{2.5}
      {1,,,,,,\hspace*{1pt},,,,,,,,,,,,,\hspace*{1pt},,N,\overline
1,,,,,,,,,,,,,,,,\hspace*{1pt},,,\hspace*{1pt},,\overline N}

    \draw[fill=black, fill opacity=0.1] (obj4) to[bend right=50]
          (obj7) to[bend left=50]
          (obj39) to[bend right=50]
          (obj41) to[bend right=50]
          (obj4);

    \draw[line width=2.5pt,black] (obj1) to[bend left=40] (obj42);
    \node at (-3.4,1.2) {\small$B$};
    \node at (-5.3,2.) {\tiny$e$};
    \node at (-1.2,1.4) {\tiny$f$};

    \draw[fill=black, fill opacity=0.1] (obj26) to[bend right=50]
          (obj29) to[bend left=50]
          (obj17) to[bend right=50]
          (obj19) to[bend right=50]
          (obj26);

    \draw[line width=2.5pt,black] (obj23) to[bend left=40] (obj20);
    \node at (-3.6,-1.2) {\small$B$};

    \draw[dotted, fill=black, fill opacity=0.1] (obj8)  to[bend
right=50] (obj16) to[bend right=30] (obj8);
    \draw[dotted, fill=black, fill opacity=0.1] (obj30)  to[bend
right=50] (obj38) to[bend right=30] (obj30);

    \end{tikzpicture}
\end{center}
\caption{A non-crossing partition in $\NCBPlus$ without zero block}
\label{fig:B7}
\end{figure}

\begin{proof}
We consider all blocks which contain $1$ in their support.
Among these, let $B$ be the one with largest width. 
Let the support of~$B$ be $[e,e+1,\dots,f]$.
We must have $e<1\le f$.
\Cref{fig:B7} illustrates the structure of such a non-crossing
partition~$\pi$.

We start with the observation that the support of~$\rotB^j(B)$
will have the form $[e+j,\dots]$, for $j=0,1,\dots,N-e$.
(In less precise terms: the ``first" element~$e$ of~$B$ is shifted by one
unit during each of the first $N-e$ applications of~$\rotB$.)

Next, we distinguish between two cases, depending on whether there is
exactly one positive element in $B$ or more than one. (It must
contain \emph{at least one} positive element.)

\medskip
\noindent
{\sc Case 1. The block $B$ contains only one positive element.}
We consider first the case where $1\in B$. If also $\overline{N}\in
B$, then one application of~$\rotB$ maps $B$ to a non-crossing
partition, in which $\rotB(B)$ is a block that contains $1$ 
\emph{and\/}~$2$, that is, \emph{two} positive elements. Thus, we are in
Case~2 which is discussed below. If $\overline{N}\notin B$
(see \Cref{fig:BB1} for a schematic illustration), then
the first application of~$\rotB$ will act on $B$ as (ordinary) rotation.
Also subsequent applications of~$\rotB$ may act as rotations, until
another application of~$\rotB$ will have to be performed 
according to Case~(1) in \Cref{prop:2B}; see \Cref{fig:10,fig:10b}.
Then the block which arose from $B$ will ``lose" its ``largest element"
(its unique positive element),
which is replaced by~$1$. Thus, we are back at the original situation
of this case and we may continue by repeating the above arguments.

\begin{figure}
\begin{center}
  \begin{tikzpicture}[scale=1]
    \polygon{(0,0)}{obj}{44}{2.5}
      {1,,,,,,,,,e,,,,,,,,,,\hspace*{1pt},,N,\overline
1,,,,,,,,,e,,,,,,,,,,\hspace*{1pt},,\overline N}

    \draw[dotted, fill=black, fill opacity=0.1] (obj42) to[bend
right=50] (obj1) to[bend left=50] (obj32) to[bend right=20] (obj42);
    \draw (obj42) to[bend right=50] (obj1) to[bend left=50] (obj32);

    \draw[dotted, fill=black, fill opacity=0.1] (obj20) to[bend
right=50] (obj23) to[bend left=50] (obj10) to[bend right=20] (obj20);
    \draw (obj20) to[bend right=50] (obj23) to[bend left=50] (obj10);

    \draw[dotted, fill=black, fill opacity=0.1] (obj2)  to[bend
right=50] (obj9) to[bend right=25] (obj2);
    \draw[dotted, fill=black, fill opacity=0.1] (obj24)  to[bend
right=50] (obj31) to[bend right=25] (obj24);

    \node at (0.5,-1.2) {\small$B$};
    \node at (-0.5,1.2) {\small$B$};
    \end{tikzpicture}
\end{center}
\caption{A non-crossing partition in $\NCBPlus$ without zero block}
\label{fig:BB1}
\end{figure}
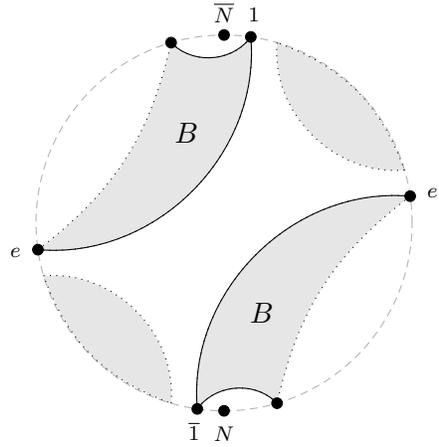

The only other possibility arises after $N-e$ applications of~$\rotB$.
Then we encounter the situation schematically illustrated in \Cref{fig:N-e}.
For the next application of~$\rotB$ we are in Case~(2) of
\Cref{prop:2B}. This implies in particular that $\overline1\in
\rotB^{N-e+1}(B)$. The ``twin block" of
$\rotB^{N-e+1}(B)$, that is, the block which consists of all
the negatives of the elements of $\rotB^{N-e+1}(B)$, then contains~$1$
and it is its largest element. This is again the same situation as 
at the very beginning of this case. The above arguments apply and
show that subsequent applications of~$\rotB$ either act as (ordinary)
rotations, or the block that arose from the original block $B$
``loses" its ``first" element (the element closest to~$\overline N$
in clockwise direction), which gets replaced by~$\overline1$.

\begin{figure}
\begin{center}
  \begin{tikzpicture}[scale=1]
    \polygon{(-4,0)}{obj}{44}{2.5}
       {1,,,,a,,,,,,,,,,,,,,,,\hspace{20pt}N-1,N,\overline{1},,,,\overline{a},,,,,,,,,,,,,,,,\overline{N-1}\hspace{20pt},\overline{N}}

    \draw[line width=2.5pt,black] (obj1) to[bend left=50] (obj44);
    \draw[line width=2.5pt,black] (obj23) to[bend left=50] (obj22);

    \draw[dash pattern=on 2pt off 2pt on 2pt off 2pt on 2pt off 2pt on
2pt off 16pt on 2pt off 2pt on 2pt off 2pt] (obj1) to[bend right=50]
(obj5);
    \draw[dash pattern=on 2pt off 2pt on 2pt off 2pt on 2pt off 2pt on
2pt off 16pt on 2pt off 2pt on 2pt off 2pt] (obj23) to[bend right=50]
(obj27);

    \draw (obj44) to[bend right=50] (obj5);
    \draw (obj22) to[bend right=50] (obj27);

    \draw[dotted] (obj6) to[bend right=20] (obj21);
    \node at ($($0.62*($(obj13)+(4.0,0)$)$)-(4.0,0)$) {\tiny$X$};
    \draw[dotted] (obj28) to[bend right=20] (obj43);
    \node at ($($0.62*($(obj35)+(4.0,0)$)$)-(4.0,0)$) {\tiny$\overline{X}$};

    \node[inner sep=0pt] at (0,0) {$\mapsto$};

    \polygon{(4,0)}{obj}{44}{2.5}
       {1,2,,,,a+1,,,,,,,,,,,,,,,,N,\overline{1},\overline{2},,,,\overline{a+1},,,,,,,,,,,,,,,,\overline{N}}

    \draw[line width=2.5pt,black] (obj1) to[bend left=20] (obj28);
    \draw[line width=2.5pt,black] (obj23) to[bend left=20] (obj6);

    \draw (obj1) to[bend right=0] (obj24);
    \draw (obj23) to[bend right=0] (obj2);

    \draw[dash pattern=on 2pt off 2pt on 2pt off 2pt on 2pt off 2pt on
2pt off 16pt on 2pt off 2pt on 2pt off 2pt] (obj2) to[bend right=50]
(obj6);
    \draw[dash pattern=on 2pt off 2pt on 2pt off 2pt on 2pt off 2pt on
2pt off 16pt on 2pt off 2pt on 2pt off 2pt] (obj24) to[bend right=50]
(obj28);

    \draw[dotted] (obj7) to[bend right=20] (obj22);
    \node at ($($0.62*($(obj14)-(4.0,0)$)$)+(4.0,0)$) {\tiny$X$};
    \draw[dotted] (obj29) to[bend right=20] (obj44);
    \node at ($($0.62*($(obj36)-(4.0,0)$)$)+(4.0,0)$) {\tiny$\overline{X}$};
    \end{tikzpicture}
\end{center}
\caption{After $N-e$ applications of $\rotB$}
\label{fig:N-e}
\end{figure}
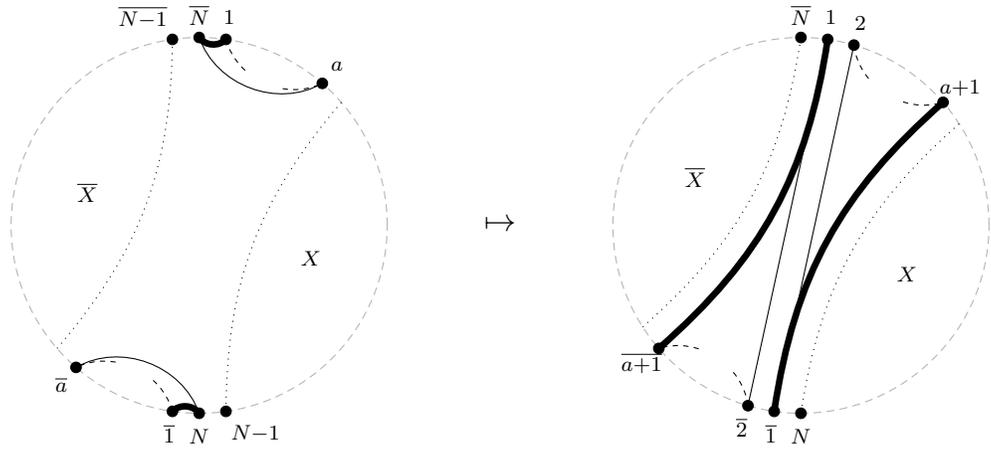

The important point is that at no point during the first $N-2$
applications of~$\rotB$ we obtain a structure that matches the
original structure of~$\pi$; cf.\ again \Cref{fig:BB1,fig:N-e}.
Thus, we return
to~$\pi$ after $N-1$ applications of~$\rotB$, but not earlier.

\medskip
\noindent
{\sc Case 2. The block $B$ contains at least two positive elements.}
In this case, it is straightforward to see that, for
$j=0,1,\dots,N-e$, the support of~$\rotB^j(B)$ equals
$[e+j,e+j+1,\dots,f+j+1]$. When we apply~$\rotB$ to $\rotB^{N-e}(\pi)$,
then we are in Case~(2) of \Cref{prop:2B}. The remaining arguments are
identical with the arguments of the last two paragraphs of the
previous case.
\end{proof}

\begin{proof}[Proof of \Cref{lem:allB}]
It is not difficult to see from the analysis of what
happens to a block under repeated application of the
pseudo-rotation~$\rotB$ in Step~3 of the proof of \Cref{lem:N-2B}
(given in \Cref{app:order-B}) that the
Construction from \Cref{sec:rotB}
does indeed yield $2\rRr $-pseudo-rotationally invariant positive
non-crossing partitions of type~$B$.

\medskip
For the converse, let $\pi$ be a $2\rRr $-pseudo-rotationally invariant 
non-crossing partitions in $\NCBPlus$. 
We have to show that 
$2\rRr $-pseudo-rotationally invariant positive
non-crossing partitions of type~$B$ cannot arise in a different
way. 
By \Cref{lem:ord=N-1}, we know that $\pi$ must contain a zero block, $Z$ say.
Since a zero block in a type~$B$ non-crossing partition is unique, it
must satisfy $\rotB^{(N-1)/\rRr}(B)=B$, that is, it remains invariant under the
pseudo-rotation $\rotB^{(N-1)/\rRr}$.

If $|Z|=2$, then it is not difficult to see directly that $\pi$
must arise from the corresponding (degenerate) case of the Construction.

Now let $|Z|\ge3$.
For some~$j$, the block $\rotB^j(Z)$ must contain $1$. Since
$|Z|\ge3$, the degenerate case of the Construction never applies
when applying~$\rotB$ to~$\pi$.
But then $1-\frac {N-1} {\rRr }k$ must also be in~$\rotB^j(Z)$.
For $k=\rRr $, we see that $1-(N-1)=\overline{2}$ is an element
of~$\rotB^{j}(Z)$, and thus also~$2$. It is easy to see that
$\rotB^{(N-1)/\rRr }$ must act by
(ordinarily) rotating the other blocks $B$ of~$\rotB^j(\pi)$ 
by $\frac {N-1} {\rRr }$ units, except when we move
across $1$ where $B$ is rotated by $\frac {N-1}
{\rRr }+1$ units. Thus we have shown that $\rotB^j(Z)$ arises from 
the Construction,
and all other elements in the orbit of~$\pi$ arise by applying the
pseudo-rotation~$\rotB$ several times.
\end{proof}

\subsection{Characterisation of pseudo-rotationally invariant positive
  non-crossing\break partitions in type~$D_n$}
\label{app:inv-D}

\begin{proof}[Proof of \Cref{lem:allD}]
Let first $\pi$ be a partition arising from Construction~1.
It is in particular a partition with a zero block.
If we remove the elements of the inner circle from~$\pi$, then what
remains is a partition in $\mNCBPlus[n-1]$ with a zero block
which is
$2\rRr$-pseudo-rotationally invariant with respect to~$\rotB$
(cf.\ the Construction in \Cref{sec:rotB}). Since
$\rotD$ acts on the elements of the outer circle in that case in the
same way as~$\rotB$, the entire partition~$\pi$ is 
$2\rRr$-pseudo-rotationally invariant with respect to~$\rotD$.

\medskip
Next, let $\pi$ be a partition arising from Construction~2.
Since in that case the partition of the elements of the inner circle
is uniquely determined, we may as well ignore the elements of the
inner circle. What remains is a partition in $\mNCBPlus[n-1]$
without zero block. By \Cref{lem:N-2B} with $N=m(n-1)$, this reduced
partition is invariant under~$\rotB^{m(n-1)-1}$. 
Since
$\rotD$ acts on the elements of the outer circle in that case in the
same way as~$\rotB$, the entire partition~$\pi$ is 
$2$-pseudo-rotationally invariant with respect to~$\rotD$.

\medskip
Finally, let $\pi$ be a partition arising from Construction~3.
It is not difficult to see from an analysis in the spirit of the
proof of \Cref{lem:N-2D} (given in \Cref{app:order-D}) that after 
${\big(m(n-1)-1\big)} /{\rRr}$
applications of~$\rotD$ we will obtain a partition that is the same
as~$\pi$, except that the elements on the inner circle may have
changed their sign. In order to see what
happened with the elements on the inner circle, we note that the
elements of the inner circle of the block~$B$ in Construction~3 are
rotated by $\frac {1} {\rRr}\big(m(n-1)-1\big)+1$ units in
counter-clockwise direction, the ``$+1$" coming from the very first
application of~$\rotD$ when an additional element on the inner circle
is added to the block. After this number of applications of~$\rotD$, before
that rotated block (in clockwise direction) 
we find again the fundamental seed, possibly
with changed sign for the elements of the inner circle. By definition,
the fundamental seed has $x$ elements on the inner circle.
Furthermore, we assume that it contains in total $my$ elements.
We know from the second assertion in \Cref{lem:Const3} that $y=n/\rRr$.
Hence, after these ${\big(m(n-1)-1\big)} /{\rRr}$ applications
of~$\rotD$, the ``first" element on the inner circle of the fundamental
seed  is
$$
\pm(mn-1)+\frac {1} {\rRr}\big(m(n-1)-1\big)+1+x
=\pm(mn-1)+(my-x-1)+1+x
=\pm(mn-1)+my,
$$
taken modulo $2m$ (the number of elements on the inner circle). 
In order to get the above equality we used~\eqref{eq:rotrel1}.
Consequently, if $y=n/\rRr$ is even
then this is $\pm(mn-1)$, which means that $\rotD^{(m(n-1)-1)/\rRr}$
leaves $\pi$ invariant. On the other hand, if $y=n/\rRr$ is odd
then this is $\mp(mn-1)$, which means that we need $2\big(m(n-1)-1\big)/\rRr$
applications of~$\rotD$ until we return to~$\pi$.

\medskip
For the converse, let $\pi$ be a $\rrr $-pseudo-rotationally invariant 
non-crossing partitions in $\mNCDPlus$. 
We have to show that 
$\rrr $-pseudo-rotationally invariant positive
non-crossing partitions of type~$D$ cannot arise in a different
way than from Constructions~1, 2, or~3 plus subsequent applications of~$\rotD$. 
If, in $\pi$, the inner circle as a whole belongs to the zero block, 
then the action of~$\rotD$ 
on~$\pi$ can be described in the following way: let $\widehat\pi$ be
the type~$B$ positive non-crossing partition that arises from~$\pi$ by
removing the elements of the inner circle from the zero block of~$\pi$. Then form
$\rotB(\widehat\pi)$. Finally add the elements of the inner circle
to the zero block of $\rotB(\widehat\pi)$. 
In other words: the pseudo-rotation~$\rotD$ acts as~$\rotB$ on the
outer circle, while the inner circle is completely contained in the
zero block anyway. By \Cref{lem:allB}, it follows that $\pi$ must
arise from Construction~1.

If, in $\pi$, the inner circle decomposes into two blocks, 
then the action of~$\rotD$ 
on~$\pi$ can be described in the following way: let $\widehat\pi$ be
the type~$B$ positive non-crossing partition that arises from~$\pi$ by
ignoring the elements of the inner circle. Then form
$\rotB(\widehat\pi)$. Finally add the elements of the inner circle
decomposed into two blocks in the unique way
determined by $\rotB(\widehat\pi)$ according to rule~(D3) in
\Cref{sec:realD}.
In other words: the pseudo-rotation~$\rotD$ acts as~$\rotB$ on the
outer circle, while the inner circle is completely determined by
the partition on the outer circle.
By \Cref{lem:ord=N-1}, it follows that $\pi$ must
arise from Construction~2 and that it is $2$-pseudo-rotationally invariant.

Finally, let $\pi$ be an element of $\mNCDPlus$
that contains a bridging block and is $\rrr$-pseudo-rotationally
invariant, i.e., $\rotD^{2(m(n-1)-1)/\rrr}(\pi)=\pi$.
If $\rrr$ is even, then
$$
\frac {2\big(m(n-1)-1\big)} {\rrr}=\frac {\big(m(n-1)-1\big)} {\rrr/2}
$$
applications on~$\pi$ bring us back to~$\pi$. By arguing in the spirit
of the proof of \Cref{lem:N-2D} (given in \Cref{app:order-D}),
it follows that~$\pi$ must arise from
Construction~3 with $\rRr=\rrr/2$ (and possibly subsequent
applications of~$\rotD$). 

On the other hand, let $\rrr$ be odd.
Let $\la$ be the multiplicative inverse of~$2$ modulo~$\rrr$. Then
there exists a positive integer~$\ka$ such that $2\la=1+\rrr\ka$.
Thus, we have
\begin{equation} \label{eq:rh-odd} 
\la\cdot\frac {2\big(m(n-1)-1\big)} {\rrr}
=\frac {2\la} {\rrr}(m(n-1)-1)
=\frac {\big(m(n-1)-1\big)} {\rrr}+\ka\big(m(n-1)-1\big).
\end{equation}
Now, by assumption $\pi$ is invariant under $\rotD^{2(m(n-1)-1)/\rrr}$.
Recall from the proof of \Cref{lem:N-2D} that $\rotD^{m(n-1)-1}$
applied to~$\pi$ yields structurally the same partition as~$\pi$,
except that the elements on the inner circle may have changed their
sign. If these two facts are combined with Equation~\eqref{eq:rh-odd},
then we see that $\rotD^{(m(n-1)-1)/\rrr}$ applied to~$\pi$
yields structurally the same partition as~$\pi$, 
except that the elements on the inner circle may have changed their
sign. However, this means that $\pi$ must arise from
Construction~3 with $\rRr=\rrr$ (and possibly subsequent
applications of~$\rotD$). 

\medskip
This finishes the proof of the theorem.
\end{proof}

\section{Proofs: Evaluations and non-negativity results for the positive $q$-Fu\ss--Catalan numbers}
\label{app:unnec}

This appendix is devoted to the proofs of the three auxiliary lemmas in \Cref{lem:1,,lem:2,,lem:2-D} giving evaluations and non-negativity results for the positive $q$-Fu\ss--Catalan numbers.

\subsection{An auxiliary result}
\label{app:unnec-aux}

We start with a strengthening of the fact that {\it rational
  $q$-Catalan polynomials} (cf.\ \cite[Sec.~5]{ArmDAB}) have
non-negative coefficients.
The proof which we give follows largely the proofs in \cite[Thm.~2]{AndrCB}
and \cite[Prop.~10.1(iii)]{RSW2004} of the above fact.


\begin{lemma} \label{lem:Andrews}
Let $a$ and $b$ be positive integers.
Then $\frac {[\gcd(a,b)]_q} {[a+b]_q}\left[\begin{smallmatrix}
a+b\\a\end{smallmatrix}\right]_q$ is a polynomial in $q$ with non-negative
integer coefficients.
\end{lemma}

\begin{proof}
We first show that the expression in
$\frac {[\gcd(a,b)]_q} {[a+b]_q}\left[\begin{smallmatrix}
a+b\\a\end{smallmatrix}\right]_q$ is a polynomial in $q$, and then derive the
non-negativity of the coefficients from the fact that the $q$-binomial
coefficient 
$\left[\begin{smallmatrix} a+b\\a\end{smallmatrix}\right]_q$ is a 
reciprocal\footnote{A polynomial
$P(q)$ in $q$ of degree $d$ is called \defn{reciprocal\/} if
$P(q)=q^dP(1/q)$.} and unimodal\footnote{\label{foot:2}%
A polynomial $P(q)=
\sum _{i=0} ^{d}p_iq^i$ is called \defn{unimodal\/} if there is an
integer $r$ with $0\le r\le d$ and $0\le p_0\le\dots\le p_r\ge\dots\ge
p_d\ge0$. It is well-known that $q$-binomial coefficients are 
unimodal; see \cite[Ex.~7.75.d]{StanBI}.} polynomial in $q$. 
In order to show polynomiality, one recalls the well-known fact that
$$
q^n-1=\prod _{d\mid n} ^{}C_d(q),
$$
where $C_d(q)$ denotes the $d$\th\ cyclotomic polynomial in $q$.
Consequently, 
$$
\frac {[\gcd(a,b)]_q} {[a+b]_q}\begin{bmatrix} a+b\\a\end{bmatrix}_q=
\prod _{d=2} ^{a+b-1}C_d(q)^{e_d},
$$
with
\begin{equation} \label{eq:SH} 
e_d=\chi(d\mid \gcd(a,b))
+\fl{\frac {a+b-1} {d}}
-\fl{\frac {a} {d}}
-\fl{\frac {b} {d}},
\end{equation}
where $\chi(\mathcal S)=1$ if $\mathcal S$ is
true and $\chi(\mathcal S)=0$ otherwise.
Let us write 
$A=\{a/d\}$ and
$B=\{b/d\}$,
where $\{\al\}:=\al-\fl{\al}$ denotes the fractional
part of~$\al$.
Using this notation, Equation~\eqref{eq:SH} becomes
$$e_d=\chi(d\mid \gcd(a,b))+\fl{A+B-\frac {1} {d}}.$$ 
This is clearly non-negative, unless 
$A=B=0$. However, in that case we have $d\mid a$ and $d\mid b$,
that is, $d\mid\gcd(a,b)$, so that $e_d$ is non-negative
also in this case.

Since we know that 
$\left[\begin{smallmatrix} a+b\\a\end{smallmatrix}\right]_q$ is a 
reciprocal and unimodal polynomial in $q$ of degree $ab$
(see \Cref{foot:2}), 
the first $\fl{ab/2}$ coefficients of
$(1-q^{\gcd(a,b)})\left[\begin{smallmatrix} a+b\\a\end{smallmatrix}\right]_q$
are non-negative, and therefore the same must be true for
$$
\frac {1-q^{\gcd(a,b)}} {1-q^{a+b}}
\begin{bmatrix} a+b\\a\end{bmatrix}_q
=\frac {[\gcd(a,b)]_q} {[a+b]_q}\begin{bmatrix} a+b\\a\end{bmatrix}_q.
$$
However, we already know that the last expression 
is actually a \emph{polynomial} in $q$, of degree
$ab-(a+b)+\gcd(a,b)$, and it is as well a
reciprocal polynomial. Hence, also the remaining coefficients are
non-negative, which is what we wanted to prove.
\end{proof}

\subsection{Polynomiality with non-negative coefficients}
\label{app:unnec-pol}

\begin{proof}[Proof of \Cref{lem:1}]
Expression~\eqref{eq:SD} is the special case of~\eqref{eq:SE} where $s=0$,
and consequently does not need to be treated separately.

\medskip
In order to establish polynomiality of Expression~\eqref{eq:SE},
we rewrite the expression~\eqref{eq:SE} in the form
{\refstepcounter{equation}\label{eq:SEa}}
\alphaeqn
\begin{gather} \label{eq:SEa.1}
\frac {[mn-n+s]_q} 
{[\gcd(mn-1,n-s-1)]_q\,[\gcd(n,s)]_q}
\kern4cm
\\
\label{eq:SEa.2}
\times
\frac {[\gcd(n,s)]_q} {[n]_q}
\begin{bmatrix} {n}\\{s}\end{bmatrix}_q\\
\label{eq:SEa.3}
\kern4cm
\times
\frac {[\gcd(mn-1,n-s-1)]_q} {[mn-1]_q}
\begin{bmatrix} {mn-1}\\ {n-s-1}\end{bmatrix}_q.
\end{gather}
\reseteqn
We note that
both  $\gcd(n,s)$ and $\gcd(mn-1,n-s-1)$ 
divide $mn-n+s$. Moreover, we claim that the two numbers
are coprime. For, if there is a $d$ dividing both 
$\gcd(n,s)$ and $\gcd(mn-1,n-s-1)$, then 
we infer that $d\mid (n-1)$, which, in combination with $d\mid n$, means
that actually $d=1$.
It is then not difficult to see
that the quotient in~\eqref{eq:SEa.1} is a polynomial in $q$
(see e.g.\ \cite[Lemma~5]{KrMuAD}). 
Clearly, \Cref{lem:Andrews} implies that both~\eqref{eq:SEa.2}
and~\eqref{eq:SEa.3} are also polynomials.

\medskip
Finally, to see that \eqref{eq:SE} is not only a polynomial, but also
has non-negative coefficients, we proceed as in the proof of
\Cref{lem:Andrews}. Since
products of reciprocal and unimodal polynomials are again reciprocal
and unimodal (see \cite[Thm.~3.9]{AndrAF}), the polynomials
$$
\begin{bmatrix} {n}\\{s}\end{bmatrix}_q
\quad \text{and}\quad 
\begin{bmatrix} {mn-1}\\ {n-s-1}\end{bmatrix}_q
$$
are reciprocal and unimodal. The first one is of degree $D_1:=s(n-s)$,
while the second is of degree $D_2:=(n-s-1)(mn-n+s)$. 
As a consequence, the first $\fl{D_1/2}$ coefficients of 
$(1-q)\left[\begin{smallmatrix} {n}\\{s}\end{smallmatrix}\right]_q$
and the first $\fl{D_2/2}$ coefficients of
$$
(1-q^{mn-n+s})
\begin{bmatrix} {mn-1}\\ {n-s-1}\end{bmatrix}_q
$$
are non-negative. But then the first $\fl{D_1/2}+\fl{D_2/2}$
coefficients of the series
$$
\frac {(1-q)(1-q^{mn-n+s})} {(1-q^{mn-1})(1-q^n)}
\begin{bmatrix} {n}\\{s}\end{bmatrix}_q
\begin{bmatrix} {mn-1}\\ {n-s-1}\end{bmatrix}_q,
$$
are also non-negative. 
However, we know that the last expression (which agrees with~\eqref{eq:SE}) 
is actually a \emph{polynomial} in $q$, of degree
$D_1+D_2-(n-1)-(n-s-1)$, and it is as well a
reciprocal polynomial. Hence, also the remaining coefficients are
non-negative, which is what we wanted to prove.

\medskip
The claims about \eqref{eq:SF} can be established in a completely analogous
fashion. To prove polynomiality, we rewrite~\eqref{eq:SF} in the form
{\refstepcounter{equation}\label{eq:SFa}}
\alphaeqn
\begin{gather} \label{eq:SFa.1} 
\frac {[mn-b_1-b_2-\dots-b_n]_q} 
{[\gcd(b_1,b_2,\dots,b_n)]_q\,[\gcd(mn-1,b_1+b_2+\dots+b_n-1)]_q}
\kern4cm\\
\label{eq:SFa.2} 
\times
\frac {[\gcd(b_1,b_2,\dots,b_n)]_q} {[b_1+b_2+\dots+b_n]_q}
\begin{bmatrix}
  {b_1+b_2+\dots+b_n}\\{b_1,b_2,\dots,b_n}\end{bmatrix}_q\\
\label{eq:SFa.3} 
\kern4cm
\times
\frac {[\gcd(mn-1,b_1+b_2+\dots+b_n-1)]_q} {[mn-1)]_q}
\begin{bmatrix} {mn-1}\\ {b_1+b_2+\dots+b_n-1}\end{bmatrix}_q.
\end{gather}
\reseteqn
It follows from \eqref{eq:SB} that $\gcd(b_1,b_2,\dots,b_n)$ divides
$(mn-b_1-b_2-\dots-b_n)$. 
Moreover, one sees directly that $\gcd(mn-1,b_1+b_2+\dots+b_n-1)$ divides
$(mn-b_1-b_2-\dots-b_n)$.
Furthermore, if there would be a $d$ dividing both
$\gcd(b_1,b_2,\dots,b_n)$ and
$\gcd(mn-1,b_1+b_2+\dots+b_n-1)$, then~\eqref{eq:SB} would imply that $d\mid
n$, so that, if this is combined with $d\mid (mn-1)$, we see that we
must have $d=1$. As earlier, as a consequence, the
fraction~\eqref{eq:SFa.1} is a polynomial in $q$. That~\eqref{eq:SFa.2} is a
polynomial in $q$ can be seen by proceeding along the lines of the
proof of \Cref{lem:Andrews}. We leave it to the reader to work
out the details. In its turn, \Cref{lem:Andrews} implies
polynomiality of~\eqref{eq:SFa.3}.
This establishes the
polynomiality of~\eqref{eq:SF}. Non-negativity of the coefficients
follows now as before of those of~\eqref{eq:SE}. We omit 
the details for the sake of brevity.
\end{proof}

\subsection{Evaluations at roots of unity}
\label{app:unnec-eval}

\begin{proof}[Proof of \Cref{lem:2}]
We have
$$
\frac {1} {[n]_q}\bmatrix (m+1)n-2\\n-1\endbmatrix_q
=\frac {(1-q^{(m+1)n-2})(1-q^{(m+1)n-3})\cdots(1-q^{mn})} 
{(1-q^{n})(1-q^{n-1})\cdots(1-q^{2})}.
$$
In order to perform the limit $q\to \om_\rrr =e^{2\pi
  i\si/(mn-2)}=e^{2\pi i/\rrr}$, we use 
the simple fact 
\begin{equation} \label{eq:SM}
\lim_{q\to \om_\rrr } \frac {1-q^\al} {1-q^\be}=
\begin{cases}
\frac \alpha \beta,&\text{if }\alpha\equiv\beta\equiv0\pmod \rrr ,\\
1,&\text{otherwise},
\end{cases}	
\end{equation}
where $\alpha,\beta$ are
non-negative integers such that 
$\alpha\equiv\beta\pmod \rrr $. Since $\rrr \mid (mn-2)$, the claimed limit
in~\eqref{eq:SI} is obtained straightforwardly.

\medskip
In order to see~\eqref{eq:SJ}, we write
\begin{equation} \label{eq:SN}
\frac {1} {[n]_q}\bmatrix mn-2\\n-1\endbmatrix_q
=\frac {(1-q^{mn-2})(1-q^{mn-3})\cdots(1-q^{mn-n})} 
{(1-q^{n})(1-q^{n-1})\cdots(1-q^{2})}.
\end{equation}
Here, we first check in which cases the number of factors
in the numerator
which vanish when $q\to e^{2\pi i\si/(mn-2)}=e^{2\pi i/\rrr}$ does not exceed
the number of those
in the denominator, because otherwise the limit of the left-hand side 
of~\eqref{eq:SN} as $q\to e^{2\pi i\si/(mn-2)}$ is zero. 
Now, the difference of the numbers of vanishing factors
in the numerator and the denominator is
$$
\fl{\frac {mn-2} {\rrr }}-\fl{\frac {mn-n-1} {\rrr }}-\fl{\frac {n} {\rrr }}=
-\fl{\frac {-n+1} {\rrr }}-\fl{\frac {n} {\rrr }}.
$$
Writing $N=\{n/\rrr \}$ for the fractional part of $n/\rrr $, this reduces to
$
-\fl{-N+\frac {1} {\rrr }},
$
which can only be zero if $N=0$ or $N=1/\rrr $, or, equivalently,
if $n\equiv0,1$~(mod~$\rrr $). If $n\equiv0$~(mod~$\rrr $), then, because of
the assumption $\rrr \mid(mn-2)$, we infer $\rrr =2$. The first alternative
on the right-hand side of~\eqref{eq:SJ} follows immediately using~\eqref{eq:SM}.
On the other hand, if $n\equiv1$~(mod~$\rrr $) then the second
alternative on the right-hand side of~\eqref{eq:SJ}
follows again from~\eqref{eq:SM}.

\medskip
Next we consider \eqref{eq:SK}. We rewrite the expression on the left-hand
side in an analogous fashion as product/quotient of factors of the
form $1-q^j$. Then, again, we count how many of these factors 
vanish when $q\to e^{2\pi i\si/(mn-2)}=e^{2\pi i/\rrr}$. Letting 
$N=\{(n-1)/\rrr \}$,
$S=\{s/\rrr \}$, and remembering that $\rrr \mid(mn-2)$, 
the difference of the numbers of vanishing factors in
the numerator and in the denominator equals
\begin{equation} \label{eq:SO}
-\fl{N-S+\frac {1} {\rrr }}-\fl{N-S}-\fl{-N+S}
=
-\fl{N-S+\frac {1} {\rrr }}+1-\chi\big(\rrr \mid(n-s-1)\big).
\end{equation}
Now there are three cases to distinguish. 

\smallskip
{\sc Case 1: $\fl{N-S+\frac {1} {\rrr }}=1$.} 
This means that $N=\frac {\rrr -1} {\rrr }$ and $S=0$, that is, 
$\rrr \mid n$ and $\rrr \mid s$. In combination with $\rrr \mid (mn-2)$, it follows that
$\rrr =2$. A further consequence is that $2=\rrr \nmid(n-s-1)$, so that~\eqref{eq:SO}
is satisfied.
The result of the substitution on the left-hand side
of~\eqref{eq:SK} is easily seen to be as asserted in the second alternative
on the right-hand side.

\smallskip
{\sc Case 2: $\fl{N-S+\frac {1} {\rrr }}=0$.}
Another way to describe the defining condition of this case is
to say that $\{n/\rrr \}\ge\{s/\rrr \}$ and not both $\rrr \mid n$ and $\rrr \mid s$.
In order \eqref{eq:SO} to be satisfied, we must have $\rrr \mid(n-s-1)$,
implying that $\{s/\rrr \}=\{(n-1)/\rrr \}\le \{n/\rrr \}$,
and that both $\rrr \mid n$ and $\rrr \mid s$ is impossible.
The result of the substitution on the left-hand
side of~\eqref{eq:SK} is as asserted in the first alternative 
on the right-hand side.

\smallskip
{\sc Case 3: $\fl{N-S_1+\frac {1} {\rrr }}=-1$.} 
Here, the expression~\eqref{eq:SO} is always positive, so
that the left-hand side of~\eqref{eq:SK} vanishes.

\medskip
Finally we address~\eqref{eq:SL}. As before,
we rewrite the expression on the left-hand
side as product/quotient of factors of the
form $1-q^j$. Subsequently we count how many of these factors 
vanish when $q\to \om_\rrr =e^{2\pi i\si/(mn-2)}=e^{2\pi i/\rrr}$. Letting 
$B_i=\{b_i/\rrr \}$, $i=1,2,\dots,n$, and again remembering that $\rrr \mid(mn-2)$,
the difference of the numbers of vanishing factors in
the numerator and in the denominator equals
\begin{equation} \label{eq:SP}
-\fl{-\sum_{k=1}^nB_k+\frac {1} {\rrr }}.
\end{equation}
Again we distinguish three cases. 

\smallskip
{\sc Case 1: $\sum_{k=1}^n B_k=0$.} In other words, we have $B_k=0$
for all~$k$, or, equivalently, $\rrr \mid b_k$ for all~$k$. 
Because of~\eqref{eq:SB}, this implies that also $\rrr \mid n$.
From $\rrr \mid (mn-2)$ it then follows that $\rrr \mid 2$, that is, that
$\rrr =2$. The result of the substitution on 
the left-hand side of~\eqref{eq:SL} is as asserted in the second
alternative on the right-hand side.

\smallskip
{\sc Case 2: $\sum_{k=1}^n B_k=\frac {1} {\rrr }$.} 
Here we have $B_k=0$ for all~$k$ but one; let $j$ be the exception.
More explicitly, we have $\rrr \mid b_k$ for all $k\ne j$ and
$b_j\equiv1$~(mod~$\rrr $). 
The result of the substitution on 
the left-hand side of~\eqref{eq:SL} is as asserted in the first
alternative on the right-hand side.

\smallskip
{\sc Case 3: $\sum_{k=1}^n B_k>\frac {1} {\rrr }$.} 
Then $\fl{-\sum_{k=1}^n B_k+\frac {1} {\rrr }}<0$, and consequently
the sum in~\eqref{eq:SP} is positive, so that the left-hand side
of~\eqref{eq:SL} vanishes.

This completes the proof of the lemma.
\end{proof}

\begin{proof}[Proof of \Cref{lem:2-B}]
We have
$$
\bmatrix (m+1)n-1\\n\endbmatrix_{q^2}
=\frac {(1-q^{2(m+1)n-2})(1-q^{2(m+1)n-4})\cdots(1-q^{2mn})} 
{(1-q^{2n})(1-q^{2n-2})\cdots(1-q^2)}.
$$
In order to perform the limit $q\to \om_\rrr =e^{2\pi
  i\si/(2mn-2)}=e^{2\pi i/\rrr}$, we use 
again the simple fact 
\begin{equation} \label{eq:SM-B}
\lim_{q\to \om_\rrr } \frac {1-q^\al} {1-q^\be}=
\begin{cases}
\frac \alpha \beta,&\text{if }\alpha\equiv\beta\equiv0\pmod \rrr ,\\
1,&\text{otherwise},
\end{cases}	
\end{equation}
where $\alpha,\beta$ are
non-negative integers such that 
$\alpha\equiv\beta\pmod \rrr $. Since $\rrr \mid (mn-1)$, the claimed limit
in~\eqref{eq:SI-B} is obtained without much effort.

\medskip
In order to see~\eqref{eq:SJ-B}, we write
\begin{equation} \label{eq:SN-B}
\bmatrix mn-1\\n\endbmatrix_{q^2}
=\frac {(1-q^{2mn-2})(1-q^{2mn-4})\cdots(1-q^{2mn-2n})} 
{(1-q^{2n})(1-q^{2n-2})\cdots(1-q^2)}.
\end{equation}
Here, we first check in which cases the number of factors
in the numerator
which vanish when $q\to e^{2\pi i\si/2(mn-1)}=e^{2\pi i/\rrr}$ does not exceed
the number of those
in the denominator, because otherwise the limit of the left-hand side 
of~\eqref{eq:SN-B} as $q\to e^{2\pi i\si/2(mn-1)}=e^{2\pi i/\rrr}$ is zero. 

If $\rrr$ is even, the difference of the numbers of vanishing factors
in the numerator and the denominator is
$$
\fl{\frac {2(mn-1)} {\rrr }}-\fl{\frac {2(mn-n-1)} {\rrr }}-\fl{\frac {2n} {\rrr }}=
-\fl{\frac {-2n} {\rrr }}-\fl{\frac {2n} {\rrr }}.
$$
Writing $N=\{2n/\rrr \}$ for the fractional part of $2n/\rrr $, this reduces to
$
-\fl{-N}
$,
which can only by zero if $N=0$.
This means that $2n\equiv0$~(mod~$\rrr $). In this case, because of
the assumption $\rrr \mid2(mn-1)$, we infer $\rrr =2$. The first alternative
on the right-hand side of~\eqref{eq:SJ-B} follows immediately using~\eqref{eq:SM-B}.

If $\rrr$ is odd, the difference of the numbers of vanishing factors
in the numerator and the denominator is
$$
\fl{\frac {mn-1} {\rrr }}-\fl{\frac {mn-n-1} {\rrr }}-\fl{\frac {n} {\rrr }}=
-\fl{\frac {-n} {\rrr }}-\fl{\frac {n} {\rrr }}.
$$
Writing $N=\{n/\rrr \}$ for the fractional part of $n/\rrr $, this reduces to
$
-\fl{-N}
$,
which can only by zero if $N=0$.
This means that $n\equiv0$~(mod~$\rrr $). In this case, because of
the assumption $\rrr \mid2(mn-1)$, we infer $\rrr =1$. The first alternative
on the right-hand side of~\eqref{eq:SJ-B} follows immediately using~\eqref{eq:SM-B}.

\medskip
Next we consider~\eqref{eq:SK-B}. We rewrite the expression on the left-hand
side in an analogous fashion as product/quotient of factors of the
form $1-q^j$. Then, again, we count how many of these factors 
vanish when $q\to e^{2\pi i\si/(mn-1)}=e^{2\pi i/\rrr}$. 

If $\rrr$ is even, then let $N=\{2n/\rrr \}$ and
$S=\{2s/\rrr \}$. Remembering that $\rrr \mid(mn-1)$, 
the difference of the numbers of vanishing factors in
the numerator and in the denominator equals
\begin{equation} \label{eq:SO-B}
-2\fl{N-S}-\fl{-N+S}.
\end{equation}
If $N<S$, then this expression equals~$2$. If $N>S$, then it is~$1$.
Only if $N=S$, or, equivalently, $2n\equiv 2s$~(mod~$\rrr$), 
the expression~\eqref{eq:SO-B} vanishes. 
The first alternative
on the right-hand side of~\eqref{eq:SK-B} follows immediately using~\eqref{eq:SM-B}.

If $\rrr$ is odd, then we let $N=\{n/\rrr \}$ and
$S=\{s/\rrr \}$. Here,
the difference of the numbers of vanishing factors in
the numerator and in the denominator equals again~\eqref{eq:SO-B},
however with the modified meanings of~$N$ and~$S$ that we consider
here. The remaining steps are completely analogous to the ones in the
previous paragraph, and they lead to the second alternative in~\eqref{eq:SK-B}.

\medskip
Finally we address \eqref{eq:SL-B}. As before,
we rewrite the expression on the left-hand
side as product/quotient of factors of the
form $1-q^j$. Subsequently we count how many of these factors 
vanish when $q\to \om_\rrr =e^{2\pi i\si/(mn-1)}=e^{2\pi i/\rrr}$. 

If $\rrr$ is even, then write
$B_i=\{2b_i/\rrr \}$, $i=1,2,\dots,n$. Again remembering that $\rrr \mid(mn-1)$,
we see that the difference of the numbers of vanishing factors in
the numerator and in the denominator equals
\begin{equation} \label{eq:SP-B}
-\fl{-\sum_{k=1}^nB_k}.
\end{equation}
This expression can only be zero if
$\sum_{k=1}^n B_k=0$, that is, if $B_k=0$
for all~$k$. Equivalently, we have $\rrr \mid 2b_k$ for all~$k$. 
The result of the substitution on 
the left-hand side of~\eqref{eq:SL-B} is as asserted in the first
alternative on the right-hand side.

If $\rrr$ is odd, then we let
$B_i=\{b_i/\rrr \}$, $i=1,2,\dots,n$. Here,
the difference of the numbers of vanishing factors in
the numerator and in the denominator turns out to again
equal~\eqref{eq:SP-B}, with the modified meaning of the $B_i$'s, however.
The remaining steps are completely analogous to the ones in the
previous paragraph, and they lead to the second alternative in~\eqref{eq:SL-B}.

\medskip
This completes the proof of the lemma.
\end{proof}

\begin{proof}[Proof of \Cref{lem:2-D}]
We have
\begin{multline*}
\frac {[2m(n-1)+n-2]_{q}} {[n]_{q}}
\begin{bmatrix} (m+1)(n-1)-1\\n-1\end{bmatrix}_{q^2}\\
=\frac {(1-q^{2m(n-1)+n-2})} {(1-q^{n})}\cdot
\frac {(1-q^{2(m+1)(n-1)-2})(1-q^{2(m+1)(n-1)-4})\cdots(1-q^{2m(n-1)})} 
{(1-q^{2n-2})\cdots(1-q^4)(1-q^2)}.
\end{multline*}
In order to perform the limit $q\to \om_\rrr =e^{2\pi
  i\si/(2m(n-1)-2)}=e^{2\pi i/\rrr}$, we use 
again the simple fact 
\begin{equation} \label{eq:SM-D}
\lim_{q\to \om_\rrr } \frac {1-q^\al} {1-q^\be}=
\begin{cases}
\frac \alpha \beta,&\text{if }\alpha\equiv\beta\equiv0\pmod \rrr ,\\
1,&\text{otherwise},
\end{cases}	
\end{equation}
where $\alpha,\beta$ are
non-negative integers such that 
$\alpha\equiv\beta\pmod \rrr $. Since $\rrr \mid (2m(n-1)-2)$, the claimed limit
in~\eqref{eq:SI-D} is obtained without much effort.

\medskip
In order to see \eqref{eq:SJ-D}, we write
\begin{multline} \label{eq:SN-D}
\frac {[2m(n-1)-n]_{q}} {[n]_{q}}
\begin{bmatrix} m(n-1)-1\\n-1\end{bmatrix}_{q^2}\\
=\frac {(1-q^{2m(n-1)-n})} {(1-q^{n})}\cdot
\frac {(1-q^{2m(n-1)-2})(1-q^{2m(n-1)-4})\cdots(1-q^{2m(n-1)-2(n-1)})} 
{(1-q^{2n-2})\cdots(1-q^4)(1-q^2)}.
\end{multline}
Here, we first check in which cases the number of factors
in the numerator
which vanish when $q\to e^{2\pi i\si/(2m(n-1)-2)}=e^{2\pi i/\rrr}$ does not exceed
the number of those
in the denominator, because otherwise the limit of the left-hand side 
of~\eqref{eq:SN-D} as $q\to e^{2\pi i\si/(2m(n-1)-2)}=e^{2\pi i/\rrr}$ is zero. 

If $\rrr$ is even, the difference of the numbers of vanishing factors
in the numerator and the denominator is
\begin{multline*}
\chi\big(\rrr\mid(2m(n-1)-n)\big)-\chi(\rrr\mid n)+
\fl{\frac {2(m(n-1)-1)} {\rrr }}
-\fl{\frac {2(m(n-1)-n)} {\rrr }}-\fl{\frac {2(n-1)} {\rrr }}\\
=\chi\big(\rrr\mid(2-n)\big)-\chi(\rrr\mid n)
-\fl{\frac {-2(n-1)} {\rrr }}-\fl{\frac {2(n-1)} {\rrr }}.
\end{multline*}
Writing $N=\{2(n-1)/\rrr \}$ for the fractional part of $2(n-1)/\rrr$, 
this reduces to
\begin{equation} \label{eq:SA-D} 
\chi\big(\rrr\mid(2-n)\big)-\chi(\rrr\mid n)
-\fl{-N}.
\end{equation}
We now distinguish several cases.

\medskip
{\sc Case 1: $\rrr=2$.} In this case, we have $N=0$, and, regardless
whether $n$ is even or odd, the expression in~\eqref{eq:SA-D} vanishes.
The result of the limit $q\to\om_\rrr$ of~\eqref{eq:SN-D}
is as asserted in the first and the second
alternative on the right-hand side of~\eqref{eq:SJ-D}, respectively.

\medskip
{\sc Case 2: $\rrr\ge4$ and $N=0$.}
The condition $N=0$ means that $\rrr\mid 2(n-1)$. Together with
$\rrr\mid (2m(n-1)-2)$, this entails that $\rrr=2$, a contradiction
to $\rrr\ge4$.

\medskip
{\sc Case 3: $\rrr\ge4$, $N>0$ and $\rrr \mid n$.} 
Here, the expression~\eqref{eq:SA-D} vanishes.
The result of the limit $q\to\om_\rrr$ of~\eqref{eq:SN-D} 
is as asserted in the third
alternative on the right-hand side of~\eqref{eq:SJ-D}.

\medskip
In all other cases, the expression~\eqref{eq:SA-D} is strictly positive.
The limit $q\to\om_\rrr$ of~\eqref{eq:SN-D} is therefore zero.

\medskip
Now we treat the case where $\rrr$ is odd.
The difference of the numbers of vanishing factors
in the numerator and the denominator of~\eqref{eq:SN-D} is
\begin{multline*}
\chi\big(\rrr\mid(2m(n-1)-n)\big)-\chi(\rrr\mid n)+
\fl{\frac {m(n-1)-1} {\rrr }}-\fl{\frac {m(n-1)-n} {\rrr }}-\fl{\frac {n-1}
  {\rrr }}\\
=
\chi\big(\rrr\mid(2m(n-1)-n)\big)-\chi(\rrr\mid n)
-\fl{\frac {-(n-1)} {\rrr }}-\fl{\frac {n-1} {\rrr }}.
\end{multline*}
Writing $N=\{(n-1)/\rrr \}$ for the fractional part of $(n-1)/\rrr $, this reduces to
\begin{equation} \label{eq:SB-D} 
\chi\big(\rrr\mid(2-n)\big)-\chi(\rrr\mid n)
-\fl{-N}.
\end{equation}
Again, we distinguish several cases.

\medskip
{\sc Case 1: $N=0$.}
The condition $N=0$ means that $\rrr\mid (n-1)$. Together with
$\rrr\mid (2m(n-1)-2)$, this entails that $\rrr=1$.
The result of the limit $q\to\om_\rrr$ of~\eqref{eq:SN-D} 
is as asserted in the fourth
alternative on the right-hand side of~\eqref{eq:SJ-D}.

\medskip
{\sc Case 2: $N>0$ and $\rrr \mid n$.} 
Here, the expression~\eqref{eq:SB-D} vanishes.
The result of the limit $q\to\om_\rrr$ of~\eqref{eq:SN-D} 
is as asserted in the fifth
alternative on the right-hand side of~\eqref{eq:SJ-D}.

\medskip
In all other cases, the expression~\eqref{eq:SB-D} is strictly positive.
The limit $q\to\om_\rrr$ of~\eqref{eq:SN-D} is therefore zero.

\medskip
This completes the proof of the lemma.
\end{proof}

\section{Proofs: Cyclic sieving for dihedral and exceptional groups}
\label{app:exc}

In this appendix, we prove \Cref{thm:cycexc}.
We treat the cyclic sieving for the dihedral groups explicitly in \Cref{sec:siev0} and the exceptional groups in detail in \Cref{sec:siev1}, each separately for~$\Krewplus$ and for~$\Krewplustilde$.

\medskip

In each situation, we exactly follow the approach detailed in the paragraph before \Cref{thm:cycexc}, based on the results in \Cref{sec:typeExc}.

\subsection{Cyclic sieving for the dihedral groups}
\label{sec:siev0}

We consider the situation for the dihedral group $I_2(h)$.
The degrees are $2,h$, and we hence have
\begin{align}
\mCatplus(I_2(h);q)=\frac 
{[hm+(h-2)]_q\, [hm]_q}
{[h]_q\, [2]_q}.\label{eq2:I2}
\end{align}

\subsubsection{Cyclic sieving for~$\Krewplus$}
\label{sec:siev0-Krewplus}

Given $(m+1)h-2 = hm + (h-2) = k \cdot \rrr$,
let $\zeta$ be a primitive $\rrr$\th\ root of unity. 
The following cases on the
right-hand side of \eqref{eq2:I2} occur:
\alphaeqn
\begin{align} 
\label{eq2:I2.1}
\lim_{q\to\zeta}\mCatplus(I_2(h);q)&=\mCatplus(I_2(h)),
&\text{if }\rrr = 1,\\
\label{eq2:I2.2}
\lim_{q\to\zeta}\mCatplus(I_2(h);q)&=\mCatplus(I_2(h)),
&\text{if }h\equiv0\text{ (mod $2$)},\ \rrr = 2,\\
\label{eq2:I2.3}
\lim_{q\to\zeta}\mCatplus(I_2(h);q)&=\tfrac {(m+1)h}2-1,
&\text{if }m\equiv h\equiv1\text{ (mod $2$)},\ \rrr = 2,\\
\label{eq2:I2.4}
\lim_{q\to\zeta}\mCatplus(I_2(h);q)&=0,
&\text{otherwise.}
\end{align}

To compute the invariant subwords, we distinguish the two situations where~$h$ is even or odd.

\medskip

First, let~$h \equiv 0$ (mod $2$).
Then~$\rrr$ is even and~$\rrr/2$ must divide~$n = 2$, so $\rrr \in \{2,4\}$.
The case where $\rrr = 2$ is clear, and it only remains to consider
the case where $\rrr = 4$.
In this case, $\invc$ consists of the~$h$ reflections and the corresponding word of~$h$ alternating copies of the two reflections~$\s$ and~$\t$ has even length (so that it starts with~$\s$ and ends with~$\t$).
The subwords of~$\invc$ are given by any two consecutive reflections in cyclically reverse order.
For example, in type~$I_2(6)$, we have
\[
  \invc = \r_1\r_2\r_3\r_4\r_5\r_6 = s,sts,ststs,tstst,tst,t,
\]
so that the reduced $\reflR$-words for the Coxeter element $c=st$ are
\[
  \r_1\r_6 = \r_2\r_1 = \r_3\r_2 = \r_4\r_3 = \r_5\r_4 = \r_6\r_5.
\]
The corresponding $\reflS$-word is $\c_A\c_B\c_A\c_B\c_A\c_B = \s\t\s\t\s\t$ with $\c_A = \s$ and $\c_B = \t$.
We thus see that all factorisations of~$\c=\s\t$ use one letter from~$\c_A$ and one letter from~$\c_B$, as above for $I_2(6)$.
However, by \Cref{lem:invariantsitting}(ii), 
none of these factorisations can correspond to an~$\rrr/2$-fold symmetry.

\medskip

Second, let~$h \equiv 1$ (mod $2$).
Then~$\rrr$ must divide~$n=2$, so $\rrr \in \{1,2\}$.
The case where $\rrr=1$ is clear, and it only remains to consider the
case where $\rrr = 2$.
In this case, $\invc$ consists of the~$h$ reflections and the corresponding word of~$h$ alternating copies of the two reflections~$\s$ and~$\t$ has odd length (so that it starts with~$\s$ and ends with~$\s$).
The subwords of~$\invc$ are given by any two consecutive reflections in cyclically reverse order.
For example, in type~$I_2(5)$, we have
\[
  \invc = \r_1\r_2\r_3\r_4\r_5 = s,sts,ststs,tst,t,
\]
so that the reduced $\reflR$-words for the Coxeter element $c=st$ are
\[
  \r_1\r_5 = \r_2\r_1 = \r_3\r_2 = \r_4\r_3 = \r_5\r_4.
\]
The corresponding $\reflS$-word is $\c_A\c_B\c_A\c_B\c_A = \s\t\s\t\s$ with $\c_A = \s$ and $\c_B = \t$.
From \Cref{lem:numbercrunching}, we obtain that $2$ divides $(m+1)h-2$ if and only if $m \equiv 1~(\text{mod }2)$, and we therefore write $m = 2a + 1$ and finally get
\[
  \ifix = \tfrac{1}{2}\big((1+1)h-2\big)+1 = h.
\]
Thus, as all factorisations are invariant in this case, we finally sum over all factorisations to obtain that for $m = 2a+1$, the number of invariant subwords of $\invc^m\invc^L$ whose product is~$c = st$ is given by
\[
  \binom{a}{1} + (h-1)\binom{a+1}{1} = h(a+1)-1,
\]
which equals~\eqref{eq2:I2.3} with $m = 2a+1$, as desired.

\subsubsection{Cyclic sieving for~$\Krewplustilde$}

Given $mh-2 = k \cdot \rrr$,
let $\zeta$ be a primitive $\rrr$\th\ root of unity. 
The following cases on the right-hand side of \eqref{eq2:I2} occur:
\begin{align} 
\label{eq1:I2.1}
\lim_{q\to\zeta}\mCatplus(I_2(h);q)&=\mCatplus(I_2(h)),
&\text{if }\rrr = 1,\\
\label{eq1:I2.2}
\lim_{q\to\zeta}\mCatplus(I_2(h);q)&=\mCatplus(I_2(h)),
&\text{if }h\equiv0\text{ (mod $2$)},\ \rrr = 2,\\
\label{eq1:I2.3}
\lim_{q\to\zeta}\mCatplus(I_2(h);q)&=\tfrac {mh}2,
&\text{if }m\not\equiv h\equiv1\text{ (mod $2$)},\ \rrr = 2,\\
\label{eq1:I2.4}
\lim_{q\to\zeta}\mCatplus(I_2(h);q)&=1,
&\text{otherwise.}
\end{align}
\reseteqn

This situation can be treated using the obtained counting formula for $\Krewplus$ in \Cref{sec:siev0-Krewplus}: the cases $\rrr = 1$ and $\rrr = 2$ with $h\equiv 0~\text{(mod }2)$ are clear.
So we only need to consider the cases $\rrr = 2$ for $h \equiv 1~\text{(mod }2)$ and $\rrr = 4$ for $h \equiv 0~\text{(mod }4)$.
But in both situations there is no single reflection that is invariant itself, so we have the same cases as in the previous situation with the parameter $m$ shifted to $m-1$, and with the empty factorisation (which is always invariant) added to both situations.
In the case where $h \equiv 1~\text{(mod }2)$ and $\rrr = 4$, this is the shift from~$0$ (no invariant factorisation) to~$1$ (only the empty invariant factorisation).
In the case where $h \equiv 0~\text{(mod }2)$ and $\rrr = 2$, this is the shift from~$\tfrac{(m+1)h}{2}-1$ (factorisation of length~$2$) to~$\tfrac{mh}{2}$ (factorisation of length~$2$ plus the empty invariant factorisation).

\subsection{Cyclic sieving for the exceptional groups}
\label{sec:siev1}

For our two positive Kreweras maps, we determine in detail the parameters in the considerations detailed at the end of \Cref{sec:typeExc}.
These are then used in the algorithm implemented in the computer algebra system SageMath\footnote{The computations are available upon request.}.

We remark that we do not explicitly mention the \emph{otherwise} case (being~$0$ for~$\Krewplus$ and $1$ for~$\Krewplustilde$, respectively) in each situation, but emphasise that these have all been computationally verified.
Moreover, we list the computed indices of invariant factorisations only in several small example situations.

\subsubsection{Cyclic sieving for~$\Krewplus$}
\label{sec:siev11}

In this situation, we treat factorisations $\r_1\ldots\r_n$ of the bipartite Coxeter element $c = c_Lc_R$ that are invariant under appropriate~$k$\th\ powers of~$\Krewplus$ with $(m+1)h-2 = k \cdot \rrr$.
From \Cref{thm:order}, we see that the order of $\Krewplus$ equals $k \cdot \rrr$ if $\psi\equiv\one$ and $k \cdot \rrr/2$ if $\psi\not\equiv\one$.
We can thus assume that $\rrr$ is even if $\psi\not\equiv\one$.
Moreover, we have seen in \Cref{prop:invariantsittingdetails} that we only need to consider parameters~$\rrr$ such that~$\rrr$ (if~$\psi\not\equiv\one$) or, respectively,~$\rrr/2$ (if~$\psi\equiv\one$) divides~$n$.

\medskip

\paragraph{\sc\bf Case $H_3$}
The degrees are $2,6,10$, and $(m+1)h-2 = 10m+8 = k \cdot \rrr$.
We hence have
\begin{align}
\mCatplus(H_3;q)=\frac 
{[10m+8]_q\, [10m+4]_q\, [10m]_q} 
{[10]_q\, [6]_q\, [2]_q} .\label{eq2:H3}
\end{align}
Let $\zeta$ be a primitive $\rrr$\th\ root of unity. 
The following cases on the right-hand side of~\eqref{eq2:H3} occur:
\alphaeqn
\begin{align} 
\label{eq2:H3.1}
\lim_{q\to\zeta}\mCatplus(H_3;q)&=\mCatplus(H_3),
&\text{if }\rrr \in\{ 1,2\},\\
\label{eq2:H3.2}
\lim_{q\to\zeta}\mCatplus(H_3;q)&=\tfrac {5m+4}3,
&\text{if }m\equiv1\text{ (mod $3$)},\ \rrr \in \{ 3,6\},\\
\label{eq2:H3.3}
\lim_{q\to\zeta}\mCatplus(H_3;q)&=0,
&\text{otherwise.}
\end{align}
\reseteqn

We may assume that $\rrr$ is even.
In this case $\rrr/2$ divides~$n=3$, implying $\rrr \in \{2,6\}$.

\medskip

The case $\rrr = 2$ is computed using all ascent-descent configurations of the~$50$ factorisations of~$c$, giving, with $m = a$,
\[
  21\binom{a+2}{3} + 28\binom{a+1}{3} + \binom{a}{3} = \mCatplus(H_3).
\]

\medskip

The case $\rrr = 6$ is given by the parameters
\begin{align*}
  y = 2,\quad m' = 1,\quad \ifix = 10.
\end{align*}
There are~$5$ factorisations of the bipartite Coxeter element that satisfy the properties in \Cref{prop:invariantsittingdetails}.
As positions of letters in $\invc$ (which lists the~$15$ reflections in the reflection ordering), these are given by
\[
  \big\{(2, 11, 5), (5, 14, 8), (8, 2, 11), (11, 5, 14), (14, 8, 2) \big\}.
\]
Their ascent-descent sequences imply that the number of $(\Krewplustilde)^k$-invariant subwords as
\[
  3\binom{a+1}{1} + 2\binom{a}{1} = 5a+3,
\]
With $m' = 1$, we obtain $m = 3a+1$, which, by \Cref{eq2:H3.2}, leads to
\[
  \tfrac{1}{3}(5m+4) = \tfrac{1}{3}(15a+9) = 5a+3,
\]
as desired.

\medskip

\paragraph{\sc\bf Case $H_4$}
The degrees are $2,12,20,30$, and $(m+1)h-2 = 30m+28 = k \cdot \rrr$. We hence have
\begin{align}
\mCatplus(H_4;q)=\frac 
{[30m+28]_q\, [30m+18]_q\, [30m+10]_q\, [30m]_q} 
{[30]_q\, [20]_q\, [12]_q\, [2]_q} .\label{eq2:H4}
\end{align}
Let $\zeta$ be a primitive $\rrr$\th\ root of unity. 
The following cases on the right-hand side of~\eqref{eq2:H4} occur:
\alphaeqn
\begin{align} 
\label{eq2:H4.1}
\lim_{q\to\zeta}\mCatplus(H_4;q)&=\mCatplus(H_4),
&\text{if }\rrr \in \{1,2\},\\
\label{eq2:H4.2}
\lim_{q\to\zeta}\mCatplus(H_4;q)&=\tfrac {(15m+14)m}4,
&\text{if } m\equiv0\text{ (mod $2$)},\ \rrr = 4,\\
\label{eq2:H4.3}
\lim_{q\to\zeta}\mCatplus(H_4;q)&=0,
&\text{otherwise.}
\end{align}
\reseteqn

We may assume that $\rrr$ is even.
In this case $\rrr/2$ divides~$n=4$, implying $\rrr \in \{2,4,8\}$.

\medskip

The case $\rrr = 2$ is computed using all ascent-descent configurations of the~$1150$ factorisations of~$c$, giving, with $m = a$,
\[
  232\binom{a+3}{4} + 842\binom{a+2}{4} + 275\binom{a+1}{4} + \binom{a}{4} = \mCatplus(H_4).
\]

\medskip

The case $\rrr = 4$ is given by the parameters
\begin{align*}
  y = 2,\quad m' = 0,\quad \ifix =  29.
\end{align*}
There are $30$ factorisations of the bipartite Coxeter element that satisfy the properties in \Cref{prop:invariantsittingdetails}.
As positions of letters in $\invc$ (which lists the~$60$ reflections in the reflection ordering), these are given by
\[
  \left\{\begin{array}{ccccc}
  (4, 44, 32, 12), & (4, 52, 32, 20), & (8, 48, 36, 16), & (8, 56, 36, 24), & (12, 52, 40, 20), \\
  (12, 60, 40, 28), & (16, 4, 44, 32), & (16, 56, 44, 24), & (20, 8, 48, 36), & (20, 60, 48, 28), \\
  (24, 4, 52, 32), & (24, 12, 52, 40), & (28, 8, 56, 36), & (28, 16, 56, 44), & (32, 12, 60, 40), \\
  (32, 20, 60, 48), & (36, 16, 4, 44), & (36, 24, 4, 52), & (40, 20, 8, 48), & (40, 28, 8, 56), \\
  (44, 24, 12, 52), & (44, 32, 12, 60), & (48, 28, 16, 56), & (48, 36, 16, 4), & (52, 32, 20, 60),\\
  (52, 40, 20, 8), & (56, 36, 24, 4), & (56, 44, 24, 12), & (60, 40, 28, 8), & (60, 48, 28, 16)
  \end{array}
  \right\}
\]
Their ascent-descent sequences imply that the number of $(\Krewplus)^k$-invariant subwords as
\[
  22\binom{a+1}{2} + 8\binom{a}{2} = (15a+7)a,
\]
With $m' = 0$, we obtain $m = 2a$, which, by \Cref{eq2:H4.2}, leads to
\[
  \tfrac{1}{4}(15m+14)m = \tfrac{1}{4}(30a+14)2a = (15a+7)a,
\]
as desired.

\medskip

The case $\rrr = 8$ is given by the parameters
\[
  y = 2,\quad m' = 2,\quad \ifix =  45.
\]
In this case, there are no factorisations of the bipartite Coxeter element that satisfy the properties in \Cref{prop:invariantsittingdetails}, in agreement with \Cref{eq2:H4.3}.

\medskip

\paragraph{\sc\bf Case $F_4$}
The degrees are $2,6,8,12$, and $(m+1)h-2 = 12m+10 = k \cdot \rrr$. We hence have
\begin{align}
\mCatplus(F_4;q)=\frac 
{[12m+10]_q\, [12m+6]_q\, [12m+4]_q\, [12m]_q} 
{[12]_q\, [8]_q\, [6]_q\, [2]_q} .\label{eq2:F4}
\end{align}
Let $\zeta$ be a primitive $\rrr$\th\ root of unity.
The following cases on the right-hand side of~\eqref{eq2:F4} occur:
\alphaeqn
\begin{align} 
\label{eq2:F4.1}
\lim_{q\to\zeta}\mCatplus(F_4;q)&=\mCatplus(F_4),
&\text{if }\rrr \in \{1,2\},\\
\label{eq2:F4.2}
\lim_{q\to\zeta}\mCatplus(F_4;q)&=0,
&\text{otherwise.}
\end{align}
\reseteqn

We may assume that $\rrr$ is even.
In this case $\rrr/2$ divides~$n=4$, implying $\rrr \in \{2,4,8\}$.

\medskip

The case $\rrr = 2$ is computed using all ascent-descent configurations of the $432$ factorisations of~$c$, giving, with $m = a$,
\[
  66\binom{a+3}{4} + 265\binom{a+2}{4} + 100\binom{a+1}{4} + \binom{a}{4} = \mCatplus(F_4).
\]

\medskip

The cases $\rrr \in \{4,8\}$ reduce to \Cref{lem:numbercrunching}(1).

\medskip

\paragraph{\sc\bf Case $E_6$}
The degrees are $2,5,6,8,9,12$, and $(m+1)h-2 = 12m+10 = k \cdot \rrr$. We hence have
\begin{align}
\mCatplus(E_6;q)=\frac 
{[12m+10]_q\, [12m+7]_q\, [12m+6]_q\, [12m+4]_q\, [12m+3]_q\, [12m]_q} 
{[12]_q\, [9]_q\, [8]_q\, [6]_q\, [5]_q\, [2]_q} .\label{eq2:E6}
\end{align}
Let $\zeta$ be a primitive $\rrr$\th\ root of unity.
The following cases on the right-hand side of~\eqref{eq2:E6} occur:
\alphaeqn
\begin{align} 
\label{eq2:E6.1}
\lim_{q\to\zeta}\mCatplus(E_6;q)&=\mCatplus(E_6),
&\text{if }\rrr=1,\\
\label{eq2:E6.2}
\lim_{q\to\zeta}\mCatplus(E_6;q)&=\tfrac {(6m+5)(3m+1)(2m+1)m}2,
&\text{if }\rrr= 2,\\
\label{eq2:E6.4}
\lim_{q\to\zeta}\mCatplus(E_6;q)&=0,
&\text{otherwise.}
\end{align}
\reseteqn

In this case $\rrr$ divides~$n=6$, implying $\rrr \in \{1,2,3,6\}$.

\medskip

The case $\rrr = 1$ is computed using all ascent-descent configurations of the~$41{,}472$ factorisations of~$c$, giving, with $m = a$,
\begin{multline*}
  418\binom{a+5}{6} + 8141\binom{a+4}{6} + 21308\binom{a+3}{6} + \\
  10778\binom{a+2}{6} + 826\binom{a+1}{6} + \binom{a}{6} = \mCatplus(E_6).
\end{multline*}

\medskip

The case $\rrr = 2$ is given by the parameters
\begin{align*}
  y = 2,\quad m' = 0,\quad \ifix =  31.
\end{align*}
There are $432$ factorisations of the bipartite Coxeter element that satisfy the properties in \Cref{prop:invariantsittingdetails}.
Their ascent-descent sequences imply that the number of $(\Krewplus)^k$-invariant subwords as
\[
  66\binom{a+3}{4} + 265\binom{a+2}{4} + 100\binom{a+1}{4} + 1\binom{a}{4} = \tfrac{1}{2}(6a + 5)(3a + 1)(2a + 1)a.
\]
With $m' = 0$, we obtain $m = a$, which yields \Cref{eq2:E6.2}, as desired.

\medskip

The cases $\rrr \in \{3,6\}$ reduce to \Cref{lem:numbercrunching}(1).

\medskip

\paragraph{\sc\bf Case $E_7$}

The degrees are $2,6,8,10,12,14,18$, and $(m+1)h-2 = 18m+16 = k \cdot \rrr$. We hence have
\begin{multline}
\mCatplus(E_7;q)=\frac 
{[18m+16]_q\, [18m+12]_q\, [18m+10]_q} 
{[18]_q\, [14]_q\, [12]_q}\\
\times
\frac 
{[18m+8]_q\, [18m+6]_q\, [18m+4]_q\, [18m]_q} 
{[10]_q\, [8]_q\, [6]_q\, [2]_q} .\label{eq2:E7}
\end{multline}
Let $\zeta$ be a primitive $\rrr$\th\ root of unity.
The following cases on the right-hand side of~\eqref{eq2:E7} occur:
\alphaeqn
\begin{align} 
\label{eq2:E7.1}
\lim_{q\to\zeta}\mCatplus(E_7;q)&=\mCatplus(E_7),
&\text{if }\rrr \in \{ 1,2 \},\\
\label{eq2:E7.2}
\lim_{q\to\zeta}\mCatplus(E_7;q)&=\tfrac {9m+8}7,
&\text{if }  m\equiv3\text{ (mod $7$)},\ \rrr\in \{7,14\},\\
\label{eq2:E7.3}
\lim_{q\to\zeta}\mCatplus(E_7;q)&=0,
&\text{otherwise.}
\end{align}
\reseteqn

We may assume that $\rrr$ is even.
In this case $\rrr/2$ divides~$n=7$, implying $\rrr \in \{2,14\}$.

\medskip

The case $\rrr = 2$ is computed using all ascent-descent configurations of the $1{,}062{,}882$ factorisations of~$c$, giving, with $m = a$,
\begin{multline*}
  2431\binom{a+6}{7} + 85800\binom{a+5}{7} + 412764\binom{a+4}{7} + 446776\binom{a+3}{7} \\
  + 110958\binom{a+2}{7} + 4152\binom{a+1}{7} + \binom{a}{7} = \mCatplus(E_7).
\end{multline*}

\medskip

The case $\rrr = 14$ is given by the parameters
\begin{align*}
  y = 2,\quad m' = 3,\quad \ifix =  36.
\end{align*}
There are $9$ factorisations of the bipartite Coxeter element that satisfy the properties in \Cref{prop:invariantsittingdetails}.
As positions of letters in $\invc$ (which lists the~$63$ reflections in the reflection ordering), these are given by
\[
  \left\{\begin{array}{ccc}
 (7, 42, 14, 49, 21, 56, 28), &
 (14, 49, 21, 56, 28, 63, 35),&
 (21, 56, 28, 63, 35, 7, 42),\\
 (28, 63, 35, 7, 42, 14, 49),&
 (35, 7, 42, 14, 49, 21, 56),&
 (42, 14, 49, 21, 56, 28, 63),\\
 (49, 21, 56, 28, 63, 35, 7),&
 (56, 28, 63, 35, 7, 42, 14),&
 (63, 35, 7, 42, 14, 49, 21)
  \end{array}
  \right\}
\]
Their ascent-descent sequences imply that the number of $(\Krewplus)^k$-invariant subwords as
\[
  5\binom{a+1}{1} + 4\binom{a}{1} = 9a+5,
\]
With $m' = 3$, we obtain $m = 7a+3$, which, by \Cref{eq2:E6.2}, leads to
\[
  \tfrac{1}{7}(9m+8) = \tfrac{1}{7}(63a+35) = 9a+5,
\]
as desired.

\medskip

\paragraph{\sc\bf Case $E_8$}
The degrees are $2,8,12,14,18,20,24,30$, and $(m+1)h-2 = 30m+28 = k
\cdot \rrr$. We hence have 
\begin{multline}
\mCatplus(E_8;q)=\frac
{[30m+28]_q\, [30m+22]_q\, [30m+18]_q\, [30m+16]_q} 
{[30]_q\, [24]_q\, [20]_q\, [18]_q}\\
\times
\frac 
{[30m+12]_q\, [30m+10]_q\, [30m+6]_q\, [30m]_q} 
{[14]_q\, [12]_q\, [8]_q\, [2]_q} .\label{eq2:E8}
\end{multline}
Let $\zeta$ be a primitive $\rrr$\th\ root of unity.
The following cases on the right-hand side of~\eqref{eq2:E8} occur:
\alphaeqn
\begin{align} 
\label{eq2:E8.1}
\lim_{q\to\zeta}\mCatplus(E_8;q)&=\mCatplus(E_8),
&\text{if }\rrr \in \{1,2\},\\
\label{eq2:E8.2}
\lim_{q\to\zeta}\mCatplus(E_8;q)&=\tfrac {(15m+14)(15m+8)(5m+2)m}{64},
&\text{if }m\equiv0\text{ (mod $2$)},\ \rrr=4,\\
\label{eq2:E8.3}
\lim_{q\to\zeta}\mCatplus(E_8;q)&=\tfrac{(15m+14)(5m+2)}{16},
&\text{if }m\equiv2\text{ (mod $4$)},\ \rrr=8,\\
\label{eq2:E8.4}
\lim_{q\to\zeta}\mCatplus(E_8;q)&=0,
&\text{otherwise.}
\end{align}
\reseteqn

We may assume that $\rrr$ is even.
In this case $\rrr/2$ divides~$n=8$, implying $\rrr \in \{2,4,8,16\}$.

\medskip

The case $\rrr = 2$ is computed using all ascent-descent configurations of the $37{,}968{,}750$ factorisations of~$c$ (which we omit here).

\medskip

The case $\rrr = 4$ is given by the parameters
\begin{align*}
  y = 2,\quad m' = 0,\quad \ifix =  57.
\end{align*}
There are $6{,}750$ factorisations of the bipartite Coxeter element that satisfy the properties in \Cref{prop:invariantsittingdetails}.
Their ascent-descent sequences imply that the number of $(\Krewplus)^k$-invariant subwords as
\[
  627\binom{a+3}{4} + 3784\binom{a+2}{4} + 2251\binom{a+1}{4} + 88\binom{a}{4} = \tfrac{1}{4}(15a + 7)(15a + 4)(5a + 1)a.
\]
With $m' = 0$, we obtain $m = 2a$, which, by \Cref{eq2:E8.2}, leads to
\begin{align*}
  \tfrac{1}{64}(15m+14)(15m+8)(5m+2)m &= \tfrac{1}{64}(30a+14)(30a+8)(10a+2)2a \\
                                     &= \tfrac{1}{4} (15a+7)(15a+4)(5a+1)a,
\end{align*}
as desired.

\medskip

The case $\rrr = 8$ is given by the parameters
\begin{align*}
  y = 2,\quad m' = 2,\quad \ifix =  89.
\end{align*}
There are $150$ factorisations of the bipartite Coxeter element that satisfy the properties in \Cref{prop:invariantsittingdetails}.
Their ascent-descent sequences imply that the number of $(\Krewplus)^k$-invariant subwords as
\[
  33\binom{a+2}{4} + 109\binom{a+1}{4} + 8\binom{a}{4} = (15a + 11)(5a + 3).
\]
With $m' = 2$, we obtain $m = 4a+2$, which, by \Cref{eq2:E8.3}, leads to
\[
  \tfrac{1}{16}(15m+14)(5m+2) = \tfrac{1}{16}(60a+44)(20a+12) = (15a+11)(5a+3),
\]
as desired.

\medskip

The case $\rrr = 16$ is given by the parameters
\[
  y = 2,\quad m' = 6,\quad \ifix =  105.
\]
In this case, there are no factorisations of the bipartite Coxeter element that satisfy the properties in \Cref{prop:invariantsittingdetails}, in agreement with \Cref{eq2:E8.4}.

\subsubsection{Cyclic sieving for~$\Krewplustilde$}
\label{sec:siev12}

In this situation, we treat factorisations $\r_1\ldots\r_{n'}$ that can be extended to factorisations of the bipartite Coxeter element $c = c_Lc_R$ and that are invariant under appropriate~$k$\th\ powers of~$\Krewplustilde$ with $(m+1)h-2 = k \cdot \rrr$.
From \Cref{thm:order}, we see that the order of~$\Krewplustilde$ equals $k \cdot \rrr$ if $\psi\equiv\one$ and $k \cdot \rrr/2$ if $\psi\not\equiv\one$.
We can thus assume that $\rrr$ is even if $\psi\not\equiv\one$.
Moreover, we have seen in \Cref{prop:invariantsittingdetails} that we need to consider all parameters~$\rrr$ such that~$\rrr$ (if~$\psi\not\equiv\one$) or, respectively,~$\rrr/2$ (if~$\psi\equiv\one$) divides some~$n' \leq n$, i.e., all parameters~$\rrr$ such that~$\rrr$ or, respectively,~$\rrr/2$ is less than or equal to~$n$.
Otherwise, there will only be the invariant trivial factorisation.

\medskip

\paragraph{\sc\bf Case $H_3$}
The degrees are $2,6,10$, and $mh-2 = 10m-2 = k\cdot\rrr$.
We hence have
\begin{align}
\mCatplus(H_3;q)=\frac 
{[10m+8]_q\, [10m+4]_q\, [10m]_q} 
{[10]_q\, [6]_q\, [2]_q} .\label{eq:H3}
\end{align}
Let $\zeta$ be a primitive $\rrr$\th\ root of unity.
The following cases on the right-hand side of~\eqref{eq:H3} occur:
\alphaeqn
\begin{align} 
\label{eq:H3.1}
\lim_{q\to\zeta}\mCatplus(H_3;q)&=\mCatplus(H_3),
&\text{if } \rrr\in\{1,2\},\\
\label{eq:H3.2}
\lim_{q\to\zeta}\mCatplus(H_3;q)&=\tfrac {5m+2}3,
&\text{if } m\equiv 2\text{ (mod $3$)},\ \rrr\in\{3,6\},\\
\label{eq:H3.3}
\lim_{q\to\zeta}\mCatplus(H_3;q)&=1,
&\text{otherwise.}
\end{align}
\reseteqn

We may assume that $\rrr$ is even.
In this case $\rrr/2 \leq n = 3$, implying $\rrr \in \{2,4,6\}$.

\medskip

The case $\rrr = 2$ is computed using all ascent-descent configurations of the $116$ factorisations that can be extended to factorisations of~$c$, giving, with $m = a+1$,
\begin{multline*}
21\binom{a+2}{3}+28\binom{a+1}{3}+\binom{a}{3}+8\binom{a+2}{2}+39\binom{a+1}{2} \\
+3\binom{a}{2}+12\binom{a+1}{1}+3\binom{a}{1}+\binom{a}{0}
= \mCatplus(H_3).
\end{multline*}

\medskip

The case $\rrr = 4$ is given by the parameters
\[
  y = 2,\quad m' = 1,\quad \ifix =  7.
\]
In this case, there is only the trivial factorisation that satisfies the properties in \Cref{prop:invariantsittingdetails}, in agreement with \Cref{eq:H3.3}.

\medskip

The case $\rrr = 6$ is given by the parameters
\begin{align*}
  y = 2,\quad m' = 2,\quad \ifix =  10.
\end{align*}
There are $6$ factorisations of the bipartite Coxeter element that satisfy the properties in \Cref{prop:invariantsittingdetails}.
As positions of letters in $\invc$ (which lists the~$15$ reflections in the reflection ordering), these are given by the empty factorisation and the $5$ factorisations of~$c$ listed in the above considerations for type~$H_3$.
Their ascent-descent sequences imply that the number of $(\Krewplustilde)^k$-invariant subwords as
\[
  3\binom{a+1}{1}+2\binom{a}{1}+\binom{a}{0}
  = 5a + 4
\]
With $m' = 2$, we obtain $m = 3a+2$, which, by \Cref{eq:H3.2}, leads to
\[
  \tfrac{1}{3}(5m+2) = \tfrac{1}{3}(15a+12) = 5a+4,
\]
as desired.

\medskip

\paragraph{\sc\bf Case $H_4$}
The degrees are $2,12,20,30$, and $mh-2 = 30m-2 = k\cdot\rrr$.
We hence have
\begin{align}
\mCatplus(H_4;q)=\frac 
{[30m+28]_q\, [30m+18]_q\, [30m+10]_q\, [30m]_q} 
{[30]_q\, [20]_q\, [12]_q\, [2]_q} .\label{eq:H4}
\end{align}
Let $\zeta$ be a primitive $\rrr$\th\ root of unity.
The following cases on the right-hand side of~\eqref{eq:H4} occur:
\alphaeqn
\begin{align} 
\label{eq:H4.1}
\lim_{q\to\zeta}\mCatplus(H_4;q)&=\mCatplus(H_4),
&\text{if }\rrr\in\{1,2\},\\
\label{eq:H4.2}
\lim_{q\to\zeta}\mCatplus(H_4;q)&=\tfrac {(5m+3)(3m+1)}4,
&\text{if }m\equiv1\text{ mod $2$},\ \rrr = 4,\\
\label{eq:H4.3}
\lim_{q\to\zeta}\mCatplus(H_4;q)&=1,
&\text{otherwise.}
\end{align}
\reseteqn

We may assume that $\rrr$ is even.
In this case $\rrr/2 \leq n = 4$, implying $\rrr \in \{2,4,6,8\}$.

\medskip

The case $\rrr = 2$ is computed using all ascent-descent configurations of the $3{,}226$ factorisations that can be extended to factorisations of~$c$, giving, with $m = a+1$,
\begin{multline*}
232\binom{a+3}{4}+842\binom{a+2}{4}+275\binom{a+1}{4}+\binom{a}{4}+42\binom{a+3}{3} \\
+760\binom{a+2}{3}+544\binom{a+1}{3}+4\binom{a}{3}+133\binom{a+2}{2}+326\binom{a+1}{2} \\
+6\binom{a}{2}+56\binom{a+1}{1}+4\binom{a}{1}+\binom{a}{0}
= \mCatplus(H_4).
\end{multline*}

\medskip

The case $\rrr = 4$ is given by the parameters
\begin{align*}
  y = 2,\quad m' = 1,\quad \ifix =  29.
\end{align*}
There are $46$ factorisations of the bipartite Coxeter element that satisfy the properties in \Cref{prop:invariantsittingdetails}.
As positions of letters in $\invc$ (which lists the~$60$ reflections in the reflection ordering), these are given by the empty factorisation, the $30$ factorisations of~$c$ listed in the above considerations for type~$H_4$, and
\[
  \left\{\begin{array}{ccccc}
    (4, 32), & (8, 36), & (12, 40), & (16, 44), & (20, 48), \\
    (24, 52), & (28, 56), & (32, 60), & (36, 4), & (40, 8), \\
    (44, 12), & (48, 16), & (52, 20), & (56, 24), & (60, 28)
  \end{array}
  \right\}
\]
Their ascent-descent sequences imply that the number of $(\Krewplustilde)^k$-invariant subwords as
\[
  22\binom{a+1}{2}+8\binom{a}{2}+7\binom{a+1}{1}+8\binom{a}{1}+\binom{a}{0}
  = (5a + 4)(3a + 2)
\]
With $m' = 1$, we obtain $m = 2a+1$, which, by \Cref{eq:H4.2}, leads to
\[
  \tfrac{1}{4}(5m+3)(3m+1) = \tfrac{1}{4}(10a+8)(6a+4) = (5a+4)(3a+2),
\]
as desired.

\medskip

The case $\rrr = 6$ reduces to \Cref{lem:numbercrunching}(1).

\medskip

The case $\rrr = 8$ is given by the parameters
\[
  y = 2,\quad m' = 3,\quad \ifix =  45.
\]
In this case, there is only the trivial factorisation that satisfies the properties in \Cref{prop:invariantsittingdetails}, in agreement with \Cref{eq:H4.3}.

\medskip

\paragraph{\sc\bf Case $F_4$}
The degrees are $2,6,8,12$, and $mh-2 = 12m-2 = k\cdot\rrr$.
We hence have
\begin{align}
\mCatplus(F_4;q)=\frac 
{[12m+10]_q\, [12m+6]_q\, [12m+4]_q\, [12m]_q} 
{[12]_q\, [8]_q\, [6]_q\, [2]_q} .\label{eq:F4}
\end{align}
Let~$\zeta$ be a primitive $\rrr$\th\ root of unity.
The following cases on the right-hand side of~\eqref{eq:F4} occur:
\alphaeqn
\begin{align} 
\label{eq:F4.1}
\lim_{q\to\zeta}\mCatplus(F_4;q)&=\mCatplus(F_4),
&\text{if }\rrr\in\{1,2\},\\
\label{eq:F4.2}
\lim_{q\to\zeta}\mCatplus(F_4;q)&=1,
&\text{otherwise.}
\end{align}
\reseteqn

We may assume that $\rrr$ is even.
In this case $\rrr/2 \leq n = 4$, implying $\rrr \in \{2,4,6,8\}$.

\medskip

The case $\rrr = 2$ is computed using all ascent-descent configurations of the $1{,}045$ factorisations that can be extended to factorisations of~$c$, giving, with $m = a+1$,
\begin{multline*}
66\binom{a+3}{4}+265\binom{a+2}{4}+100\binom{a+1}{4}+\binom{a}{4}+10\binom{a+3}{3} \\
+224\binom{a+2}{3}+194\binom{a+1}{3}+4\binom{a}{3}+35\binom{a+2}{2}+115\binom{a+1}{2} \\
+6\binom{a}{2}+20\binom{a+1}{1}+4\binom{a}{1}+\binom{a}{0}
= \mCatplus(F_4).
\end{multline*}

\medskip

The cases $\rrr \in \{4,6,8\}$ reduce to \Cref{lem:numbercrunching}(1).

\medskip

\paragraph{\sc\bf Case $E_6$}
The degrees are $2,5,6,8,9,12$, and $mh-2 = 12m-2 = k\cdot\rrr$.
We hence have
\begin{align}
\mCatplus(E_6;q)=\frac 
{[12m+10]_q\, [12m+7]_q\, [12m+6]_q\, [12m+4]_q\, [12m+3]_q\, [12m]_q} 
{[12]_q\, [9]_q\, [8]_q\, [6]_q\, [5]_q\, [2]_q} .\label{eq:E6}
\end{align}
Let~$\zeta$ be a primitive $\rrr$\th\ root of unity.
The following cases on the right-hand side of~\eqref{eq:E6} occur:
\alphaeqn
\begin{align} 
\label{eq:E6.1}
\lim_{q\to\zeta}\mCatplus(E_6;q)&=\mCatplus(E_6),
&\text{if }\rrr=1,\\
\label{eq:E6.2}
\lim_{q\to\zeta}\mCatplus(E_6;q)&=\tfrac {(6m+5)(3m+1)(2m+1)m}{2},
&\text{if }\rrr=2,\\
\label{eq:E6.3}
\lim_{q\to\zeta}\mCatplus(E_6;q)&=\tfrac {12m+3}5,
&\text{if }m\equiv1\text{ (mod $5$)},\ \rrr=5,\\
\label{eq:E6.4}
\lim_{q\to\zeta}\mCatplus(E_6;q)&=1,
&\text{otherwise.}
\end{align}
\reseteqn

In this case $\rrr \leq n = 6$, implying $\rrr \in \{1,2,3,4,5,6\}$.

\medskip

The case $\rrr = 1$ is computed using all ascent-descent configurations of the $104{,}653$ factorisations that can be extended to factorisations of~$c$, giving, with $m = a+1$,
\begin{multline*}
418\binom{a+5}{6}+814\binom{a+4}{6}+21308\binom{a+3}{6}+10778\binom{a+2}{6} \\
+826\binom{a+1}{6}+\binom{a}{6}+7\binom{a+5}{5}+2466\binom{a+4}{5}+18258\binom{a+3}{5} \\
+18272\binom{a+2}{5}+2463\binom{a+1}{5}+6\binom{a}{5}+70\binom{a+4}{4}+4039\binom{a+3}{4} \\
+10375\binom{a+2}{4}+278\binom{a+1}{4}+15\binom{a}{4}+175\binom{a+3}{3}+2226\binom{a+2}{3} \\
+1467\binom{a+1}{3}+20\binom{a}{3}+135\binom{a+2}{2}+354\binom{a+1}{2}+15\binom{a}{2} \\
+30\binom{a+1}{1}+6\binom{a}{1}+\binom{a}{0}
= \mCatplus(E_7).
\end{multline*}

\medskip

The case $\rrr = 2$ is given by the parameters
\begin{align*}
  y = 2,\quad m' = 1,\quad \ifix =  31.
\end{align*}
There are $1045$ factorisations of the bipartite Coxeter element that satisfy the properties in \Cref{prop:invariantsittingdetails}.
Their ascent-descent sequences imply that the number of $(\Krewplustilde)^k$-invariant subwords as
\begin{multline*}
  66\binom{a+3}{4}+265\binom{a+2}{4}+100\binom{a+1}{4}+\binom{a}{4}+10\binom{a+3}{3} \\
  +224\binom{a+2}{3}+194\binom{a+1}{3}+4\binom{a}{3}+35\binom{a+2}{2}+115\binom{a+1}{2} \\
  +6\binom{a}{2}+20\binom{a+1}{1}+4\binom{a}{1}+\binom{a}{0}
  = \frac{(6a + 11)(3a + 4)(2a + 3)(a + 1)}{2}.
\end{multline*}
With $m' = 1$, we obtain $m = a+1$, which, by \Cref{eq:E6.2}, leads to
\begin{align*}
  \tfrac{1}{2}(6m+5)(3m+1)(2m+1)m &= \tfrac{1}{2}(6a+11)(3a+4)(2a+3)(a+1),
\end{align*}
as desired.

\medskip

The cases $\rrr \in \{3,4\}$ reduce to \Cref{lem:numbercrunching}(1).

\medskip

The case $\rrr = 5$ is given by the parameters
\begin{align*}
  y = 1,\quad m' = 1,\quad \ifix =  7.
\end{align*}
There are $13$ factorisations of the bipartite Coxeter element that satisfy the properties in \Cref{prop:invariantsittingdetails}.
As positions of letters in $\invc$ (which lists the~$36$ reflections in the reflection ordering), these are given by the empty factorisation and
\[
  \left\{\begin{array}{cccc}
 (1, 8, 13, 20, 25), &
 (2, 7, 14, 19, 26), &
 (7, 14, 19, 26, 31), &
 (8, 13, 20, 25, 32), \\
 (13, 20, 25, 32, 2), &
 (14, 19, 26, 31, 1), &
 (19, 26, 31, 1, 8), &
 (20, 25, 32, 2, 7), \\
 (25, 32, 2, 7, 14), &
 (26, 31, 1, 8, 13), &
 (31, 1, 8, 13, 20), &
 (32, 2, 7, 14, 19)
  \end{array}
  \right\}
\]
Their ascent-descent sequences imply that the number of $(\Krewplustilde)^k$-invariant subwords as
\[
  2\binom{a+1}{1}+10\binom{a}{1}+\binom{a}{0} = 12a+3.
\]
With $m' = 1$, we obtain $m = 5a+1$, which, by \Cref{eq:E6.3}, leads to
\[
  \tfrac{1}{5}(12m+3) = \tfrac{1}{5}(60a+15) = 12a+3,
\]
as desired.

\medskip

The case $\rrr = 6$ reduces to \Cref{lem:numbercrunching}(1).

\medskip

\paragraph{\sc\bf Case $E_7$}
The degrees are $2,6,8,10,12,14,18$, and $mh-2 = 18m-2 = k\cdot\rrr$.
We hence have
\begin{multline}
\mCatplus(E_7;q)=\frac 
{[18m+16]_q\, [18m+12]_q\, [18m+10]_q} 
{[18]_q\, [14]_q\, [12]_q}\\
\times
\frac 
{[18m+8]_q\, [18m+6]_q\, [18m+4]_q\, [18m]_q} 
{[10]_q\, [8]_q\, [6]_q\, [2]_q} .\label{eq:E7}
\end{multline}
Let~$\zeta$ be a primitive $\rrr$\th\ root of unity.
The following cases on the right-hand side of~\eqref{eq:E7} occur:
\alphaeqn
\begin{align} 
\label{eq:E7.1}
\lim_{q\to\zeta}\mCatplus(E_7;q)&=\mCatplus(E_7),
&\text{if }\rrr\in\{1,2\},\\
\label{eq:E7.2}
\lim_{q\to\zeta}\mCatplus(E_7;q)&=\tfrac{(9m+5)(3m+1)}8,
&\text{if }m\equiv1\text{ (mod $2$)},\ \rrr=4,\\
\label{eq:E7.3}
\lim_{q\to\zeta}\mCatplus(E_7;q)&=\tfrac {9m+3}4,
&\text{if }m\equiv1\text{ (mod $4$)},\ \rrr=8,\\
\label{eq:E7.4}
\lim_{q\to\zeta}\mCatplus(E_7;q)&=\tfrac {9m+4}5,
&\text{if }m\equiv4\text{ (mod $5$)},\ \rrr\in\{5,10\},\\
\label{eq:E7.5}
\lim_{q\to\zeta}\mCatplus(E_7;q)&=\tfrac {9m+6}7,
&\text{if }m\equiv4\text{ (mod $7$)},\ \rrr\in\{7,14\},\\
\label{eq:E7.6}
\lim_{q\to\zeta}\mCatplus(E_7;q)&=1,
&\text{otherwise.}
\end{align}
\reseteqn

We may assume that $\rrr$ is even.
In this case $\rrr/2 \leq n = 7$, implying $\rrr \in \{2,4,6,8,10,12,\break 14\}$.

\medskip

The case $\rrr = 2$ is computed using all ascent-descent configurations of the $2{,}702{,}332$ factorisations that can be extended to factorisations of~$c$ (which we omit here).

\medskip

The case $\rrr = 4$ is given by the parameters
\begin{align*}
  y = 2,\quad m' = 1,\quad \ifix =  29.
\end{align*}
There are $37$ factorisations of the bipartite Coxeter element that satisfy the properties in \Cref{prop:invariantsittingdetails}.
As positions of letters in $\invc$ (which lists the~$63$ reflections in the reflection ordering), these are given by the empty factorisation and
\[
  \left\{\begin{array}{ccccc}
 \multicolumn{5}{c}{(1, 29),
 (8, 36),
 (15, 43),
 (22, 50),
 (29, 57),
 (36, 1),
 (43, 8),
 (50, 15),
 (57, 22), }\\
 (1, 15, 29, 43), &
 (1, 43, 29, 8), &
 (1, 50, 29, 15), &
 (8, 22, 36, 50), &
 (8, 50, 36, 15), \\
 (8, 57, 36, 22), &
 (15, 1, 43, 29), &
 (15, 29, 43, 57), &
 (15, 57, 43, 22), &
 (22, 1, 50, 29), \\
 (22, 8, 50, 36), &
 (22, 36, 50, 1), &
 (29, 8, 57, 36), &
 (29, 15, 57, 43), &
 (29, 43, 57, 8), \\
 (36, 15, 1, 43), &
 (36, 22, 1, 50), &
 (36, 50, 1, 15), &
 (43, 22, 8, 50), &
 (43, 29, 8, 57), \\
 (43, 57, 8, 22), &
 (50, 1, 15, 29), &
 (50, 29, 15, 57), &
 (50, 36, 15, 1), &
 (57, 8, 22, 36), \\
 \multicolumn{5}{c}{(57, 36, 22, 1),
 (57, 43, 22, 8)}
 \end{array}
  \right\}
\]
Their ascent-descent sequences imply that the number of $(\Krewplustilde)^k$-invariant subwords as
\[
  2\binom{a+2}{2}+20\binom{a+1}{2}+5\binom{a}{2}+4\binom{a+1}{1}+5\binom{a}{1}+\binom{a}{0} = \frac{(9a + 7)(3a + 2)}{2}.
\]
With $m' = 1$, we obtain $m = 2a+1$, which, by \Cref{eq:E7.2}, leads to
\[
  \tfrac{1}{8}(9m+5)(3m+1) = \tfrac{1}{8}(18a+14)(6a+4) = \tfrac{1}{2}(9a+7)(3a+2),
\]
as desired.

\medskip

The case $\rrr = 6$ reduce to \Cref{lem:numbercrunching}(1).

\medskip

The case $\rrr = 8$ is given by the parameters
\begin{align*}
  y = 2,\quad m' = 1,\quad \ifix =  15.
\end{align*}
There are $10$ factorisations of the bipartite Coxeter element that satisfy the properties in \Cref{prop:invariantsittingdetails}.
As positions of letters in $\invc$ (which lists the~$63$ reflections in the reflection ordering), these are given by the empty factorisation and
\[
  \left\{\begin{array}{ccc}
(1, 15, 29, 43), &
 (8, 22, 36, 50), &
 (15, 29, 43, 57), \\
 (22, 36, 50, 1), &
 (29, 43, 57, 8), &
 (36, 50, 1, 15), \\
 (43, 57, 8, 22), &
 (50, 1, 15, 29), &
 (57, 8, 22, 36)
 \end{array}
  \right\}
\]
Their ascent-descent sequences imply that the number of $(\Krewplustilde)^k$-invariant subwords as
\[
  2\binom{a+1}{1}+7\binom{a}{1}+\binom{a}{0} = 9a+3.
\]
With $m' = 1$, we obtain $m = 4a+1$, which, by \Cref{eq:E7.3}, leads to
\[
  \tfrac{1}{4}(9m+3) = \tfrac{1}{4}(36a+12) = 9a+3,
\]
as desired.

\medskip

The case $\rrr = 10$ is given by the parameters
\begin{align*}
  y = 2,\quad m' = 4,\quad \ifix =  50.
\end{align*}
There are $10$ factorisations of the bipartite Coxeter element that satisfy the properties in \Cref{prop:invariantsittingdetails}.
As positions of letters in $\invc$ (which lists the~$63$ reflections in the reflection ordering), these are given by the empty factorisation and
\[
  \left\{\begin{array}{ccc}
 (7, 56, 42, 28, 14), &
 (14, 63, 49, 35, 21), &
 (21, 7, 56, 42, 28), \\
 (28, 14, 63, 49, 35), &
 (35, 21, 7, 56, 42), &
 (42, 28, 14, 63, 49), \\
 (49, 35, 21, 7, 56), &
 (56, 42, 28, 14, 63), &
 (63, 49, 35, 21, 7)
 \end{array}
  \right\}
\]
Their ascent-descent sequences imply that the number of $(\Krewplustilde)^k$-invariant subwords as
\[
  7\binom{a+1}{1}+2\binom{a}{1}+\binom{a}{0} = 9a+8.
\]
With $m' = 4$, we obtain $m = 5a+4$, which, by \Cref{eq:E7.4}, leads to
\[
  \tfrac{1}{5}(9m+4) = \tfrac{1}{5}(45a+40) = 9a+8,
\]
as desired.

\medskip

The cases $\rrr = 12$ reduces to \Cref{lem:numbercrunching}(1).

\medskip

The case $\rrr = 14$ is given by the parameters
\begin{align*}
  y = 2,\quad m' = 4,\quad \ifix =  36.
\end{align*}
There are $10$ factorisations of the bipartite Coxeter element that satisfy the properties in \Cref{prop:invariantsittingdetails}.
As positions of letters in $\invc$ (which lists the~$63$ reflections in the reflection ordering), these are given by the empty factorisation and
\[
  \left\{\begin{array}{ccc}
 (7, 42, 14, 49, 21, 56, 28), &
 (14, 49, 21, 56, 28, 63, 35), &
 (21, 56, 28, 63, 35, 7, 42), \\
 (28, 63, 35, 7, 42, 14, 49), &
 (35, 7, 42, 14, 49, 21, 56), &
 (42, 14, 49, 21, 56, 28, 63), \\
 (49, 21, 56, 28, 63, 35, 7), &
 (56, 28, 63, 35, 7, 42, 14), &
 (63, 35, 7, 42, 14, 49, 21)
 \end{array}
  \right\}
\]
Their ascent-descent sequences imply that the number of $(\Krewplustilde)^k$-invariant subwords as
\[
  5\binom{a+1}{1}+4\binom{a}{1}+\binom{a}{0} = 9a + 6.
\]
With $m' = 4$, we obtain $m = 7a+4$, which, by \Cref{eq:E7.5}, leads to
\[
  \tfrac{1}{7}(9m+6) = \tfrac{1}{7}(63a+42) = 9a+6,
\]
as desired.

\medskip

\paragraph{\sc\bf Case $E_8$}
The degrees are $2,8,12,14,18,20,24,30$, and $mh-2 = 30m-2 = k\cdot\rrr$.
We hence have
\begin{multline}
\mCatplus(E_8;q)=\frac 
{[30m+28]_q\, [30m+22]_q\, [30m+18]_q\, [30m+16]_q} 
{[30]_q\, [24]_q\, [20]_q\, [18]_q}\\
\times
\frac 
{[30m+12]_q\, [30m+10]_q\, [30m+6]_q\, [30m]_q} 
{[14]_q\, [12]_q\, [8]_q\, [2]_q} . \label{eq:E8}
\end{multline}
Let~$\zeta$ be a primitive $\rrr$\th\ root of unity.
The following cases on the right-hand side of~\eqref{eq:E8} occur:
\alphaeqn
\begin{align} 
\label{eq:E8.1}
\lim_{q\to\zeta}\mCatplus(E_8;q)&=\mCatplus(E_8),
&\text{if }\rrr\in\{1,2\},\\
\label{eq:E8.2}
\lim_{q\to\zeta}\mCatplus(E_8;q)&=\tfrac {(15m+11)(5m+3)(5m+1)(3m+1)}{64},
&\text{if } m\equiv1\text{ (mod $2$)},\ \rrr=4,\\
\label{eq:E8.3}
\lim_{q\to\zeta}\mCatplus(E_8;q)&=\tfrac{(15m+11)(5m+1)}{16},
&\text{if }m\equiv3\text{ (mod $4$)},\ \rrr=8,\\
\label{eq:E8.4}
\lim_{q\to\zeta}\mCatplus(E_8;q)&=\tfrac {15m+6}7,
&\text{if }m\equiv1\text{ (mod $7$)},\ \rrr\in\{7,14\},\\
\label{eq:E8.5}
\lim_{q\to\zeta}\mCatplus(E_8;q)&=1,
&\text{otherwise.}
\end{align}
\reseteqn

We may assume that $\rrr$ is even.
In this case $\rrr/2 \leq n = 8$, implying $\rrr \in \{2,4,6,8,10,12,\break 14,16\}$.

\medskip

The case $\rrr = 2$ is computed using all ascent-descent configurations of the $97{,}126{,}171$ factorisations that can be extended to factorisations of~$c$ (which we omit here).

\medskip

The case $\rrr = 4$ is given by the parameters
\begin{align*}
  y = 2,\quad m' = 1,\quad \ifix =  57.
\end{align*}
There are $10{,}771$ factorisations of the bipartite Coxeter element that satisfy the properties in \Cref{prop:invariantsittingdetails}.
Their ascent-descent sequences imply that the number of $(\Krewplustilde)^k$-invariant subwords as
\begin{multline*}
  627\binom{a+3}{4}+3784\binom{a+2}{4}+2251\binom{a+1}{4}+88\binom{a+0}{4} \\
  +14\binom{a+3}{3}+1198\binom{a+2}{3}+1987\binom{a+1}{3}+176\binom{a+0}{3} \\
  +42\binom{a+2}{2}+446\binom{a+1}{2}+112\binom{a+0}{2}+21\binom{a+1}{1} \\
  +24\binom{a+0}{1}+1\binom{a+0}{0} = \tfrac{1}{4}(15a + 13)(5a + 4)(5a + 3)(3a + 2).
\end{multline*}
With $m' = 1$, we obtain $m = 2a+1$, which, by \Cref{eq:E8.2}, leads to
\begin{align*}
  \tfrac{1}{64}(15m+11)(5m+3)(5m+1)(3m+1) &= \tfrac{1}{64}(30a+26)(10a+8)(10a+6)(6a+4) \\
                                          &= \tfrac{1}{4} (15a+13)(5a+4)(5a+3)(3a+2),
\end{align*}
as desired.

\medskip

The case $\rrr = 6$ reduce to \Cref{lem:numbercrunching}(1).

\medskip

The case $\rrr = 8$ is given by the parameters
\begin{align*}
  y = 2,\quad m' = 3,\quad \ifix =  89.
\end{align*}
There are $181$ factorisations of the bipartite Coxeter element that satisfy the properties in \Cref{prop:invariantsittingdetails}.
Their ascent-descent sequences imply that the number of $(\Krewplustilde)^k$-invariant subwords as
\begin{multline*}
  33\binom{a+2}{2}+109\binom{a+1}{2}+8\binom{a+0}{2}+22\binom{a+1}{1}\\
  +8\binom{a+0}{1}+1\binom{a+0}{0} = (15a + 14)(5a + 4).
\end{multline*}
With $m' = 3$, we obtain $m = 4a+3$, which, by \Cref{eq:E8.3}, leads to
\[
  \tfrac{1}{16}(15m+11)(5m+1) = \tfrac{1}{16}(60a+56)(20a+16) = (15a+14)(5a+4),
\]
as desired.

\medskip

The cases $\rrr \in \{10,12\}$ reduce to \Cref{lem:numbercrunching}(1).

\medskip

The case $\rrr = 14$ is given by the parameters
\begin{align*}
  y = 2,\quad m' = 1,\quad \ifix =  17.
\end{align*}
There are $16$ factorisations of the bipartite Coxeter element that satisfy the properties in \Cref{prop:invariantsittingdetails}.
As positions of letters in $\invc$ (which lists the~$120$ reflections in the reflection ordering), these are given by the empty factorisation and
\[
  \left\{\begin{array}{c@{\hskip 3pt}c@{\hskip 3pt}c}
 (4, 20, 36, 52, 68, 84, 100), &
 (12, 28, 44, 60, 76, 92, 108), &
 (20, 36, 52, 68, 84, 100, 116), \\
 (28, 44, 60, 76, 92, 108, 4), &
 (36, 52, 68, 84, 100, 116, 12), &
 (44, 60, 76, 92, 108, 4, 20), \\
 (52, 68, 84, 100, 116, 12, 28), &
 (60, 76, 92, 108, 4, 20, 36), &
 (68, 84, 100, 116, 12, 28, 44), \\
 (76, 92, 108, 4, 20, 36, 52), &
 (84, 100, 116, 12, 28, 44, 60), &
 (92, 108, 4, 20, 36, 52, 68), \\
 (100, 116, 12, 28, 44, 60, 76), &
 (108, 4, 20, 36, 52, 68, 84), &
 (116, 12, 28, 44, 60, 76, 92)
 \end{array}
  \right\}
\]
Their ascent-descent sequences imply that the number of $(\Krewplustilde)^k$-invariant subwords as
\[
  2\binom{a+1}{1}+13\binom{a}{1}+\binom{a}{0} = 15a + 3.
\]
With $m' = 1$, we obtain $m = 7a+1$, which, by \Cref{eq:E8.4}, leads to
\[
  \tfrac{1}{7}(15m+6) = \tfrac{1}{7}(105a+21) = 15a+3,
\]
as desired.

\medskip

The case $\rrr = 16$ is given by the parameters
\[
  y = 2,\quad m' = 7,\quad \ifix =  105.
\]
In this case, there is only the trivial factorisation that satisfies the properties in \Cref{prop:invariantsittingdetails}, in agreement with \Cref{eq:E8.5}.

\end{document}